\documentclass[10pt,a4paper]{article}
\usepackage[utf8]{inputenc}
\usepackage[T1]{fontenc}
\usepackage{geometry}
\usepackage{amsthm}
\usepackage[french,english]{babel}
\usepackage{amssymb}
\usepackage{lmodern}
\usepackage{amsmath,amsfonts}
\usepackage{mathrsfs} 
\usepackage{dsfont}
\usepackage{stmaryrd}
\DeclareMathOperator{\dive}{div} 
\usepackage{graphicx}
\usepackage{fullpage}
\usepackage[nottoc,notlot,notlof]{tocbibind}
\usepackage[colorlinks=true,urlcolor=black,linkcolor=black,citecolor=black]{hyperref}
\numberwithin{equation}{section}
\usepackage{mathtools,array}
\usepackage{verbatimbox}
\usepackage{color}
\usepackage{bm}
\usepackage{tikz-cd}
\usepackage{pst-node}
\usetikzlibrary{matrix}
\usepackage{mathtools}
\usepackage{verbatimbox}
\usepackage{diagbox}
\usepackage{array}
\newcolumntype{C}{>{$\displaystyle} c <{$}}
\usepackage{makecell}
\usepackage{float}
\usepackage{tabu}
\usepackage{cancel}
\usepackage{lscape}
\usepackage[babel=true]{csquotes}
\makeatletter
\def\env@dmatrix{\hskip -\arraycolsep
	\let\@ifnextchar\new@ifnextchar
	\def\arraystretch{2}%
	\array{*{\c@MaxMatrixCols}{>{\displaystyle}c}}%
}

\makeatother
\usepackage{enumitem}
\setlist[itemize]{noitemsep, nolistsep}

\begin{document}

	\renewcommand{\thefootnote}{\fnsymbol{footnote}}
	
	\title{On the Morse Index of Branched Willmore Spheres in $3$-Space}
	\author{Alexis Michelat\footnote{Department of Mathematics, ETH Zentrum, CH-8093 Z\"{u}rich, Switzerland.}}
	\date{\today}
	
	\maketitle
	
	\vspace{-0.5em}
	
	\begin{abstract}
		We develop a general method to compute the Morse index of branched Willmore spheres and show that the Morse index is equal to the index of certain matrix whose dimension is equal to the number of ends of the dual minimal surface. As a corollary, we find that for all \emph{immersed} Willmore spheres $\vec{\Phi}:S^2\rightarrow \mathbb{R}^3$ such that $W(\vec{\Phi})=4\pi n$, we have $\mathrm{Ind}_{W}(\vec{\Phi})\leq n-1$. 
	\end{abstract}

	\tableofcontents
	\vspace{0cm}
	\begin{center}
	\end{center}
	
	\theoremstyle{plain}
	\newtheorem*{theorem*}{Theorem}
	\newtheorem{theorem}{Theorem}[section]
	\newenvironment{theorembis}[1]
	{\renewcommand{\thetheorem}{\ref{#1}$'$}%
		\addtocounter{theorem}{-1}%
		\begin{theorem}}
		{\end{theorem}}
	\renewcommand*{\thetheorem}{\Alph{theorem}}
	\newtheorem{lemme}[theorem]{Lemma}
	\newtheorem{propdef}[theorem]{Definition-Proposition}
	\newtheorem{prop}[theorem]{Proposition}
	\newtheorem{cor}[theorem]{Corollary}
	\theoremstyle{definition}
	\newtheorem*{definition}{Definition}
	\newtheorem{defi}[theorem]{Definition}
	\newtheorem{rem}[theorem]{Remark}
	\newtheorem*{conjecture}{Conjecture}
	\newtheorem{rems}[theorem]{Remarks}
	\newtheorem*{remast}{Remark}
	\newtheorem{exemple}[theorem]{Example}
	\renewcommand\hat[1]{%
		\savestack{\tmpbox}{\stretchto{%
				\scaleto{%
					\scalerel*[\widthof{\ensuremath{#1}}]{\kern-.6pt\bigwedge\kern-.6pt}%
					{\rule[-\textheight/2]{1ex}{\textheight}}
				}{\textheight}%
			}{0.5ex}}%
		\stackon[1pt]{#1}{\tmpbox}
	}
	\parskip 1ex
	\newcommand{\totimes}{\ensuremath{\,\dot{\otimes}\,}}
	\newcommand{\vc}[3]{\overset{#2}{\underset{#3}{#1}}}
	\newcommand{\conv}[1]{\ensuremath{\underset{#1}{\longrightarrow}}}
	\newcommand{\A}{\ensuremath{\vec{A}}}
	\newcommand{\B}{\ensuremath{\vec{B}}}
	\newcommand{\C}{\ensuremath{\mathbb{C}}}
	\newcommand{\D}{\ensuremath{\nabla}}
	\newcommand{\E}{\ensuremath{\vec{E}}}
	\newcommand{\I}{\ensuremath{\mathbb{I}}}
	\newcommand{\Q}{\ensuremath{\vec{Q}}}
	\newcommand{\z}{\ensuremath{\bar{z}}}
	\newcommand{\hh}{\ensuremath{\mathscr{H}}}
	\newcommand{\h}{\ensuremath{\vec{h}}}
	\newcommand{\vol}{\ensuremath{\mathrm{vol}}}
	\newcommand{\hs}[3]{\ensuremath{\left\Vert #1\right\Vert_{\mathrm{H}^{#2}(#3)}}}
	\newcommand{\R}{\ensuremath{\mathbb{R}}}
	\renewcommand{\P}{\ensuremath{\mathbb{P}}}
	\newcommand{\N}{\ensuremath{\mathbb{N}}}
	\newcommand{\Z}{\ensuremath{\mathbb{Z}}}
	\newcommand{\p}[1]{\ensuremath{\partial_{#1}}}
	\newcommand{\Res}{\ensuremath{\mathrm{Res}}}
	\newcommand{\lp}[2]{\ensuremath{\mathrm{L}^{#1}(#2)}}
	\renewcommand{\wp}[3]{\ensuremath{\left\Vert #1\right\Vert_{\mathrm{W}^{#2}(#3)}}}
	\newcommand{\np}[3]{\ensuremath{\left\Vert #1\right\Vert_{\mathrm{L}^{#2}(#3)}}}
	\newcommand{\nc}[3]{\ensuremath{\left\Vert #1\right\Vert_{C^{#2}(#3)}}}
	\renewcommand{\Re}{\ensuremath{\mathrm{Re}\,}}
	\renewcommand{\Im}{\ensuremath{\mathrm{Im}\,}}
	\newcommand{\diam}{\ensuremath{\mathrm{diam}\,}}
	\newcommand{\leb}{\ensuremath{\mathscr{L}}}
	\newcommand{\supp}{\ensuremath{\mathrm{supp}\,}}
	\renewcommand{\phi}{\ensuremath{\vec{\Phi}}}
	\renewcommand{\H}{\ensuremath{\vec{H}}}
	\renewcommand{\L}{\ensuremath{\mathscr{L}}}
	\renewcommand{\lg}{\ensuremath{\mathscr{L}_g}}
	\renewcommand{\ker}{\ensuremath{\mathrm{Ker}}}
	\renewcommand{\epsilon}{\ensuremath{\varepsilon}}
	\renewcommand{\bar}{\ensuremath{\overline}}
	\newcommand{\s}[2]{\ensuremath{\langle #1,#2\rangle}}
	\newcommand{\pwedge}[2]{\ensuremath{\,#1\wedge#2\,}}
	\newcommand{\bs}[2]{\ensuremath{\left\langle #1,#2\right\rangle}}
	\newcommand{\scal}[2]{\ensuremath{\langle #1,#2\rangle}}
	\newcommand{\sg}[2]{\ensuremath{\left\langle #1,#2\right\rangle_{\mkern-3mu g}}}
	\newcommand{\n}{\ensuremath{\vec{n}}}
	\newcommand{\ens}[1]{\ensuremath{\left\{ #1\right\}}}
	\newcommand{\lie}[2]{\ensuremath{\left[#1,#2\right]}}
	\newcommand{\g}{\ensuremath{g}}
	\newcommand{\e}{\ensuremath{\vec{e}}}
	\newcommand{\ig}{\ensuremath{|\vec{\mathbb{I}}_{\phi}|}}
	\newcommand{\ik}{\ensuremath{\left|\mathbb{I}_{\phi_k}\right|}}
	\newcommand{\w}{\ensuremath{\vec{w}}}
	\renewcommand{\v}{\ensuremath{\vec{v}}}
	\renewcommand{\tilde}{\ensuremath{\widetilde}}
	\newcommand{\vg}{\ensuremath{\mathrm{vol}_g}}
	\newcommand{\im}{\ensuremath{\mathrm{W}^{2,2}_{\iota}(\Sigma,N^n)}}
	\newcommand{\imm}{\ensuremath{\mathrm{W}^{2,2}_{\iota}(\Sigma,\R^3)}}
	\newcommand{\timm}[1]{\ensuremath{\mathrm{W}^{2,2}_{#1}(\Sigma,T\R^3)}}
	\newcommand{\tim}[1]{\ensuremath{\mathrm{W}^{2,2}_{#1}(\Sigma,TN^n)}}
	\renewcommand{\d}[1]{\ensuremath{\partial_{x_{#1}}}}
	\newcommand{\dg}{\ensuremath{\mathrm{div}_{g}}}
	\renewcommand{\Res}{\ensuremath{\mathrm{Res}}}
	\newcommand{\un}[2]{\ensuremath{\bigcup\limits_{#1}^{#2}}}
	\newcommand{\res}{\mathbin{\vrule height 1.6ex depth 0pt width
			0.13ex\vrule height 0.13ex depth 0pt width 1.3ex}}
	\newcommand{\ala}[5]{\ensuremath{e^{-6\lambda}\left(e^{2\lambda_{#1}}\alpha_{#2}^{#3}-\mu\alpha_{#2}^{#1}\right)\left\langle \nabla_{\vec{e}_{#4}}\vec{w},\vec{\mathbb{I}}_{#5}\right\rangle}}
	\setlength\boxtopsep{1pt}
	\setlength\boxbottomsep{1pt}
	\newcommand\norm[1]{%
		\setbox1\hbox{$#1$}%
		\setbox2\hbox{\addvbuffer{\usebox1}}%
		\stretchrel{\lvert}{\usebox2}\stretchrel*{\lvert}{\usebox2}%
	}
	\allowdisplaybreaks
	\newcommand*\mcup{\mathbin{\mathpalette\mcapinn\relax}}
	\newcommand*\mcapinn[2]{\vcenter{\hbox{$\mathsurround=0pt
				\ifx\displaystyle#1\textstyle\else#1\fi\bigcup$}}}
	\def\Xint#1{\mathchoice
		{\XXint\displaystyle\textstyle{#1}}%
		{\XXint\textstyle\scriptstyle{#1}}%
		{\XXint\scriptstyle\scriptscriptstyle{#1}}%
		{\XXint\scriptscriptstyle\scriptscriptstyle{#1}}%
		\!\int}
	\def\XXint#1#2#3{{\setbox0=\hbox{$#1{#2#3}{\int}$ }
			\vcenter{\hbox{$#2#3$ }}\kern-.58\wd0}}
	\def\ddashint{\Xint=}
	\newcommand{\dashint}[1]{\ensuremath{{\Xint-}_{\mkern-10mu #1}}}
	\newcommand\ccancel[1]{\renewcommand\CancelColor{\color{red}}\cancel{#1}}
	\newcommand\colorcancel[2]{\renewcommand\CancelColor{\color{#2}}\cancel{#1}}
	\renewcommand{\thetheorem}{\thesection.\arabic{theorem}}

	\section{Introduction}
	
	It was proposed by Tristan Rivi\`{e}re in \cite{eversion} to study the topology of immersions of surfaces into Euclidean space by means of a quasi-Morse function (say $\mathscr{L}$). Fix a closed surface $M^2$ and let $\mathrm{Imm}(M^2,\R^n)$ be the space of smooth immersions $\phi:M^2\rightarrow \R^n$. We look for a Lagrangian $\mathscr{L}:\mathrm{Imm}(M^2,\R^n)\rightarrow \R$ satisfying the following properties for all $\phi:M^2\rightarrow \R^n$: 
	\begin{itemize}[noitemsep]
		\itemsep-0.5em 
		\item[(1)] $\mathscr{L}(\phi+\vec{c})=\mathscr{L}(\phi)$ for all $\vec{c}\in \R^n$ (translation invariance)\\
		\item[(2)] $\mathscr{L}(R\phi)=\mathscr{L}(\phi)$ for all $R\in \mathrm{O}(n)$ (rotation invariance)\\
		\item[(3)] $\mathscr{L}(\lambda \phi)=\mathscr{L}(\phi)$ for all $\lambda>0$ (scaling invariance).
	\end{itemize}
	Indeed, an immersion does not change geometrically when one translates, rotates or dilates it.
	
    Now, assume that $n=3$. To an immersed surface one can attach two natural quantities: the principal curvatures $\kappa_1,\kappa_2$ (introduced by Euler in $1760$ \cite{euler}) which are the maximum and the minimum of the curvature of normal section of the surface at a given point. Then we define the mean curvature $H$ and Gauss curvature $K$ (introduced by Meusnier in $1776$ \cite{meusnier}) by
    \begin{align*}
    	H=\frac{\kappa_1+\kappa_2}{2},\quad \text{and}\;\, K=\kappa_1\kappa_2.
    \end{align*}
    Thanks to the third property, $\mathscr{L}$ must be a quadratic expression of the principal curvatures (see also \cite{mondinonguyen} for a more general study of conformal invariants of Euclidean space), which says that up to scaling
    \begin{align*}
    	\mathscr{L}(\phi)=\int_{M^2}\left(H^2+\lambda\,K\right)d\mathrm{vol}_{g}
    \end{align*}
    for some $\lambda\in \R$, 
    where $g=g_{\phi}=\phi^{\ast}g_{\R^3}$. Thanks to Gauss-Bonnet theorem,
    \begin{align*}
    	\int_{M^2}K\,d\vg=2\pi\chi(M^2)
    \end{align*}
    is a constant independent of the immersion. Therefore, up to constants, the only non-trivial such quasi-Morse function is
    \begin{align*}
    	\mathscr{L}(\phi)=\int_{\Sigma}H^2d\vg,
    \end{align*}
    which is generally denoted by $\mathscr{L}=W$ and is called the Willmore energy. 
    
    This Lagrangian actually first appeared in the work of Germain and Poisson in $1811$ and $1814$ respectively in their work about elasticity (\cite{germain}, \cite{poisson}). It was considered by many geometers in the following years, including in important work of Navier (\cite{navier}). For more information on the history in which these considerations of elasticity emerged, we refer to the comprehensive work of Todhunter (\cite{todhunter}). 
    Poisson was the first one to obtain the correct Euler-Lagrange equation, more than $100$ years before Blaschke and Thomsen, who attributed it to Schadow in $1922$ (\cite{thomsen}, \cite{blaschke}). He also found in $1814$ the first version of Gauss-Bonnet theorem, and his student Rodrigues computed the following year the exact constant $4\pi$ for ellipsoids, but unfortunately made a sign mistake and  found $8\pi$ for tori (\cite{rodrigues1}, \cite{rodrigues2}). The famous memoir of Gauss on the subject of the curvature of surfaces appeared only in $1827$  (\cite{gauss}), and Gauss-Bonnet in a published form in $1848$ (\cite{bonnet}).
    
    This Lagrangian only reappeared in $1965$ in Willmore's work who proposed the famous conjecture about minimisers of the Willmore energy for tori (\cite{willmore1}), which was finally proved in $2012$ by Marques-Neves (\cite{marqueswillmore}).
    
    In higher codimension, we can also define the Willmore energy as follows 
    \begin{align*}     
         W(\phi)=\int_{\Sigma}|\H|^2d\vg,
    \end{align*}
    where $\H$ is the mean curvature vector (the half-trace of the second fundamental form). It has the fundamental property of being invariant under conformal transformations (of ambient space). In particular, as minimal surfaces ($\H=0$) are absolute minimisers, inversions of complete minimal surfaces with finite total curvature are Willmore surfaces (though they may have branch points in general). Furthermore, Bryant showed that all immersions of the sphere in $\R^3$ are inversions of complete minimal surfaces with embedded planar ends (\cite{bryant}). 
    
    Now, a basic problem that we can address is to try to understand the following quantities : let $\gamma\in \pi_k(\mathrm{Imm}(M^2,\R^n))$ be a non-zero class (of regular homotopy of immersions) and let 
    \begin{align*}
    	\beta_{\gamma}=\inf_{\{\phi_t\}\simeq \gamma}\sup_{t\in S^k}W(\phi_t).
    \end{align*}
    Then one would like to understand if we can estimate these numbers and get some information on the critical immersions realising them (if this is possible to realise the \emph{width} of these min-max problems). 
    
    The first non-trivial number is given as follows : let $M^2=S^2$, $n=3$, and $\gamma\in \pi_1(\mathrm{Imm}(S^2,\R^n))\simeq \Z\times \Z_2$ be a non-trivial class (Smale, \cite{smale}). Then we define
    \begin{align}\label{beta0}
    	\beta_{\gamma}=\inf_{\{\phi_t\}\simeq \gamma}\sup_{t\in [0,1]}W(\phi_t).
     \end{align}
    By the work of Smale, the space of immersions from the round sphere $S^2$ in three-space $\R^3$ is path-connected ($\pi_0(\mathrm{Imm}(S^2,\R^3))=\ens{0}$), we have
    \begin{align*}
    	\beta_0=\inf_{\{\phi_t\}\in \Omega}\sup_{t\in [0,1]}W(\phi_t)
    \end{align*}
    where $\Omega$ is the set of path $\{\phi_t\}_{t\in [0,1]}\subset \mathrm{Imm}(S^2,\R^3)$ such that $\phi_0=\iota$ and $\phi_1=-\iota$, where $\iota:S^2\rightarrow \R^3$ is the standard embedding of the round sphere. These two min-max widths are equal since the Froissart-Morin eversion generates $\pi_1(\mathrm{Imm}(S^2,\R^3))$ (see \cite{eversion}). We will explain in the following what can be said about this problem in general and show a path to determine \eqref{beta0} and find which immersions may realise it. In relationship with these quantities, Kusner proposed the following conjecture.
    
    \begin{conjecture}[Kusner, $1980$'s \cite{kusner}]
    	We have $\beta_0=16\pi$, and an optimal path is given by a Willmore gradient flow starting from the inversion of Bryant's minimal surface with $4$ embedded ends.
    \end{conjecture}
    
    Thanks to Bryant's classification (\cite{classification}) and our extension to a large class of branched Willmore spheres (\cite{classification}, \cite{sagepaper}), it makes particularly sense to compute the index of inversions of complete minimal surfaces with finite total curvature in $\R^3$. Indeed, we have the following result.
    
    \begin{theorem}[Rivi\`{e}re \cite{eversion}, M. \cite{index2}, \cite{index3}]
    	There exists compact true branched Willmore spheres\\ $\phi_1,\cdots,\phi_p,\vec{\Psi}_1,\cdots,\vec{\Psi}_q:S^2\rightarrow \R^3$ such that 
    	\begin{align}\label{12}
    		\beta_0=\sum_{i=1}^{p}W(\phi_p)+\sum_{j=1}^{q}\left(W(\vec{\Psi}_j)-4\pi \theta_j
    		\right)
    	\end{align}
    	and 
    	\begin{align*}
    		\sum_{i=1}^{p}\mathrm{Ind}_W(\phi_i)+\sum_{j=1}^{q}\mathrm{Ind}_W(\vec{\Psi}_j)\leq 1,
    	\end{align*}
    	where $\theta_j=\theta_0(\vec{\Psi}_j,p_j)\in \N$ is the multiplicity of $\vec{\Psi}_j$ at some point $p_j\in \vec{\Psi}_j(S^2)$.
    \end{theorem}

    Here, recall that a Willmore surface $\phi:\Sigma\rightarrow \R^n$ has no first residue if for all path $\gamma$ around a branch point $p$ of $\phi$ (which does not contain or intersect other branch points)
    \begin{align*}
    	\vec{\gamma}_0(\phi,p)=\frac{1}{4\pi}\,\Im\int_{\gamma}\left(\partial \H+|\H|^2\partial\phi+2\,g^{-1}\otimes\s{\H}{\h_0}\otimes \bar{\partial}\phi\right)=0.
    \end{align*}
    We refer to \cite{riviere1}, \cite{quanta} and \cite{classification} for more information on this quantity. 
    
    This theorem shows that the previous conjecture should be interpreted as follows.
    
    \begin{conjecture}
    	Let $\phi_1,\cdots,\phi_p,\vec{\Psi}_1,\cdots,\vec{\Psi}_q$ be given by \eqref{12}. Then $p=1$, $q=0$ and $\phi_1$ is the inversion of Bryant's minimal surface with $4$ embedded planar ends. 
    \end{conjecture}

	\section{Main results}
	
	\renewcommand*{\thetheorem}{\Alph{theorem}}
	
	If $\vec{\Psi}:\Sigma\rightarrow \R^3$ is a branched Willmore sphere, we write for all normal admissible variations $\vec{v}=v\n\in \mathscr{E}_{\vec{\Psi}}(\Sigma,\R^3)$ (see Section \ref{admissible} for a precise definition)
	\begin{align*}
		Q_{\vec{\Psi}}(v)=D^2W(\vec{\Psi})(\vec{v},\vec{v})
	\end{align*}
	the quadratic form of the second derivative of the Willmore energy $W$ at $\vec{\Psi}$. Then we define the Willmore Morse index as the maximum dimension of sub-vector spaces of $\mathscr{E}_{\vec{\Psi}}(\Sigma,\R^3)$ on which $Q_{\vec{\Psi}}$ is negative definite.
	
	\begin{theorem}\label{ta}
		Let $\Sigma$ be a closed Riemann surface and let $\vec{\Psi}:\Sigma\rightarrow \R^3$ be a branched Willmore surface, $g={\vec{\Psi}}^{\ast}g_{\R^3}$ be the induced metric on $\Sigma$ and assume that $\vec{\Psi}$ is the inversion of a complete minimal surface $\phi:\Sigma\setminus\ens{p_1,\cdots,p_n}\rightarrow \R^3$ with embedded ends and fix some smooth metric $g_0$ on $\Sigma$. Assume that $0\leq m\leq n$ is fixed such that $p_1\,\cdots,p_m$ are catenoid ends, while $p_{m+1},\cdots,p_n$ are planar ends, and for all $1\leq j\leq m$, let $\beta_j=\mathrm{Flux}(\phi,p_j)\in \R^{\ast}$ be the flux of $\phi$ at $p_j$. 
		There exists a symmetric matrix $\Lambda(\vec{\Psi})\in \mathrm{Sym}_n(\R)$ defined by
		\begin{align*}
		\Lambda(\vec{\Psi})=\begin{pmatrix}
		\tilde{\beta}_1^2&\lambda_{1,2} &\cdots &\cdots & \cdots &\cdots &\lambda_{1,n}\\
		\lambda_{1,2} & \tilde{\beta}_2^2& \cdots &\cdots &\cdots &\cdots &\lambda_{2,n}\\
		\vdots& \ddots & \ddots &\ddots &\ddots &\ddots &\vdots \\
		\lambda_{1,m} &\cdots &\cdots & \tilde{\beta}_m^2 &\cdots&\cdots &\lambda_{m,n}\\
		\lambda_{1,m+1}& \cdots &\cdots &\cdots &0 &\cdots &\lambda_{m+1,n}\\
		\vdots& \ddots& \ddots &\ddots &\ddots & \ddots &\vdots\\
		\lambda_{1,n} &\lambda_{2,n} &\cdots &\cdots  &\cdots & \cdots & 0
		\end{pmatrix},\quad \tilde{\beta}_j^2=\frac{4}{2n+1}\beta_j^2
		\end{align*}
		with the following property. For all $a=(a_1,\cdots,a_n)\in \R^n$, there exists an admissible variation 
		\begin{align*}
		v\in W^{2,2}(\Sigma,d\mathrm{vol}_{g_0})\cap\ens{v:|dv|_g\in L^{\infty}(\Sigma,d\mathrm{vol}_{g_0})\;\, \text{and}\;\, \Delta_gv\in L^2(\Sigma,d\vg)},
		\end{align*}
		such that $(v(p_1),\cdots,v(p_n))=(a_1,\cdots,a_n)$ and
		\begin{align}\label{remq0}
		Q_{\vec{\Psi}}(v)=16\pi\sum_{j=1}^{n}\beta_j^2v^2(p_j)+4\pi(2n+1)\sum_{1\leq i,j\leq n}^{}\lambda_{i,j}v(p_i)v(p_j).
		\end{align}
		Therefore, we have
		\begin{align*}         
		    \mathrm{Ind}_{W}(\vec{\Psi})=\mathrm{Ind}\, \Lambda(\vec{\Psi})\leq n.
		\end{align*}
		Furthermore, if $\vec{\Psi}$ is a smooth immersion then $m=0$ and we have
		\begin{align*}
			\mathrm{Ind}_{W}(\vec{\Psi})=\mathrm{Ind}\,\Lambda(\vec{\Psi})\leq n-1,
		\end{align*}
 		where $\mathrm{Ind}\,\Lambda(\vec{\Psi})$ is the number of negative eigenvalues of $\Lambda(\vec{\Psi})$.
	\end{theorem}


    \begin{remast}
    	This theorem was first presented in detail on November $13$, $2018$ at the Institute for Advanced Study in the seminar \emph{Variational Methods in Geometry Seminar}
    	
    	\url{https://www.math.ias.edu/seminars/abstract?event=138881}
    	
    	The video was uploaded and is freely available on the internet since then at the following links:
    	
    	\url{https://video.ias.edu/varimethodsgeo/2018/1113-AlexisMichelat}
    	
    	\url{https://www.youtube.com/watch?v=lAYcy22OIec}
    	
    	The interested reader will find at $1$:$05$ the main theorem, at $1$:$31$ and $1$:$35$ the special negative variations with logarithm behaviour at the ends and at $1$:$39$ the additional term coming out for variations including a logarithm term.
    \end{remast}

    \begin{remast}
    	There are examples of complete minimal surfaces of genus $1$ with planar ends discovered by Costa and Shamaev (\cite{costa2}, \cite{shamaev}). Kusner and Schmitt also studied the moduli space of such minimal surfaces in detail (see \cite{kusnerschmitt}), and showed in particular that there are no examples with three planar ends (this is the first non-trivial case thanks to Schoen's theorem on the characterisation of the catenoid as the only complete minimal surface with $2$ embedded ends \cite{schoenPlanar}). They all have an even number of ends (at least $4$). In fact,  all values of ends $2n\geq 4$ are attained. 
    \end{remast}

    \begin{cor}\label{tb}
    	Let $\phi:S^2\rightarrow \R^3$ be a Willmore immersion. Then
    	\begin{align*}
    		\mathrm{Ind}_{W}(\phi)\leq \frac{1}{4\pi}W(\phi)-1.
    	\end{align*}
    \end{cor}
    
    In general, we can obtain a general bound which generalised \cite{indexS3} to the case of branched Willmore surfaces.
    
    \begin{theorem}\label{tc}
    	Let $\vec{\Psi}:\Sigma\rightarrow \R^3$ be a branched Willmore surface and assume that $\vec{\Psi}$ is the inversion of a complete minimal surface with finite total curvature $\phi:\Sigma\setminus\ens{p_1,\cdots,p_n}\rightarrow \R^3$. Then there a universal symmetric matrix $\Lambda=\Lambda(\vec{\Psi})=\ens{\lambda_{i,j}}_{1\leq i,j\leq n}$  such that for all smooth $v\in W^{2,2}(\Sigma)$ and admissible normal variation $\vec{v}=v\n_{\vec{\Psi}}\in \mathscr{E}_{\phi}(\Sigma,\R^3)$
    	\begin{align*}
    	Q_{\vec{\Psi}}(v)=Q_{\vec{\Psi}}(v_0)+4\pi\sum_{1\leq i,j\leq n}^{}\lambda_{i,j}v(p_i)v(p_j),
    	\end{align*}
    	for some $v_0\in W^{2,2}(\Sigma)$ such that $v_0(p_i)=0$ for all $1\leq i\leq n$.
    	In particular, we have
    	\begin{align}\label{remq1}
    	\mathrm{Ind}_W(\vec{\Psi})\leq n=\frac{1}{4\pi}W(\vec{\Psi})-\frac{1}{2\pi}\int_{\Sigma}K_{g}d\mathrm{vol}_{g}+\chi(\Sigma).
    	\end{align}
    \end{theorem}
    
    \begin{remast}
    For \emph{true} branched immersions with ends of multiplicity at most  $2$, we have
    \begin{align*}
    	\mathrm{Ind}_W(\vec{\Psi})\leq n-1=\frac{1}{4\pi}W(\vec{\Psi})-\frac{1}{2\pi}\int_{\Sigma}K_{g}d\mathrm{vol}_{g}+\chi(\Sigma)-1.
    \end{align*}    
    by showing that $\lambda_{i,i}=0$ for ends of multiplicity $2$ in \eqref{remq1}. See Section \ref{multiplicity2} for the proof (Theorem \ref{mult2}).  
    \end{remast}
    
    We can generalise Theorem \ref{ta} to the branched case at the price of getting a possibly weaker bound.
    
    \begin{theorem}\label{td}
    Let $\phi:\Sigma\setminus\ens{p_1,\cdots,p_n}\rightarrow \R^3$ be a complete minimal surface with finite total curvature, and $\vec{\Psi}=\iota\circ \phi:\Sigma\rightarrow \R^3$ be a compact inversion of $\phi$, and let $\Lambda(\vec{\Psi})=\ens{\lambda_{i,j}}_{1\leq i,j\leq n}\in \mathrm{Sym}_n(\R)$ be the matrix given by Theorem 
    \ref{tc}. Then there exists a matrix $\{\tilde{\lambda}_{i,j}\}_{1\leq i,j\leq n}\in \mathrm{Sym}_n(\R)$ such that $\tilde{\lambda}_{i,i}=0$ for all $1\leq i\leq n$ with the following property. Define
    \begin{align*}
    \tilde{\Lambda}(\vec{\Psi})=\begin{pmatrix}
    \vspace{0.5em}\lambda_{1,1}&\lambda_{1,2}+2n\,\tilde{\lambda}_{1,2} &  \cdots  &\lambda_{1,n}+2n\,\tilde{\lambda}_{1,n}\\
    \vspace{0.5em}\lambda_{1,2}+2n\,\tilde{\lambda}_{1,2} & \lambda_{2,2}& \cdots  &\lambda_{2,n}+2n\,\tilde{\lambda}_{2,n}\\
    \vspace{0.5em}\vdots& \vdots & \ddots &\vdots \\
    \vspace{0.5em}\lambda_{1,n}+2n\,\tilde{\lambda}_{1,n} &\lambda_{2,n}+2n\,\tilde{\lambda}_{1,n} &  \cdots & \lambda_{n,n}.
    \end{pmatrix}.
    \end{align*}
    Then for all $\vec{a}=(a_1,\cdots,a_n)\in \R^n$, there exists an admissible variation $v\in W^{2,2}(\Sigma)$ (such that $\vec{v}=v\n_{\vec{\Psi}}\in \mathscr{E}_{\vec{\Psi}}(\Sigma,\R^3)$) such that $(v(p_1),\cdots,v(p_n))=(a_1,\cdots,a_n)$ and
    \begin{align*}
    	Q_{\vec{\Psi}}(v)=4\pi\sum_{i=1}^{n}\lambda_{i,i}v^2(p_i)+4\pi\sum_{\substack{1\leq i,j\leq n\\ i\neq j}}^{}\left(\lambda_{i,j}+2n\,\tilde{\lambda}_{i,j}\right)v(p_i)v(p_j).
    \end{align*}
    Therefore, we have
    \begin{align*}
    \mathrm{Ind}\,\tilde{\Lambda}(\vec{\Psi})\leq \mathrm{Ind}_{W}(\vec{\Psi})\leq \mathrm{Ind}\,\Lambda(\vec{\Psi}).
    \end{align*}
    \end{theorem}

\begin{remast}
    It seems very likely that $\tilde{\lambda}_{i,j}=\lambda_{i,j}$ (for $1\leq i\neq j\leq n$), but our current methods do not allow us to check if this fact holds or not in general (it holds for embedded ends by Theorem \ref{ta}).
\end{remast}

\textbf{Acknowledgement.} I would like to thank my advisor Tristan Rivière for his constant support and for suggesting the analogy with the renormalised energy appearing in the Ginzburg-Landau model from super-conductivity  (see for example \cite{bethuelbrezishelein2}, \cite{pacariv} and \cite{serfaty}).

\textbf{Added in proof.} Recently Jonas Hirsch and Elena M\"{a}der-Baumdicker wrote a paper on this subject in the special case of minimal surfaces with flat ends (\cite{hirsch}). 

\renewcommand{\thetheorem}{\thesection.\arabic{theorem}}

\section{The second derivative of the Willmore energy as a renormalised energy}\label{admissible}

Let $\Sigma$ be a closed Riemann surface, $n\in \N$, $p_1,\cdots,p_n\in \Sigma$ be fixed distinct points and $\phi:\Sigma\setminus\ens{p_1,\cdots,p_n}\rightarrow \R^3$ be a complete minimal surface with finite total curvature and assume without loss of generality that $0\notin \phi(\Sigma\setminus\ens{p_1,\cdots,p_n})\subset \R^3$. Then the inversion 
\begin{align*}
	\vec{\Psi}=\frac{\phi}{|\phi|^2}:\Sigma\rightarrow \R^3
\end{align*}
is a compact branched Willmore surface. Now, recall that we defined in \cite{index3} a notion of admissible variations of the Willmore energy as the maximum set of variations for which the second derivative of the Willmore energy is well-defined.  

\begin{theorem}[\cite{index3}]
	Let $\Sigma$ be a closed Riemann surface and let $\vec{\Psi}:\Sigma\rightarrow \R^d$ be a branched Willmore immersion and let $g=\vec{\Psi}^{\ast}g_{\R^d}$ be the induced metric. Then the second derivative $D^2W(\vec{\Psi})$ is well-defined at some point
	\begin{align*}
		\w=\mathscr{E}_{\vec{\Psi}}(\Sigma,\R^n)=W^{2,2}\cap W^{1,\infty}(\R^n)\cap\ens{\w:\w(p)\in T_{\vec{\Psi}(p)}\R^n\;\, \text{for all}\;\, p\in \Sigma}
	\end{align*}
	if and only if
	\begin{align*}
		\w\in L^{\infty}(\Sigma,g_0)\qquad \text{and}\;\, \mathscr{L}_g\w\in L^2(\Sigma,d\vg),
	\end{align*}
	where $g_0$ is any fixed smooth metric on $\Sigma$ and $\mathscr{L}_g\w=\Delta_{g}^{\n}\w+\mathscr{A}(\w)$ is the Jacobi operator and $\mathscr{A}$ is the Simons operator. We denote by $\mathrm{Var}(\vec{\Psi})$ this space of admissible variations.
\end{theorem}

Notice that at a branch point of multiplicity $\theta_0\geq 1$, the condition are equivalent to
\begin{align*}
	\frac{|d\w|}{|z|^{\theta_0-1}}\in L^{\infty}(D^2),\qquad \text{and}\;\, \frac{\Delta_g^{\n}\w}{|z|^{\theta_0-1}}\in L^2(D^2).
\end{align*}
In particular, if $\w$ is a smooth variation, the conditions are equivalent to
\begin{align*}
	\w=\w(0)+\Re\left(\vec{\gamma}z^{\theta_0}\right)+O(|z|^{\theta_0+1}).
\end{align*}

We can now define the Willmore Morse index as follows (see \cite{index3}).
\begin{defi}
	Let $\Sigma$ be a closed Riemann surface and let $\phi:\Sigma\rightarrow \R^n$ be a branched Willmore immersion. Then Willmore index of $\phi$, denoted by $\mathrm{Ind}_{W}(\phi)$, is equal to the dimension of the maximal sub-vector space $V\subset \mathscr{E}_{\vec{\Psi}}(\Sigma,\R^n)$ on which the quadratic form second variation $Q_{\vec{\Psi}}(\,\cdot\,)=D^2W(\phi)(\,\cdot\,,\,\cdot\,)$ is negative definite. 
\end{defi}

Now, thanks to Proposition $4.5$ of \cite{indexS3}, for all $\vec{v}=v\n_{\vec{\Psi}}\in \mathrm{Var}(\vec{\Psi})$, we have
\begin{align*}
	Q_{\vec{\Psi}}(v)=D^2W(\vec{\Psi})(\vec{v},\vec{v})=\int_{\Sigma}\bigg\{\frac{1}{2}\left(\Delta_gu-2K_gu\right)^2d\vg-d\left(\left(\Delta_gu+2K_gu\right)\ast du-\frac{1}{2}\ast d|du|_g^2\right)\bigg\}.
\end{align*}
where $u=|\phi|^2v$. In particular, thanks to Stokes theorem, we have
\begin{align}\label{limit}
	Q_{\vec{\Psi}}(v)=\lim\limits_{\epsilon\rightarrow 0}\left(\frac{1}{2}\int_{\Sigma_{\epsilon}}\left(\Delta_gu-2K_gu\right)^2d\vg+\sum_{i=1}^{n}\int_{\partial B_{\epsilon}(p_i)}\left(\left(\Delta_gu+2K_gu\right)\ast du-\frac{1}{2}\ast d|du|_g^2\right)\right).
\end{align}
where
\begin{align*}
\Sigma_{\epsilon}=\Sigma\setminus\bigcup_{i=1}^n\bar{B}_{\epsilon}(p_i).
\end{align*}
In particular, the limit \eqref{limit} exists for all such $\vec{v}\in \mathrm{Var}(\vec{\Psi})$. Here, the balls $\bar{B}_{\epsilon}(p_i)$ are fixed following the following definition for some covering $(U_1,\cdots,U_n)$ of $\ens{p_1,\cdots,p_n}$ fixed once and for all.

\begin{defi}
	We say that a family of chart domains $(U_1,\cdots,U_n)$ is a covering of $\ens{p_1,\cdots,p_n}\subset \Sigma$ if $p_i\in U_i$ for all $1\leq i\leq n$ and $U_i\cap U_j=\varnothing$ for all $1\leq i\leq n$. For all $1\leq i\leq n$ if $\varphi_i:U_i\rightarrow B_{\C}(0,1)\subset \C$ is a complex chart such that $\varphi_i(p_i)=0$ and $\varphi_i(U_i)=B_{\C}(0,1)$, we define for all $0<\epsilon<1$
	\begin{align*}
		B_{\epsilon}(p_i)=\varphi_i^{-1}(B_{\C}(0,\epsilon)).
	\end{align*}
	This definition is independent of the chart $\varphi_i:U_i\rightarrow B_{\C}(0,1)\subset \C$ such that $\varphi_i(U_i)=B_{\C}(0,1)$ and $\varphi_i(p_i)=0$.
\end{defi}

The independence of the chart $\varphi_i$ with the above properties is a trivial consequence of Schwarz lemma (se \cite{indexS3} for more details).

\section{Decomposition of the renormalised energy}

We fix a Willmore surface $\vec{\Psi}:\Sigma\rightarrow \R^3$ which is the inversion of a complete minimal surface $\phi:\Sigma\setminus\ens{p_1,\cdots,p_n}$ with finite total curvature.
We fix $v\in W^{2,2}(\Sigma)$ (such that $\vec{v}=v\n_{\vec{\Psi}}\in \mathrm{Var}(\vec{\Psi})$), and as in the introduction, for all $\epsilon>0$ small enough, we consider the following minimisation problem
\begin{align}\label{E1}
\inf_{w\in\mathscr{E}_{\epsilon}(p_i)} \frac{1}{2}\int_{\Sigma\setminus\bar{B}_\epsilon(p_i)}(\Delta_g w-2K_gw)^2d\vg
\end{align}
where the class of admissible functions is
\begin{align*}
\mathscr{E}_\epsilon(p_i)=\mathrm{W}^{2,2}(\Sigma\setminus \bar{B}_\epsilon(p_i))\cap \ens{w: 
	\left\{\begin{alignedat}{2}
	w&=u&&\text{on}\;\, \partial B_\epsilon(p_i)\\
	\partial_\nu w&=\partial_\nu u \qquad  &&\text{on}\;\, \partial B_\epsilon(p_i)\\
	\end{alignedat}\right.}
\end{align*}
Notice that for an end $p_j$ (for some $1\leq j\leq n$) of multiplicity $m\geq 1$ of a complete minimal surface with finite total curvature $\phi:\Sigma\rightarrow \ens{p_1,\cdots,p_n}\rightarrow \R^n$, in any complex chart $z:B(0,1)\subset \C\rightarrow \Sigma$ such that $z(0)=p_j$, there exists $\vec{A}_0\in \C^n\setminus\ens{0}$ (depending on $z$) such that 
\begin{align*}
	\phi(z)=\Re\left(\frac{\vec{A}_0}{z^m}\right)+O(|z|^{1-m})
\end{align*}
for $m\geq 2$, while for $m=1$ there exists $\vec{\gamma}_0\in \R^n$ such that
\begin{align*}
	\phi(z)=\Re\left(\frac{\vec{A}_0}{z}\right)+\vec{\gamma}_0\log|z|+O(1).
\end{align*}
Therefore, we have up to scaling 
\begin{align*}
	e^{2\lambda}=2|\p{z}\phi|^2=\frac{1}{|z|^{2(m+1)}}\left(1+O(|z|)\right).
\end{align*}
In particular, we deduce that 
\begin{align*}
	K_g=-\Delta_g\lambda=O(|z|^{2(m+1)}),
\end{align*}
and 
\begin{align*}
	\mathscr{L}_g=\Delta_g-2K_g=e^{-2\lambda}\left(\Delta+2\Delta\lambda\right)=|z|^{2(m+1)}\left(1+O(|z|)\right)\left(\Delta+O(1)\right),
\end{align*}
so $\mathscr{L}_g$ is not elliptic in a neighbourhood of $p_j$. Therefore, we will have to consider another problem than \eqref{E1}.

Recall first by definition of $B_{\epsilon}(p_i)$ that $B_{1}(p_i)\cap B_{1}(p_j)=\varnothing$ for all $1\leq i\neq j\leq n$. Therefore, for all $0<\epsilon<1$, and for all $0<\delta<\epsilon$, and $1\leq i\leq n$ consider the domain
\begin{align*}
	\Sigma_{\epsilon,\delta}^i=\Sigma\setminus \left(\bar{B}_{\epsilon}(p_i)\cup \bigcup_{j\neq i}\bar{B}_{\delta}(p_i)\right).
\end{align*}
We will also write
\begin{align*}
	\Sigma_{\epsilon}^i=\bigcup_{\delta>0}\Sigma_{\epsilon,\delta}^i=\Sigma\setminus\left(\bar{B}_{\epsilon}(p_i)\cup\ens{p_1,\cdots,p_n}\right).
\end{align*}
Then $\mathscr{L}_g$ and $\mathscr{L}_g^2$ are strongly elliptic on $\Sigma_{\epsilon,\delta}^i$ and have the uniqueness for the Cauchy problem \textit{i.e.} if $\mathscr{L}_gu=0$ (resp. $\mathscr{L}_g^2u=0$) and $u=0$ on some open $U\subset \Sigma_{\epsilon,\delta}^i$, then $u=0$  (this fact was first proved in general by J. Simons \cite{simons}), thanks to a classical theorem of Smale (see \cite{morsmale} and \cite{choe}) there exists $0<\epsilon_0$ such that for all $0<\epsilon<\epsilon_0$, there exists $0<\delta(\epsilon)<\epsilon$ such that for all $0<\delta<\delta(\epsilon)$, the operators $\mathscr{L}_g$ and $\mathscr{L}_g^2$ have no kernel on $\Sigma_{\epsilon,\delta}^i$ for all $1\leq i\leq n$.  
More precisely, the only solution of each of the two following problems
\begin{align}\label{1}
\left\{\begin{alignedat}{2}
\mathscr{L}_gu&=0\quad &&\text{in}\;\,\Sigma_{\epsilon,\delta}^i\\
u&=0\quad &&\text{on}\;\, \partial \Sigma_{\epsilon,\delta}^i,
\end{alignedat}\right.
\end{align}
and
\begin{align}\label{2}
\left\{\begin{alignedat}{2}
\mathscr{L}_g^2u&=0\quad &&\text{in}\;\,\Sigma_{\epsilon,\delta}^i\\
u&=0\quad &&\text{on}\;\, \partial \Sigma_{\epsilon,\delta}^i\\
\partial_{\nu}u&=0 &&\text{on}\;\, \partial \Sigma_{\epsilon,\delta}^i
\end{alignedat}\right.
\end{align}
is the trivial solution $u=0$.   
Therefore, thanks to the Fredholm alternative (see \cite{brezis}, IX.$23$) for all $1\leq i\leq n$ and all but finitely many $0<\epsilon<\epsilon_0$ there exists a unique minimiser $u_{\epsilon,\delta}^i$ of \eqref{E1} such that
\begin{equation}\label{EQepsilon}
\left\{\begin{alignedat}{2}
\mathscr{L}_g^2u_{\epsilon,\delta}^i&=0  && \text{on}\;\,\Sigma_{\epsilon,\delta}^i \\
u_{\epsilon,\delta}^i&=u&&\text{on}\;\, \partial B_{\epsilon}(p_i)\\
\partial_\nu u_{\epsilon,\delta}^i&=\partial_\nu u \quad  &&\text{on}\;\, \partial B_{\epsilon}(p_i)\\
u_{\epsilon,\delta}^i&=0\qquad &&\text{on}\;\, \partial B_{\delta}(p_j)\;\,\text{for all}\;\, 1\leq j\neq i\leq n\\
\partial_{\nu}u_{\epsilon,\delta}^i&=0\qquad &&\text{on}\;\, \partial B_{\delta}(p_j)\;\,\text{for all}\;\, 1\leq j\neq i\leq n
\end{alignedat}\right.
\end{equation}
where $u=|\phi|^2v$ and $\mathscr{L}_g=\Delta_g-2K_g$ is the Jacobi operator of the minimal surface $\phi:\Sigma\setminus\ens{p_1,\cdots,p_n}\rightarrow\R^3$. In particular, we fix $0<\epsilon<\epsilon_0$ and we assume $0<\delta<\delta_0(\epsilon)<\epsilon$. Furthermore, notice that $u_{\epsilon,\delta}^i$ is the unique solution to the variational problem
\begin{align}\label{E2}
	\inf_{w\in \mathscr{E}_{\epsilon,\delta}(p_i)}\frac{1}{2}\int_{\Sigma_{\epsilon,\delta}^i}\left(\Delta_gw-2K_gw\right)^2d\vg
\end{align}
where 
\begin{align*}
	\mathscr{E}_{\epsilon,\delta}(p_i)=W^{2,2}(\Sigma_{\epsilon,\delta}^i)\cap\ens{w:\left\{\begin{alignedat}{2}
		w&=u\qquad &&\text{on}\;\, \partial B_{\epsilon}(p_i)\\
		\partial_{\nu}w&=\partial_{\nu}u\qquad &&\text{on}\;\, \partial B_{\epsilon}(p_i)\\
		w&=0\qquad &&\text{on}\;\, \partial B_{\epsilon}(p_j)\;\, \text{for all}\;\, 1\leq j\neq i\leq n\\
		\partial_{\nu}w&=0\qquad &&\text{on}\;\, \partial B_{\epsilon}(p_j)\;\, \text{for all}\;\, 1\leq j\neq i\leq n.
		\end{alignedat}\right.}
\end{align*}

\subsection{Estimate of the singular energy of the minimisers}

Recall the definition
\begin{align*}
	Q_{\vec{\Psi}}(v)=D^2W(\vec{\Psi})(v\n_{\vec{\Psi}},v\n_{\vec{\Psi}}),
\end{align*}
for some admissible variation $\vec{v}=v\n_{\vec{\Psi}}\in \mathscr{E}_{\vec{\Psi}}(\Sigma,\R^3)$. Fix some $1\leq i\leq n$ and assume that $p_i$ has multiplicity $m_i\geq 1$. Then there exists $\alpha_i>0$ such that 
\begin{align}\label{expansion0}
	&|\phi|^2=\frac{\alpha_i^2}{|z|^{2m_i}}\left(1+O(|z|)\right)\nonumber\\
	&g=\frac{m_i^2\alpha_i^2}{|z|^{2(m_i+1)}}\left(1+O(|z|)\right)\nonumber\\
	&K_g=O(|z|^{2m_i+2}).
\end{align}
For all $v\in C^{\infty}(\Sigma)$, define $u=|\phi|^2v$. Then there exists a function $f_{\epsilon}:\R^{(m_i+1)(2m_i+1)}\rightarrow \R$, such that 
\begin{align*}
	-\int_{\partial B_{\epsilon}(p_i)}\left(\left(\Delta_gu+2K_gu\right)\ast du-\frac{1}{2}\ast\,d|du|_g|^2\right)=f_{\epsilon}(J_{p_i}^{2m_i}v),
\end{align*}
where $J_{p_i}^{2m_i}$ is the jet of $v$ of order $2m_i$ at $p_i\in \Sigma$. However, thanks to the Sobolev embedding $W^{2,2}(\Sigma)\hookrightarrow C^0(\Sigma)$ and the absence of Sobolev embedding $W^{2,2}(\Sigma)\not\hookrightarrow C^1(\Sigma)$, $f_{\epsilon}$ can only depend on $v(p_i)$. Furthermore, as $Q_{\vec{\Psi}}(v)$ is quadratic in $v$, we deduce that there exists $Q_{\epsilon}^i\in \R$ such that 
\begin{align*}
    f_{\epsilon}(J_{p_i}^{2m_i}v)=Q_{\epsilon}^iv^2(p_i).
\end{align*}
See \cite{indexS3} for an explicit argument and Section \ref{renormalised} for the explicit singular energy associated to minimal surfaces with embedded ends. As $\vec{v}$ is admissible and $v$ is smooth, there exists $\gamma\in \R$ such that 
\begin{align*}
	v=v(p_i)+\Re(\gamma z^{m_i})+O(|z|^{m_i+1}).
\end{align*}
The previous expansion \eqref{expansion0} shows that 
\begin{align*}
	u=|\phi|^2u=|\phi|^2v(p_i)+\Re\left(\frac{\alpha_i^2\bar{\gamma}}{z^{m_i}}\right)+O(|z|^{1-m_i}).
\end{align*}
Therefore, as $\phi$ is conformal and harmonic, we have $\Delta_g|\phi|^2=4$ and
\begin{align*}
	\mathscr{L}_g u=\left(4-2K_g|\phi|^2\right)v(p_i)+O(|z|^{m_i+1}).
\end{align*}
This expansion implies that 
\begin{align*}
	\frac{1}{2}\int_{B_{\epsilon_0}\setminus \bar{B}_{\epsilon}(p_i)}\left(\lg u\right)^2d\vg&=\frac{1}{2}\int_{B_{\epsilon_0}\setminus \bar{B}_{\epsilon}(0)}\left(\left(4-2K_g|\phi|^2\right)v(p_i)+O(|z|^{m_i+1})\right)^2\frac{m_i^2\alpha_i^2(1+O(|z|))}{|z|^{2m_i+2}}|dz|^2\\
	&=\pi \int_{\epsilon}^{\epsilon_0}\left(\left(4-2K_g|\phi|^2\right)^2v^2(p_i)+8v(p_i)O(r^{m_i+1})+O(r^{2(m_i+1)})\right)\frac{m_i^2\alpha_i^2dr}{r^{2m_i+1}}\\
	&=8\pi m_i\alpha_i^2\frac{v^2(p_i)}{\epsilon^{2m_i}}+O\left(\frac{1}{\epsilon^{2(m_i-1)}}\right)\conv{\epsilon\rightarrow 0}\infty.
\end{align*}
Notice that $v(p_i)=0$ implies that $\mathscr{L}_g u=O(|z|^{m_i+1})$ and 
\begin{align*}
	\frac{1}{2}\int_{B_{\epsilon_0}\setminus \bar{B}_{\epsilon}(p_i)}\left(\lg u\right)^2d\vg=O(\epsilon_0)
\end{align*}
which shows that the limit \eqref{limit} reduces to
\begin{align*}
	Q_{\vec{\Psi}}(v)=\frac{1}{2}\int_{\Sigma}\left(\lg u\right)^2d\vg<\infty.
\end{align*}
Therefore, for all $v\in C^{\infty}(\Sigma)$ such that $\vec{v}=v\n_{\vec{\Psi}}$ be admissible, we have
\begin{align*}
	Q_{\vec{\Psi}}(v)\geq 0.
\end{align*}
Therefore, we deduce the first extension of \cite{indexS3} to the case of branched surfaces.
\begin{theorem}
	Let $\Sigma$ be a closed Riemann surface and $\vec{\Psi}:\Sigma\rightarrow \R^3$ be a branched Willmore surface. Assume that $\vec{\Psi}$ is the inversion of a complete minimal surface with finite total curvature $\phi:\Sigma\setminus\ens{p_1,\cdots,p_n}\rightarrow \R^3$. Then we have
	\begin{align}\label{upends}
		\mathrm{Ind}_{W}(\vec{\Psi})\leq n=\frac{1}{4\pi}W(\vec{\Psi})-\frac{1}{2\pi}\int_{\Sigma}K_{g_{\vec{\Psi}}}\,d\mathrm{vol}_{g_{\vec{\Psi}}}+\chi(\Sigma).
	\end{align}
\end{theorem}
\begin{proof}
	Write $g=\phi^{\ast}g_{\R^3}$ be the induced metric on $\Sigma\setminus\ens{p_1,\cdots,p_n}$. 
	The preceding argument shows that for $\vec{v}=v\n_{\vec{\Psi}}\in \mathscr{E}_{\vec{\Psi}}(\Sigma,\R^3)$ and $v:\Sigma\rightarrow \R$ smooth such that $v(p_i)=0$ for all $1\leq i\leq n$, we have
	\begin{align*}
		Q_{\vec{\Psi}}(v)=\frac{1}{2}\int_{\Sigma}\left(\lg \left(|\phi|^2v\right)\right)^2d\vg\geq 0.
	\end{align*}
	Now, let $\vec{v}=v\n_{\vec{\Psi}}\in \mathscr{E}_{\vec{\Psi}}(\Sigma,\R^3)$ be an arbitrary variation such that $v(p_i)=0$ for all $1\leq i\leq n$. As 
	\begin{align}\label{infty0}
		\Delta_{g_{\vec{\Psi}}}^{\n_{\vec{\Psi}}}\vec{v}=\left(\Delta_{g_{\vec{\Psi}}}v\right)\n_{\vec{\Psi}}\in L^2(\Sigma,g_{\vec{\Psi}})
	\end{align}
	and
	\begin{align}\label{infty}
		g^{-1}\otimes dv\in L^{\infty}(\Sigma,g_0),
	\end{align}
	where $g_0$ is a fixed smooth metric on $\Sigma$. The estimate \eqref{infty} must be interpreted as follows. If $p_i$ corresponds to an end of multiplicity $m_i\geq 1$ of $\phi$, then $\vec{\Psi}$ admits a branch point of multiplicity $\theta_0=m_i$ at $p_i$, and \eqref{infty} means that in the chart $\varphi_i:U_i\rightarrow B(0,1)\subset \C$
	\begin{align}\label{infty2}
		\frac{|\D v|}{|z|^{\theta_0-1}}\in L^{\infty}(B(0,1)).
	\end{align}
	we deduce that in particular $v\in W^{2,2}\cap W^{1,\infty}(\Sigma)$. Therefore, let $\ens{v_{k}}_{k\in \N}\in C^{\infty}$ such that
	\begin{align*}
		v_k\conv{k\rightarrow \infty}v\qquad \text{in}\;\, W^{2,2}(\Sigma).
	\end{align*}
	Then up to a subsequence, we deduce that (up to taking a subsequence) $\D^2 v_k\conv{k\rightarrow \infty} \D^2v$ almost everywhere on $\Sigma$. In $U_i$ we have an expansion for some $\gamma_{j_1,j_2}^k\in \R$ (as $\phi$ is smooth)
	\begin{align}\label{expansion01}
		v_k=v_k(p_i)+\sum_{\substack{j_1,j_2\geq 0\\ 1\leq j_1+j_2\leq \theta_0}}\Re\left(\gamma_{i,j_1,j_2}^kz^{j_1}\z^{j_2}\right)+O(|z|^{\theta_0+1}).
	\end{align}
	As $v_k\conv{k\rightarrow \infty} v$ in $C^{0}(\Sigma)$, and by \eqref{infty2}, we deduce (as $v(p_i)=0$) that $v_k(p_i)\conv{k\rightarrow \infty}0$ and $\gamma_{i,j_1,j_2}^k\conv{k\rightarrow \infty}0$ for all $1\leq j_2+j_2\leq \theta_0-1$. Furthermore, as by \eqref{infty0}
	\begin{align}\label{infty1}
		\frac{\Delta v}{|z|^{\theta_0-1}}\in L^2(B(0,1)),
	\end{align} 
	and $\Delta v_k\conv{k\rightarrow \infty}\Delta v$ almost everywhere, we also find by \eqref{infty1} that $\gamma_{i,j_1,j_2}^k\conv{k\rightarrow \infty}0$ for all $j_1+j_2=\theta_0$ such that $(j_1,j_2)\notin \ens{(\theta_0,0),(0,\theta_0)}$. Finally, this implies that if $\rho_i$ is a cutoff function such that $\rho_i=1$ on $\varphi_i^{-1}(B(0,1/2))\subset U_i$ and $\mathrm{supp}(\rho_i)\subset U_i$, and
	\begin{align*}
		\tilde{v}_k=v_k-\sum_{i=1}^{n}\rho_i\bigg\{v_k(p_i)+\sum_{\substack{j_1,j_2\geq 0\\ 1\leq j_1+j_2\leq m_i}}\Re\left(\gamma_{i,j_1,j_2}^k\varphi_i^{j_1}\bar{\varphi_i}^{j_2}\right)\bigg\}\in C^{\infty}(\Sigma),
	\end{align*}
	also satisfies
	\begin{align*}
	    &\tilde{v}_k(p_i)=0\qquad \text{for all }\;\, 1\leq i\leq n,\quad \text{and}\quad
		\tilde{v}_k\conv{k\rightarrow \infty}v\qquad\text{strongly in}\;\, W^{2,2}(\Sigma),
	\end{align*}
	and furthermore, by the expansion \eqref{expansion01}, we deduce that 
	\begin{align*}
		\tilde{v}_k\n_{\vec{\Psi}}\in \mathscr{E}_{\vec{\Psi}}(\Sigma,\R^3).
	\end{align*}
	Therefore, $\tilde{v}_k$ is an admissible variation of $\vec{\Psi}$, and by the preceding discussion we have
	\begin{align}\label{var}
		{Q}_{\vec{\Psi}}(\tilde{v}_k)=\frac{1}{2}\int_{\Sigma}\left(\lg\left(|\phi|^2\tilde{v}_k\right)\right)^2d\vg\geq 0.
	\end{align}
	Now, by the strong $W^{2,2}$ convergence and as $\tilde{v}_k$ is admissible, we have (see for example the explicit formula for $Q_{\vec{\Psi}}$ in \cite{indexS3} or \cite{index3})
	\begin{align*}
		Q_{\vec{\Psi}}(\tilde{v}_k)\conv{k\rightarrow \infty}Q_{\vec{\Psi}}(v).
	\end{align*}
	Then \eqref{var} implies that $Q_{\vec{\Psi}}(v)\geq 0$, but notice also that by Fatou lemma
	\begin{align*}
		Q_{\vec{\Psi}}(v)=\liminf\limits_{k\rightarrow\infty}Q_{\vec{\Psi}}(\tilde{v}_k)\geq \frac{1}{2}\int_{\Sigma}\liminf_{k\rightarrow \infty}\left(\lg \left(|\phi|^2\tilde{v}_k\right)\right)^2d\vg=\frac{1}{2}\int_{\Sigma}\left(\lg\left(|\phi|^2v\right)\right)^2d\vg\geq 0.
	\end{align*}
	This observation concludes the proof of the theorem, as the last equality in \eqref{upends} comes from the Li-Yau inequality (\cite{lieyau}) and the Jorge-Meeks formula (\cite{jorge}).
\end{proof}
\begin{rem}
	The proof of the theorem shows in particular that for all admissible variations $\vec{v}=v\n_{\vec{\Psi}}\in \mathscr{E}_{\vec{\Psi}}(\Sigma,\R^3)$ such that $v(p_i)=0$ for all $1\leq i\leq n$, 
	\begin{align*}
		\frac{1}{2}\int_{\Sigma}\left(\lg\left(|\phi|^2v\right)\right)^2d\vg\leq Q_{\vec{\Psi}}(v)<\infty.
	\end{align*}
\end{rem}
Therefore, we have $Q_{\epsilon}^i>0$ for all $\epsilon>0$ small enough, as by \eqref{limit} we have
\begin{align}\label{limit2}
	Q_{\vec{\Psi}}(v)=\lim\limits_{\epsilon\rightarrow 0}\left(\frac{1}{2}\int_{\Sigma_{\epsilon}}(\lg u)^2d\vg-\sum_{i=1}^{n}Q_{\epsilon}^iv^2(p_i)\right)
\end{align}
In particular, as the metric $g$ is real analytic on all compact subset $K\subset \Sigma\setminus\ens{p_1,\cdots,p_n}$, we deduce that for all $1\leq i\leq n$
\begin{align}\label{welldefined}
	\frac{1}{2}\int_{B_{\epsilon_0}\setminus\bar{B}_{\epsilon}(p_i)}\left(\lg u\right)^2d\vg=Q_{\epsilon}^iv^2(p_i)+O(1).
\end{align}

The following theorem is the analogous of Theorem V.$1,2,3$ \cite{bethuelbrezishelein2}. Here, the vortices are already fixed and correspond to the points $p_1,\cdots,p_n\in \Sigma$ where the metric of the corresponding minimal surface degenerates. We first obtain an estimate of the singular energy by a geometric argument, and show that the Jacobi operator of the minimiser $u_{\epsilon,\delta}^i$ is bounded in $L^2$ away from $p_i$. This will allow us to pass to the limit to a limit function as $\delta\rightarrow 0$ and $\epsilon\rightarrow 0$.

\begin{theorem}\label{firstestimate}
	Let $0<\epsilon<\epsilon_0$ and $0<\delta<\delta(\epsilon)<\epsilon$ and $u_{\epsilon,\delta}^i$ be the unique solution of \eqref{EQepsilon}. Then there exists a non-decreasing function $\omega:\R_+\rightarrow\R_+$ which is continuous at $0$ and such that $\omega(0)=0$ (independent of $\epsilon$ and $\delta$) such that 
	\begin{align}\label{singularenergy}
		\left|\frac{1}{2}\int_{\Sigma_{\epsilon,\delta}^i}\left(\lg u_{\epsilon,\delta}^i\right)^2d\vg-Q_{\epsilon}^iv^2(p_i)\right|\leq \omega\left(\wp{v}{2,2}{\Sigma}\right)
	\end{align}
	and
	\begin{align}\label{awayvortices}
		\frac{1}{2}\int_{\Sigma_{\epsilon_0,\delta}}\left(\lg u_{\epsilon,\delta}^i\right)^2d\vg\leq \omega\left(\wp{v}{2,2}{\Sigma}\right).
	\end{align}
\end{theorem}
\begin{proof}
Recalling that $\Sigma_{\epsilon}=\Sigma\setminus \bigcup_{i=1}^n\bar{B}_\epsilon(p_i)$, we define the continuous bilinear form $B_\epsilon:W^{2,2}(\Sigma_\epsilon)\times W^{2,2}(\Sigma_\epsilon)\rightarrow\R$ by
\begin{align*}
B_\epsilon(u_1,u_2)=\frac{1}{2}\int_{\Sigma_\epsilon}^{}\mathscr{L}_g u_1\,\mathscr{L}_gu_2\,d\vg
\end{align*}
and let $Q_\epsilon:W^{2,2}(\Sigma_\epsilon)\rightarrow \R$ be the associated quadratic form. Then we have
\begin{align}\label{est2}
Q_{\vec{\Psi}}(v)=\lim\limits_{\epsilon\rightarrow 0}Q_\epsilon(u)-\sum_{i=1}^{n}Q_\epsilon^iv^2(p_i)
\end{align}
and the limit is well-defined. Now, fix a cutoff function $\rho_i\geq 0$ such that 
\begin{align*}
	\rho_{i}=1\qquad \text{on}\;\, B_{\epsilon_0/2}(p_i),\qquad \text{and}\;\, \supp (\rho_i)\subset B_{\epsilon_0}(p_i).
\end{align*} 
Notice in particular that for all $1\leq i\leq n$, $0<\epsilon<\epsilon_0$ and $0<\delta<\delta(\epsilon)$, we have by \eqref{welldefined}
\begin{align}\label{boundenergy1}
	\frac{1}{2}\int_{\Sigma_{\epsilon,\delta}^i}\left(\lg u_{\epsilon,\delta}^i\right)^2d\vg\leq\frac{1}{2}\int_{\Sigma_{\epsilon,\delta}}\left(\lg(\rho_i u)\right)^2d\vg=\frac{1}{2}\int_{B_{\epsilon_0}\setminus\bar{B}_{\epsilon}(p_i)}\left(\lg\left( \rho_i u\right)\right)^2d\vg=Q_{\epsilon}^iv^2(p_i)+O(1).
\end{align}
Now, if $0<\epsilon<\epsilon_0$ and $0<\delta<\delta(\epsilon)<\epsilon$ define
\begin{align*}
u_{\epsilon,\delta}=u-\sum_{i=1}^{n}u_{\epsilon,\delta}^i,
\end{align*}
We have
\begin{align*}
Q_\epsilon(u)&=B_\epsilon\left(u_{\epsilon,\delta}+\sum_{i=1}^{n}u_{\epsilon,\delta}^i,u_{\epsilon,\delta}+\sum_{i=1}^{n}u_{\epsilon,\delta}^i\right)\\
&=Q_\epsilon(u_{\epsilon,\delta},u_{\epsilon,\delta})+\sum_{i=1}^{n}Q_\epsilon(u_{\epsilon,\delta}^i)+2\sum_{i=1}^{n}B_\epsilon(u_{\epsilon,\delta},u_{\epsilon,\delta}^i)+\sum_{1\leq i\neq j\leq n}B_\epsilon(u_{\epsilon,\delta}^i,u_{\epsilon,\delta}^j).
\end{align*}
Integrating by parts, we find
\begin{align*}
B_\epsilon(u_{\epsilon,\delta},u_{\epsilon,\delta}^i)=\frac{1}{2}\int_{\partial B_\epsilon(p_i)}^{}\left(u_{\epsilon,\delta}\,\partial_\nu (\mathscr{L}_g u_{\epsilon,\delta}^i)-\partial_\nu u_{\epsilon,\delta} (\mathscr{L}_gu_{\epsilon,\delta}^i)\right)d\hh^1.	
\end{align*}
As
\begin{align*}
u_{\epsilon,\delta}=-\sum_{j\neq i}^{}u_{\epsilon,\delta}^j,\quad \partial_\nu u_{\epsilon,\delta}=-\sum_{j\neq i}^{}\partial_\nu u_{\epsilon,\delta}^j\quad \text{on}\;\partial B_\epsilon(p_i),
\end{align*}
we deduce that $B_\epsilon(u_{\epsilon,\delta},u_{\epsilon,\delta}^i)$ does not contain a quadratic term of the form $C_\epsilon v^2(p_i)$, as the functions $u_{\epsilon,\delta}^j$ ($j\neq i$) are independent of $v(p_i)$. A similar argument applies for $B_\epsilon(u_\epsilon^i,u_\epsilon^j)$ ($i\neq j$), so we deduce by \eqref{limit2} that the only possibility for the limit \eqref{est2} to be finite is that
\begin{align}\label{boundenergy2}
Q_\epsilon(u_\epsilon^i)=Q_\epsilon^iv^2(p_i)+O(1)
\end{align}
where $O(1)$ is a quantity bounded independently of $0<\epsilon<\epsilon_0$ and $0<\delta<\delta(\epsilon)$. Therefore, combining \eqref{boundenergy2} with \eqref{boundenergy1}, we deduce that for all $0<\delta<\delta(\epsilon)<\epsilon$
\begin{align*}
\frac{1}{2}\int_{\Sigma_{\epsilon,\delta}}\left(\lg u_{\epsilon,\delta}^i\right)^2d\vg=Q_{\epsilon}^iv^2(p_i)+O(1),
\end{align*}
where $O(1)$ is a quantity bounded independently of $0<\epsilon<\epsilon_0$ and $0<\delta<\delta(\epsilon)$.
Therefore, we deduce that
\begin{align}\label{apriori}
	Q_{\epsilon}^iv^2(p_i)+O(1)&=\int_{\Sigma_{\epsilon,\delta}^i}\left(\lg (\rho_i u)\right)^2d\vg=\int_{\Sigma_{\epsilon,\delta}^i}\left(\lg(\rho_iu-u_{\epsilon,\delta}^i)\right)^2d\vg\nonumber\\
	&+2\int_{\Sigma_{\epsilon,\delta}^i}\lg\left(\rho_iu-u_{\epsilon,\delta}^i\right)\lg(u_{\epsilon,\delta}^i)d\vg
	+\int_{\Sigma_{\epsilon,\delta}^i}\left(\lg u_{\epsilon,\delta}^i\right)^2d\vg\nonumber\\
	&=\int_{\Sigma_{\epsilon,\delta}^i}\left(\lg(\rho_iu-u_{\epsilon,\delta}^i)\right)^2d\vg+2\int_{\Sigma_{\epsilon,\delta}^i}\lg\left(\rho_iu-u_{\epsilon,\delta}^i\right)\lg(u_{\epsilon,\delta}^i)d\vg+Q_{\epsilon}^iv^2(p_i)+O(1).
\end{align}
Furthermore, the boundary conditions imply that $u_{\epsilon,\delta}^i=u=\rho_i u$ on $\partial B_{\epsilon}(p_i)$ and $\partial_{\nu}u_{\epsilon,\delta}^iu=\partial_{\nu}i=\partial_{\nu}(\rho_i u)=0$, while for all $j\neq i$, $u_{\epsilon,\delta}^i=\partial_{\nu}u_{\epsilon,\delta}^i=\rho_iu=\partial_{\nu}(\rho_i u)=0$. Therefore, we deduce as $\lg^2u_{\epsilon,\delta}^i=0$ that 
\begin{align}\label{aprioribis}
	&\int_{\Sigma_{\epsilon,\delta}^i}\lg\left(\rho_iu-u_{\epsilon,\delta}^i\right)\lg(u_{\epsilon,\delta}^i)d\vg=\int_{\Sigma_{\epsilon,\delta}^i}(\rho_i u-u_{\epsilon,\delta}^i)\lg^2u_{\epsilon,\delta}^i\,d\vg\nonumber\\
	&+\int_{\partial B_{\epsilon}(p_i)}\left(\rho_i u-u_{\epsilon,\delta}^i\right)\partial_{\nu}\left(\lg u_{\epsilon,\delta}^i\right)-\partial_{\nu}\left(\rho_i u-u\right)\lg u_{\epsilon,\delta}^i\,d\mathscr{H}^1\nonumber\\
	&+\sum_{j\neq i}^{}\int_{\partial B_{\delta}(p_j)}\left(\rho_i u-u_{\epsilon,\delta}^i\right)\partial_{\nu}\left(\lg u_{\epsilon,\delta}^i\right)-\partial_{\nu}\left(\rho_i u-u\right)\lg u_{\epsilon,\delta}^i\,d\mathscr{H}^1\nonumber\\
	&=0
\end{align}
Therefore, \eqref{apriori} and \eqref{aprioribis} imply that 
\begin{align*}
	\int_{\Sigma_{\epsilon,\delta}}(\lg(\rho_i-u_{\epsilon,\delta}^i))^2d\vg=O(1)
\end{align*}
and as $\supp(\rho_i)\subset B_{\epsilon_0}(p_i)$, we deduce that 
\begin{align*}
	\int_{\Sigma_{\epsilon_0,\delta}}\left(\lg u_{\epsilon,\delta}^i\right)^2d\vg=O(1),
\end{align*}
or in other words
\begin{align*}
	\limsup_{\epsilon\rightarrow 0}\limsup_{\delta\rightarrow 0}\int_{\Sigma_{\epsilon_0,\delta}}\left(\lg u_{\epsilon,\delta}^i\right)^2d\vg<\infty.
\end{align*}
Furthermore, as the error terms are continuous in $v\in W^{2,2}(\Sigma)$ (such that $\v=v\n_{\vec{\Psi}}$), we deduce that there exists a modulus of continuity $\omega=\omega_{\vec{\Psi}}:\R_+\rightarrow \R_+$ independent of $0<\epsilon<\epsilon_0$ and $0<\delta<\delta(\epsilon)<\epsilon$ (that we can take non-decreasing and continuous at $0$) such that 
\begin{align*}
	\left|\frac{1}{2}\int_{\Sigma_{\epsilon,\delta}}\left(\lg u_{\epsilon,\delta}^i\right)^2d\vg-Q_{\epsilon}^iv^2(p_i)\right|\leq \omega(\wp{v}{2,2}{\Sigma})
\end{align*}
and
\begin{align*}
	\frac{1}{2}\int_{\Sigma_{\epsilon_0},\delta}\left(\lg u_{\epsilon,\delta}^i\right)^2d\vg\leq \omega(\wp{v}{2,2}{\Sigma}).
\end{align*}
This concludes the proof of the theorem. 
\end{proof}

\begin{rem}
	Notice that the preceding proof implies that the limits of
	$
	B_\epsilon(u_{\epsilon,\delta},u_{\epsilon,\delta}^i)
	$
	and 
	$
	B_\epsilon(u_{\epsilon,\delta}^i,u_{\epsilon,\delta}^j)
	$
	($i\neq j$)	are well-defined as $\epsilon\rightarrow 0$ (and $0<\delta<\delta(\epsilon)<\epsilon$). 
\end{rem}

\subsection{Local estimates near the ends}

As the operators $\lg$ and $\lg^2$ are uniformly elliptic on $\Sigma_{\epsilon}$ for all $\epsilon>0$, the only difficult estimates come from the asymptotic behaviour near the vortices $p_i$ (for $1\leq i\leq n$). As the estimates depend on the chart, we fix some covering $(U_1,\cdots,U_n)\subset \Sigma$ by domains of charts $\Sigma$ such that $p_i\in U_i$ for all $1\leq i\leq n$ and all estimates will be taken with respect to a complex chart $\varphi_i:U_i\rightarrow B(0,1)\subset \C$ such that $\varphi_i(p_i)=0$.

\begin{theorem}\label{ipptheorem}
	Let $1\leq i\leq n$ be a fixed integer and $u_{\epsilon,\delta}^i$ be the solution of \eqref{EQepsilon} for some $0<\epsilon<\epsilon_0$ and $0<\delta<\delta(\epsilon)$. Let $1\leq j\neq i\leq n$ and assume that the end of $\phi$ has multiplicity $m\geq 1$, and  define in the chart $U_j$ the function $v_{\epsilon,\delta}^i=e^{-\lambda}u_{\epsilon,\delta}^i$. Then there exists real analytic functions $\zeta_0,\zeta_2:B(0,1)\rightarrow \R$ and $\vec{\zeta}_1,\vec{\zeta}_3:B(0,1)\rightarrow \R^2$ and a universal constant $C=C(U_j,\vec{\Psi})>0$ depending only on the chosen chart $U_j$ around $p_j$ and on $\vec{\Psi}$ such that 
	\begin{align}
		&\int_{B_1\setminus \bar{B}_{\delta}(0)}\left(\Delta v_{\epsilon,\delta}^i-2(m+1)\left(\frac{x}{|x|^2}+\D\zeta_0\right)\cdot \D v_{\epsilon,\delta}^i+\frac{(m+1)^2}{|x|^2}\left(1+x\cdot \vec{\zeta}_1\right)v_{\epsilon,\delta}^i\right)^2dx\nonumber\\
		&\leq \int_{\Sigma_{\epsilon_0,\delta}}\left(\lg u_{\epsilon,\delta}^i\right)^2d\vg\leq C\,\omega\left(\wp{v}{2,2}{\Sigma}\right)\label{square1}\\
		&\int_{B_1\setminus\bar{B}_{\delta}(0)}\left(\Delta v_{\epsilon,\delta}^i+(m+1)(m-1)\frac{v_{\epsilon,\delta}^i}{|x|^2}\right)^2dx+4(m+1)(m-1)\int_{B_1\setminus \bar{B}_{\delta}(0)}\left(\frac{x}{|x|^2}\cdot \D v_{\epsilon,\delta}^i-\frac{v_{\epsilon,\delta}^i}{|x|^2}\right)^2dx\nonumber\\
		&-\int_{B_1\setminus\bar{B}_{\delta}(0)}\left(\D\zeta_2\cdot\D v_{\epsilon,\delta}^i-\frac{x}{|x|^2}\cdot \vec{\zeta}_3\,v_{\epsilon,\delta}^i\right)^2dx\leq C\,\omega\left(\wp{v}{2,2}{\Sigma}\right).\label{square2}
	\end{align}
\end{theorem}
\begin{proof}
	As the end has multiplicity $m\geq 1$, there exists $\alpha_i>0$ and $\alpha_0\in \C$ such that 
	\begin{align*}
		&e^{2\lambda}=\alpha_i^2|z|^{-2(m+1)}\left(1+2\,\Re\left(\alpha_0z
		\right)+O(|z|^2)\right)\\
		&e^{2\lambda}K_g=O(1).
	\end{align*} 
	Furthermore, let $\zeta:B(0,1)\rightarrow \R$ be the real analytic function such that 
	\begin{align*}
		\lambda(z)=-(m+1)\log|z|+\zeta(z).
	\end{align*}
	Notice that $\zeta$ is real-analytic by the Weierstrass parametrisation \cite{dierkes}. 
	Then we have
	\begin{align*}
		|\D\lambda|^2=\frac{(m+1)^2}{|x|^2}-2(m+1)\frac{x}{|x|^2}\cdot \D\zeta+|\D\zeta|^2=\frac{(m+1)^2}{|x|^2}\left(1-\frac{2}{(m+1)}x\cdot\D\zeta+\frac{1}{(m+1)^2}|x|^2|\D\zeta|^2\right)
	\end{align*}
	Therefore, we have if $u_{\epsilon,\delta}^i=e^{\lambda}v_{\epsilon,\delta}^i$ as $\Delta\zeta=-e^{2\lambda}K_g$ for all $0<\kappa<1$
	\begin{align}\label{twobounds1}
		&\int_{B_1\setminus \bar{B}_{\delta}(0)}\left(\Delta_gu_{\epsilon,\delta}^i-2K_gv_{\epsilon,\delta}^i\right)^2d\vg=\int_{B_1\setminus\bar{B}_{\delta}(0)}\left(e^{-\lambda}\Delta(e^{\lambda}v_{\epsilon,\delta}^i)-2e^{2\lambda}K_gv_{\epsilon,\delta}^i\right)^2d\vg\nonumber\\
		&=\int_{B_1\setminus \bar{B}_{\delta}(0)}\left(\Delta v_{\epsilon,\delta}^i+2\s{\D\lambda}{\D v_{\epsilon,\delta}^i}+\left(|\D\lambda|^2-3e^{2\lambda}K_g\right)v_{\epsilon,\delta}^i\right)^2dx\nonumber\\
		&=\int_{B_1\setminus \bar{B}_{\delta}(0)}\bigg(\Delta v_{\epsilon,\delta}^i-2(m+1)\left(\frac{x}{|x|^2}-\frac{1}{(m+1)}\D\zeta\right)\cdot \D v_{\epsilon,\delta}^i\nonumber\\
		&+\frac{(m+1)^2}{|x|^2}\left(1-\frac{2}{(m+1)}x\cdot \D\zeta+\frac{1}{(m+1)^2}|x|^2\left(|\D\zeta|^2+3\Delta\zeta\right)\right)v_{\epsilon,\delta}^i\bigg)^2dx\nonumber\\
		&\geq (1-\kappa)\int_{B_1\setminus\bar{B}_{\delta}(0)}\left(\Delta v_{\epsilon,\delta}^i-2(m+1)\frac{x}{|x|^2}\cdot\D v_{\epsilon,\delta}^i+\frac{(m+1)^2}{|x|^2}u\right)^2dx\nonumber\\
		&+\left(1-\frac{1}{\kappa}\right)\frac{1}{(m+1)^4}\int_{B_1\setminus\bar{B}_{\delta}(0)}\left(2(m+1)^3\D\zeta\cdot \D v_{\epsilon,\delta}^i+\left(-2(m+1)\frac{x}{|x|^2}\cdot\D\zeta+(|\D\zeta|^2+3\,\Delta \zeta)\right)v_{\epsilon,\delta}^i\right)^2dx,
	\end{align}
	where we used the inequality for all $a,b\in \R$ and $0<\kappa<1$
	\begin{align*}
		(a+b)^2\geq (1-\kappa)a^2+\left(1-\frac{1}{\kappa}\right)b^2.
	\end{align*}
	In particular, the first estimate follows directly from \eqref{awayvortices} of Theorem \ref{firstestimate}, with
	\begin{align*}
		&\zeta_0=-\frac{1}{(m+1)}\zeta,\\
		&\zeta_1=-\frac{2}{(m+1)^2}\D\zeta+\frac{x}{(m+1)^2}(|\D\zeta|^2+3\Delta\zeta).
	\end{align*}
	Now, thanks to the computations of Lemma, we have 
	\begin{align}\label{twobounds2}
		&\int_{B_1\setminus\bar{B}_{\delta}(0)}\left(\Delta v_{\epsilon,\delta}^i-2(m+1)\frac{x}{|x|^2}\cdot\D v_{\epsilon,\delta}^i+\frac{(m+1)^2}{|x|^2}u\right)^2dx\nonumber\\
		&=\int_{B_1\setminus \bar{B}_{\delta}(0)}\left(\Delta v_{\epsilon,\delta}^i+(m+1)(m-1)\frac{v_{\epsilon,\delta}^i}{|x|^2}\right)^2dx+4(m+1)(m-1)\int_{B_1\setminus \bar{B}_{\delta}(0)}\left(\frac{x}{|x|^2}\cdot \D v_{\epsilon,\delta}^i-\frac{v_{\epsilon,\delta}^i}{|x|^2}\right)^2dx\nonumber\\
		&-\int_{S^1}\left(v_{\epsilon,\delta}^i\,\partial_{\nu}(\mathscr{L}_m v_{\epsilon,\delta}^i)-\partial_{\nu}v_{\epsilon,\delta}^i\,\mathscr{L}_m v_{\epsilon,\delta}^i\right)d\mathscr{H}^1
		+\int_{S^1}\left(v_{\epsilon,\delta}^i\partial_{\nu}(\Delta v_{\epsilon,\delta}^i)-\partial_{\nu}(v_{\epsilon,\delta}^i)\Delta v_{\epsilon,\delta}^i\right)d\mathscr{H}^1\nonumber\\
		&+4(m+1)(m-1)\int_{S^1}\left((v_{\epsilon,\delta}^i)^2- v_{\epsilon,\delta}^i\,\partial_{\nu}^2v_{\epsilon,\delta}^i\right)\,d\mathscr{H}^1\nonumber\\
		&=\int_{B_1\setminus \bar{B}_{\delta}(0)}\left(\Delta v_{\epsilon,\delta}^i+(m+1)(m-1)\frac{v_{\epsilon,\delta}^i}{|x|^2}\right)^2dx+4(m+1)(m-1)\int_{B_1\setminus \bar{B}_{\delta}(0)}\left(\frac{x}{|x|^2}\cdot \D v_{\epsilon,\delta}^i-\frac{v_{\epsilon,\delta}^i}{|x|^2}\right)^2dx\nonumber\\
		&-\int_{S^1}\left(v_{\epsilon,\delta}^i\,\partial_{\nu}(\left(\mathscr{L}_m-\Delta\right) v_{\epsilon,\delta}^i)-\partial_{\nu}v_{\epsilon,\delta}^i\,\left(\mathscr{L}_m-\Delta\right) v_{\epsilon,\delta}^i\right)d\mathscr{H}^1\nonumber\\
		&+4(m+1)(m-1)\int_{S^1}\left((v_{\epsilon,\delta}^i)^2- v_{\epsilon,\delta}^i\,\partial_{\nu}^2v_{\epsilon,\delta}^i\right)\,d\mathscr{H}^1.
	\end{align}
	Now, if 
	\begin{align*}
		\tilde{\mathscr{L}}_g=e^{\lambda}\mathscr{L}_g(e^{\lambda}\,\cdot\,),
	\end{align*}
	we have (as $\Delta\lambda=-e^{2\lambda}K_g$)
	\begin{align*}
		\tilde{\mathscr{L}}_g=\Delta+2\s{\D\lambda}{\D \,\cdot\,}+(|\D\lambda|^2+3\Delta\lambda),
	\end{align*}
	so $v_{\epsilon,\delta}^i$ solves
	\begin{align*}
		\tilde{\mathscr{L}}_g^{\ast}\tilde{\mathscr{L}}_gv_{\epsilon,\delta}^i=0,
	\end{align*}
	where one checks that there exists polynomial functions $P_{k,l}:\R^2\times \R^4\rightarrow \mathrm{M}_2(\R)$, $\vec{Q}:\R^2\times\R^4\rightarrow \R^2$ and $R:\R^2\times \R^4\rightarrow \R$ such that  
	\begin{align*}
		\tilde{\mathscr{L}}_g\mathscr{L}_gv_{\epsilon,\delta}^i=\Delta^2v_{\epsilon,\delta}^i+\vec{P}(\D\lambda,\D^2\lambda)\cdot\D^2v_{\epsilon,\delta}^i+\vec{Q}(\D\lambda,\D^2\lambda)\cdot\D v_{\epsilon,\delta}^i +R(\D\lambda,\D^2\lambda)\, v_{\epsilon,\delta}^i.
	\end{align*}
	Therefore, thanks to elliptic regularity and \eqref{awayvortices}, we deduce that 
	\begin{align*}
		\int_{S^1}\left(v_{\epsilon,\delta}^i\,\partial_{\nu}(\left(\mathscr{L}_m-\Delta\right) v_{\epsilon,\delta}^i)-\partial_{\nu}v_{\epsilon,\delta}^i\,\left(\mathscr{L}_m-\Delta\right) v_{\epsilon,\delta}^i\right)d\mathscr{H}^1
		-4(m+1)(m-1)\int_{S^1}\left((v_{\epsilon,\delta}^i)^2- v_{\epsilon,\delta}^i\,\partial_{\nu}^2v_{\epsilon,\delta}^i\right)\,d\mathscr{H}^1
	\end{align*}
	is uniformly bounded in $0<\epsilon<\epsilon_0$ and $0<\delta<\delta(\epsilon)<\epsilon$. Furthermore, there exists $C_0=C_0(U_j,\vec{\Psi})>0$ such that 
	\begin{align}\label{twobounds3}
		&\bigg|\int_{S^1}\left(v_{\epsilon,\delta}^i\,\partial_{\nu}(\left(\mathscr{L}_m-\Delta\right) v_{\epsilon,\delta}^i)-\partial_{\nu}v_{\epsilon,\delta}^i\,\left(\mathscr{L}_m-\Delta\right) v_{\epsilon,\delta}^i\right)d\mathscr{H}^1\nonumber\\
		&-4(m+1)(m-1)\int_{S^1}\left((v_{\epsilon,\delta}^i)^2- v_{\epsilon,\delta}^i\,\partial_{\nu}^2v_{\epsilon,\delta}^i\right)\,d\mathscr{H}^1\bigg|\leq C_0\,\omega\left(\wp{v}{2,2}{\Sigma}\right).
	\end{align}
	Therefore, we have by \eqref{square1}, \eqref{twobounds1} and \eqref{twobounds3}
	\begin{align*}
		&\int_{B_1\setminus\bar{B}_{\delta}(0)}\left(\Delta v_{\epsilon,\delta}^i+(m+1)(m-1)\frac{v_{\epsilon,\delta}^i}{|x|^2}\right)^2dx
		+4(m+1)(m-1)\int_{B_{1}\setminus \bar{B}_{\delta}(0)}\left(\frac{x}{|x|^2}\cdot \D v_{\epsilon,\delta}^i-\frac{v_{\epsilon,\delta}^i}{|x|^2}\right)^2dx\\
		&-\frac{1}{\kappa}\frac{1}{(m+1)^4}\int_{B_1\setminus\bar{B}_{\delta}(0)}\left(2(m+1)^3\D\zeta\cdot \D v_{\epsilon,\delta}^i+\left(-2(m+1)\frac{x}{|x|^2}\cdot \D \zeta+\left(|\D\zeta|^2+3\Delta\zeta\right)\right)v_{\epsilon,\delta}^i\right)^2dx\\
		&\leq \frac{1}{1-\kappa}C_1\,\omega(\wp{v}{2,2}{\Sigma}).
	\end{align*}
	Choose now $\kappa=\dfrac{1}{4}$, and  define
	\begin{align*}
		\zeta_2&=4(m+1)\zeta\\
		\vec{\zeta}_3&=\frac{4}{(m+1)}x\cdot\D\zeta-\frac{2}{(m+1)^2}x\left(|\D\zeta|^2+3\Delta\zeta\right),
	\end{align*}
	we find
	\begin{align*}
		&\int_{B_1\setminus\bar{B}_{\delta}(0)}\left(\Delta v_{\epsilon,\delta}^i+(m+1)(m-1)\frac{v_{\epsilon,\delta}^i}{|x|^2}\right)^2dx+4(m+1)(m-1)\int_{B_{1}\setminus \bar{B}_{\delta}(0)}\left(\frac{x}{|x|^2}\cdot \D v_{\epsilon,\delta}^i-\frac{v_{\epsilon,\delta}^i}{|x|^2}\right)^2dx\\
		&-\int_{B_1\setminus\bar{B}_{\delta}(0)}\left(\D\zeta_2\cdot\D v_{\epsilon,\delta}^i-\frac{x}{|x|^2}\cdot \vec{\zeta}_3\,v_{\epsilon,\delta}^i\right)^2dx
		\leq \frac{4}{3}C_1\,\omega\left(\wp{v}{2,2}{\Sigma}\right).
	\end{align*}
	This concludes the proof of the theorem.
\end{proof}

\subsection{Indicial roots analysis : case of embedded ends}

The following theorem is the analogous of Theorem VI.$1$ of \cite{bethuelbrezishelein2} and Theorem $1$ of \cite{bethuelbrezishelein} from Ginzburg-Landau theory.

\begin{theorem}\label{indicielles1} Assume that the minimal surface $\phi$ of Theorem  \ref{ta} has embedded ends.
	Then there exists $v_{\epsilon}^i\in C^{\infty}(\Sigma\setminus(\bar{B}_{\epsilon}(p_i)\cup\ens{p_1,\cdots,p_n}))$
	such that for all compact $K\subset \Sigma_{\epsilon}^i$, we have (up to a subsequence as $\delta\rightarrow 0$)
	\begin{align*}
	v_{\epsilon,\delta}^i\conv{\delta\rightarrow 0}v_{\epsilon}^i\qquad \text{in}\;\, C^l(K)\;\, \text{for all}\;\, l\in\N.
	\end{align*} 
	Furthermore, for all $j\neq i$, we have an expansion in $U_j$ as  
	\begin{align*}
		v_{\epsilon}^i(z)=\Re\left(\gamma_0z+\gamma_1z^2\right)+\gamma_2|z|^2+\gamma_3|z|^2\log|z|+\varphi_{\epsilon}(z)
	\end{align*}
	for some real-analytic function $\varphi_{\epsilon}$ such that $\varphi_{\epsilon}(z)=O(|z|^3)$. Therefore, if $u_{\epsilon}^i=|\phi|^2v_{\epsilon}^i$, there exists  $a_{i,j},b_{i,j}\in \R$ and $c_{i,j},d_{i,j}\in \C$ and $\psi_{\epsilon}\in C^{\infty}(B(0,1)\setminus\ens{0})$ such that 
	\begin{align*}
	u_{\epsilon}^i(z)=\Re\left(\frac{c_{i,j}}{z}+d_{i,j}\frac{\z}{z}\right)+a_{i,j}\log|z|+b_{i,j}+\psi_{\epsilon}(z).
	\end{align*}
	and for all $l\in \N$, 
	\begin{align*}
		|\D^l\psi_{\epsilon}(z)|=O(|z|^{1-l}).
	\end{align*}
\end{theorem}
\begin{rem}
	Although $a_{i,j},b_{i,j},c_{i,j}$ and $d_{i,j}$ depends on $\epsilon$, we remove this explicit dependence for the sake of simplicity of notation.
\end{rem}
\begin{proof}
	\textbf{Step 1: Indicial roots analysis.}
	
	We make computations as previously in the previously fixed chart $\varphi_j:U_j\rightarrow B(0,1)\subset \C$ such that $\varphi_j(p_j)=0$.
	By \cite{pacariv} (p. $25$) the asymptotic expansion of $v_{\epsilon,\delta}^i$ at $0$ depends only on the linearised operator of $e^{\lambda}\mathscr{L}_g(e^{\lambda}\,\cdot\,)$, which is as $\phi$ has embedded ends
	\begin{align*}
		\mathscr{L}=\Delta-4\frac{x}{|x|^2}\cdot \D+\frac{4}{|x|^2}.
	\end{align*}
	Therefore, without loss of generality, we can assume that $\mathscr{L}^{\ast}\mathscr{L} v_{\epsilon,\delta}^i=0$. Taking polar coordinates $(r,\theta)$ centred at the origin, recall that 
	\begin{align*}
	\Delta=\partial_r^2+\frac{1}{r}\partial_r+\frac{1}{r^2}\partial^2_\theta.
	\end{align*}
    Therefore, we have
	\begin{align*}
	\mathscr{L}&=\partial_r^2+\frac{1}{r}\partial_r+\frac{1}{r^2}\partial_\theta^2-4\frac{1}{r}\partial_r+\frac{4}{r^2}
	=\partial_r^2-\frac{3}{r}\partial_r+\frac{4}{r^2}+\frac{1}{r^2}\partial_\theta^2,
	\end{align*}
	and
	\begin{align*}
	\mathscr{L}^\ast=\partial_r^2+\frac{5}{r}\partial_r+\frac{4}{r^2}+\frac{1}{r^2}\partial_\theta^2.
	\end{align*}
	Projecting to $\mathrm{Vect}(e^{ik\cdot})$ (where $k\in\Z$ is fixed), the operator $\mathscr{L}$ (resp. $\mathscr{L}^\ast$) becomes
	\begin{align*}
	\mathscr{L}_k=\partial_r^2-\frac{3}{r}\partial_r+\frac{4-k^2}{r^2}\quad \left(\text{resp.}\;\,\mathscr{L}^\ast_k=\partial_r^2+\frac{5}{r}\partial_r+\frac{4-k^2}{r^2}\right)
	\end{align*}
	and we define for all $k\in\Z$ the functions $v_{\epsilon,\delta}^{i}(k,\cdot):(\delta,1)\rightarrow\mathbb{C}$ by
	\begin{align*}
	v_{\epsilon,\delta}^{i}(r,\theta)=\sum_{k\in\Z}^{}v_{\epsilon,\delta}^{i}(k,r)e^{ik\theta}.
	\end{align*}
	As $\mathscr{L}_k^\ast \mathscr{L}_k v_{\epsilon,\delta}^i(k,\cdot)=0$, and the space of solutions to $\mathscr{L}_k^\ast \mathscr{L}_k u=0$ is four-dimensional, we only need to find a basis of solutions to $\mathscr{L}_k^\ast \mathscr{L}_k u=0$ to obtain all possible asymptotic behaviour at the origin. 
	
	Let $\alpha\in\C$ fixed, we have
	\begin{align}\label{lk}
	&\mathscr{L}_k r^\alpha=\alpha(\alpha-1)r^{\alpha-2}-3\alpha r^{\alpha-2}+(4-k^2)r^{\alpha-2}=(\alpha^2-4\alpha+4-k^2)r^{\alpha-2}\nonumber\\
	&\mathscr{L}_k^\ast\mathscr{L}_k r^\alpha=(\alpha^2-k^2)(
	\alpha^2-4\alpha+4-k^2)r^{\alpha-4}.
	\end{align}
	so the basis of solutions to $\mathscr{L}^{\ast}\mathscr{L}_ku=0$ is given by
	\begin{align*}
	u_1^\pm=r^{\alpha_k^\pm},\quad u_2^{\pm}=r^{\beta_k^\pm}
	\end{align*}
	where
	\begin{align*}
	\alpha_k^\pm=2\pm|k|,\quad \beta_k^\pm=\pm |k|.
	\end{align*}
	In particular, for $k=0$, we need to find two other solutions. For $k=0$, we have
	\begin{align*}
	\mathscr{L}_0^\ast\mathscr{L}_0=\partial_r^4+\frac{2}{r}\partial_r^3-\frac{1}{r^2}\partial_r^2+\frac{1}{r^3}\partial_r
	\end{align*}
	so one easily check that a basis of solutions of $\mathscr{L}^{\ast}\mathscr{L}u=0$ is given by
	\begin{align*}
	1,r^2,\log (r),r^2\log(r)
	\end{align*}
	and that furthermore, 
	\begin{align*}
	\mathscr{L}_0(r^2)=\mathscr{L}_0\left(r^2\log(r)\right)=0.
	\end{align*}
	Finally, for $|k|=1$, as $\ens{\alpha_{k}^{+},\alpha_{k}^{-}}=\ens{1,3}$ and $\ens{\beta_{k}^+,\beta_{k}^-}=\ens{-1,1}$, we only have three solutions and we need to find an additional one. As $\mathrm{Ker}(\mathscr{L}_1^{\ast})=\ens{r^{-3},r^{-1}}$, we need to find a solution $u$ such that $\mathscr{L}u\neq 0$, $\mathscr{L}u\in \mathrm{Ker}(\mathscr{L}_1^{\ast})$ and $u\notin\mathrm{Span}_{\R}(r^{-1},r,r^3)$. One checks directly that this additional solution is given by 
	\begin{align*}
		u(r)=r\log(r),
	\end{align*}
	which satisfies indeed
	\begin{align*}
		\mathscr{L}u=\frac{1}{r}-\frac{3}{r}\left(\log(r)+1\right)+\frac{3}{r^2}\left(r\log(r)\right)=-\frac{2}{r}\in \mathrm{Ker}(\mathscr{L}^{\ast}).
	\end{align*}
	
	Notice that these computations also show that $\mathscr{L}^{\ast}\mathscr{L}=\Delta^2$, but we did not want to use this result directly to obtain a formally similar proof in the case of ends of higher multiplicity.
	
	\textbf{Step 2: Estimate on the biharmonic components.}
	
	Now, recall that 
	\begin{align*}
		\int_{B_1\setminus \bar{B}_{\delta}(0)}\left(\Delta v_{\epsilon,\delta}^i-2(m+1)\left(\frac{x}{|x|^2}+\D\zeta_0\right)\cdot \D v_{\epsilon,\delta}^i+\frac{(m+1)^2}{|x|^2}\left(1+x\cdot \vec{\zeta}_1\right)v_{\epsilon,\delta}^i\right)^2dx\leq C\,\omega\left(\wp{v}{2,2}{\Sigma}\right).
	\end{align*}
	Now, let $\gamma,\gamma_k^{1},\gamma_k^{2},\gamma_k^{3},\gamma_k^{4}\in \C$ (for $k\in \Z$) be such that 
	\begin{align}\label{dev}
		v_{\epsilon,\delta}^i(r,\theta)&=\sum_{k\in \Z^{\ast}}^{}\left(\gamma_k^{1}r^{2+k}+\gamma_{k}^2r^{2-k}+\gamma_k^{3}r^{k}+\gamma_{k}^{4}r^{-k}\right)e^{ik\theta}+\left(\gamma \,re^{i\theta}+\bar{\gamma}\,re^{-i\theta}\right)\log(r)\nonumber\\
		& +\gamma_0^1+\gamma_0^2\log(r)+\gamma_0^3r^2+\gamma_0^4r^2\log(r)
		.
	\end{align}
	As $v_{\epsilon,\delta}^{i}$ is real, we have for all $k\in \Z$
	\begin{align*}
		\gamma_{-k}^2=\bar{\gamma_{k}^1}\qquad \text{and}\;\, \gamma_{-k}^4=\bar{\gamma_k^3}.
	\end{align*}
	Now, we have
	\begin{align}\label{linear1}
		\mathscr{L}v_{\epsilon,\delta}^i=4\sum_{k\in \Z^{\ast}}^{}\left(-(k-1)\gamma_k^3r^{k-2}+(k+1)\gamma_k^4r^{-k-2}\right)e^{ik\theta}-2\left(\frac{\gamma}{r} e^{i\theta}+\frac{\bar{\gamma}}{r}e^{-i\theta}\right)+\frac{4\gamma_0^1}{r^2}+\frac{4\gamma_0^2}{r^2}(\log(r)-1).
	\end{align}
	Therefore, as 
	\begin{align*}
		\tilde{\mathscr{L}}_g=\mathscr{L}-4\D\zeta_0\cdot\D+(m+1)^2\frac{x}{|x|^2}\cdot\vec{\zeta}_1
	\end{align*}
	for two real-analytic functions $\zeta_0:B(0,1)\rightarrow \R$ and $\vec{\zeta}_1:B(0,1)\rightarrow \R$, we deduce from \eqref{linear1} that 
	\begin{align*}
		\tilde{\mathscr{L}}_gv_{\epsilon,\delta}^i&=4\sum_{k\in \Z\setminus\ens{0,1}}^{}-(k-1)\gamma_k^3r^{k-2}(1+O(r))e^{ik\theta}+4\sum_{k\in \Z\setminus\ens{-1,0}}^{}(k+1)\gamma_k^4r^{-k-2}(1+O(r))e^{ik\theta}\\
		&-2\left(\frac{\gamma}{r}(1+O(r))e^{i\theta}+\frac{\bar{\gamma}}{r}(1+O(r))e^{-i\theta}\right)
		+\frac{4(\gamma_0^1-\gamma_0^2)}{r^2}(1+O(r))
		+\frac{4\gamma_0^1}{r^2}\log(r)(1+O(r)).
	\end{align*}
	Notice that the first two sums do not involve powers in $1/r$ (this justifies why there are no cross terms between these two sums and the remaining terms).
	Now fix some $0<R<1$ such that the \enquote{$O(1)$ functions} be bounded by $1/2$ (in absolute value) on $B_{R}\setminus\bar{B}_{\delta}(0)$. 
	Then we have by Parseval identity
	\begin{align}\label{singularintegral}
		&\int_{B_{R}\setminus\bar{B}_{\delta}(0)}\left(\tilde{\mathscr{L}}_gv_{\epsilon,\delta}^i\right)^2dx=32\pi\sum_{k\in \Z^{\ast}}^{}\int_{\delta}^R\bigg((k-1)^2|\gamma_k^3|^2r^{2(k-2)}+(k+1)^2|\gamma_k^4|^2r^{-2(k+2)}\nonumber\\
		&+2(k+1)(k-1)\,\Re\left(\gamma_k^3\bar{\gamma_k^4}\right)r^{-4}\bigg)(1+O(r))rdr
		+16\pi\int_{\delta}^R\left(\frac{|\gamma|^2}{r^2}(1+O(r))\right)rdr\nonumber\\
		&+32\pi\int_{\delta}^R\left(\frac{(\gamma_0^1-\gamma_0^2)^2}{r^4}+\frac{(\gamma_0^1)^2}{r^4}\log^2(r)+2(\gamma_0^1-\gamma_0^2)\gamma_0^1\frac{\log(r)}{r^4}\right)(1+O(r))rdr\nonumber\\
		&=16\pi\sum_{k\in \Z\setminus\ens{0,1}}^{}(k-1)|\gamma_k^3|^2\left(R^{2(k-1)}(1+O(R))-\delta^{2(k-1)}(1+O(\delta))\right)\nonumber\\
		&+16\pi\sum_{k\in \Z\setminus\ens{-1,0}}^{}-(k+1)|\gamma_k^4|^2\left(R^{-2(k+1)}(1+O(R))-\delta^{-2(k+1)}(1+O(\delta))\right)\nonumber\\
		&+16\pi (\gamma_0^1-\gamma_0^2)^2\left(\frac{1}{\delta^2}(1+O(\delta))-\frac{1}{R^2}(1+O(R))\right)\nonumber\\
		&+8\pi(\gamma_0^1)^2\left(\frac{1}{\delta^2}(1+O(\delta))\left(2\log^2(\delta)+2\log(\delta)+1\right)-\frac{1}{R^2}(1+O(R))\left(\log^2(R)+2\log(R)+1\right)\right)\nonumber\\
		&+8\pi \gamma_0^1(\gamma_0^1-\gamma_0^2)\left(\frac{1}{\delta^2}\left(1+O(\delta)\right)(2\log(\delta)+1)-\frac{1}{R^2}\left(1+O(R)\right)\left(2\log(R)+1\right)\right)\nonumber\\
		&+32\pi\sum_{k\in \Z^{\ast}}^{}\Re\left(\gamma_k^3\bar{\gamma_k^4}\right)\left(\frac{1}{\delta^2}\left(1+O(\delta)\right)-\frac{1}{R^2}\left(1+O(R)\right)\right)\nonumber\\
		&=16\pi\sum_{k\geq 2}^{}(|k|-1)|\gamma_k^3|^2R^{2(|k|-1)}\left(1+O(R)-\left(\frac{\delta}{R}\right)^{2(|k|-1)}(1+O(\delta))\right)\nonumber\\
		&+16\pi\sum_{k\leq -1}^{}(|k|+1)|\gamma_k^3|^2\frac{1}{\delta^{2(|k|+1)}}\left(1+O(\delta)-\left(\frac{\delta}{R}\right)^{2(|k|+1)}(1+O(R))\right)\nonumber\\
		&+16\pi\sum_{k\geq 1}^{}(|k|+1)|\gamma_k^4|^2\frac{1}{\delta^{2(|k|+1)}}\left(1+O(\delta)-\left(\frac{\delta}{R}\right)^{2(|k|+1)}(1+O(R))\right)\nonumber\\
		&+16\pi\sum_{k\geq -2}^{}(|k|-1)|\gamma_k^4|^2R^{2(|k|-1)}\left(1+O(R)-\left(\frac{\delta}{R}\right)^{2(|k|-1)}(1+O(\delta))\right)\nonumber\\
		&+16\pi|\gamma|^2\left(\log\left(\frac{1}{\delta}\right)(1+O(\delta))+\log(R)\left(1+O(R)\right)\right)\nonumber\\
		&+16\pi (\gamma_0^1-\gamma_0^2)^2\frac{1}{\delta^2}\left(1+O(\delta)-\left(\frac{\delta}{R}\right)(1+O(R))\right)\nonumber\\
		&+8\pi(\gamma_0^1)^2\left(\frac{1}{\delta^2}(1+O(\delta))\left(2\log^2(\delta)+2\log(\delta)+1\right)-\frac{1}{R^2}(1+O(R))\left(\log^2(R)+2\log(R)+1\right)\right)\nonumber\\
		&+8\pi \gamma_0^1(\gamma_0^1-\gamma_0^2)\left(\frac{1}{\delta^2}\left(1+O(\delta)\right)(2\log(\delta)+1)-\frac{1}{R^2}\left(1+O(R)\right)\left(2\log(R)+1\right)\right)\nonumber\\
		&+32\pi\sum_{k\in \Z^{\ast}}^{}(k+1)(k-1)\Re\left(\gamma_k^3\bar{\gamma_k^4}\right)\frac{1}{\delta^2}\left(1+O(\delta)-\left(\frac{\delta}{R}\right)^2\left(1+O(R)\right)\right).
	\end{align}
	As the quantity in the left-hand side of \eqref{singularintegral} is bounded independently of $0<\delta<R$, we deduce that for all $k\geq 1$, and some uniform constant $C>0$ 
	\begin{align}\label{singularvanish}
		&\left|\gamma_{-|k|}^3\right|\leq C\delta^{|k|+1}\conv{\delta\rightarrow 0}0\nonumber\\
		&\left|\gamma_{|k|}^4\right|\leq C\delta^{|k|+1}\conv{\delta\rightarrow  0}0.
	\end{align}
	Notice that the second estimate follows from the first one as $\gamma_k^4=\bar{\gamma_{-k}^3}$. Furthermore, as 
	\begin{align*}
		\sum_{k\geq 1}^{}(|k|+1)|\gamma_{-k}^3|^2\frac{1}{\delta^{2(|k|+1)}}<\infty
	\end{align*}
	and is bounded independently of $\delta>0$, there exists $C>0$ such that 
	\begin{align}\label{precisedsingular}
		|\gamma_{-k}^3|\leq \frac{C}{|k|+1}\delta^{|k|+1}. 
	\end{align}
	Now, we see the next order of singularity is given by $\log^2(\delta)/\delta^2$, so we have
	\begin{align}\label{othersing1}
		|\gamma_0^1|\leq C\frac{\delta}{\log\left(\frac{1}{\delta}\right)}\conv{\delta\rightarrow  0}0.
	\end{align}
	Another singular term is 
	\begin{align*}
		\frac{1}{\delta^2}\sum_{k\in \Z\setminus\ens{-1,0,1}}^{}(k+1)(k-1)\Re(\gamma_k^3\bar{\gamma_k^4}),
	\end{align*}
	but \eqref{singularvanish} implies that (as the $\gamma_k$ are also uniformly bounded)
	\begin{align*}
		&\frac{1}{\delta^2}\left|\sum_{k\in \Z\setminus\ens{-1,0,1}}^{}(k+1)(k-1)\Re(\gamma_k^3\bar{\gamma_k^4})\right|\leq \frac{C}{\delta^2}\sum_{k\geq 2}^{}(|k|+1)(|k|-1)\delta^{|k|+1}\\
		&\leq C\sum_{|k|\geq 2}^{}(|k|+1)|k|\delta^{|k|-1}=C\left(\frac{3\delta}{1-\delta}-\frac{6\delta^2}{(1-\delta)}+\frac{2\delta^3}{(1-\delta)^3}\right)\leq 4C\delta\conv{\delta\rightarrow 0}0.
	\end{align*}
	for $0<\delta<1$ small enough. The next singular term is 
	\begin{align*}
		16\pi (\gamma_0^1-\gamma_0^2)^2\frac{1}{\delta^2}\left(1+O(\delta)-\left(\frac{\delta}{R}\right)(1+O(R))\right),
	\end{align*}
	and using \eqref{othersing1}, we deduce that 
	\begin{align*}
		(\gamma_0^1-\gamma_0^2)^2\frac{1}{\delta^2}=\frac{|\gamma_0^2|^2}{\delta^2}+O\left(\frac{1}{\log^2(\frac{1}{\delta})}\right),
	\end{align*}
	so we deduce that 
	\begin{align}\label{othersing2}
		|\gamma_0^2|\leq C\delta\conv{\delta\rightarrow 0}0.
	\end{align}
	Finally, the last singular term is
	\begin{align*}
		16\pi |\gamma|^2\log\left(\frac{1}{\delta}\right)
	\end{align*}
	and we deduce that 
	\begin{align}\label{othersing3}
		|\gamma|\leq C\sqrt{\frac{1}{\log\left(\frac{1}{\delta}\right)}}\conv{\delta\rightarrow 0}0.
	\end{align}
	\textbf{Step 3: Estimates on the harmonic components.} 
	Now, we have the inequality (from Theorem \ref{ipptheorem})
	\begin{align}\label{bounded}
		\int_{B_R\setminus\bar{B}_{\delta}(0)}\left(\Delta v_{\epsilon,\delta}^i\right)^2dx-\int_{B_R\setminus \bar{B}_{\delta}(0)}\left(\D\zeta_2\cdot\D v_{\epsilon,\delta}^i-\frac{x}{|x|^2}\cdot \vec{\zeta}_3\,v_{\epsilon,\delta}^i\right)^2dx\leq C\,\omega\left(\wp{v}{2,2}{\Sigma}\right).
	\end{align}
	Thanks to \eqref{dev}
	\begin{align*}
		\Delta v_{\epsilon,\delta}^i=4\sum_{k\in \Z^{\ast}}^{}\left((k+1)\gamma_k^1r^{k}-(k-1)\gamma_{k}^2r^{-k}\right)e^{ik\theta}+2\left(\frac{\gamma}{r}e^{i\theta}+\frac{\bar{\gamma}}{r}e^{-i\theta}\right)+4\gamma_0^3+4\gamma_0^4(\log r+1).
	\end{align*}
	Furthermore, we have
	\begin{align*}
		\D\zeta_2\cdot\D v_{\epsilon,\delta}^i-\frac{x}{|x|^2}\cdot \vec{\zeta}_3\,v_{\epsilon,\delta}^i&=\sum_{k\in \Z^{\ast}}^{}\bigg(\gamma_{k}^1O(r^{1+k})+\gamma_{k}^2O(r^{1-k})
		+\gamma_{k}^3O(r^{k-1})+\gamma_{k}^4O(r^{-k-1})\bigg)e^{ik\theta}\\
		&+\gamma\, O(1)+\gamma_0^1O\left(\frac{1}{r}\right)+\gamma_0^2\log(r)O\left(\frac{1}{r}\right)+\gamma_0^3O(r)+\gamma_0^4\log r O(r).
	\end{align*}
	Therefore, we have
	\begin{align}\label{newbound4}
		&\int_{B_{R}\setminus \bar{B}_{\delta}(0)}\left(\Delta v_{\epsilon,\delta}^i\right)^2dx=32\pi\sum_{k\in \Z^{\ast}}^{}\int_{\delta}^{R}\bigg((k+1)^2|\gamma_k^1|^2r^{2k}+(k-1)^2|\gamma_k^2|^2r^{-2k}\nonumber\\
		&-2(k+1)(k-1)\Re\left(\gamma_k^1\bar{\gamma_k^2}\right)\bigg)rdr
		+16\pi \int_{\delta}^{R}\left(\frac{|\gamma|^2}{r^2}\left(1+O(r)\right)\right)rdr\nonumber\\
		&
		+32\pi \int_{\delta}^{R}\left(|\gamma_0^3+\gamma_0^4|^2+|\gamma_0^4|^2\log^2(r)+2(\gamma_0^3+\gamma_0^4)\gamma_0^4\log(r)\right)rdr.
	\end{align}
	Notice that the second integral involving the square of the radial component of $\Delta v_{\epsilon,\delta}^i$ is bounded, so we can neglect this term. Now, we also have as $|a+b+c+d|^2\leq 4\left(|a|^2+|b|^2+|c|^2+|d|^2\right)$ for all $a,b,c,d\in \C$ and by Parseval identity
	\begin{align}\label{newbound0}
		&\int_{B_R\setminus\bar{B}_{\delta}(0)}\left(\D\zeta_2\cdot \D v_{\epsilon,\delta}^i-\frac{x}{|x|^2}\cdot\vec{\zeta}_3\,v_{\epsilon,\delta}^i\right)^2dx\nonumber\\
		&\leq 8\pi\sum_{k\in \Z^{\ast}}^{}\int_{\delta}^{R}\bigg(|\gamma_k^1|^2O(r^{2+2k})+|\gamma_{k}^2|^2O(r^{2-2k})+|\gamma_k^3|^2O(r^{2k-2})+|\gamma_k^4|^2O(r^{-2k-2})\bigg)rdr\nonumber\\
		&+8\pi\int_{\delta}^{R}\left(|\gamma_0^1|^2O\left(\frac{1}{r^2}\right)+|\gamma_0^2|^2O\left(\frac{\log^2(r)}{r^2}\right)+|\gamma_0^3|^2O(r^2)+|\gamma_0^4|^2O(r^2\log^2(r))\right)rdr+|\gamma|\,O(R).
	\end{align}
	Now notice that 
	\begin{align}\label{newbound1}
		\sum_{k\geq 1}^{}\int_{\delta}^R\left(|\gamma_k^3|^2O(r^{2k-2})+|\gamma_{-k}^4|^2O(r^{2k-2})\right)rdr
	\end{align}
	is bounded in $\delta$, and \eqref{precisedsingular} imply that 
	\begin{align}\label{newbound2}
		&\left|\sum_{k\geq 1}^{}\left(|\gamma_{-k}^3|^2O(r^{-2k-2})+|\gamma_k^4|O(r^{-2k-2})\right)rdr\right|\leq C\sum_{k\geq 1}^{}\int_{\delta}^{R}\frac{\delta^{2k+2}}{(|k|+1)^2}r^{-2k-1}dr\nonumber\\
		&=C\sum_{k\geq 1}^{}\frac{\delta^2}{2k(k+1)^2}\left(1-\left(\frac{\delta}{R}\right)^{2k}\right)\leq C\left(1-\frac{\pi^2}{12}\right)\delta^2\conv{\delta\rightarrow 0}0,
	\end{align}
	where we used
	\begin{align*}
		\sum_{k=1}^{\infty}\frac{1}{k(k+1)^2}=\sum_{k=1}^{\infty}\left(\frac{1}{k}-\frac{1}{k+1}\right)-\sum_{k=1}^{\infty}\frac{1}{(k+1)^2}=1-(\zeta(2)-1)=2-\zeta(2)=2-\frac{\pi^2}{6}.
	\end{align*}
	Finally, by \eqref{newbound0}, \eqref{newbound1} and \eqref{newbound2}, we deduce that there exists $C>0$ (independent of $\delta$ and $\epsilon$) such that 
	\begin{align}\label{newbound5}
	    &\bigg|8\pi\sum_{k\in \Z^{\ast}}^{}\int_{\delta}^{R}\bigg(|\gamma_k^1|^2O(r^{2+2k})+|\gamma_{k}^2|^2O(r^{2-2k})+|\gamma_k^3|^2O(r^{2k-2})+|\gamma_k^4|^2O(r^{-2k-2})\bigg)rdr\nonumber\\
		&+8\pi\int_{\delta}^{R}\left(|\gamma_0^1|^2O\left(\frac{1}{r^2}+|\gamma_0^2|^2O\left(\frac{\log^2(r)}{r^2}\right)\right)+|\gamma_0^3|^2O(r^2)+|\gamma_0^4|^2O(r^2\log^2(r))\right)rdr\nonumber\\
		&- 8\pi\sum_{k\in \Z^{\ast}}^{}\int_{\delta}^R\bigg(|\gamma_k^1|^2O(r^{2+2k})+|\gamma_{k}^2|^2O(r^{2-2k})\bigg)rdr\bigg|\leq C
	\end{align}
	Therefore, by \eqref{newbound4}, \eqref{newbound0}, \eqref{newbound5}, and \eqref{othersing3} (for the term in $|\gamma|^2\log(1/\delta)^{-1}$) we have
	\begin{align}\label{bounded2}
		&\int_{B_R\setminus \bar{B}_{\delta}(0)}\left(\Delta v_{\epsilon,\delta}^i\right)^2dx-\int_{B_R\setminus \bar{B}_{\delta}(0)}\left(\D\zeta_2\cdot \D v_{\epsilon,\delta}^i
		-\frac{x}{|x|^2}\cdot \D\vec{\zeta}_3\,v_{\epsilon,\delta}^i\right)^2dx\nonumber\\
		&\geq 32\pi\sum_{k\in \Z^{\ast}}^{}\int_{\delta}^R\left((k+1)^2|\gamma_k^1|^2r^{2k}(1+O(r^2))+(k-1)^2|\gamma_k^2|^2r^{-2k}(1+O(r^2))-2(k+1)(k-1)\Re\left(\gamma_k^1\bar{\gamma_k^2}\right)\right)rdr\nonumber\\
		&+8\pi\int_{\delta}^{R}\left(|\gamma_{-1}^1|^2O(1)+|\gamma_1^2|^2O(1)\right)rdr\nonumber\\
		&+32\pi \int_{\delta}^{R}\left(|\gamma_0^3+\gamma_0^4|^2+|\gamma_0^4|^2\log^2(r)+2(\gamma_0^3+\gamma_0^4)\gamma_0^4\log(r)\right)rdr-C.
	\end{align}
	As previously, the terms involving positive powers of $k$ are bounded, and
	\begin{align*}
		\sum_{k\leq -2}\int_{\delta}^R(k+1)^2|\gamma_k^1|^2r^{2k}(1+O(r^2))rdr=\frac{1}{2}\sum_{k\leq -2}^{}(|k|-1)|\gamma_k^1|^2\frac{1}{\delta^{2(|k|-1)}}\left(1+O(\delta^2)-\left(\frac{\delta}{R}\right)^{2(|k|-1)}(1+O(R^2))\right),
	\end{align*}
	so for all $k\geq 2$, as \eqref{bounded} implies that \eqref{bounded2} is bounded independently of $\delta$ (and $\epsilon$), we deduce that for some universal constant $C$ (independent of $0<\epsilon<\epsilon_0$ and $0<\delta<\delta(\epsilon)<\epsilon$).   
	\begin{align*}
		|\gamma_{-|k|}^1|\leq C\delta^{|k|-1}\conv{\delta\rightarrow 0}0.
	\end{align*}
	and as $\gamma_{-k}^2=\bar{\gamma_{k}^1}$, we also have for all $k\geq 2$
	\begin{align*}
		|\gamma_{k}^2|\leq C\delta^{|k|-1}\conv{\delta\rightarrow 0}0.
	\end{align*}
	
	\textbf{Step 4: Conclusion and limit as $\delta\rightarrow 0$.}
	
	Finally, we deduce from the two previous steps that 
	\begin{align*}
		&v_{\epsilon,\delta}^i=\left(\gamma_1^2+\gamma_1^3\right)re^{i\theta}+
		\left(\gamma_{-1}^1+\gamma_{-1}^4\right)re^{-i\theta}+\gamma_0^3r^2+\gamma_0^4r^2\log(r)\\
		&+\gamma_{1}^1r^3e^{3i\theta}+\gamma_{-1}^{2}r^{3}e^{-3i\theta}
		+\sum_{k\geq 2}^{}\bigg(\left(\gamma_k^1r^{2+k}+\gamma_k^3r^{k}\right)e^{ik\theta}+\left(\gamma_{-k}^2r^{2+k}+\gamma_{-k}^4r^{k}\right)e^{-ik\theta}\bigg)\\
		&+\sum_{k\leq -2}^{}\bigg(\left(\gamma_{k}^2r^{2-|k|}+\gamma_{k}^4r^{-|k|}\right)e^{ik\theta}+\left(\gamma_{-k}^1r^{2-|k|}+\gamma_{-k}^3r^{-|k|}\right)e^{-ik\theta}\bigg)\\
		&+\gamma_0^1+\gamma_0^2\log(r)+\left(\gamma\, e^{i\theta}+\bar{\gamma}\,e^{-i\theta}\right)r\log(r).
	\end{align*}
	Thanks if the previous estimates, all coefficients are bounded, and for all fixed $(r,\theta)\in B_{R}\setminus \bar{B}_{\delta}(0)$, 
	\begin{align}\label{vanishsingular}
		&\bigg|\sum_{k\leq -2}^{}\bigg(\left(\gamma_{k}^2r^{2-|k|}+\gamma_{k}^4r^{-|k|}\right)e^{ik\theta}+\left(\gamma_{-k}^1r^{2-|k|}+\gamma_{-k}^3r^{-|k|}\right)e^{-ik\theta}\bigg)\\
		&+\gamma_0^1+\gamma_0^2\log(r)+\left(\gamma\, e^{i\theta}+\bar{\gamma}\,e^{-i\theta}\right)r\log(r)\bigg|\conv{\delta\rightarrow 0}0.
	\end{align}
	Furthermore, as the operator $\tilde{\mathscr{L}}_g^{\ast}\tilde{\mathscr{L}}_g$ is uniformly elliptic on $\Sigma_{\epsilon}$ for all fixed $0<\epsilon<\epsilon_0$ and thanks to the uniform bound , we deduce that up to a subsequence, there exists $v_{\epsilon}^i\in C^{\infty}(\Sigma\setminus(\bar{B}_{\epsilon}(p_i)\cup\ens{p_1,\cdots,p_n}))$
	such that for all compact $K\subset \Sigma\setminus(\bar{B}_{\epsilon}(p_i)\cup\ens{p_1,\cdots,p_n})$), 
	\begin{align*}
		v_{\epsilon,\delta}^i\conv{\delta\rightarrow 0}v_{\epsilon}^i\qquad \text{in}\;\, C^l(K)\;\, \text{for all}\;\, l\in\N.
	\end{align*}
	Furthermore, as $\delta\rightarrow 0$, \eqref{vanishsingular} implies that 
	\begin{align*}
		v_{\epsilon}^i&=\left(\gamma_1^2+\gamma_1^3\right)re^{i\theta}+
		\left(\gamma_{-1}^1+\gamma_{-1}^4\right)re^{-i\theta}+\gamma_0^3r^2+\gamma_0^4r^2\log(r)\\
		&+\gamma_{1}^1r^3e^{3i\theta}+\gamma_{-1}^{2}r^{3}e^{-3i\theta}
		+\sum_{k\geq 2}^{}\bigg(\left(\gamma_k^1r^{2+k}+\gamma_k^3r^{k}\right)e^{ik\theta}+\left(\gamma_{-k}^2r^{2+k}+\gamma_{-k}^4r^{k}\right)e^{-ik\theta}\bigg)\\
		&=\Re(\gamma_0z+\gamma_1z^2)+\gamma_2|z|^2+\gamma_3|z|^2\log|z|+\varphi(z),\\
		&=\Re(\gamma_0z+\gamma_1z^2)+\gamma_2|z|^2+\gamma_3|z|^2\log|z|+O(|z|^3),
	\end{align*}
	where $\varphi$ is real analytic and $\varphi(z)=O(|z|^3)$. Finally, by the Weierstrass parametrisation, if $\phi$ has embedded ends, we can assume (\cite{schoenPlanar}) that up to rotation $\phi$ admits the following expansion for some $\alpha>0$ and $\beta\in \R$ 
	\begin{align*}
		\phi(z)=\Re\left(\frac{\alpha_j}{z}+O(|z|),\frac{i\alpha_j}{z}+O(|z|),\beta\log|z|\right).
	\end{align*}
	Therefore, we have
	\begin{align*}
		\p{z}\phi(z)=\frac{1}{2}\left(-\frac{\alpha_j}{z^2}+O(1),-\frac{i\alpha_j}{z^2}+O(1),\frac{\beta}{z}+O(1)\right)
	\end{align*}
	and
	\begin{align*}
		e^{2\lambda}=2|\p{z}\phi|^2=\frac{\alpha_j^2}{|z|^4}+O\left(\frac{1}{|z|^2}\right)=\frac{\alpha_j^2}{|z|^4}\left(1+O(|z|^2)\right).
	\end{align*}
	Therefore, we have
	\begin{align*}
		u_{\epsilon}^i&=e^{\lambda}v_{\epsilon,\delta}^i=\frac{\alpha_j}{|z|^2}(1+O(|z|^2))\left(\Re\left(\gamma_0z+\gamma_1z^2\right)+\gamma_2|z|^2+\gamma_3|z|^2\log|z|+O(|z|^3)\right)\\
		&=\Re\left(\frac{\alpha_j\bar{\gamma_0}}{z}+\alpha_j\bar{\gamma_1}\frac{\z}{z}\right)+\alpha_j\gamma_2+\alpha_j\gamma_3\log|z|+O(|z|)
	\end{align*}
	and this concludes the proof of the theorem. 
\end{proof}
\begin{rem}\label{outsidevortices}
	Notice that as $|z|^2,|z|^2\log|z|,\Re(\gamma_0z)\in \mathrm{Ker}(\mathscr{L})$, we have
	\begin{align*}
		\mathscr{L}v_{\epsilon}^i=-4\,\Re\left(\bar{\gamma_1}\frac{\z}{z}\right)+O(|z|), 
	\end{align*} 
	which implies that
	\begin{align*}
		&\mathscr{L}_gu_{\epsilon}^i=e^{-\lambda}\tilde{\mathscr{L}}_g v_{\epsilon}^i=-4\alpha_j\,\Re\left(\gamma_1z^2\right)+O(|z|^3)\\
		&\partial_{\nu}\left(\mathscr{L}_gu_{\epsilon}^i\right)=-\frac{8\alpha_j}{|z|}\Re\left(\gamma_1z^2\right)+O(|z|^2),
	\end{align*}
	where we used
	\begin{align}
		\partial_{\nu}=\frac{x}{|x|}\cdot \D=\frac{(z+\z)}{|z|}(\p{z}+\p{\z})+\frac{(z-\z)}{2i}i(\p{z}-\p{\z})=\frac{1}{|z|}\left(z\,\p{z}+\z\,\p{\z}\right)=\frac{2}{|z|}\Re\left(z\,\p{z}\left(\,\cdot\,
		\right)\right).
	\end{align}
\end{rem}

\subsection{Indicial roots analysis: case of ends of higher multiplicity}

\begin{theorem}\label{expansionends}
	Let $1\leq i\leq n$ and $1\leq j\neq i\leq n$, and assume that $\phi$ has an end of multiplicity $m\geq 2$ at $p_j$, and define $v_{\epsilon,\delta}^i$ in $U_j$ as $u_{\epsilon,\delta}^i=e^{\lambda}v_{\epsilon,\delta}^i$. Then there exists $v_{\epsilon}^i\in C^{\infty}(\Sigma\setminus(\bar{B}_{\epsilon}(p_i)\cup\ens{p_1,\cdots,p_n}))$
	such that for all compact $K\subset \Sigma_{\epsilon}^i$, we have (up to a subsequence as $\delta\rightarrow 0$)
	\begin{align*}
	v_{\epsilon,\delta}^i\conv{\delta\rightarrow 0}v_{\epsilon}^i\qquad \text{in}\;\, C^l(K)\;\, \text{for all}\;\, l\in\N.
	\end{align*} 
	Furthermore, for all $j\neq i$, we have an expansion in $U_j$ as  
	\begin{align*}
	v_{\epsilon}^i(z)=|z|^{m+1}\sum_{k=1}^{m}\Re\left(\frac{\gamma_{i,j,k}^0}{z^k}\right)+|z|^{1-m}\sum_{k=0}^{m}\Re\left(\gamma_{i,j,k}^1z^{m+k}\right)+\gamma_{i,j}^2|z|^{m+1}+\gamma_{i,j}^3|z|^{m+1}\log|z|+\varphi_{\epsilon}(z)
	\end{align*}
	for some real-analytic function $\varphi_{\epsilon}$ such that $\varphi_{\epsilon}(z)=O(|z|^{m+2})$. Furthermore, we have an expansion 
	\begin{align*}
		u_{\epsilon}^i(z)&=e^{\lambda}v_{\epsilon}^i
		=\Re\left(\frac{c_{i,j}}{z^m}\right)+\sum_{1-m\leq k+l\leq 0}^{}\Re\Big(c_{i,j,k,l}z^{k}\z^l\Big)+a_{i,j}\log|z|+\psi_{\epsilon}(z),
	\end{align*}
	for some $\psi_{\epsilon}\in C^{\infty}(B(0,1)\setminus\ens{0})$ such that  for all $l\in \N$
	\begin{align*}
		\D^l\psi_{\epsilon}=O(|z|^{1-l}),
	\end{align*}
	and the $c_{i,j,k,l}$ are almost all zero, that is all but finitely many as $j,k\in \Z$ and $1-m\leq j+k\leq 0$. 
\end{theorem}
\begin{proof}
	\textbf{Step 1: Indicial roots analysis. }
	
	We have the expansion
	\begin{align}\label{devconforme}
		e^{2\lambda}=\frac{\alpha_j^2}{|z|^{2m+2}}\left(1+2\,\Re\left(\alpha_0z
		\right)+O(|z|^2)\right).
	\end{align}
	Now, let $v_{\epsilon}^i$ such that $u_{\epsilon}^i=e^{\lambda}v_{\epsilon}^i$. Then we have as 
	\begin{align*}
	e^{\lambda}\Delta_{g}u_{\epsilon}^i=e^{-\lambda}\Delta_{g_0}\left(e^{\lambda}v_{\epsilon}^i\right)=\Delta v_{\epsilon}^i+2\D\lambda\cdot\D v_{\epsilon}^i+\left(3\Delta\lambda+|\D\lambda|^2\right)v_{\epsilon}^i.
	\end{align*}
	Now we have 
	\begin{align*}
	\lambda=-(m+1)\log|z|+\log(\alpha_j)+\log\left(1+O(|z|)\right),
	\end{align*}
	so we have $\Delta\lambda\in L^{\infty}(D^2)$ and
	\begin{align*}
	\D\lambda=-(m+1)\frac{x}{|x|^2}+O(1)
	|\D\lambda|^2=\frac{(m+1)^2}{|x|^{2}}\left(1+O(|x|)\right)
	\end{align*}
	so we obtain
	\begin{align*}
	e^{\lambda}\Delta_gu_{\epsilon}^i=\left(\Delta_{g}-2(m+1)\left(\frac{x}{|x|^2}+O(1)\right)\cdot \D +\left(\frac{(m+1)^2}{|x|^2}+O\left(\frac{1}{|x|}\right)\right)\right)v_{\epsilon}^i
	\end{align*}
	As $e^{2\lambda}K_g=O(1)$, we finally get 
	\begin{align*}
	e^{\lambda}\mathscr{L}_gu_{\epsilon}^i=\left(\Delta -2(m+1)\left(\frac{x}{|x|^2}+O(1)\right)\cdot\D+\left(\frac{(m+1)^2}{|x|^2}+O\left(\frac{1}{|x|}\right)\right)\right)v_{\epsilon}^i.
	\end{align*}
   Now, denote by $\mathscr{L}_m$ the elliptic operator with regular singularities (see \cite{pacariv})
	\begin{align*}
	\mathscr{L}_m=\Delta-2(m+1)\frac{x}{|x|^2}\cdot \D+\frac{(m+1)^2}{|x|^2}.
	\end{align*}
	As 
	\begin{align}\label{harmonic}
	\frac{x}{|x|^2}=\D\,\log|x|
	\end{align}
	and $\log$ is harmonic on $D^2\setminus\ens{0}$, we have
	\begin{align*}
	\mathscr{L}_m^{\ast}=\Delta+2(m+1)\frac{x}{|x|^2}\cdot\D+\frac{(m+1)^2}{|x|^2},
	\end{align*}
	where $\mathscr{L}_m^{\ast}$ is the formal adjoint of $\mathscr{L}_m$. As the indicial roots of on operator of the form
	\begin{align*}
	\Delta+\frac{x+b(x)}{|x|^2}\cdot\D+\frac{c(x)}{|x|^2}
	\end{align*}
	where $b$ and $c$ are $C^{\infty}$ and $b(x)=O(|x|^2)$ only depends on $c(0)$ and is independent of $b$ (\cite{pacariv}).
	
	Therefore, the indicial roots of $e^{\lambda}\mathscr{L}_ge^{\lambda}\left(e^{\lambda}\mathscr{L}_g(e^{\lambda}\,\cdot\,)\right)$, giving all possible asymptotic behaviour of a solution of $\mathscr{L}_g^2u_{\epsilon}^i=0$ in $D^2\setminus\ens{0}$ are the same of the indicial roots of the operator $\mathscr{L}_m^{\ast}\mathscr{L}_m$. Therefore, consider first a solution $v$ of 
	\begin{align}\label{e1}
	\mathscr{L}_m^{\ast}\mathscr{L}_mv=0.
	\end{align}
	First, recall that
	\begin{align*}
	&\Delta=\partial^2_r+\frac{1}{r}\partial_r+\frac{\Delta_{S^1}}{r^2}\\
	&\mathscr{L}_m=\partial_r^2-\frac{(2m+1)}{r}\partial_r+\frac{\Delta_{S^1}+(m+1)^2}{r^2}.
	\end{align*}
	Therefore, for all $k\in \N$ the projection $\mathscr{L}_{m,k}$ on $\mathrm{Span}(e^{i\cdot k})$ of $\mathscr{L}_m$ is given by 
	\begin{align*}
	\mathscr{L}_{m,k}=\partial_r^2-\frac{(2m+1)}{r}\partial_ r+\frac{(m+1)^2-k^2}{r^2}.
	\end{align*}
	and we first look for solutions of the form
	\begin{align*}
	v(r)=r^{\alpha}
	\end{align*}
	for some $\alpha\in \C$. We have by a direct computation
	\begin{align*}
	\mathscr{L}_{m,k}v &=\left(\alpha(\alpha-1)-2(m+1)\alpha+(m+1)^2-k^2\right)r^{\alpha-2}\\
	&=\left(\alpha^2-2(m+1)\alpha+(m+1)^2-k^2\right)r^{\alpha-2}\\
	&=\left(\alpha-(m+1+|k|)\right)\left(\alpha-(m+1-|k|)\right)r^{\alpha-2}.
	\end{align*} 
	Therefore, for all $k\in \Z\setminus\ens{0}$, we have two linearly independent solutions
	\begin{align*}
	v_{k,0}(r)=r^{m+1+|k|},\quad v_{k,1}(r)=r^{m+1-|k|}.
	\end{align*}
	Now, we compute if $\alpha'=\alpha+2$
	\begin{align*}
	\mathscr{L}_{m,k}^{\ast}\mathscr{L}_mr^{\alpha'}&=\mathscr{L}_{m,k}^{\ast}(\alpha'-(m+1+|k|)(\alpha'-(m+1-|k|))r^{\alpha}\\
	&=\left(\alpha(\alpha-1)+(2m+3)\alpha+(m+1)^2-k^2\right)(\alpha'-(m+1+|k|))(\alpha'-(m+1-|k|))r^{\alpha}\\
	&=\left(\alpha^2+2(m+1)\alpha+(m+1)^2-k^2\right)r^{\alpha}\\
	&=(\alpha-(-(m+1)+|k|))(\alpha-((-m+1)-|k|))(\alpha'-(m+1+|k|))(\alpha'-(m+1-|k|))
	\end{align*}
	and we find two independent solution for $k\in \Z$
	\begin{align*}
	v_{k,2}(r)=r^{-m+1+|k|},\quad v_{k,3}(r)=r^{-m+1-|k|}.
	\end{align*}
	Now, for $k=0$, we need to find two additional solution and one check immediately that
	\begin{align*}
	r^{m+1}\log(r),\quad r^{1-m}\log(r)
	\end{align*}
	are two additional solutions. Furthermore, notice that when $|k|=m$,
	\begin{align*}
		\ens{r^{m+1+|k|},r^{m+1-|k|},r^{1-m+|k|},r^{1-m-|k|}}=\ens{r^{1-2m},r,r^{2m+1}}
	\end{align*}
	so we need to find another solution. As $\mathrm{Ker}(\mathscr{L}_{m,m}^{\ast})=\mathrm{Span}(r^{-(m+1)+m},r^{-(m+1)-m})=\mathrm{Span}(r^{-1},r^{-(2m+1)})$, we compute that 
	\begin{align*}
		\mathscr{L}_{m,m}\left(r\log(r)\right)=\frac{1}{r}-\frac{(2m+1)}{r}\left(\log(r)+1\right)+\frac{(m+1)^2-m^2}{r^2}\left(r\log(r)\right)=-\frac{2m}{r}\in \mathrm{Ker}(\mathscr{L}^{\ast}_{m,m}).
	\end{align*}
	Therefore, for $|k|=m$, we have the basis of solutions
	\begin{align*}
		r^{1-2m},r,r^{2m+1},r\log(r).
	\end{align*}
	so we find the additional solution $r\log(r)$ when $|k|=m$. 
	Therefore, we finally get 
	\begin{align}\label{newdev}
		v_{\epsilon,\delta}^i(r,\theta)&=\sum_{k\in \Z^{\ast}}^{}\left(\gamma_k^1r^{m+1+k}+\gamma_{k}^2r^{m+1-k}+\gamma_k^3r^{1-m+k}+\gamma_k^4r^{1-m-k}\right)e^{ik\theta}\nonumber\\
		&+\left(\gamma\, e^{im\theta}+\bar{\gamma}\,e^{-im\theta}\right)r\log(r)+\gamma_0^1r^{1-m}+\gamma_0^2r^{1-m}\log(r)+\gamma_0^3r^{m+1}+\gamma_0^4r^{m+1}\log(r).
	\end{align}
	
	\textbf{Step 2: Estimate coming from $\mathscr{L}_m v_{\epsilon,\delta}^i\in L^2$.}
	As
	\begin{align}\label{negative}
	&\int_{B_1\setminus \bar{B}_{\delta}(0)}\left(\Delta v_{\epsilon,\delta}^i-2(m+1)\left(\frac{x}{|x|^2}+\D\zeta_0\right)\cdot \D v_{\epsilon,\delta}^i+\frac{(m+1)^2}{|x|^2}\left(1+x\cdot \vec{\zeta}_1\right)v_{\epsilon,\delta}^i\right)^2dx\nonumber\\
	&\leq \int_{\Sigma_{\epsilon_0,\delta}}\left(\lg u_{\epsilon,\delta}^i\right)^2d\vg\leq C\,\omega\left(\wp{v}{2,2}{\Sigma}\right)\nonumber\\
	&\int_{B_1\setminus\bar{B}_{\delta}(0)}\left(\Delta v_{\epsilon,\delta}^i+(m+1)(m-1)\frac{v_{\epsilon,\delta}^i}{|x|^2}\right)^2dx+4(m+1)(m-1)\int_{B_1\setminus \bar{B}_{\delta}(0)}\left(\frac{x}{|x|^2}\cdot \D v_{\epsilon,\delta}^i-\frac{v_{\epsilon,\delta}^i}{|x|^2}\right)^2dx\nonumber\\
	&-\int_{B_1\setminus\bar{B}_{\delta}(0)}\left(\D\zeta_2\cdot\D v_{\epsilon,\delta}^i-\frac{x}{|x|^2}\cdot \vec{\zeta}_3\,v_{\epsilon,\delta}^i\right)^2dx\leq C\,\omega\left(\wp{v}{2,2}{\Sigma}\right).
	\end{align}
	and $m\geq 2$, we deduce by the same argument as Theorem \ref{indicielles1} that the following three integrals are bounded uniformly in $\epsilon$ and $\delta$
	\begin{align*}
		&\int_{B_1\setminus \bar{B}_{\delta}}\left(\mathscr{L}_mv_{\epsilon,\delta}^i\right)^2dx=\int_{B_1\setminus\bar{B}_{\delta}(0)}\left(\Delta v_{\epsilon,\delta}^i-2(m+1)\frac{x}{|x|^2}\cdot \D v_{\epsilon,\delta}^i+\frac{(m+1)^2}{|x|^2}v_{\epsilon,\delta}^i\right)^2dx\\
		&\int_{B_1\setminus\bar{B}_{\delta}(0)}\left(\Delta v_{\epsilon,\delta}^i+(m+1)(m-1)\frac{v_{\epsilon,\delta}^i}{|x|^2}\right)^2dx\\
		&\int_{B_1\setminus \bar{B}_{\delta}(0)}\left(\frac{x}{|x|^2}\cdot \D v_{\epsilon,\delta}^i-\frac{v_{\epsilon,\delta}^i}{|x|^2}\right)^2dx.
	\end{align*}
	Now define
	\begin{align*}
		&\tilde{\mathscr{L}}_m=\Delta+\frac{(m+1)(m-1)}{|x|^2}\\
		&\mathscr{D}=\frac{x}{|x|^2}\cdot \D-\frac{1}{|x|^2}.
	\end{align*}
	Furthermore, notice that for all $k\in \Z^{\ast}$, if $P_k$ is the projection on $\mathrm{Span}(e^{ik}\,\cdot\,)$, then 
	\begin{align}\label{orthogonalkernel}
		\mathrm{Ker}(P_k\mathscr{L}_m)=\mathrm{Span}\left(r^{m+1+k},r^{m+1-k}\right),\qquad \mathrm{Ker}(P_k\tilde{\mathscr{L}_m})\cap \mathrm{Ker}(P_k\mathscr{D}) =\mathrm{Span}(r^{m-1+k},r^{m-1-k}).
	\end{align}
	Furthermore, for all $\alpha\in \Z$, 
	\begin{align*}
		\mathscr{D}(r^{\alpha}\log(r))=\frac{1}{r}\partial_r\left(r^{\alpha}\log(r)\right)-r^{\alpha-2}\log(r)=(\alpha-1)r^{\alpha-2}\log(r)+r^{\alpha-2},
	\end{align*}
	so we deduce that the coefficients $\gamma_0^1$ and $\gamma_0^2$ vanish when $\delta\rightarrow 0$, as $|x|^{(1-m)-2}=|x|^{-(m+1)}\notin L^2(B(0,1))$. Furthermore, thanks to \eqref{orthogonalkernel} and the proof of Theorem \ref{indicielles1}, we deduce that whenever a power $\alpha=m+1+k,m+1-k,1-m+k,1-m+k$ satisfies
	\begin{align*}
		\alpha\leq 0,
	\end{align*}
	then the corresponding   coefficient $\gamma_k^j$ vanishes as ${\delta\rightarrow 0}$. 
	Notice that all powers $r^{m+1+k}$, $r^{m+1-k}$, $r^{1-m+k}$ and $r^{1-m-k}$ are all distinct, except when $|k|=m$, where the powers become either
	\begin{align*}
		r^{2m+1}, r, r, r^{1-2m}
	\end{align*}
	or
	\begin{align*}
		r,r^{2m+1},r^{1-2m},r.
	\end{align*}
	Notice also that the coefficient $\gamma$ in \eqref{newdev} also vanishes as $\gamma r\log(r)\,e^{\pm m\theta}\notin \mathrm{Ker}(\mathscr{L}_m)$ (and using the same argument as in the proof of Theorem \ref{explicit}).
	So we have a remaining coefficient in $\Re(\gamma_0z)$ in the expansion of $v_{\epsilon,\delta}^i$ as $\delta\rightarrow 0$, as $re^{\pm i m   \theta}\in \mathrm{Ker}(\mathscr{L}_m)\cap\mathrm{Ker}(\tilde{\mathscr{L}}_m)\cap \mathrm{Ker}(\mathscr{D})$. Finally, we deduce that as $\delta\rightarrow 0$, $v_{\epsilon,\delta}^i\conv{\delta\rightarrow 0} v_{\epsilon}^i\in C^{l}_{\mathrm{loc}}(\Sigma_{\epsilon}^i)$ for all $l\in \N$ such that 
	\begin{align*}
		v_{\epsilon}^i&=r^{m+1}\sum_{\substack{k\in \Z^{\ast}\\ k\geq -m}}^{}r^k(\gamma_k^1e^{ik\theta}+\gamma_{-k}^2e^{-ik\theta})+r^{1-m}\sum_{k\geq m}^{}r^k(\gamma_k^3e^{ik\theta}+\gamma_{-k}^4e^{-ik\theta})+\gamma_0^3r^{m+1}+\gamma_0^4r^{m+1}\log(r)\\
		&=2\,r^{m+1}\sum_{\substack{k\in \Z^{\ast}\\k\geq -m}}^{}\Re\left(\gamma_k^1z^k\right)+2\,r^{1-m}\sum_{k\geq m}^{}\Re\left(\gamma_k^3z^k\right)+\gamma_0^3r^{m+1}+\gamma_0^4r^{m+1}\log(r),
	\end{align*} 
	where we used $\gamma_{-k}^2=\bar{\gamma_k^1}$ and $\gamma_{-k}^4=\bar{\gamma_{k}^3}$. 
	The last expansion of $u_{\epsilon}^i$ follows directly from this estimate using \eqref{devconforme}.
\end{proof}

Finally, we obtain in the following theorem the expansion as $\epsilon\rightarrow 0$ of the previously obtained function $u_{\epsilon}^i$. Notice the shift of notation for $v_{\epsilon}^i$.

\begin{theorem}\label{expansionends2}
	Let $u_{\epsilon}^i\in C^{\infty}(\Sigma_{\epsilon}^i)$ be the function constructed in Theorem \ref{expansionends}, and $v_{\epsilon}^i\in C^{\infty}(\Sigma_{\epsilon}^i)$ be the global function such that $u_{\epsilon}^i=|\phi|^2v_{\epsilon}^i$. Then there exists $v_{0}^i\in W^{2,2}(\Sigma)$ such that up to a subsequence, 
	\begin{align*}
		v_{\epsilon}^i\conv{\epsilon\rightarrow 0}v_{0}^i\qquad \text{in}\;\, C^l(\Sigma\setminus\ens{p_1,\cdots,p_n})\;\, \text{for all}\;\, l\in \N.
	\end{align*}
	Furthermore, we have $v_0^i(p_j)=0$ for $j\neq i$, $v_0^i(p_i)=v(p_i)$, and for all $1\leq j\leq n$, and if $p_j$ has multiplicity $m\geq 1$,  $v_0^i$ admits the following expansion in $U_j$ for some $\gamma_{i,j}^0,\gamma_{i,j,k,l}\in \C$ ($k,l\in \N$) and $\gamma_{i,j}^1\in \R$
	\begin{align*}
		v_0^i(z)=v(p_i)\delta_{i,j}+\Re\left(\gamma_{i,j}^0z^m\right)+\sum_{m+1\leq k+l\leq 2m}^{}\Re\left(\gamma_{i,j,k,l}z^k\z^l\right)+\gamma_{i,j}^1|z|^{2m}\log|z|+O(|z|^{2m+1}\log|z|).
	\end{align*}
	Furthermore, if $m=1$, there exists $\gamma_{i,j}^0,\gamma_{i,j}^1\in \C$ and $\gamma_{i,j}^2,\gamma_{i,j}^3\in \R$ such that 
	\begin{align*}
		v_0^i(z)=v(p_j)\delta_{i,j}+\Re(\gamma_{i,j}^0z+\gamma_{i,j}^1z^2)+\gamma_{i,j}^{2}|z|^2+\gamma_{i,j}^3|z|^2\log|z|+O(|z|^3).
	\end{align*}
	In particular, for all $1\leq i\leq n$, the variation $\vec{v}_i=v_0^i\n_{\vec{\Psi}}$ is an admissible variation of the branched Willmore surface $\vec{\Psi}:\Sigma\rightarrow \R^3$.
\end{theorem}
\begin{proof}
	The first claim on $v_{\epsilon}^i$ follows directly from the uniform bound \eqref{awayvortices} and a standard diagonal argument. Furthermore, as $v_{\epsilon}^i=v=v(p_i)+O(\epsilon)$ on $\partial B_{\epsilon}(p_i)$, we deduce that $v_{0}^i(p_i)=v(p_i)$. Finally, the expansion in $U_j$ follows from Theorem, as 
	\begin{align*}
		u_{\epsilon}^i=\Re\left(\frac{c_{i,j}}{z^m}\right)+\sum_{1-m\leq k+l\leq 0}^{}\Re\Big(c_{i,j,k,l}z^{k}\z^l\Big)+a_{i,j}\log|z|+\psi_{\epsilon}(z)
	\end{align*}
	and as $|\phi|^2=\beta_0^2|z|^{-2m}(1+O(|z|))$ (for some $\beta_0>0$), we find that for some $\gamma_{i,j,\epsilon}^0,\gamma_{i,j,k,l,\epsilon}\in \C$ and $\gamma_{i,j,\epsilon}^1\in \R$
	\begin{align*}
		v_{\epsilon}^i=\Re\left(\gamma_{i,j,\epsilon}^0z^m\right)+\sum_{m+1\leq j+k\leq 2m}^{}\Re\left(\gamma_{i,j,k,l,\epsilon}z^k\z^l\right)+\gamma_{i,j,\epsilon}^1|z|^{2m}\log|z|+O(|z|^{2m+1}\log|z|),
	\end{align*}
	so as $\epsilon\rightarrow 0$, by the strong convergence $\gamma_{i,j,k,l,\epsilon}\rightarrow \gamma_{i,j,k,l}\in \C$ and we get the expected expansion. Finally, the indicial root analysis shows that 
	\begin{align*}
		\D^2v_0^i=\D^2\Re\left(\gamma_{i,j}^0z^m\right)+O(|z|^{m-1}\log|z|)=O(\log|z|)\in \bigcap_{p<\infty}L^p(\Sigma)
	\end{align*}
	and as $v_0^i\in C^{\infty}(\Sigma\setminus\ens{p_1,\cdots,p_n})$, we deduce that 
	\begin{align*}
		v\in \bigcap_{p<\infty}W^{2,p}(\Sigma)
	\end{align*}
	and this concludes the proof of the theorem. 
\end{proof}
\begin{rem}
	We emphasize that the variations $v^i_0\in W^{2,2}(\Sigma)$ are admissible at a branch point $p\in \Sigma$ of order $\theta_0\geq 1$ corresponds to an end $p_j$ (for some $1\leq j\leq n$) of multiplicity $m=\theta_0\geq 1$, and the previous theorem shows that in $U_j$
	\begin{align*}
		v_0^i(z)=v(p_i)\delta_{i,j}+\Re\left(\gamma_{i,j}z^{\theta_0}\right)+O(|z|^{\theta_0+1}\log|z|),
	\end{align*}
	so these variations are indeed admissible by the discussion in Section \ref{admissible}. For more details on this important technical point, we refer to \cite{index3}. Notice that in general, at a branch point of multiplicity $m\geq 2$, we have 
	\begin{align*}
		\D^{m+1}v\in L^{\infty}(B(0,1))
	\end{align*}
	which implies that
	\begin{align*}
		v\in C^{m,1}(B(0,1))
	\end{align*}
	while for $m=1$, 
	\begin{align*}
		\D^2v=O(\log|z|)
	\end{align*}
	so that 
	\begin{align*}
		v\in \bigcap_{\alpha<1}C^{1,\alpha}(B(0,1)),
	\end{align*}
	but $v\notin C^{1,1}(B(0,1))$ in general. 
\end{rem}

\begin{defi}
	For all admissible variation $v\in W^{2,2}(\Sigma)$ of $\vec{\Psi}$ we denote by $u_0^i=|\phi|^2v_0^i$, where $v_0^i\in W^{2,2}(\Sigma)$ is the admissible variation of $\vec{\Psi}$ constructed in Theorem \ref{expansionends2}.
\end{defi}

\section{Renormalised energy for minimal surfaces with embedded ends}\label{renormalised}

\subsection{Explicit computation of the the singular energy}

First recall the definition of flux of a complete minimal surface. 
\begin{defi}\label{flux}
	Let $\Sigma$ be a closed Riemann surface, $p_1,\cdots,p_n\in \Sigma$ be fixed points and $\phi:\Sigma\rightarrow \R^d$ be a complete minimal surface with finite total curvature. For all $1\leq j\leq n$, we define the flux of $\phi$ at $p_j$ by
	\begin{align*}
		\mathrm{Flux}(\phi,p_j)=\frac{1}{\pi}\Im\int_{\gamma}\partial\phi\in \R^d
	\end{align*}
	where $\gamma\subset \Sigma\setminus\ens{p_1,\cdots,p_n}$ is a fixed contour around $p_j$ that does not enclosed other points $p_k$ for some $k\neq j$.
\end{defi}
By the Weierstrass parametrisation, we have at an end of multiplicity $m\geq 1$ for some $\vec{A}_0\in \C^d\setminus\ens{0}$ and $\vec{A}_1,\cdots,\vec{A}_m\in \C^d$ and $\vec{\gamma}_0\in \R^d$
\begin{align*}
	\phi(z)=\sum_{j=0}^{m}\Re\left(\frac{\vec{A}_j}{z^{m-j}}\right)+\vec{\gamma}_0\log|z|+O(|z|),
\end{align*}
and we compute
\begin{align*}
	\int_{S^1}\partial\phi=\int_{S^1}\frac{\vec{\gamma}_0}{2}\frac{dz}{z}=\pi i\,\vec{\gamma}_0.
\end{align*}
Therefore, we have
\begin{align*}
	\mathrm{Flux}(\phi,p_j)=\vec{\gamma}_0\in \R^d
\end{align*}
is a well-defined quantity independent of the chart.

\begin{theorem}\label{embeddedends}
	Let $\Sigma$ a compact Riemann surface, $\phi:\Sigma\setminus\ens{p_1,\cdots,p_n}\rightarrow \R^3$ a minimal surface with $n$ embedded ends $p_1,\cdots,p_n\in\Sigma$ and exactly $m$ catenoid ends $p_1,\cdots,p_{m}\in\Sigma$ ($0\leq m\leq n$). Let $\vec{\Psi}:\Sigma\rightarrow S^3$ a Willmore surface in $S^3$ obtained by inverse stereographic projection of $\phi$. Then the index quadratic form $Q_{\vec{\Psi}}:W^{2,2}(\Sigma)\rightarrow\R$ of $\vec{\Psi}$ satisfies for all $v\in  C^2(\Sigma)$ the identity
	\begin{align}\label{GLbranche}
	Q_{\vec{\Psi}}(v)=&\lim\limits_{\epsilon\rightarrow 0}\bigg\{\frac{1}{2}\int_{\Sigma_\epsilon^2}^{}(\Delta_gu-2K_gu)^2d\vg-8\pi\sum_{i=1}^{n}\frac{\alpha_i^2}{\epsilon^2}v^2(p_i)-16\pi\sum_{j=1}^{m}\beta_j^2\log\left(\frac{1}{\epsilon}\right)v^2(p_j)\nonumber\\
	&+16\pi\sum_{j=1}^{m}\beta_j^2v^2(p_j)\bigg\},
	\end{align}
	where $u=|\phi|^2v$, and $\Sigma_\epsilon=\Sigma\setminus\un{i=1}{n}\bar{B}_{\epsilon}(p_j)$, where the $\bar{B}_{\epsilon}(p_i)$ are chosen as in \cite{indexS3} (with respect to a fixed covering $U_1,\cdots,U_n$ of $p_1,\cdots,p_n$), and 
	\begin{align*}
		\beta_j=|\mathrm{Flux}(\phi,p_j)|.
	\end{align*} 
\end{theorem}
\begin{proof}
	We make the computation of the residue at a catenoid ends, as the residue at a planar end will be the same if we simply take the formula by plugging $0$ at the place of the catenoid residue. This simple fact will become clear at the end of the proof. there exists $\beta_j\in\R\setminus\ens{0}$ ($m\leq j\leq m+n$) such that in a conformal local chart $D^2\subset\R^2\rightarrow B_r(p_j)$,
	\begin{align}
	\phi(x)=\left(\alpha_j\frac{x}{|x|^2}+O(|x|),\beta_j\log|x|+a\cdot x+O(|x|^2)\right)
	\end{align}
	where $a\in\R^2$. 
	Furthermore, in the remaining terms, as $Q:W^{2,2}(\Sigma)\rightarrow \R$ is continuous, and by Sobolev embedding theorem, $W^{2,2}(\Sigma)\hookrightarrow C^0(S^2)$ but does not embed in $C^1(\Sigma)$, we deduce that $Q$ cannot depend on the higher derivatives of $v$ at $p_i$. Therefore, one only needs to compute
	\begin{align*}
	\int_{S^1_r}v^2\left(\Delta_g|\phi|^2\star d|\phi|^2-\frac{1}{2}\star d|d|\phi|^2|^2\right)=\int_{S^1_r}v^2\star d\left(4|\phi|^2-\frac{1}{2}|d|\phi|^2|^2\right)+O\left(\frac{\log(r)}{r}\right)
	\end{align*}
	because as $\phi$ is harmonic, we have
	\begin{align*}
	\Delta_g|\phi|^2=4
	\end{align*}
	while
	\begin{align*}
	|\phi|^2=|x|^2+\beta_j^2\log^2\left(\frac{|x|}{\alpha_j}\right)+O\left(\frac{\log|x|}{|x|}\right)
	\end{align*}
	and
	\begin{align*}
	\D|\phi|^2=2x+2\beta_j^2\frac{x}{|x|^2}\log\left(\frac{|x|}{\alpha_j}\right)+O\left(\frac{\log|x|}{|x|^2}\right)
	\end{align*}
	so
	\begin{align*}
	|d|\phi|^2|^2=4\left(|x|^2+2\beta_j^2\log\left(\frac{|x|}{\alpha_j}\right)+O\left(\frac{\log|x|}{|x|}\right)\right)
	\end{align*}
	therefore
	\begin{align*}
	4|\phi|^2-\frac{1}{2}|d|\phi|^2|&=4\left(|x|^2+\beta_j^2\log^2\left(\frac{|x|}{\alpha_j}\right)\right)-2\left(|x|^2+2\beta_j^2\log\left(\frac{|x|}{\alpha_j}\right)\right)+O\left(\frac{\log|x|}{|x|}\right)\\
	&=2|x|^2+4\beta_j^2\log\left(\frac{|x|}{\alpha_j}\right)\left(\log\left(\frac{|x|}{\alpha_j}\right)-1\right)+O\left(\frac{\log|x|}{|x|}\right)
	\end{align*}
	We deduce that
	\begin{align*}
	\star d\left(4|\phi|^2-\frac{1}{2}|d|\phi|^2|^2\right)=4(x_1dx_2-x_2dx_1)+8\beta_j^2\log\left(\frac{|x|}{\alpha_j}\right)\frac{x_1dx_2-x_2dx_1}{|x|^2}-4\beta_j^2\frac{x_1dx_2-x_2dx_1}{|x|^2}
	\end{align*}
	Finally, one gets
	\begin{align}\label{residu}
	\int_{S_r^1}v^2\left(\Delta_g|\phi|^2\star d|\phi|^2-\frac{1}{2}\star d|d|\phi|^2|^2\right)&=8\pi r^2v^2(p_j)+16\pi\beta_j^2\log\left(\frac{r}{\alpha_j}\right)v^2(p_j)-8\pi\beta_j^2v^2(p_j)+O\left(\frac{\log r}{r}\right)
	\end{align}
	and for a planar end, we have the same expression with $\beta_j=0$. This translates if $u=|\phi|^2v$ as
	\begin{align*}
	Q_\epsilon(v)&=\frac{1}{2}\int_{\Sigma_\epsilon}(\Delta_gu-2K_gu)^2d\vg-\int_{\Sigma_\epsilon}^{}d\left((\Delta_gu+2K_gu)\star dw-\frac{1}{2}\star d|du|_g^2\right)\\
	&=\frac{1}{2}\int_{\Sigma_\epsilon}^{}(\Delta_gu-2K_gu)^2d\vg-8\pi\sum_{i=1}^{n+m}\frac{\alpha_i^2}{\epsilon^2}v^2(p_i)-16\pi\sum_{j=1}^{m}\beta_j^2\log\left(\frac{1}{\epsilon}\right)v^2(p_j)\\
	&+16\pi\sum_{j=1}^{m}\beta_j^2v^2(p_j)+O(\epsilon\log\epsilon)
	\end{align*}
	and this gives
	\begin{align*}
	Q(v)&=\lim\limits_{\epsilon\rightarrow 0}\frac{1}{2}\int_{\Sigma_\epsilon}^{}(\Delta_gu-2K_gu)^2d\vg-8\pi\sum_{i=1}^{n}\frac{\alpha_i^2}{\epsilon^2}v^2(p_i)-16\pi\sum_{j=1}^{m}\beta_j^2\log\left(\frac{1}{\epsilon}\right)v^2(p_j)\\
	&+16\pi\sum_{j=1}^{m}\beta_j^2v^2(p_j)
	\end{align*}
	which concludes the proof of the theorem.	
\end{proof}
\begin{rem}\label{rem1}
	We can easily see that the preceding expression if well-defined directly, even is we already know that it is (as this quantity is the second variation of a compact Willmore surface for an admissible direction). Indeed, 
	\begin{align*}
	\mathscr{L}_g|\phi|^2&=\Delta_g|\phi|^2-2K_g|\phi|^2=4+2\frac{\beta_j^2}{|x|^2}+O\left(\frac{1}{|x|^4}\right)\\
	\mathscr{L}_g^2|\phi|^2&=8\frac{\beta_j^2}{|x|^4}+2\frac{\beta_j^2}{|x|^4}\left(4+2\frac{\beta_j^2}{|x|^2}\right)+O\left(\frac{1}{|x|^6}\right)\\
	&=16\frac{\beta_j^2}{|x|^4}+O\left(\frac{1}{|x|^6}\right)
	\end{align*}
	Therefore
	\begin{align*}
	\frac{1}{2}\int_{B_r\setminus B_1(0)}\left(\mathscr{L}_g|\phi|^2\right)^2d\vg
	&=\frac{1}{2}\int_{B_r\setminus B_1(0)}\left(4+2\frac{\beta_j^2}{|x|^2}\right)^2dx+O(1)
	=\frac{1}{2}\int_{B_r\setminus B_1(0)}^{}\left(16+16\frac{\beta_j^2}{|x|^2}\right)dx\\
	&=8\pi r^2+16\beta_j^2\log r+O(1)
	\end{align*}
	and this proves by \eqref{residu} that \eqref{GLbranche} makes sense.
\end{rem}

\subsection{Renormalised energy identity}

\begin{theorem}\label{explicit}
	Under the hypothesis of Theorem \ref{ta}, assume that $\phi$ has embedded ends. 
	There exists a symmetric $\{\lambda_{i,j}\}_{1\leq i,j\leq n}$ with zero diagonal terms independent of $v$, and a function $v_0\in W^{2,2}(\Sigma)$ vanishing on $\ens{p_1,\cdots,p_n}$ such that
	\begin{align*}
	Q(u)=Q(u_0)+16\pi\sum_{i=1}^{m}\beta_i^2v^2(p_i)+4\pi\sum_{1\leq i,j\leq n}^{}\lambda_{i,j}v(p_i)v(p_j)
	\end{align*}
	where $u_0=|\phi|^2v_0$ and $Q(u_0)\geq 0$ and $Q(u_0)=0$ if $\mathscr{L}_g^2 u=0$. In particular we deduce that the index cannot be more that $n-1$.
\end{theorem}
\begin{proof}
	We fix $\epsilon>0$ small enough such that the ball $\ens{\bar{B}_{2\epsilon}(p_i)}_{1\leq i\leq n}$ are disjoint, and we define the following symmetric bilinear form $B_\epsilon:\mathrm{W}^{2,2}(\Sigma_\epsilon)\times \mathrm{W}^{2,2}(\Sigma_\epsilon)\rightarrow\R$
	\begin{align*}
	B_\epsilon(u_1,u_2)=\int_{\Sigma_\epsilon}\mathscr{L}_g u_1\, \mathscr{L}_gu_2\,d\vg
	\end{align*}
	and $Q_\epsilon:\mathrm{W}^{2,2}(S^2_\epsilon)\rightarrow\R$ the associated quadratic form. We note that
	\begin{align*}
	Q(u)=\lim\limits_{\epsilon\rightarrow 0}Q_\epsilon(u)-8\pi\sum_{i=1}^{n}\frac{\alpha_i^2}{\epsilon^2}v^2(p_i)-16\pi\sum_{j=1}^{m}\beta_j^2\log\left(\frac{1}{\epsilon}\right)v^2(p_j)
	+16\pi\sum_{j=1}^{m}\beta_j^2v^2(p_j)
	\end{align*}
	if $u=|\phi|^2 v$. We now define
	$\displaystyle
	u_\epsilon=u-\sum_{i=1}^{n}u_\epsilon^i
	$
	\begin{align}\label{decomposition}
	Q_\epsilon(u)=Q_\epsilon\left(u_\epsilon+\sum_{i=1}^{n}u_\epsilon^i\right)
	=Q_\epsilon(u_\epsilon)+\sum_{i=1}^{n}Q_\epsilon(u_\epsilon^i)+\sum_{i=1}^{n}B_\epsilon(u_\epsilon,u_\epsilon^i)+\frac{1}{2}\sum_{1\leq i\neq j\leq n}^{}B_\epsilon(u_\epsilon^i,u_\epsilon^j).
	\end{align}
	\textbf{Step 1} : Estimation of $Q_\epsilon(u_\epsilon)$. We first remark that $Q_\epsilon(u_\epsilon)$ cannot depend on the derivatives of $v$ at $p_1,\cdots,p_n$ by Sobolev embedding theorem. Therefore, each time we differentiate $v_\epsilon^i$, we know that analogous cancellations as observed by the explicit computations in \cite{indexS3} will actually make these residues vanish as $\epsilon\rightarrow 0$. Whenever one of these terms occur, we shall neglect them.
	
	For all $1\leq i\leq n$, let $v_\epsilon^i\in C^\infty(\bar{B_{2\epsilon}\setminus \bar{B}_\epsilon(p_i)})$ such that $u_\epsilon^i=|\phi|^2v_\epsilon^i$ on $\bar{B}_{2\epsilon}(p_i)\setminus B_\epsilon(p_i)$.  We fix a chart $D^2\rightarrow B_r(p_i)$. We recall that close to $p_i$, we have
	\begin{align*}
	|\phi(x)|^2&=\frac{\alpha_i^2}{|x|^2}+\beta_i^2\log^2|x|+O(|x|\log|x|)
	\end{align*}
	Then we deduce by the Dirichlet boundary condition that
	\begin{align*}
	v_\epsilon^i=v\quad \text{on}\;\, \partial B_\epsilon(p_i)
	\end{align*}
	and
	\begin{align*}
	\partial_\nu u_\epsilon^i=\partial_\nu|\phi|^2 v_\epsilon^i+|\phi|^2 \partial_\nu v_\epsilon^i=\partial_\nu|\phi|^2 v
	+|\phi|^2 \partial_\nu v_\epsilon^i\quad \text{on}\;\,\partial B_\epsilon(p_i)
	\end{align*}
	and as
	\begin{align*}
	\partial_\nu u_\epsilon^i=\partial_\nu (|\phi|^2) v+|\phi|^2\partial_\nu v\quad \text{on}\;\, \partial B_\epsilon(p_i)\\
	\end{align*}
	we also have
	\begin{align*}
	\partial_\nu v_\epsilon^i=\partial_\nu v\quad \text{on}\;\,\partial B_\epsilon(p_i)
	\end{align*}
	so
	\begin{align*}
	u_\epsilon^i&=\left(\frac{\alpha_i^2}{\epsilon^2}+\beta_i^2\log^2(\epsilon)\right)v+O(1)\\
	\partial_\nu u_\epsilon^i&=\left(-2\frac{\alpha_i^2}{\epsilon^3}+2\frac{\beta_i^2}{\epsilon}\log\epsilon\right)v+\left(\frac{\alpha_i^2}{\epsilon^2}+\beta_i^2\log^2\epsilon\right)\partial_{\nu}v+O(1)
	\end{align*}
	then on $B_{2\epsilon}(p_i)\setminus \bar{B}_\epsilon(p_i)$, we have
	\begin{align*}
	\Delta_gu_\epsilon^i&=\Delta_g(|\phi|^2v_\epsilon^i)
	=4v_\epsilon^i+2\s{d|\phi|^2}{d v_\epsilon^i}+|\phi|^2\Delta_g v_\epsilon^i\\
	&=4v_\epsilon^i-4\frac{|x|^4}{\alpha_i^2}\left(\frac{\alpha_i^2}{|x|^3}+\beta_i^2\frac{\log\left(\frac{1}{|x|}\right)}{|x|}\right)\frac{x}{|x|}\cdot \D v_\epsilon^i+|x|^2\Delta v_\epsilon^i+O(|x|^4\log^2|x|)\\
	&=4v_\epsilon^i-4|x|\left(1+\frac{\beta_i^2}{\alpha_i^2}|x|^2\log\left(\frac{1}{|x|}\right)\right)\frac{x}{|x|}\cdot \D v_\epsilon^i+|x|^2 \Delta v_\epsilon^i+O(|x|^4\log^2|x|)
	\end{align*}
	and
	\begin{align*}
	-2K_gu_\epsilon^i=2\frac{\beta_i^2}{\alpha_i^2}|x|^2 v_\epsilon^i+O(|x|^4)
	\end{align*}
	so on $\partial B_\epsilon(p_i)$,
	\begin{align*}
	\mathscr{L}_gu_\epsilon^i=2\left(2+\frac{\beta_i^2}{\alpha_i^2}\epsilon^2\right)v-4\epsilon\left(1+\frac{\beta_i^2}{\alpha_i^2}\epsilon^2\log\left(\frac{1}{\epsilon}\right)\right)\partial_\nu v+\epsilon^2\Delta v_\epsilon^i+O(\epsilon^4\log^2\epsilon)
	\end{align*}
	Therefore
	\begin{align*}
	\frac{x}{|x|}\cdot \D\left(\Delta_g u_\epsilon^i\right)&=4\frac{x}{|x|}\cdot \D v_\epsilon^i-4\left(1+\frac{\beta_i^2}{\alpha_i^2}|x|^2\log\left(\frac{1}{|x|}\right)\right)\frac{x}{|x|}\cdot \D v_\epsilon^i\\
	&-4\frac{\beta_i^2}{\alpha_i^2}|x|^2\left(1+2\log\left(\frac{1}{|x|}\right)\right)\frac{x}{|x|}\cdot \D v_\epsilon^i\\
	&-4|x|\left(1+\frac{\beta_i^2}{\alpha_i^2}|x|^2\log\left(\frac{1}{|x|}\right)\right)\left(\frac{x}{|x|}\right)^t D^2 v_\epsilon^i\left(\frac{x}{|x|}\right)+2|x|\Delta v_\epsilon^i+|x|x\cdot \D\Delta v_\epsilon^i+O(|x|^3\log^2|x|)
	\end{align*}
	while
	\begin{align*}
	\frac{x}{|x|}\cdot \D(-2K_gu_\epsilon^i)=4\frac{\beta_i^2}{\alpha_i^2}|x|v_\epsilon^i+2\frac{\beta_i^2}{\alpha_i^2}|x|x\cdot \D v_\epsilon^i
	\end{align*}
	so on $\partial B_\epsilon(p_j)$, we have
	\begin{align}\label{precise}
	\partial_\nu (\mathscr{L}_gu_\epsilon^i)&=4\epsilon\frac{\beta_i^2}{\alpha_i^2}v-4\frac{\beta_i^2}{\alpha_i^2}\epsilon^2\left(\frac{1}{2}+3\log\left(\frac{1}{\epsilon}\right)\right)\partial_\nu v-4\epsilon\left(1+\frac{\beta_i^2}{\alpha_i^2}\log\left(\frac{1}{\epsilon}\right)\right)\left(\frac{x}{\epsilon}\right)^tD^2v_\epsilon^i\left(\frac{x}{\epsilon}\right)+2\epsilon\Delta v_\epsilon^i\nonumber\\
	&+\epsilon^2\partial_{\nu}\Delta v_\epsilon^i+O(\epsilon^3\log^2\epsilon)
	\end{align}
	as we can neglect all terms containing derivatives of $v$, we can replace $v$ by $v(p_i)$ and replace $\partial_\nu v$ by $0$, which gives
	\begin{align}
	&u_\epsilon^i\partial_\nu(\mathscr{L}_g u_\epsilon^i)-(\partial_\nu u_\epsilon^i)\mathscr{L}_gu_\epsilon^i=\left(\frac{\alpha_i^2}{\epsilon^2}+\beta_i^2\log^2(\epsilon)\right)v(p_i)\bigg(4\epsilon\frac{\beta_i^2}{\alpha_i^2}v(p_i)-4\epsilon\left(1+\frac{\beta_i^2}{\alpha_i^2}\log\left(\frac{1}{\epsilon}\right)\right)\left(\frac{x}{\epsilon}\right)^tD^2v_\epsilon^i\left(\frac{x}{\epsilon}\right)\nonumber\\&+2\epsilon\Delta v_\epsilon^i\bigg)\nonumber
	+2\left(\frac{\alpha_i^2}{\epsilon^3}+\frac{\beta_i^2}{\epsilon}\log\left(\frac{1}{\epsilon}\right)\right)v(p_i)\left(\left(4+2\frac{\beta_i^2}{\alpha_i^2}\epsilon^2\right)v(p_i)+\epsilon^2\Delta v_\epsilon^i\right)+O(\log^2\epsilon)\nonumber\\
	&=8\frac{\alpha_i^2}{\epsilon^3}v^2(p_i)+8\frac{\beta_i^2}{\epsilon}\log\left(\frac{1}{\epsilon}\right)v^2(p_i)+8\frac{\beta_i^2}{\epsilon}v^2(p_i)
	+2\frac{\alpha_i^2}{\epsilon}\Delta v_\epsilon^i v(p_i)\nonumber\\
	&+\frac{\alpha_i^2}{\epsilon^2}\left(-4\epsilon\left(1+\frac{\beta_i^2}{\alpha_i^2}\log\left(\frac{1}{\epsilon}\right)\right)\left(\frac{x}{\epsilon}\right)^t D^2 v_\epsilon^i\,\left(\frac{x}{\epsilon}\right)+2\epsilon\Delta v_\epsilon^i\right)v(p_i)+O(\log^2\epsilon).
	\end{align}
	Furthermore, we note that if
	\begin{align*}
	D^2 v_\epsilon^i=\begin{pmatrix}
	a_{1,1}&a_{1,2}\\
	a_{2,1}&a_{2,2}
	\end{pmatrix}+O(\epsilon)
	\end{align*}
	then
	\begin{align*}
	\int_{\partial B_\epsilon(p_i)}^{}x^t D^2 v_\epsilon^i x\,d\hh^1
	&=\epsilon^3\int_{0}^{2\pi}\left(a_{1,1}\cos^2(\theta)+a_{2,2}\sin^2(\theta)+\frac{a_{1,2}+a_{2,1}}{2}\sin(2\theta)\right)d\theta+O(\epsilon^4)\\
	&=\pi(a_{1,1}+a_{2,2})\epsilon^3+O(\epsilon^4)
	\end{align*}
	while
	\begin{align*}
	\int_{\partial B_\epsilon(p_i)}\Delta v_\epsilon^i=2\pi(a_{1,1}+a_{2,2})\epsilon+O(\epsilon^2)
	\end{align*}
	so if we write $\delta_i=a_{1,1}+a_{2,2}$, we have
	\begin{align*}
	&\int_{\partial B_\epsilon(p_i)}^{}\frac{\alpha_i^2}{\epsilon^2}\left(-4\epsilon\left(1+\frac{\beta_i^2}{\alpha_i^2}\log\left(\frac{1}{\epsilon}\right)\right)\left(\frac{x}{\epsilon}\right)^t D^2 v_\epsilon^i\,\left(\frac{x}{\epsilon}\right)+2\epsilon\Delta v_\epsilon^i\right)d\hh^1\\
	&=-4\left(\alpha_i^2+\beta_i^2\log\left(\frac{1}{\epsilon}\right)\right)
	\pi\delta_i+2\alpha_i^2\cdot 2\pi\delta_i+O(\epsilon\log\epsilon)\\
	&=-4\pi \beta_i^2\log\left(\frac{1}{\epsilon}\right)\delta_i+O(\epsilon\log\epsilon)
	\end{align*}
	we obtain finally
	\begin{align}\label{diagonale}
	&Q_\epsilon(u_\epsilon^i)=\frac{1}{2}\int_{\partial B_\epsilon(p_i)}^{}u_\epsilon^i\partial_\nu(\mathscr{L}_g u_\epsilon^i)-(\partial_\nu u_\epsilon^i)\mathscr{L}_g u_\epsilon^i d\hh^1\nonumber\\
	&=\frac{8\pi \alpha_i^2}{\epsilon^2}v^2(p_i)+8\pi\beta_i^2\log\left(\frac{1}{\epsilon}\right)v^2(p_i) +8\pi\beta_i^2v^2(p_i)-2\pi\beta_i^2\log\left(\frac{1}{\epsilon}\right)\delta_i v(p_i)+2\pi\beta_i^2\delta_iv(p_i)+O(\epsilon\log\epsilon).
	\end{align}
	Now thanks to the asymptotic behaviour of $\ens{v_\epsilon^i}_{1\leq i\leq n}$, we know that $B_\epsilon(u_\epsilon,u_\epsilon^i)$, and $B_\epsilon(u_\epsilon^i,u_\epsilon^j)$ are bounded terms, so for the energy to be finite, we must have 
	\begin{align*}
	Q_\epsilon(u_\epsilon^i)=\frac{8\pi\alpha_i^2}{\epsilon^2}v^2(p_i)+16\pi\beta_i^2\log\left(\frac{1}{\epsilon}\right)+O(1)
	\end{align*}
	which imposes
	\begin{align*}
	\delta_i=-4v(p_i)
	\end{align*}
	and we get
	\begin{align}
	Q_\epsilon(u_\epsilon^i)=\frac{8\pi\alpha_i^2}{\epsilon^2}v^2(p_i)+16\pi\beta_i^2\log\left(\frac{1}{\epsilon}\right)v^2(p_i)+O(\epsilon\log^2\epsilon).
	\end{align}
	
	\textbf{Step 2} : Estimation of $B_\epsilon(u_\epsilon^i,u_\epsilon^j)$ for $i\neq j$.
	
	For all $1\leq i,j\leq n$, and $k\neq i,j$ we have by Theorem \ref{indicielles1}
	\begin{align*}
	\int_{\partial B_\epsilon(p_k)}^{}u^i_\epsilon\partial_\nu(\mathscr{L}_g u_\epsilon^j)d\hh^1=\int_{\partial B_\epsilon(p_k)}^{}O\left(\frac{1}{|x|}\right)O(|x|^2)d\hh^1=O(\epsilon^2\log\epsilon)
	\end{align*}
	and likewise
	\begin{align*}
	\int_{\partial B_\epsilon(p_k)}^{}\partial_\nu(u_\epsilon^i)\mathscr{L}_gu_\epsilon^j d\hh^1=\int_{\partial B_\epsilon(p_k)}^{}O\left(\frac{1}{|x|^2}\right)O(|x|^3)d\hh^1=O(\epsilon^2)
	\end{align*}
	therefore
	\begin{align}\label{error}
	\sum_{k\neq i,j}^{}\int_{\partial B_\epsilon(p_k)}\left(u_\epsilon^i \partial_\nu(\mathscr{L}_g u_\epsilon^j)-\partial_\nu(u_\epsilon^i)\mathscr{L}_g u_\epsilon^j\right)d\hh^1=O(\epsilon^2\log\epsilon)
	\end{align}
	So we need only to consider the boundary integrals for $B_\epsilon(p_i)$ and $B_\epsilon(p_j)$. We have up to $O(\epsilon^3\log\epsilon)$ error terms by \eqref{error}
	\begin{align*}
	\int_{\Sigma_\epsilon}\mathscr{L}_gu_\epsilon^i\mathscr{L}_gu_\epsilon^j&=\int_{\partial B_\epsilon(p_i)}\left(u_\epsilon^i\partial_\nu (\mathscr{L}_g u_\epsilon^j)-(\partial_\nu u_\epsilon^i)\mathscr{L}_gu_\epsilon^j \right)d\hh^1+\int_{\partial B_\epsilon(p_j)}\left(u_\epsilon^i\partial_\nu (\mathscr{L}_gu_\epsilon^j)-(\partial_\nu u_\epsilon^i)\mathscr{L}_g u_\epsilon^j \right)d\hh^1
	\end{align*}
	and
	\begin{align*}
	\int_{\partial B_\epsilon(p_i)}\left(u_\epsilon^i\partial_\nu (\mathscr{L}_g u_\epsilon^j)-(\partial_\nu u_\epsilon^i)\mathscr{L}_gu_\epsilon^j \right)d\hh^1=\int_{\partial B_\epsilon(p_i)}^{}O\left(\frac{1}{|x|^2}\right)O(|x|^2)-O\left(\frac{1}{|x|^3}\right)O(|x|^3)d\hh^1=O(\epsilon)
	\end{align*}
	so by symmetry, we have
	\begin{align*}
	\frac{1}{2}B_\epsilon(u_\epsilon^i,u_\epsilon^j)&=\frac{1}{2}\int_{\partial B_\epsilon(p_j)}\left(u_\epsilon^i\partial_\nu (\mathscr{L}_gu_\epsilon^j)-(\partial_\nu u_\epsilon^i)\mathscr{L}_g u_\epsilon^j \right)d\hh^1+O(\epsilon)\\
	&=\frac{1}{2}\int_{\partial B_\epsilon(p_i)}\left(u_\epsilon^j\partial_\nu(\mathscr{L}_gu_\epsilon^i)d\mathscr{H}^1-\left(\partial_\nu u_\epsilon^i\right)\mathscr{L}_g u_\epsilon^i\right)+O(\epsilon).
	\end{align*}
	From now on, we find useful to use complex notations. Recall the expansion on $\partial B_{\epsilon}(p_j)$
	\begin{align*}
		u_{\epsilon}^i=\Re\left(\frac{c_{i,j}}{z}+d_{i,j}\frac{\z}{z}\right)+a_{i,j}\log|z|+b_{i,j}+O(|z|).
	\end{align*}
	Furthermore, as 
	\begin{align*}
		|\phi|^2=\frac{\alpha_j^2}{|z|^2}+\beta_i^2\log^2|z|+O(|z|\log|z|)
	\end{align*}
	we have 
	\begin{align*}
		u_{\epsilon}^j=|\phi|^2(v(p_j)+\Re(\gamma z)+O(|z|^2))=\frac{\alpha_j^2}{|z|^2}\left(v(p_j)+\Re(\gamma z)+O(|z|^2\log^2|z|)\right).
	\end{align*}
    Furthermore, we have
    \begin{align*}
    	e^{2\lambda}=\frac{\alpha_i^2}{|z|^4}\left(1+O(|z|)\right),
    \end{align*}
    so we deduce that
    \begin{align*}
    	&\Delta_g u_{\epsilon}^j=4v(p_j)+O(|z|^2\log^2|z|).
    \end{align*}
    Furthermore, as $K_g=O(|z|^4)$, we have also
    \begin{align*}
    	K_g u_{\epsilon}^j=O(|z|^2),
    \end{align*}
    so we get
    \begin{align*}
    	&\mathscr{L}_gu_{\epsilon}^j=4v(p_j)+O(|z|^2\log|z|)\\
    	&\partial_{\nu}(\mathscr{L}_g u_{\epsilon}^j)=O(|z|\log^2|z|),
    \end{align*}
    so we recover a weak form of \eqref{precise} (however sufficient for our purpose here). 
	Now we note that
	\begin{align*}
	\int_{\partial B_\epsilon(p_j)}^{}u_\epsilon^i\partial_\nu(\mathscr{L}_g u_\epsilon^j)d\hh^1=\int_{\partial B_\epsilon(p_j)}^{}\left(\Re\left(\frac{c_{\epsilon}^i}{z}\right)+O(\log|z|)\right)O(|z|\log^2|z|)d\hh^1=O(\epsilon\log^2\epsilon).
	\end{align*}
	Now, notice that for all smooth $\varphi:B(0,1)\rightarrow \R$, we have
	\begin{align*}
	\partial_{\nu}\varphi&=\frac{x_1}{|x|}\cdot\p{x_1}\varphi+\frac{x_2}{|x|}\cdot \p{x_2}\varphi\\
	&=\frac{1}{|z|}\left(\frac{(z+\z)}{2}(\partial+\bar{\partial})\varphi+\frac{(z-\z)}{2i}i(\partial-\bar{\partial})\varphi\right)=\frac{1}{|z|}(z\partial \varphi+\z\bar{\partial}\varphi)=\frac{2}{|z|}\Re\left(z\p{z}\varphi\right).
	\end{align*}
	Therefore, we have (as $\z/z$ has no radial component)
	\begin{align*}
	\partial_\nu u_\epsilon^i=-\frac{2}{|z|}\Re\left(\frac{c_{i,j}}{z}\right)+\frac{a_{i,j}}{|z|}+O(1),
	\end{align*}
	while
	\begin{align*}
	\mathscr{L}_gu_\epsilon^j=4v(p_j)+O(|z|^2\log|z|),
	\end{align*}
	therefore
	\begin{align*}
	\int_{\partial B_\epsilon(p_j)}\partial_\nu u_\epsilon^i\mathscr{L}_gu_\epsilon^j\,d\hh^1&=\int_{\partial B_\epsilon(p_j)}\left(-\frac{c_{i,j}}{2|z|z}-\frac{\bar{c_{i,j}}}{2|z|\z}+\frac{a_{i,j}}{|z|}+O(\log|z|)\right)(4v(p_j)+O(|z|^2\log|z|))d\hh^1\\
	&=8\pi a_{i,j}v(p_j)+O(\epsilon\log^2\epsilon).
	\end{align*}
	Therefore by symmetry
	\begin{align*}
	\frac{1}{2}B_\epsilon(u_\epsilon^i,u_\epsilon^j)=-4\pi a_{i,j}v(p_j)+O(\epsilon\log^2\epsilon)=-4\pi a_{j,i}(p_i)+O(\epsilon\log^2\epsilon)
	\end{align*}
	so
	\begin{align*}
	a_{\epsilon}^i(p_j)v(p_j)=a_{\epsilon}^{j}v(p_i)
	\end{align*}
	therefore there exists $\lambda_{i,j}\in\R$ such that $a_{i,j}=\lambda_{i,j}v(p_i)$, $a_{j,i}=\lambda_{i,j}v(p_j)$ and we deduce that
	\begin{align}\label{horsdiagonale}
	\frac{1}{2}B_\epsilon(p_i,p_j)=-4\pi\lambda_{i,j}v(p_i)v(p_j)+O(\epsilon\log\epsilon).
	\end{align}
	We note that these notations imply that for $r>0$ small enough, for all $1\leq i\leq n$, for all $j\neq i$ we have on any conformal chart $D^2\rightarrow \bar{B}_r(p_j)$
	\begin{align*}
	u_\epsilon^i(z)=\Re\left(\frac{c_{i,j}}{z}+d_{i,j}\frac{\z}{z}\right)+\lambda_{i,j}v(p_i) \log |z|+b_{i,j}+O(|z|).
	\end{align*}
	
	\textbf{Step 3} : Estimation of $B_\epsilon(u_\epsilon,u_\epsilon^i)$ for $1\leq i\leq n$. 
	
	We note that the boundary conditions imply that for all $1\leq i\leq n$, we have on $\partial B_\epsilon(p_i)$ (for some $\gamma_i,\tilde{\gamma}_i\in \C$ and $b_i\in \R$)
	\begin{align*}
	u_\epsilon&=u-\sum_{j=1}^{n}u_\epsilon^j=-\sum_{j\neq i}^{}u_\epsilon^j=\Re\left(\frac{\gamma_i}{z}+\tilde{\gamma}_i\frac{\z}{z}\right)-\sum_{j\neq i}^{}\lambda_{i,j}v(p_j)\log |z|+b_i+O(|z|)\\
	&=\Re\left(\frac{\gamma_i}{z}\right)-\sum_{j\neq i}^{}\lambda_{i,j}v(p_j)\log|z|+O(1)\\
	\partial_\nu u_\epsilon&=-\frac{1}{\epsilon}\Re\left(\frac{\gamma_i}{z}\right)-\frac{1}{\epsilon}\sum_{j\neq i}^{}\lambda_{i,j}v(p_j)+O(1),
	\end{align*}
	where we used $\partial_{\nu}\left(\tilde{\gamma}_i\dfrac{\z}{z}\right)=0$.
	As by the Remark \ref{outsidevortices} for all $j\neq i$, we have on $\partial_\nu B_\epsilon(p_j)$
	\begin{align*}
	&\mathscr{L}_g u_\epsilon^i=O(\epsilon^2)\\
	&\partial_\nu (\mathscr{L}_g u_\epsilon^i)=O(\epsilon)
	\end{align*}
	we deduce that
	\begin{align*}
	\int_{\partial B_\epsilon(p_j)}^{}u_\epsilon\partial_\nu (\mathscr{L}_g u_\epsilon^i)-\partial_\nu u_\epsilon \mathscr{L}_gu_\epsilon^i d\hh^1=O(\epsilon),
	\end{align*}
	and as on $\partial B_\epsilon(p_i)$
	\begin{align*}
	&\mathscr{L}_gu_\epsilon^i=4v(p_i)+O(\epsilon^2\log^2\epsilon)\\
	&\partial_\nu (\mathscr{L}_gu_\epsilon^i)=O(\epsilon\log^2\epsilon)
	\end{align*}
	we have
	\begin{align*}
	\int_{\partial B_\epsilon(p_i)}^{}u_\epsilon \partial_\nu (\mathscr{L}_gu_\epsilon^i)d\hh^1=\int_{\partial B_\epsilon(p_i)}^{}O(\log^2\epsilon)d\hh^1=O(\epsilon\log^2\epsilon)
	\end{align*}
	so finally, as $\mathscr{L}_g^2 u_\epsilon^i=0$ on $\Sigma_{\epsilon}$
	\begin{align*}
	&B_\epsilon(u_\epsilon,u_\epsilon^i)=\int_{\Sigma_\epsilon} \mathscr{L}_g u_\epsilon\, \mathscr{L}_gu_\epsilon^i\,d\vg=\int_{\Sigma_\epsilon}^{}u_\epsilon \mathscr{L}_g^2 u_\epsilon^i\,d\vg+\sum_{j=1}^{n}\int_{\partial B_\epsilon(p_j)}^{}u_\epsilon\partial_\nu(\mathscr{L}_g u_\epsilon^i)-\partial_\nu u_\epsilon\,\mathscr{L}_g u_\epsilon^i\, d\hh^1\\
	&=\int_{\partial B_\epsilon(p_i)}^{}-\frac{1}{\epsilon}\left(-\Re\left(\frac{\gamma_i}{z}\right)-\sum_{j\neq i}^{}\lambda_{i,j}v(p_j)+O(\epsilon)\right)(4v(p_i)+O(\epsilon))d\hh^1+O(\epsilon\log^2\epsilon)\\
	&=8\pi\sum_{j\neq i}^{}\lambda_{i,j}v(p_i)v(p_j)+O(\epsilon\log^2\epsilon),
	\end{align*}
	where we used by obvious symmetry
	\begin{align*}
		\int_{\partial B(0,\epsilon)}\Re\left(\frac{\gamma_i}{z}\right)d\mathscr{H}^1=0.
	\end{align*}
	Therefore for all $1\leq i\leq n$, one has
	\begin{align}\label{claim3}
	B_\epsilon(u_\epsilon,u_\epsilon^i)=8\pi\sum_{j\neq i}^{}\lambda_{i,j}v(p_i)v(p_j)+O(\epsilon\log\epsilon).
	\end{align}
	\textbf{Conclusion :}  We have finally by \eqref{decomposition}, \eqref{diagonale}, \eqref{horsdiagonale}, \eqref{claim3}
	\begin{align*}
	Q_\epsilon(u)&=Q_\epsilon(u_\epsilon)+8\pi\sum_{i=1}^{n}\frac{\alpha_i^2}{\epsilon^2}v^2(p_i)+16\pi\sum_{j=1}^{m}\beta_j^2\log\left(\frac{1}{\epsilon}\right)v^2(p_j)\\
	&+2\sum_{i=1}^{n}\left(4\pi\sum_{j\neq i}^{}\lambda_{i,j}v(p_i)v(p_j)\right)-4\pi\sum_{i\neq j}^{}\lambda_{i,j}v(p_i)v(p_j)+O(\epsilon\log^2\epsilon)\\
	&=Q_\epsilon(u_\epsilon)+8\pi\sum_{i=1}^{n}\frac{\alpha_i^2}{\epsilon^2}v^2(p_i)++16\pi\sum_{j=1}^{m}\beta_j^2\log\left(\frac{1}{\epsilon}\right)v^2(p_j)`+4\pi\sum_{i\neq j}^{}\lambda_{i,j}v(p_i)v(p_j)+O(\epsilon\log^2\epsilon)
	\end{align*}
	and finally
	\begin{align*}
	Q(u)&=\lim\limits_{\epsilon\rightarrow 0}\left(Q_\epsilon(u)-8\pi\sum_{i=1}^{n}\frac{\alpha_i^2}{\epsilon^2}v^2(p_i)-16\pi\sum_{j=1}^{m}\beta_j^2\log\left(\frac{1}{\epsilon}\right)v^2(p_j)
    +16\pi\sum_{j=1}^{m}\beta_j^2v^2(p_j)\right)\\
	&=Q(u_0)+16\pi\sum_{j=1}^{m}\beta_j^2v^2(p_j)+4\pi\sum_{i\neq j}^{}\lambda_{i,j}v(p_i)v(p_j),
	\end{align*}
	which concludes the proof, as the last claim follows from the fact that
	\begin{align*}
	\int_{\partial B_\epsilon(p_i)}^{}u_\epsilon \partial_\nu(\mathscr{L}_g u_\epsilon)-\partial_\nu u_\epsilon\, (\mathscr{L}_g u_\epsilon)\,d\hh^1=\int_{\partial B_\epsilon(p_i)}^{}O(\log\epsilon)O(\epsilon^2)-O\left(\frac{1}{\epsilon}\right)O(\epsilon^3)d\hh^1=O(\epsilon^3\log\epsilon).
	\end{align*}
	which concludes the proof of the theorem.
\end{proof}
We deduce from the preceding theorem an improvement of theorem
\begin{cor}\label{residuasymptotique}
	For all $1\leq i\leq n$, there exists $\lambda_{i,j}\in \R$ such that for all $j\neq i$, for all $0<\epsilon<\epsilon_0$, on every complex chart around $p_j$ there exists $c_{i,j},d_{i,j}\in \C$ and $b_{i,j}\in \R$ such that
	\begin{align}\label{lambda}
	u_\epsilon^i(z)=\Re\left(\frac{c_{i,j}}{z}+d_{i,j}\frac{\z}{z}\right)+\lambda_{i,j}v(p_i)\log|z|+b_{i,j}+O(|z|\log|z|).
	\end{align}
\end{cor}

\section{Equality of the Morse index for inversions of minimal surfaces with embedded ends}

\begin{theorem}\label{embedded}
	Let $\Sigma$ be a closed Riemann surface, $\phi:\Sigma\setminus\ens{p_1,\cdots,p_n}\rightarrow \R^3$ be a complete minimal surface with finite total curvature and embedded ends , and $\vec{\Psi}:\Sigma\rightarrow \R^3$ be its inversion.  Assume that $0\leq m\leq n$ is fixed such that $p_1\,\cdots,p_m$ are catenoid ends, while $p_{m+1},\cdots,p_n$ are planar ends, and for all $1\leq j\leq m$, let $\beta_j=\mathrm{Flux}(\phi,p_j)\in \R^{\ast}$ be the flux of $\phi$ at $p_j$. 
	Let $\Lambda(\vec{\Psi})\in \mathrm{Sym}(\R^n)$ be the symmetric matrix defined (see Corollary \ref{residuasymptotique}) by
	\begin{align*}
		A(\vec{\Psi})=\begin{pmatrix}
		\tilde{\beta}_1^2&\lambda_{1,2} &\cdots &\cdots & \cdots &\cdots &\lambda_{1,n}\\
		\lambda_{1,2} & \tilde{\beta}_2^2& \cdots &\cdots &\cdots &\cdots &\lambda_{2,n}\\
		\vdots& \ddots & \ddots &\ddots &\ddots &\ddots &\vdots \\
		\lambda_{1,m} &\cdots &\cdots & \tilde{\beta}_m^2 &\cdots&\cdots &\lambda_{m,n}\\
		\lambda_{1,m+1}& \cdots &\cdots &\cdots &0 &\cdots &\lambda_{m+1,n}\\
		\vdots& \ddots& \ddots &\ddots &\ddots & \ddots &\vdots\\
		\lambda_{1,n} &\lambda_{2,n} &\cdots &\cdots  &\cdots & \cdots & 0
		\end{pmatrix},\quad \tilde{\beta}_j^2=\frac{4}{2n+1}\beta_j^2.
	\end{align*}
	Then for all $a=(a_1,\cdots,a_n)\in \R^n$, there exists $v=v_a\in W^{2,2}(\Sigma)$ such that $(v(p_1),\cdots,v(p_n))=(a_1,\cdots,a_n)$ and
	\begin{align}\label{enfin}
		Q_{\vec{\Psi}}(v)=16\pi\sum_{j=1}^{n}\beta_j^2v^2(p_j)+4\pi(2n+1)\sum_{1\leq i,j\leq n}^{}\lambda_{i,j}v(p_i)v(p_j).
	\end{align}
	Therefore, we have $
	\mathrm{Ind}_{W}(\vec{\Psi})=\mathrm{Ind}\,A(\vec{\Psi})
	$, where the index $\mathrm{Ind}$ of a matrix is the number of its negative eigenvalues.
\end{theorem}
\begin{proof}
		Let $v\in C^2(\Sigma)$ be such that $v(p_i)\neq 0$ and consider $u_0=\sum_{i=1}^{n}u_0^i=|\phi|^2\sum_{i=1}^{n}v_0^i$ obtained in Theorem \ref{expansionends2}. We assume for simplicity that the end is planar, as the computation for a catenoid end would be identical up to the addition of two extra terms. Recall now that in the chart $U_i$ around $p_i$, we have for all $j\neq i$
		\begin{align*}
		&|\phi|^2=\frac{\beta_0^2}{|z|^2}\left(1+O(|z|^2)\right)\\
		&u_{0}^i=\frac{\beta_0^2}{|z|^2}\left(v(p_i)+O(|z|)\right)
		&u_{0}^j=\Re\left(\frac{\gamma_{i,j}}{z}+\tilde{\gamma}_{i,j}\frac{\z}{z}\right)+\lambda_{i,j}v(p_j)\log|z|+\mu_{i,j}+O(|z|).
		\end{align*}
		Therefore, we find
		\begin{align*}
		&v_0^i=v(p_i)+O(|z|)\\
		&v_0^j=\Re\left(\frac{\gamma_{i,j}}{\beta_0^2}z+\frac{\bar{\tilde{\gamma}_{i,j}}}{\beta_0^2}z^2\right)+\frac{\lambda_{i,j}}{\beta_0^2}|z|^2\log|z|+\frac{\mu_{i,j}}{\beta_i^2}|z|^2+O(|z|^3).
		\end{align*}
		Furthermore, as $v_0^i$ is regular at $p_i$, this implies if $u_0=|\phi|^2v_0$ that there exists $\gamma_0\in \R$ and $\zeta_0,\zeta_1\in \C$ such that
		\begin{align}\label{key}
		v_0=v(p_i)+2\,\Re\left(\zeta_0z+\zeta_1z^2\right)+\gamma_0|z|^2+\sum_{j\neq i}^{}\frac{\lambda_{i,j}}{\beta_0^2}v(p_j)|z|^2\log|z|+O(|z|^3).
		\end{align}
		Therefore, one needs to compute the renormalised energy for variations not only $C^2$ but also of the form given by \eqref{key}.
		
		Let $v\in W^{2,2}(\Sigma)$ such that
		\begin{align}\label{elena}
			v=v(p_i)+2\,\Re\left(\zeta_0z+\zeta_1z^2\right)+\gamma_0|z|^2+\gamma_1|z|^2\log|z|+O(|z|^3).
		\end{align}
		We will now compute $Q_{\vec{\Psi}}(v)$ for the  variation $v$ in \eqref{elena}.		
		 Now, recall that at a planar end there exists $\beta_0^2>0$ and $\alpha_0\in \C$ such that
		\begin{align*}
			|\phi|^2=\frac{\beta_0^2}{|z|^2}\left(1+2\,\Re\left(\alpha_0z^2\right)+O(|z|^3)\right).
		\end{align*}
		As $\phi$ is minimal, we deduce that 
		\begin{align*}
			g=e^{2\lambda}|dz|^2=\partial\bar{\partial}|\phi|^2=\frac{\beta_0^2}{|z|^4}\left(1-2\,\Re\left(\alpha_0z^2\right)+O(|z|^3)\right).
		\end{align*}
		Therefore, we have
		\begin{align*}
			u&=|\phi|^2v=\frac{\beta_0^2}{|z|^2}\left(v(p_i)+2\,\Re\left(\zeta_0z+\left(\alpha_0v(p_i)+\zeta_1\right)z^2\right)\right)+\beta_0^2\gamma_0+\beta_0^2\gamma_1\log|z|+O(|z|)\\
			&=\frac{\beta_0^2}{|z|^2}\left(v(p_i)+2\,\Re\left(\zeta_0z+\zeta_2z^2\right)\right)+\beta_0^2\gamma_0+\beta_0^2\gamma_1\log|z|+O(|z|)
		\end{align*}
		and (as $|z|^{-2}\Re(\zeta_0z)=\Re(\bar{\zeta_0}z^{-1})$ is harmonic)
		\begin{align}\label{epsilon2}
			\Delta u&=\frac{4\beta_0^2}{|z|^4}\left(v(p_i)-2\,\Re\left(\zeta_2z^2\right)+O(|z|^3)\right)\nonumber\\
			\Delta_gu&=e^{-2\lambda}\Delta u=\frac{|z|^4}{\beta_0^2}\left(1+2\,\Re\left(\alpha_0z^2\right)+O(|z|^3)\right)\times \frac{4\beta_0^2}{|z|^4}\left(v(p_i)-2\,\Re\left(\zeta_2z^2\right)+O(|z|^3)\right)\nonumber\\
			&=4v(p_i)+2\,\Re\left(\left(\alpha_0v(p_i)-\zeta_2\right)z^2\right)+O(|z|^3)\nonumber\\
			&=4v(p_i)+2\,\Re\left(\zeta_3z^2\right)+O(|z|^3)
		\end{align}
		Now, we have 
		\begin{align*}
			\partial u&=-\frac{\beta_0^2}{z|z|^2}\left(v(p_i)+\bar{\zeta_0}\z-\zeta_2z^2+\bar{\zeta_2}\z^2\right)dz+\frac{\beta_0^2\gamma_1}{2}\frac{dz}{z}+O(1)\\
			&=-\frac{\beta_0^2}{z|z|^2}\left(v(p_i)+\bar{\zeta_0}\z-2i\,\Im\left(\zeta_2z^2\right)\right)dz+\frac{\beta_0^2\gamma_1}{2}\frac{dz}{z}+O(1).
		\end{align*}
		This implies that we have for some $\lambda_0,\lambda_1\in \C$
		\begin{align*}
			\Delta_g\left(|\phi|^2v\right)\partial\left(|\phi|^2v\right)=\Delta_gu\,\left(\partial u\right)=-\frac{4\beta_0^2}{z|z|^2}\left(v^2(p_i)+\bar{\zeta_0}v(p_i)\z+\lambda_0z^2+\lambda_1\z^2\right)dz+2\beta_0^2\gamma_1v(p_i)\frac{dz}{z}+O(1).
		\end{align*}
		This implies that
		\begin{align}\label{finite1}
			\Im\int_{\partial B(0,\epsilon)}\Delta_g\left(|\phi|^2v\right)\partial\left(|\phi|^2v\right)=-\frac{8\pi\beta_0^2}{\epsilon^2}v^2(p_i)+4\pi\beta_0^2\gamma_1v(p_i)+O(\epsilon).
		\end{align}
		Now, we compute
		\begin{align*}
			|\partial u|^2&=\left|-\frac{\beta_0^2}{z|z|^2}\left(v(p_i)+\bar{\zeta_0}\z-2i\,\Im\left(\zeta_2z^2\right)\right)dz+\frac{\beta_0^2\gamma_0}{2}\frac{dz}{z}+O(1)\right|^2\\
			&=\bigg(\frac{\beta_0^4}{|z|^6}\left(v^2(p_i)+2\,\Re\left(\zeta_0v(p_i)z\right)+|\zeta_0|^2|z|^2+O(|z|^3)\right)-\frac{\beta_0^4\gamma_0}{|z|^4}\bigg)|dz|^2\\
			&=\frac{\beta_0^4}{|z|^6}\left(v^2(p_i)+2\,\Re\left(\zeta_0v(p_i)z\right)+(|\zeta_0|^2-\gamma_0)|z|^2+O(|z|^3)\right)
		\end{align*}
		Therefore, we obtain
		\begin{align*}
			|\partial u|_g^2&=e^{-2\lambda}|\p{z}u|^2=\frac{|z|^4}{\beta_0^2}\left(1+2\,\Re\left(\alpha_0z^2\right)+O(|z|^3)\right)\times \frac{\beta_0^4}{|z|^6}\left(v^2(p_i)+2\,\Re\left(\zeta_0v(p_i)z\right)+(|\zeta_0|^2-\gamma_0)|z|^2+O(|z|^3)\right)\\
			&=\frac{\beta_0^2}{|z|^2}\bigg(v^2(p_i)+2\,\Re\left(\zeta_0v(p_i)z+\alpha_0z^2\right)\bigg)+\left(|\zeta_0|^2-\gamma_0\right)+O(|z|)
		\end{align*}
		and notice the constant $(|\zeta_0|^2-\gamma_0)$. Finally, we find for some $\lambda_2,\lambda_3\in \C$
		\begin{align*}
			\partial|\partial u|_g^2=-\frac{\beta_0^2}{z|z|^2}\left(v^2(p_i)+\bar{\zeta_0}v(p_i)\z+\lambda_2z^2+\lambda_3\z^2\right)dz+O(1)
		\end{align*}
		Finally, we have
		\begin{align*}
			-\partial\left|d\left(|\phi|^2v\right)\right|_g^2&=-\partial|du|_g^2=-4\,\partial|\partial u|_g^2\\
			&=\frac{4\beta_0^2}{z|z|^2}\left(v^2(p_i)+\bar{\zeta_0}v(p_i)\z+\lambda_2z^2+\lambda_3\z^2\right)dz+O(1).
		\end{align*}
		and
		\begin{align}\label{finite2}
			\Im\int_{\partial B(0,\epsilon)}-\partial\left|d\left(|\phi|^2v\right)\right|_g^2=\frac{8\pi\beta_0^2}{\epsilon^2}+O(\epsilon).
		\end{align}
		Gathering \eqref{finite1} and \eqref{finite2} we obtain as $K_g=O(|z|^6)$ by planarity of the end (notice the factor $2$ in front of the Laplacian)
		\begin{align}\label{newres}
			\Im\int_{\partial B(0,\epsilon)}2\left(\Delta_gu+2K_gu\right)\partial u-\partial|du|_g^2&=\Im\int_{\partial B(0,\epsilon)}2\left(\Delta_gu\right)\partial u-\partial|du|_g^2+O(\epsilon)\nonumber\\
			&=2\left(-\frac{8\pi\beta_0^2}{\epsilon^2}+4\pi\beta_0^2\gamma_1v(p_i)\right)+\frac{8\pi\beta_0^2}{\epsilon^2}+O(\epsilon)\nonumber\\
			&=-\frac{8\pi\beta_0^2}{\epsilon^2}+8\pi\beta_0^2\gamma_1v(p_i)+O(\epsilon).
		\end{align}
		Now, coming back to \eqref{key}, we see that for $v_0$ written above we have (with $\beta_i$ replace by $\beta_0$)
		\begin{align*}
			\gamma_1=\sum_{j\neq i}^{}\frac{\lambda_{i,j}}{\beta_0^2}v(p_j),
		\end{align*}
		so that
		\begin{align}\label{cont}
			8\pi\beta_0^2\gamma_1v(p_i)=8\pi\sum_{j\neq i}^{}\lambda_{i,j}v(p_i)v(p_j).
		\end{align}
		Therefore, each end will bring this new contribution \eqref{cont} by, so we obtain (as $\lambda_{i,i}=0$)
		\begin{align*}
			Q_{\vec{\Psi}}(v_0)&=8\pi\sum_{i=1}^{n}\sum_{j\neq i}^{}\lambda_{i,j}v(p_i)v(p_j)+4\pi\sum_{1\leq i,j\leq n}^{}\lambda_{i,j}v(p_i)v(p_j)\\
			&=4\pi(2n+1)\sum_{1\leq i,j\leq n}^{}\lambda_{i,j}v(p_i)v(p_j).
		\end{align*}
		Indeed, let $u_{\epsilon}^i$ the solution of \eqref{EQepsilon} where $v$ is replaced by $v_0$. We notice as $\lg^2u_0=0$ that
		\begin{align*}
			&Q_{\epsilon}\Big(u_0-\sum_{j=1}^{n}u_{\epsilon}^j\Big)=\frac{1}{2}\int_{\Sigma\setminus\cup_{i=1}^n\bar{B}_{\epsilon}(p_i)}\Big(\lg\Big(u_0-\sum_{j=1}^{n}u_{\epsilon}^i\Big)\Big)^2d\vg\\
			&=\sum_{i=1}^{n}\int_{\partial B_{\epsilon}(p_i)}\bigg(\Big(u_0-\sum_{j=1}^{n}u_{\epsilon}^j\Big)\partial_{\nu}\Big(\lg \Big(u_0-\sum_{j=1}^{n}u_{\epsilon}^j\Big)\Big)-\partial_{\nu}\Big(u_0-\sum_{j=1}^{n}u_{\epsilon}^j\Big)\lg\Big(u_0-\sum_{j=1}^{n}u_{\epsilon}^j\Big)\bigg)d\mathscr{H}^1.
		\end{align*}
		Now, recall that on $\partial B_{\epsilon}(p_i)$, we have
		\begin{align*}
			&u_{\epsilon}^i=|\phi|^2v_0,\quad \partial_{\nu}(u_{\epsilon}^i)=\partial_{\nu}\left(|\phi|^2v_0\right)
		\end{align*}
		Furthermore, the computations of \eqref{epsilon2} imply that 
		\begin{align*}     
		     \lg u_{\epsilon}^i=4v(p_i)+O(\epsilon^2)\qquad \lg u_{0}=4v(p_i)+O(\epsilon^2)
	    \end{align*}
	    which implies that 
	    \begin{align*}
	    	\lg(u_0-u_{\epsilon}^i)=O(\epsilon^2)
	    \end{align*}
	    while for all $j\neq i$
	    \begin{align*}
	    	u_{\epsilon}^j=O\left(\frac{1}{\epsilon}\right),\quad \partial_{\nu}u_{\epsilon}^j=O\left(\frac{1}{\epsilon^2}\right).
	    \end{align*}
	    Therefore, we have on $\partial B_{\epsilon}(p_i)$ for all $1\leq i\leq n$
	    \begin{align*}
	    	&\Big(u_0-\sum_{j=1}^{n}u_{\epsilon}^j\Big)=O\left(\frac{1}{\epsilon}\right),\quad \partial_{\nu}\Big(u_0-\sum_{j=1}^{n}u_{\epsilon}^j\Big)=O\left(\frac{1}{\epsilon^2}\right)\\
	    	&\lg u_{\epsilon}^i=O(\epsilon^2),\quad 
	    	\partial_{\nu}\left(\lg u_{\epsilon}^i\right)=O(\epsilon)\\
	    	&\Big(u_0-\sum_{j=1}^{n}u_{\epsilon}^j\Big)\partial_{\nu}\Big(\lg\Big(u_0-\sum_{j=1}^{n}u_{\epsilon}^j\Big)\Big)-\partial_{\nu}\Big(u_0-\sum_{j=1}^{n}u_{\epsilon}^j\Big)\lg\Big(u_0-\sum_{j=1}^{n}u_{\epsilon}^j\Big)\\
	    	&=O\left(\frac{1}{\epsilon}\right)\times O(\epsilon)-O\left(\frac{1}{\epsilon^2}\right)\times O(\epsilon^2)=O(1)
	    \end{align*}
	    and finally
	    \begin{align*}
	    	Q_{\epsilon}\Big(u_0-\sum_{j=1}^{n}u_{\epsilon}^j\Big)=O(\epsilon)\conv{\epsilon\rightarrow 0}0.
	    \end{align*}
	    which shows by the analysis of Theorem \ref{explicit} the announced formula \eqref{enfin}.
\end{proof}

\section{Jacobi fields associated to the Universal Matrix $\Lambda$}

\begin{theorem}
	Let $\phi:\Sigma\setminus\ens{p_1,\cdots,p_n}\rightarrow\R^3$ be a complete minimal surface with finite total curvature and embedded ends (not necessarily planar) and $\vec{\Psi}:\Sigma\rightarrow \R^3$ be its inversion. Let $\ens{\lambda_{i,j}}_{1\leq i,j\leq n}$ be the matrix with $0$ diagonal of Theorem \ref{explicit}, and assume that $\lambda_{i,j}\neq 0$. Then there exists a Jacobi field $\omega_0^{i,j}:\Sigma\setminus\ens{p_1,\cdots,p_n}\rightarrow \R$ such that $\omega_{0}^{i,j}$ for all $k\in\ens{1,\cdots,n}$, in every complex chart around $p_i$ there exists $\zeta_{i,j}^k\in \C$ and $\mu_{i,j}^k\in \R$ such that 
	\begin{align*}
		\omega_{0}^{i,j}=\Re\left(\frac{\zeta_{i,j}^k}{z}\right)+\lambda_{i,j}\left(\delta_{i,k}+\delta_{j,k}\right)\log|z|+\mu_{i,j}^k+O(|z|),
	\end{align*} 
	where $\delta_{i,k}$ is the Kronecker symbol.
\end{theorem}
\begin{proof}
	Fix a covering $(U_1,\cdots,U_n)$ such that $p_i\in U_i$ for all $1\leq i\leq n$ and $U_i\cap U_j=\varnothing$ for all $1\leq i\neq j\leq n$, and we assume that $U_i$ is a domain of chart for all $1\leq i\leq n$, \textit{i.e.} that there exists a complex diffeomorphism $f_i:U_i\rightarrow D^2\subset \C$ such that $f_i(p_i)=0$ and $f_i(U_i)=D^2$. Notice that for a function $u\in C^{2,\alpha}(\Sigma\setminus\ens{p_1,\cdots,p_n})$ admitting the following expansion in the chart $(U_i,f_i)$
	\begin{align}\label{devlog}
		u=\Re\left(\frac{\zeta}{z}+\kappa\frac{\z}{z}\right)+\lambda\log|z|+\mu+O(|z|)
	\end{align}
	for some $\lambda,\mu\in \R$ and $\zeta,\kappa\in \C$. The constant $\lambda$ does not depends on $U_i$ and $f_i$, and $\mu$ does not depends on $f_i$ as a complex change of chart $D^2\rightarrow D^2$ fixing $0$ is a rotation, while the expansion \eqref{devlog} is invariant under rotations.
	
	Now, assume that $n\geq 2$ and fix $1\leq i\neq j\leq n$. For all $\epsilon>0$ small enough, there exists thanks to Theorem \ref{expansionends} a solution $u_{\epsilon}^{i,j}\in C^{\infty}(\Sigma\setminus(\bar{B}_{\epsilon}(p_i)\cup  \bar{B}_{\epsilon}(p_j)\cup\ens{p_1,\cdots,p_n})$ such that 
	\begin{align*}
	\left\{\begin{alignedat}{2}
		&\lg^2u_{\epsilon}^{i,j}=0\quad &&\text{in}\;\, \Sigma\setminus(\bar{B}_{\epsilon}(p_i)\cup \bar{B}_{\epsilon}(p_j)\cup\ens{p_1,\cdots,p_n})\\
		&u_{\epsilon}^{i,j}=|\phi|^2\quad &&\text{on}\;\, \partial B_{\epsilon}(p_i)\cup \partial\bar{B}_{\epsilon}(p_j)\\
		&\partial_{\nu}u_{\epsilon}^{i,j}=\partial_{\nu}\left(|\phi|^2\right)\quad &&\text{on}\;\, \partial B_{\epsilon}(p_i)\cup \partial B_{\epsilon}(p_j).
		\end{alignedat}\right.
	\end{align*}
	Furthermore, by the argument of Theorem \ref{explicit}, we have (this would also be true for catenoid ends up to an additional singular term in $\log$, notice that $v(p_i)=v(p_j)=1$ here)
	\begin{align*}
		\frac{1}{2}\int_{\Sigma_{\epsilon}}\left(\lg u_{\epsilon}^{i,j}\right)^2d\vg=
		\frac{8\pi \alpha_i^2}{\epsilon^2}+\frac{8\pi\alpha_j^2}{\epsilon^2}+O(\epsilon\log\epsilon).
	\end{align*}
	 By the previous indicial root analysis of Theorem \ref{indicielles1}, for all $k\neq i,j$ there exists $\lambda_{i,j}^k\in \R$ such that we have in $f_k^{-1}(\bar{B}_{\C}(0,\epsilon))\subset U_k$ the expansion
	\begin{align*}
		&u_{\epsilon}^{i,j}=\Re\left(\frac{\zeta_{i,j}^k}{z}+\kappa_{i,j}^k\frac{\z}{z}\right)+\lambda_{i,j}^k\log|z|+O(1)\\
	    &\lg u_{\epsilon}^{i,j}=\Re\left(\tilde{\kappa}_{i,j}^kz^2\right)+O(\epsilon^3) \\
	    &\partial_{\nu}(\lg u_{\epsilon}^{i,j})=\frac{2}{|z|}\Re\left(\tilde{\kappa}_{i,j}^iz^2\right)+O(\epsilon^2).
	\end{align*}
	Recall also that for all $1\leq l\neq k\leq n$, we have the expansion in $f_l^{-1}(\bar{B}(0,\epsilon))\subset U_l$
	\begin{align*}
		&u_{\epsilon}^{k}=\Re\left(\frac{\zeta_k^l}{z}+\kappa_k^l\frac{\z}{z}\right)+\lambda_{k,l}\log|z|+O(1)\\
		&\lg u_{\epsilon}^k=\Re(\tilde{\kappa}_k^lz^2)+O(\epsilon^3)\\
		&\partial_{\nu}(\lg u_{\epsilon}^k)=\frac{2}{|z|}\Re\left(\tilde{\kappa}_k^lz^2\right)+O(\epsilon^2).
	\end{align*}
	Notice that these expansions imply that for all $l\neq i,j,k$ we have on $\partial B_{\epsilon}(p_l)$ the estimate
	\begin{align*}
		u_{\epsilon}^{i,j}\partial_{\nu}\left(\lg u_{\epsilon}^k\right)-\partial_{\nu}\left(u_{\epsilon}^{i,j}\right)\lg u_{\epsilon}^k=O\left(\frac{1}{\epsilon}\right)O(\epsilon)-O\left(\frac{1}{\epsilon^2}\right)O(\epsilon^2)=O(1),
	\end{align*}
	which implies that for all $l\neq i,j,k$
	\begin{align}\label{estimate1}
		\int_{\partial B_{\epsilon}(p_l)}\left(u_{\epsilon}^{i,j}\partial_{\nu}\left(\lg u_{\epsilon}^k\right)-\partial_{\nu}\left(u_{\epsilon}^{i,j}\right)\lg u_{\epsilon}^k\right)\,d\mathscr{H}^1=O(\epsilon).
	\end{align}
	 Now, recall that on $\partial B_{\epsilon}(p_i)$ we have 
	 \begin{align*}
	 	|\phi|^2=\frac{\alpha_i^2}{|z|^2}+O(\log^2|z|).
	 \end{align*}
	 Therefore, we have
	 \begin{align*}
	 	\partial_{\nu}|\phi|^2=-\frac{2\alpha_i^2}{|z|^3}+O\left(\frac{\log|z|}{|z|}\right).
	 \end{align*}
	 Therefore, we deduce that we have on $\partial B_{\epsilon}(p_i)$ thanks to the boundary conditions
	 \begin{align}\label{dev2}
	 	&u_{\epsilon}^{i,j}\partial_{\nu}\left(\lg u_{\epsilon}^k\right)-\partial_{\nu}\left(u_{\epsilon}^{i,j}\right)\lg u_{\epsilon}^k=|\phi|^2\partial_{\nu}\left(\lg u_{\epsilon}^k\right)-\partial_{\nu}\left(|\phi|^2\right)\lg u_{\epsilon}^k\nonumber\\
	 	&=\left(\frac{\alpha_i^2}{|z|^2}+O(\log^2|z|)\right)\times \left(\frac{2}{|z|}\Re\left(\tilde{\kappa}_k^lz^2\right)+O(\epsilon^2)
	 	\right)-\left(-\frac{2\alpha_i^2}{|z|^3}+O\left(\frac{\log|z|}{|z|}\right)\right)\times\left(\frac{2}{|z|}\Re\left(\tilde{\kappa}_k^lz^2\right)+O(|z|^3)\right)\nonumber\\
	 	&=6\frac{\alpha_i^2}{|z|^3}\Re\left(\tilde{\kappa}_k^lz^2\right)+O(1).
	 \end{align}
	 As for all $c\in \C$
	 \begin{align}\label{notradial}
	 	\int_{\partial B(0,\epsilon)}\frac{1}{|z|^3}\Re\left(cz^2\right)d\mathscr{H}^1=\frac{1}{\epsilon}\Re\left(c\int_{0}^{2\pi}e^{2i\theta}d\theta\right)=0,
	 \end{align}
	 we deduce from \eqref{dev2} and \eqref{notradial} that 
	 \begin{align}\label{estimate2}
	 	\int_{\partial B_{\epsilon}(p_i)}\left(u_{\epsilon}^{i,j}\partial_{\nu}\left(\lg u_{\epsilon}^k\right)-\partial_{\nu}\left(u_{\epsilon}^{i,j}\right)\lg u_{\epsilon}^k\right)\,d\mathscr{H}^1=O(\epsilon)
	 \end{align}
	 and by symmetry
	 \begin{align}\label{estimate3}
	 	\int_{\partial B_{\epsilon}(p_j)}\left(u_{\epsilon}^{i,j}\partial_{\nu}\left(\lg u_{\epsilon}^k\right)-\partial_{\nu}\left(u_{\epsilon}^{i,j}\right)\lg u_{\epsilon}^k\right)\,d\mathscr{H}^1=O(\epsilon).
	 \end{align}
	 Finally, using the expansion on $\partial B_{\epsilon}(p_k)$
	 \begin{align*}
	 	&\mathscr{L}_gu_{\epsilon}^k=4+O(\epsilon^2)\\
	    &\mathscr{L}_g u_{\epsilon}^k=O(\epsilon)
	 \end{align*}
	 together with the previous estimates \eqref{estimate1}, \eqref{estimate2} and \eqref{estimate3}, we deduce that 
	\begin{align}\label{lambdaijk1}
		&\int_{\Sigma_{\epsilon}}\lg u_{\epsilon}^{i,j}\lg u_{\epsilon}^kd\vg=\sum_{l=1}^{n}\int_{\partial B_{\epsilon}(p_l)}\left(u_{\epsilon}^{i,j}\partial_{\nu}\left(\lg u_{\epsilon}^k\right)-\partial_{\nu}\left(u_{\epsilon}^{i,j}\right)\lg u_{\epsilon}^k\right)\,d\mathscr{H}^1\\
		&=\int_{\partial B_{\epsilon}(p_k)}\left(u_{\epsilon}^{i,j}\partial_{\nu}(\lg u_{\epsilon}^k)-\partial_{\nu}(u_{\epsilon}^{i,j})\lg u_{\epsilon}^k\right)d\mathscr{H}^1+O(\epsilon)\nonumber\\
		&=\int_{\partial B_{\epsilon}(p_i)}\left(O\left(\frac{1}{\epsilon^2}\right)O(\epsilon^2)-O\left(\frac{1}{\epsilon^3}\right)O(\epsilon^3)\right)d\mathscr{H}^1
		+\int_{\partial B_{\epsilon}(p_j)}\left(O\left(\frac{1}{\epsilon^2}\right)O(\epsilon^2)-O\left(\frac{1}{\epsilon^3}\right)O(\epsilon^3)\right)d\mathscr{H}^1\nonumber\\
		&+\int_{\partial B_{\epsilon}(p_k)}\left(O\left(\frac{1}{\epsilon}\right)O(\epsilon)-\frac{1}{\epsilon}\left(-\Re\left(\frac{\zeta_{i,j}^k}{z}\right)+\lambda_{i,j}^k\right)\left(4+O(\epsilon^2)\right)\right)d\mathscr{H}^1\nonumber\\
		&=-8\pi\lambda_{i,j}^k+O(\epsilon).
	\end{align}
	Likewise, we have
	\begin{align}\label{lambdaijk2}
		&\int_{\Sigma_{\epsilon}}\lg u_{\epsilon}^{i,j}\lg u_{\epsilon}^kd\vg=-8\pi(\lambda_{k,i}+\lambda_{k,j})+O(\epsilon).
	\end{align}
	By symmetry of $\lambda_{p,q}$, we obtain by \eqref{lambdaijk1} and \eqref{lambdaijk2} for all $i\neq j$ and $k\neq i,j$
	\begin{align}\label{lambdaijk}
		\lambda_{i,j}^k=\lambda_{i,k}+\lambda_{j,k}.
	\end{align}
	Now, let $\omega_{\epsilon}^{i,j}=u_{\epsilon}^{i,j}-u_{\epsilon}^j-u_{\epsilon}^j$. Thanks to \eqref{lambdaijk}, we have for all $k\neq i,j$ on $\partial B_{\epsilon}(p_k)$
	\begin{align*}
		\omega_{\epsilon}^{i,j}=\Re\left(\frac{\zeta_{i,j}^k}{z}\right)+\lambda_{i,j}^k\log|z|-\lambda_{i,k}\log|z|-\lambda_{j,k}\log|z|+O(1)=\Re\left(\frac{\zeta_{i,j}^k}{z}\right)+O(1),
	\end{align*}
	while on $\partial B_{\epsilon}(p_i)$, we have
	\begin{align}\label{twolog}
		&\omega_{\epsilon}^{i,j}=|\phi|^2-|\phi|^2-u_{\epsilon}^j=-\lambda_{j,i}\log|z|+O(1)=-\Re\left(\frac{\zeta_j^i}{z}\right)-\lambda_{i,j}\log|z|+O(1)\nonumber\\
		&\partial_{\nu}\omega_{\epsilon}^{i,j}=\partial_{\nu}|\phi|^2-\partial_{\nu}|\phi|^2-\partial_{\nu}u_{\epsilon}^j=\frac{1}{|z|}\left(\Re\left(\frac{\zeta_j^i}{z}\right)-\lambda_{i,j}\right)+O(1).
	\end{align}
	while on $\partial B_{\epsilon}(p_j)$
	\begin{align*}
		&\omega_{\epsilon}^{i,j}=-\Re\left(\frac{\zeta_i^j}{z}\right)-\lambda_{i,j}\log|z|+O(1)\\
		&\partial_{\nu}\omega_{\epsilon}^{i,j}=\frac{1}{|z|}\left(\Re\left(\frac{\zeta_j^i}{z}\right)-\lambda_{i,j}\right)+O(1).
	\end{align*}
	Furthermore, as we have on $\partial B_{\epsilon}(p_i)$
	\begin{align*}
		&\lg u_{\epsilon}^{i,j}=4+O(\epsilon^2)\\
		&\lg u_{\epsilon}^{i}=4+O(\epsilon^2)\\
		&\lg u_{\epsilon}^j=O(\epsilon^2)
	\end{align*}
	we have on $\partial B_{\epsilon}(p_i)$
	\begin{align*}
		&\lg \omega_{\epsilon}^{i,j}=O(\epsilon^2)\\
		&\partial_{\nu}(\mathscr{L}_g\omega_{\epsilon}^{i,j})=O(\epsilon)
	\end{align*}
	In particular, we have the estimates on $\partial B_{\epsilon}(p_i)\cup \partial B_{\epsilon}(p_j)$ (by symmetry of the previous estimates)
	\begin{align}\label{bound0}
		\omega_{\epsilon}^{i,j}\partial_{\nu}\left(\lg \omega_{\epsilon}^{i,j}\right)-\partial_{\nu}\omega_{\epsilon}^{i,j}\lg \omega_{\epsilon}^{i,j}=O\left(\frac{1}{\epsilon}\right)O(\epsilon)-O\left(\frac{1}{\epsilon^2}\right)O(\epsilon^2)=O(1).
	\end{align}
	We also easily deduce by the preceding arguments that for all $k\neq i,j$
	\begin{align}\label{bound1}
		\int_{\partial B_{\epsilon}(p_k)}\left(\omega_{\epsilon}^{i,j}\partial_{\nu}\left(\lg \omega_{\epsilon}^{i,j}\right)-\partial_{\nu}\omega_{\epsilon}^{i,j}\lg \omega_{\epsilon}^{i,j}\right)\,d\mathscr{H}^1=O(\epsilon).
	\end{align}
	Finally, we find by \eqref{bound0} and \eqref{bound1}
	\begin{align*}
		\int_{\Sigma_{\epsilon}}\left(\lg \omega_{\epsilon}^{i,j}\right)^2d\vg&=\sum_{k=1}^{n}\int_{\partial B_{\epsilon}(p_k)}\left(\omega_{\epsilon}^{i,j}\partial_{\nu}\left(\lg \omega_{\epsilon}^{i,j}\right)-\partial_{\nu}\left(\omega_{\epsilon}^{i,j}\right)\lg\omega_{\epsilon}^{i,j}\right)d\mathscr{H}^1
		=O(\epsilon).
	\end{align*}
	Therefore, we have
	\begin{align*}
		\lim\limits_{\epsilon\rightarrow 0}\int_{\Sigma_{\epsilon}}\left(\lg \omega_{\epsilon}^{i,j}\right)^2d\vg\conv{\epsilon\rightarrow 0}0,
	\end{align*}
	which implies (by the proof of Theorem \ref{expansionends2} and Fatou lemma) that (up to a subsequence) $\omega_{\epsilon}^{i,j}\conv{\epsilon\rightarrow 0}\omega_{0}^{i,j}\in C^{\infty}(\Sigma\setminus\ens{p_1,\cdots,p_n})$ where $\omega_0^{i,j}$ is a Jacobi field, \textit{i.e.} $\lg \omega_{0}^{i,j}=0$. Furthermore, by \eqref{lambdaijk} and \eqref{twolog}, $\omega_0^{i,j}$ is bounded at $p_k$ for all $k\neq i$, while in $U_i$ (resp. $U_j$), we have an expansion of the form
	\begin{align}\label{expansion}
		\omega_{0}^{i,j}=-\Re\left(\frac{\zeta_{i,j}^i}{z}+\kappa_{i,j}^i\frac{\z}{z}\right)-\lambda_{i,j}\log|z|-\mu_{i,j}^i+O(|z|).
	\end{align}
	In particular, we have $\omega_0^{i,j}\in C^{\infty}(\Sigma\setminus\ens{p_i,p_j})$. Furthermore, notice that we have near $p_j$
	\begin{align*}
		\mathscr{L}_g=\Delta_g-2K_g=e^{-2\lambda}\left(\Delta-2e^{2\lambda }K_g\right)=e^{-2\lambda}\left(\Delta+O(1)\right)
	\end{align*}
	as $e^{2\lambda}=\dfrac{\alpha_i}{|z|^4}\left(1+O(|z|^2)\right)$ and $K_g=O(|z|^4)$ (at an embedded end, not necessarily planar). As $\mathscr{L}_g\omega_{0}^{i,j}=0$, we deduce from the expansion \eqref{expansion} that 
	\begin{align*}
		\Delta \omega_{0}^{i,j}=4\,\Re\left(\kappa_{i,j}^i\frac{1}{z^2}\right)+O\left(\frac{1}{|z|}\right)
	\end{align*}
	and trivially
	\begin{align*}
		-2e^{2\lambda}K_g\,\omega_{0}^{i,j}=O\left(\frac{1}{|z|}\right)
	\end{align*}
	which implies that
	\begin{align*}
		0=e^{2\lambda}\lg\omega_{0}^{i,j}=4\,\Re\left(\kappa_{i,j}^i\frac{1}{z^2}\right)+O\left(\frac{1}{|z|}\right).
	\end{align*}
	Therefore, we deduce that 
	\begin{align*}
		\kappa_{i,j}^i=0
	\end{align*}
	and
	\begin{align}\label{finalexpansion}
		\omega_0^{i,j}=-\Re\left(\frac{\zeta_{i,j}^i}{z}\right)-\lambda_{i,j}\log|z|-\mu_{i,j}^i+O(|z|).
	\end{align}
	The conclusion of the theorem following by replacing $\omega_{0}^{i,j}$ by $-\omega_0^{i,j}$. 
\end{proof}

\section{Renormalised energy for ends of arbitrary multiplicity}\label{higher}

\begin{theorem}\label{general}
	Let $\phi:\Sigma\setminus\ens{p_1,\cdots,p_n}\rightarrow \R^3$ be a complete minimal surface with finite total curvature and $\vec{\Psi}=\iota\circ \phi:\Sigma\rightarrow \R^3$ be a compact inversion of $\phi$. Then there a universal symmetric matrix $\Lambda=\Lambda(\vec{\Psi})=\ens{\lambda_{i,j}}_{1\leq i,j\leq n}$ with such that for smooth all normal variation $\vec{v}=v\n_{\vec{\Psi}}\in \mathscr{E}_{\phi}(S^2,\R^m)$
	\begin{align*}
	D^2W(\vec{\Psi})(\vec{v},\vec{v})=Q_{\vec{\Psi} }(v_0)+4\pi\sum_{1\leq i,j\leq n}^{}\lambda_{i,j}v(p_i)v(p_j),
	\end{align*}
	for some $v_0\in W^{2,2}(\Sigma,\R)$ such that $v(p_j)=0$ for all $1\leq j\leq n$.
	In particular, we have
	\begin{align*}
	\mathrm{Ind}_W(\vec{\Psi})\leq \mathrm{Ind}\Lambda(\vec{\Psi})\leq  n=\frac{1}{4\pi}W(\vec{\Psi})-\frac{1}{2\pi}\int_{S^2}K_{g}d\mathrm{vol}_{g}+\chi(\Sigma).
	\end{align*}
\end{theorem}
\begin{proof}
	As previously, fix a some residues charts $(U_1,\cdots,U_n)$ around $p_1,\cdots, p_n$, and assume that $p_i$ has multiplicity $m_i\geq 1$ for all $1\leq i\leq n$, and fix some $1\leq i\leq m$. Then the same argument as in Theorem \ref{s4} shows that there exists $\alpha_i>0$ and $\alpha_i^l,\gamma_i^l\in \R$ and $\beta_i\in \R$ such that for all smooth variation $v$ (if $u=|\phi|^2v$)
	\begin{align*}
		&\int_{\partial B_{\epsilon}(p_i)}\left(\left(\Delta_g u+2K_gu\right)\ast dw-\frac{1}{2}\ast d|du|_g^2\right)\\
		&=-4\pi\left\{ \frac{\alpha_i^2}{\epsilon^{2m_i}}\left\{1+\sum_{l
			_1=1}^{2m_i}\alpha_i^l\epsilon^l\left(1+\gamma_i^l\log\left(\frac{1}{\epsilon}\right)\right)\right\}+\beta_i^2\log\left(\frac{1}{\epsilon}\right)\right\}v^2(p_i)+O\left(\epsilon\log^2\epsilon\right).
	\end{align*}
	In particular, we deduce that 
	\begin{align}\label{generallimit}
		D^2W(\vec{\Psi})&=\lim\limits_{\epsilon\rightarrow 0}\bigg\{\frac{1}{2} \int_{\Sigma_{\epsilon}}\left(\Delta_gw-2K_gw\right)^2d\vg\nonumber\\
		&-4\pi\sum_{i=1}^m\left\{ \frac{\alpha_i^2}{\epsilon^{2m_i}}\left\{1+\sum_{l
			_1=1}^{2m_i}\alpha_i^l\epsilon^l\left(1+\gamma_i^l\log\left(\frac{1}{\epsilon}\right)\right)\right\}+\beta_i^2\log\left(\frac{1}{\epsilon}\right)\right\}v^2(p_i)\bigg\}.
	\end{align}	
	As the limit in \eqref{generallimit} exists, retaking the notations of Theorem \ref{explicit} we find that for all $1\leq i\leq n$
	\begin{align}\label{part1}
		Q_{\epsilon}(u_{\epsilon}^i)=4\pi\left\{ \frac{\alpha_i^2}{\epsilon^{2m_i}}\left\{1+\sum_{l
			_1=1}^{2m_i}\alpha_i^l\epsilon^l\left(1+\gamma_i^l\log\left(\frac{1}{\epsilon}\right)\right)\right\}+\beta_i^2\log\left(\frac{1}{\epsilon}\right)\right\}v^2(p_i)+c+o\left(\epsilon\log^2\epsilon\right)
	\end{align}
	for some finite constant $c$, otherwise the limit would be $+\infty$ or $-\infty$ by taking variations non-zero at only one end, as the other contributions are only depend on quadratic expressions $v(p_i)v(p_j)$ for $i\neq j$, so they cannot involve singular terms which are all involving quadratic expressions $v(p_i)^2$. Furthermore, by the Sobolev embedding theorem, we have $c=\lambda_{i,i}v^2(p_i)$ for some $\lambda_{i,i}\in \R$. 
	
	For the sake of simplicity, we will remove the indices $i$ of the multiplicities $m_i$ ($1\leq i\leq n$).
	
	First notice that for all $k\neq i,j$, we have if $p_k$ has multiplicity $m_k=m\geq 2$ (for $m=1$ this was already treated previously)
	\begin{align*}
		&u_{\epsilon}^i=\Re\left(\frac{\gamma_{i,k}}{z^m}\right)+O(|z|^{1-m})\\
		&u_{\epsilon}^j=\Re\left(\frac{\gamma_{j,k}}{z^m}\right)+O(|z|^{1-m}).\\
		&\partial_{\nu}u_{\epsilon}^i=O(|z|^{-(m+1)})\\
		&\partial_{\nu}u_{\epsilon}^j=O(|z|^{-(m+1)})
	\end{align*}
	As $e^{\lambda}=\alpha_k^2|z|^{-2(m+1)}\left(1+O(|z|)\right)$, and $K_g=O(|z|^{2(m+1)})$. Therefore, we have by the harmonicity of $\Re(c\,z^{-m})$ for all $c\in \C$
	\begin{align*}
		&\Delta u_{\epsilon}^i=O(|z|^{-(m+1)})\\
		&\Delta u_{\epsilon}^j=O(|z|^{-(m+1)}),
	\end{align*}
	so that (as $\Delta_g=e^{-2\lambda}\Delta$)
	\begin{align*}
		&\mathscr{L}_gu_{\epsilon}^i=\Delta_gu_{\epsilon}^i-2K_gu_{\epsilon}^i=O(|z|^{m+1})\\
		&\mathscr{L}_gu_{\epsilon}^j=O(|z|^{m+1}).
	\end{align*}
	Therefore, we have
	\begin{align*}
		&u_{\epsilon}^i\,\partial_{\nu}\left(\mathscr{L}_gu_{\epsilon}^j\right)=O(|z|^{-m})\times O(|z|^{m})=O(1)\\
		&\partial_{\nu}\left(u_{\epsilon}^i\right)\,\mathscr{L}_gu_{\epsilon}^j=O(|z|^{-(m+1)})\times O(|z|^{m+1})=O(1).
	\end{align*}
	This implies that 
	\begin{align*}
		u_{\epsilon}^i\,\partial_{\nu}\left(\mathscr{L}_gu_{\epsilon}^j\right)-\partial_{\nu}\left(u_{\epsilon}^i\right)\,\mathscr{L}_gu_{\epsilon}^j=O(1),
	\end{align*}
	and 
	\begin{align*}
		\int_{\partial B_{\epsilon}(p_k)}u_{\epsilon}^i\,\partial_{\nu}\left(\mathscr{L}_gu_{\epsilon}^j\right)-\partial_{\nu}\left(u_{\epsilon}^i\right)\,\mathscr{L}_gu_{\epsilon}^j\,d\mathscr{H}^1=O(\epsilon).
	\end{align*}
	As the indices $i$ and $j$ do not play any role, we also have 
	\begin{align*}
		\int_{\partial B_{\epsilon}(p_k)}u_{\epsilon}^j\,\partial_{\nu}\left(\mathscr{L}_gu_{\epsilon}^i\right)-\partial_{\nu}\left(u_{\epsilon}^j\right)\,\mathscr{L}_gu_{\epsilon}^i\,d\mathscr{H}^1=O(\epsilon).
	\end{align*}
	
	Now, as $v$ is an admissible variation, we have at $p_i$ (if $p_i$ has multiplicity $m_i=m$)
	\begin{align*}
		v_{\epsilon}^i=v(p_i)+\Re(\gamma_iz^m)+O(|z|^{m+1}).
	\end{align*} 
	Therefore, we have
	\begin{align*}
		u_{\epsilon}^i=|\phi|^2v(p_i)+\Re\left(\frac{\tilde{\gamma_i}}{z^m}\right)+O(|z|^{1-m})
	\end{align*}
	which implies that 
	\begin{align*}
		\Delta_g u_{\epsilon}^i=4v(p_i)+|z|^{2m+2}(1+O(|z|))\times O(|z|^{-(m+1)})=4v(p_i)+O(|z|^{m+1}).
	\end{align*}
	Therefore, we have as $K_g=O(|z|^{2(m+1)})$
	\begin{align*}
		&\mathscr{L}_g u_{\epsilon}^i=(4-2K_g|\phi|^2)v(p_i)+O(|z|^{m+1})\\
		&\partial\left(\mathscr{L}_gu_{\epsilon}^i\right)=-\partial_{\nu}\left(2K_g|\phi|^2\right)v(p_i)+O(|z|^{m})
	\end{align*}
	This implies that 
	\begin{align}\label{higheripp1}
		\partial_{\nu}u_{\epsilon}^j\,\mathscr{L}_gu_{\epsilon}^i-u_{\epsilon}^i\partial_{\nu}\left(\mathscr{L}_gu_{\epsilon}^i\right)&=\frac{1}{|z|}\left(-m\,\Re\left(\frac{c_{i,j}}{z^m}\right)+\sum_{1-m\leq k+l\leq 0}\Re\left(kc_{k,l}^jz^{k}\z^l\right)+\gamma_{i,j}\right)\left(4-2\,K_g|\phi|^2\right)v(p_i)\nonumber\\
		&+\left(\Re\left(\frac{c_{i,j}}{z^m}\right)+\sum_{1-m\leq k+l\leq 0}\Re\left(c_{k,l}^jz^k\z^l\right)+\gamma_{i,j}\log|z|\right)\partial_{\nu}\left(2\,K_g|\phi|^2\right)v(p_i).
	\end{align}
	Notice that the quantity
	\begin{align*}
		B_{\epsilon}(u_{\epsilon}^i,u_{\epsilon}^j)
	\end{align*}
	is bounded as $\epsilon\rightarrow 0$. Therefore, cancellations occurs as we integrate \eqref{higheripp1}. Furthermore, there is a non-trivial contribution coming from (as $K_g|\phi|^2=O(|z|^2)$)
	\begin{align}\label{higheripp2}
		\int_{\partial B_{\epsilon}(p_i)}\frac{\gamma_{i,j}}{|z|}\left(4-K_g|\phi|^2\right)v(p_i)\,d\mathscr{H}^1=8\pi\gamma_{i,j}v(p_i)+O(\epsilon^2).
	\end{align}
	As $B_{\epsilon}(u_{\epsilon}^i,u_{\epsilon}^j)$ is bounded, we deduce that  there exists $\mu_{i,j}\in \R$ (\emph{a priori} different from $\gamma_{i,j}$ if the multiplicity $m$ satisfies $m\geq 2$) such that 
	\begin{align}\label{losttrack}
		\int_{\partial B_{\epsilon}(p_i)}\partial_{\nu}u_{\epsilon}^j\,\mathscr{L}_gu_{\epsilon}^i-u_{\epsilon}^i\partial_{\nu}\left(\mathscr{L}_gu_{\epsilon}^i\right)\,d\mathscr{H}^1=8\pi\mu_{i,j}v(p_i)+O(\epsilon).
	\end{align}
	Finally, recall if $u_{\epsilon}^i=e^{\lambda}w_{\epsilon}^i$ that in our fixed chart near $p_j$ (if $p_j$ has multiplicity $m_j=m\geq 2$)
	\begin{align*}
	w_{\epsilon}^i(z)=|z|^{m+1}\sum_{k=1}^{m}\Re\left(\frac{\gamma_0^k}{z^k}\right)+|z|^{1-m}\sum_{k=0}^{m-1}\Re\left(\gamma_1^jz^{m+k}\right)+\gamma_2|z|^{m+1}+\gamma_3|z|^{m+1}\log|z|+O(|z|^{m+2}).
    \end{align*}
    Recalling that 
    \begin{align*}
    	\mathscr{L}_m=\Delta-2(m+1)\frac{x}{|x|^2}\cdot \D+\frac{(m+1)^2}{|z|^2}
    \end{align*}
    and introducing the notation
    \begin{align*}
    	\mathscr{L}_g u_{\epsilon}^i=e^{-\lambda}\mathscr{L}w_{\epsilon}^i.
    \end{align*}
    Recalling that for all $k\geq 1$
    \begin{align*}
    	|z|^{m+1}\Re\left(\frac{\gamma_0^j}{z^k}\right),|z|^{1-m}\Re(\gamma_1^0z^m),|z|^{m+1},|z|^{m+1}\log|z|\in \mathrm{Ker}(\mathscr{L}_m),
    \end{align*}
    we deduce that 
    \begin{align*}
    	\mathscr{L}_m w_{\epsilon}^i=|z|^{-(m+1)}\sum_{j=1}^{m-1}\Re(\tilde{\gamma}_1^kz^{m+j})+O(|z|^{m}).
    \end{align*}
    Now, recall that $\phi$ admits an expansion of the following form (up to translation)
    \begin{align*}
    	\phi(z)=\sum_{k=0}^{m-1}\Re\left(\frac{\vec{A}_j}{z^{m-k}}\right)+O(|z|).
    \end{align*}
    Therefore, we have (as $\s{\vec{A}_0}{\vec{A}_1}=0$)
    \begin{align*}
    	|\phi(z)|^2=\frac{1}{2}\sum_{k=0}^{m-1}\frac{|\vec{A}_k|^2}{|z|^{2(m-k)}}+\frac{1}{2}\sum_{0\leq k<l\leq m-1}^{}\Re\left(\frac{\s{\bar{\vec{A}_k}}{\vec{A}_l}}{\z^{m-k}\z^{m-l}}\right)+\frac{1}{2}\sum_{\substack{0\leq k<l\\(k,l)\neq (0,1)}}^{}\Re\left(\frac{\s{\vec{A}_k}{\vec{A}_l}}{z^{2m-k-l}}\right)+O(|z|^{1-m}).
    \end{align*}
    As $e^{2\lambda}=\p{z\z}^2|\phi|^2$, we find
    \begin{align*}
    	&e^{2\lambda}=\frac{1}{2}\sum_{k=0}^{m-1}\frac{(m-k)^2|\vec{A}_k|^2}{|z|^{2(m+1-k)}}+\frac{1}{2}\sum_{0\leq k<l\leq m-1}^{}(m-k)(m-l)\Re\left(\frac{\s{\bar{\vec{A}_k}}{\vec{A}_l}}{\z^{m+1-k}z^{m+1-l}}\right)+O(|z|^{-(m+1)})\\
    	&=\frac{m^2|\vec{A}_0|^2}{2|z|^{2(m+1)}}\left(1+\sum_{k=1}^{m-1}\left(1-\frac{k}{m}\right)\frac{|\vec{A}_k|^2}{|\vec{A}_0|^2}|z|^{2k}+\sum_{0\leq k<l\leq m-1}^{}\left(1-\frac{k}{m}\right)\left(1-\frac{l}{m}\right)\Re\left(\frac{\s{\bar{\vec{A}_k}}{\vec{A}_l}}{|\vec{A}_0|^2}z^{l}\z^{k}\right)+O(|z|^{m+1})\right)
    \end{align*}
    Therefore, we have (up to normalisation $m^2|\vec{A}_0|^2=2$) for some $\alpha_{k,l}\in \C$ and $\beta_k\in \R$
    \begin{align*}
    	e^{-\lambda}=|z|^{m+1}\left(1+\sum_{k=1}^{(m-1)/2}\beta_k|z|^{2k}+\sum_{\substack{0\leq k< l\leq m-1\\k+l\leq m}}^{}\Re\left(\alpha_{k,l}z^{k}\z^l\right)+O(|z|^{m+1})\right),
    \end{align*}
    and
    \begin{align*}
    	&\mathscr{L}_gu_{\epsilon}^i=\left(1+\sum_{k=1}^{(m-1)/2}\beta_k|z|^{2k}+\sum_{0\leq k< l\leq m-1}^{}\Re\left(\alpha_{k,l}z^{k}\z^l\right)+O(|z|^{m+1})\right)\times \left(\sum_{k=1}^{m-1}\Re(\tilde{\gamma}_1^kz^{m+k})+O(|z|^{2m+1})\right)\\
    	&=\Re(\tilde{\gamma}_1^1z^{m+1})+\sum_{\substack{0\leq k< l\leq m-1\\ m+2\leq k+l}}\Re\left(\tilde{\alpha}_{k,l}z^k\z^l\right)+O(|z|^{2m+1}).
    \end{align*}
    Furthermore, we have on $\partial B_{\epsilon}(p_j)$
    \begin{align*}
    	u_{\epsilon}^j=|\phi|^2v(p_i)+O(|z|^{-m}),
    \end{align*}
    Therefore, we deduce that 
    \begin{align*}
    	u_{\epsilon}^j\,\partial_{\nu}\left(\mathscr{L}_gu_{\epsilon}^i\right)-\partial_{\nu}\left(u_{\epsilon}^j\right)\mathscr{L}_gu_{\epsilon}^i=\left(|\phi|^2\partial_{\nu}\left(\mathscr{L}_gu_{\epsilon}^i\right)-\partial_{\nu}\left(|\phi|^2\right)\mathscr{L}_gu_{\epsilon}^i\right)v(p_j)+O(1)
    \end{align*}
    As $B_{\epsilon}(u_{\epsilon}^i,u_{\epsilon}^j)$ is bounded, we deduce as  previously that there exists $\nu_{i,j}\in \R$ such that 
    \begin{align*}
    	\int_{\partial B_{\epsilon}(p_j)}\,u_{\epsilon}^j\,\partial_{\nu}\left(\mathscr{L}_gu_{\epsilon}^i\right)-\partial_{\nu}\left(u_{\epsilon}^j\right)\mathscr{L}_gu_{\epsilon}^i\,d\mathscr{H}^1=8\pi\nu_{i,j}v(p_j)+O(\epsilon).
    \end{align*}
	Therefore, we deduce that (as we may have also ends of multiplicity $1$, there is an additional error in $O(\epsilon\log\epsilon)$)
	\begin{align*}
		B_{\epsilon}(u_{\epsilon}^i,u_{\epsilon}^j)=8\pi \mu_{i,j}v(p_i)+8\pi\nu_{i,j}v(p_j)+O(\epsilon\log\epsilon).
	\end{align*}
	Now assume that $v(p_j)=0$ and $v(p_i)\neq 0$. Then 
	\begin{align*}
		B_{\epsilon}(u_{\epsilon}^i,u_{\epsilon}^j)=8\pi \mu_{i,j}v(p_i)+O(\epsilon\log\epsilon)
	\end{align*}
	and by symmetry, in $i$ and $j$, we deduce that 
	\begin{align*}
		B_{\epsilon}(u_{\epsilon}^i,u_{\epsilon}^j)=8\pi \mu_{j,i}v(p_j)+O(\epsilon\log\epsilon)=0.
	\end{align*}
	Therefore, $v(p_j)=0$ implies that $\mu_{i,j}=0$, which shows that there exists $\lambda_{i,j}^1\in \R$ such that 
	\begin{align*}
		\mu_{i,j}=\lambda_{i,j}^1v(p_j).
	\end{align*}
	Furthermore, by symmetry of the argument, we deduce that there exists $\lambda_{i,j}^2\in \R$ such that 
	\begin{align*}
		\nu_{i,j}=\lambda_{i,j}^2v(p_i).
	\end{align*}
	Therefore, if $\lambda_{i,j}^3=\lambda_{i,j}^1+\lambda_{i,j}^2$, we deduce that 
	\begin{align}\label{losttrack2}
		B_{\epsilon}(u_{\epsilon}^i,u_{\epsilon}^j)=8\pi  \lambda_{i,j}^3v(p_i)v(p_j)+O(\epsilon\log\epsilon).
	\end{align}
	Likewise, we find by the previous argument and the proof of Theorem \ref{explicit} that there exists $\lambda_{i,j}^4\in \R$ such that 
	\begin{align}\label{losttrack3}
		B_{\epsilon}(u_{\epsilon},u_{\epsilon}^i)=4\pi\sum_{j\neq i}^{}\lambda_{i,j}^4v(p_i)v(p_j)+O(\epsilon\log\epsilon).
	\end{align}
	Combining \eqref{losttrack2} and \eqref{losttrack3}, we deduce if $\lambda_{i,j}=\lambda_{i,j}^3+\lambda_{i,j}^4$ that 
	\begin{align}\label{part2}
		\sum_{i=1}^nB_{\epsilon}(u_{\epsilon},u_{\epsilon}^i)+\frac{1}{2}\sum_{1\leq i\neq j\leq n}B_{\epsilon}(u_{\epsilon}^i,u_{\epsilon}^j)=4\pi\sum_{1\leq i\neq j\leq n}\lambda_{i,j}v(p_i)v(p_j)+O(\epsilon\log^2\epsilon).
	\end{align}
	Therefore, \eqref{part1} and \eqref{part2} show that 
	\begin{align*}
		Q_{\epsilon}(u)=Q_{\epsilon}(u_{\epsilon})+4\pi\sum_{1\leq i, j\leq n}\lambda_{i,j}v(p_i)v(p_j)+O(\epsilon\log^2\epsilon)
	\end{align*}
	and finally
	\begin{align*}
		Q(u)=Q(u_0)+\sum_{1\leq i,j\leq n}\lambda_{i,j}v(p_i)v(p_j).
	\end{align*}
	This concludes the proof of the theorem.   
\end{proof}

We deduce as previously the following corollary.
\begin{cor}\label{higher2}
	For all $1\leq i\leq n$, there exists $\tilde{\lambda}_{i,j}\in \R$ such that for all $j\neq i$, for all $0<\epsilon<\epsilon_0$, on every complex chart around $p_j$ there exists $c_{i,j},c_{i,j,k,l}\in \C$ 
	\begin{align*}
	u_{\epsilon}^i(z)&
	=\Re\left(\frac{c_{i,j}}{z^m}\right)+\sum_{1-m\leq k+l\leq 0}^{}\Re\Big(c_{i,j,k,l}z^{k}\z^l\Big)+\tilde{\lambda}_{i,j}\log|z|+\psi_{\epsilon}(z),
	\end{align*}
	where $\psi_{\epsilon}\in C^{\infty}(B(0,1)\setminus\ens{0})$ such that for all $l\in \N$
	\begin{align*}
		\D^l\psi_{\epsilon}=O(|z|^{1-l}).
	\end{align*}
\end{cor}

For ends of higher multiplicity $m\geq 2$, we do not know \emph{a priori} if $\tilde{\lambda}_{i,j}=\lambda_{i,j}$, where $\lambda_{i,j}\in\R$ is given by Theorem. Nevertheless, the proof of Theorem \ref{embedded} implies the following result.

\begin{theorem}
	Let $\phi:\Sigma\setminus\ens{p_1,\cdots,p_n}\rightarrow \R^3$ be a complete minimal surface with finite total curvature, and $\vec{\Psi}=\iota\circ \phi:\Sigma\rightarrow \R^3$ be a compact inversion of $\phi$, and let $\Lambda(\vec{\Psi})=\ens{\lambda_{i,j}}_{1\leq i,j\leq n}\in \mathrm{Sym}_n(\R)$ be the matrix given by Theorem 
	\ref{general}, and $\ens{\lambda_{i,j}}_{1\leq i\neq j\leq n}$ be given by Corollary \ref{higher2}. Define
	\begin{align*}
		\tilde{\Lambda}(\vec{\Psi})=\begin{pmatrix}
		\vspace{0.5em}\lambda_{1,1}&\lambda_{1,2}+2n\,\tilde{\lambda}_{1,2} &  \cdots  &\lambda_{1,n}+2n\,\tilde{\lambda}_{1,n}\\
		\vspace{0.5em}\lambda_{1,2}+2n\,\tilde{\lambda}_{1,2} & \lambda_{2,2}& \cdots  &\lambda_{2,n}+2n\,\tilde{\lambda}_{2,n}\\
		\vspace{0.5em}\vdots& \vdots & \ddots &\vdots \\
		\vspace{0.5em}\lambda_{1,n}+2n\,\tilde{\lambda}_{1,n} &\lambda_{2,n}+2n\,\tilde{\lambda}_{1,n} &  \cdots & \lambda_{n,n}.
		\end{pmatrix}
	\end{align*}
	Then we have
	\begin{align*}
		\mathrm{Ind}\,\tilde{\Lambda}(\vec{\Psi})\leq \mathrm{Ind}_{W}(\vec{\Psi})\leq \mathrm{Ind}\,\Lambda(\vec{\Psi}).
	\end{align*}
\end{theorem}
\begin{proof}
	As in the proof of Theorem \ref{embedded}, it suffices to compute the renormalised energy for variations $v$ of the form (if $p_i$ has multiplicity $m\geq 1$)
	\begin{align*}
		v=v(p_i)+\Re(\gamma z^{m})+\cdots+|z|^{2m}\tilde{\gamma}\log|z|+O(1).
	\end{align*}
	where $\cdots$ indicate terms of lower order that $|z|^{2m}\log|z|$ and higher than $|z|^{m}$. Now notice if $u_0^i$ is the limit of $u_{\epsilon}^i$ that 
	\begin{align*}
		u_0^i=|\phi|^2v(p_i)+\Re\left(\frac{\zeta_0}{z^m}\right)+\cdots+\sum_{j\neq i}^{}\tilde{\lambda}_{i,j}v(p_j)\log|z|+O(1).
	\end{align*}
	Therefore, as $v_0^i=|\phi|^{-2}u_0^i$ is admissible, we just need to compute if $u=u_0^i$ the additional constant term in 
	\begin{align*}
		\int_{\partial B_{\epsilon}(p_i)}\left(\Delta_g u+2K_gu\right)\ast du-\frac{1}{2}\ast d\,|du|_g^2=\Im\int_{\partial B_{\epsilon}(p_i)}2\left(\Delta_gu+2K_gu\right)\partial u-\partial\,|du|_g^2
	\end{align*}
	coming from the addition of the $\log$ term. As $\log|z|$ is harmonic, we have
	\begin{align*}
		\Delta_gu+2K_gu=4v(p_i)+O(|z|^2),
	\end{align*}
	which implies that the additional constant term in 
	\begin{align}
		\Im\int_{\partial B_{\epsilon}(p_i)}2\left(\Delta_gu+2K_gu\right)\partial u
	\end{align}
	is
	\begin{align}\label{npart1}
		2\,\Im\int_{\partial B(0,\epsilon)}\left(4v(p_i)+O(|z|^2)\right)\left(\frac{1}{2}\sum_{j\neq i}\tilde{\lambda}_{i,j}v(p_j)\frac{dz}{z}\right)=8\pi \sum_{j\neq i}^{}\tilde{\lambda}_{i,j}v(p_i)v(p_j).
	\end{align}
	Furthermore, as 
	\begin{align*}
		|du|_g^2
	\end{align*}
	has no component on $\log|z|$ in its Taylor expansion, there is no constant term in
	\begin{align*}
		\Im\int_{\partial B_{\epsilon}(p_i)}\partial |du|_g^2.
	\end{align*}
	Finally, the proof of Theorem \ref{embeddedends} shows that if
	\begin{align*}
		u_0=\sum_{i=1}^{n}u_0^i
	\end{align*}
	and $u_0=|\phi|^2v_0$ that 
	\begin{align*}
		Q_{\vec{\Psi}}(v_0)&=4\pi\sum_{1\leq i,j\leq n}^{}\lambda_{i,j}v(p_i)v(p_j)+\sum_{i=1}^{n}8\pi\sum_{j\neq i}^{}\tilde{\lambda}_{i,j}v(p_i)v(p_j)\\
		&=4\pi\sum_{i=1}^{n}\lambda_{i,i}v^2(p_i)+4\pi\sum_{\substack{1\leq i,j\leq n\\ i\neq j}}^{}(\lambda_{i,j}+2n\,\tilde{\lambda}_{i,j})v(p_i)v(p_j).
	\end{align*}
	This identity completed the proof of the theorem. 
\end{proof}

\begin{rem}
	This is very likely that $\tilde{\lambda}_{i,j}=\lambda_{i,j}$, which would give the exact analogous of Theorem \ref{embedded} for higher order ends, but the argument here does not permit to check this. An explicit computation for precise examples permits nevertheless to compute the additional contributions in $\lambda_{i,j}$ from $\tilde{\lambda}_{i,j}$. 
\end{rem}

\section{Morse index estimate for Willmore spheres in $S^4$}

Recall that we have from \cite{index3} we have the formula (valid in $\mathscr{D}'(\Sigma)$) for all weak immersion $\phi\in \mathscr{E}(\Sigma,\R^m)$
\begin{align*}
\frac{d^2}{dt^2}\left(K_{g_t}d\mathrm{vol}_{g_t}\right)_{|t=0}&=d\,\Im\,\bigg(2\s{\Delta_g^{\perp}\w+4\,\Re\left(g^{-2}\otimes \left(\bar{\partial}\phi\totimes \bar{\partial}\w\right)\otimes \h_0\right)}{\partial\w}-\partial\left(|\D^{\perp}\w|_g^2\right)\\
&-2\s{d\phi}{d\w}_g\left(\s{\H}{\partial\w}
-g^{-1}\otimes\left(\h_0\totimes \bar{\partial}\w\right)\right)-8\,g^{-1}\otimes\left(\partial\phi\totimes\partial\w\right)\otimes\s{\H}{\bar{\partial}\w}\bigg)
\end{align*}	
In particular, as $\s{d\phi}{d\w}_g=-2\s{\H}{\w}$, for a minimal surface ($\H=0$), we obtain
\begin{align*}
\frac{d^2}{dt^2}\left(K_{g_t}d\mathrm{vol}_{g_t}\right)_{|t=0}&=d\,\Im\,\left(2\s{\Delta_g^{\perp}\w+4\,\Re\left(g^{-2}\otimes \left(\bar{\partial}\phi\totimes \bar{\partial}\w\right)\otimes \h_0\right)}{\partial\w}-\partial\left(|\D^{\perp}\w|_g^2\right)\right)\\
&=d\,\Im\,\left(2\s{\Delta_g^{\perp}\w-2\,\Re\left(g^{-2}\otimes \s{\w}{\bar{\h_0}}\otimes \h_0\right)}{\partial\w}-\partial\left(|\D^{\perp}\w|_g^2\right)\right)\\
&=d\,\Im\left(2\s{\Delta_g^{\perp}\w-\mathscr{A}(\w)}{\partial\w}-\partial\left(|\D^{\perp}\w|^2_g\right)\right)\\
&=d\,\left(\s{\Delta_g^{\perp}\w-\mathscr{A}(\w)}{\star\,d\w}-\frac{1}{2}\,\star\, d\left(|\D^{\perp}\w|_g^2\right)\right),
\end{align*}
where $\mathscr{A}(\w)$ is the Simon's operator. Observe that the sign is different from the Jacobi operator $\mathscr{L}_g$ of the associated minimal surface, acting on normal sections of the pull-back bundle $\phi^{\ast}T\R^m$ as
\begin{align*}
\mathscr{L}_g=\Delta^{\perp}_g+\mathscr{A}(\w).
\end{align*}
Specialising further to the codimension $1$ case $m=3$, as the minimal immersion that we consider is orientable, it is also two-sided and the unit normal furnishes a global trivialisation of the normal bundle so $\w=w\n$ for some $w\in W^{2,2}(S^2)$ and we get
\begin{align*}
\frac{d^2}{dt^2}\left(K_{g_t}d\mathrm{vol}_{g_t}\right)_{|t=0}=d\,\Im\Big(2\left(\Delta_gw+2K_g\,w\right)\,\partial w-\partial\left(|dw|_g^2\right)\Big)=d\left(\left(\Delta_gw+2K_g\,w\right)\,\star dw-\frac{1}{2}\star\,d|dw|_g^2\right),
\end{align*}
and we recover the computation of \cite{indexS3}.
We have the following generalisation of the afore-cited result to $S^4$.

\begin{theorem}\label{s4}
	Let $\vec{\Psi}:S^2\rightarrow S^4$ be a Willmore sphere, and $n\in \N$ such that $W(\vec{\Psi})=4\pi n$ and assume that $\phi$ is conformally minimal in $\R^4$. Then we have $\mathrm{Ind}_W(\vec{\Psi})\leq n$.
\end{theorem}	
\begin{proof}
	First, use some stereographic projection avoiding $\vec{\Psi}(S^2)\subset S^4$ to assume that $\vec{\Psi}:S^2\rightarrow \R^4$ is a Willmore sphere. By Montiel's classification, let $\phi:S^2\setminus\ens{p_1,\cdots,p_n}\rightarrow \R^4$ be the complete minimal surface $\vec{\Psi}:S^2\rightarrow \R^4$ is the inversion, which we assume centred at $0\in \R^4$ up to translation. Thanks to the argument of \cite{indexS3}, for all normal variation $\vec{v}\in \mathscr{E}_{\vec{\Psi}}(S^2,T\R^4)$, we have
	\begin{align*}
	D^2W(\vec{\Psi})(\vec{v},\vec{v})&=\int_{S^2\setminus\ens{p_1,\cdots,p_n}}\bigg\{ \frac{1}{2}|\Delta_g^{\perp}\w+\mathscr{A}(\w)|^2d\vg-d\,\Im\,\Big(2\s{\Delta_g^{\perp}\w-\mathscr{A}(\w)}{\partial\w}
	-\partial\left(|\D^{\perp}\w|_g^2\right)\Big)\bigg\},
	\end{align*}
	where 
	\begin{align*}
	\w=\mathscr{I}_{\phi}(\vec{v})=|\phi|^2\vec{v}-2\s{\phi}{\vec{v}}\phi\,.
	\end{align*}
	Then at every end, we have an expansion (up to translation)
	\begin{align*}
	\phi(z)=\Re\left(\frac{\vec{A}_0}{z}\right)+O(|z|)
	\end{align*}
	for some $\vec{A}_0\in \C^4\setminus \ens{0}$. Then $\s{\vec{A}_0}{\vec{A}_0}=0$ and we find that for some $\alpha_j>0$
	\begin{align*}
	|\phi|^2=\frac{\alpha_j^2}{|z|^2}+O(|z|).
	\end{align*}
	Thanks to the Sobolev embedding $W^{2,2}(S^2)\rightarrow C^0(S^2)$ and as $W^{2,2}(S^2)$ does not embed in $C^1(S^2)$ in general, we deduce that for all smooth $\vec{v}\in \mathscr{E}_{\Psi}(S^2,T\R^4)$, the residue
	\begin{align*}
	\int_{\partial B_{\epsilon}(p_j)}\Im\left(2\s{\Delta_g^{\perp}\w-\mathscr{A}(\w)}{\partial \vec{w}}-\partial\left(|\D^{\perp}\vec{w}|^2_g\right)\right)
	\end{align*}
	only depend on $\alpha_j$, $\epsilon>0$ and $\vec{v}(p_j)$, up to a negligible term as $\epsilon\rightarrow 0$ (it cannot depend on higher derivatives of $\vec{v}$ at $p_j$). Furthermore, as we can assume $\vec{v}$ to be a normal variation, we deduce that $\s{\vec{A}_0}{\vec{v}(p_j)}=0$. In particular, we deduce that 
	\begin{align*}
	\int_{\partial B_{\epsilon}(p_j)}\Im\left(2\s{\Delta_g^{\perp}\w-\mathscr{A}(\w)}{\partial \vec{w}}-\partial\left(|\D^{\perp}\vec{w}|^2_g\right)\right)&=|\vec{v}(p_j)|^2\,\Im\int_{\partial B_{\epsilon}(p_j)}2\,\Delta_g|\phi|^2\partial |\phi|^2-\partial\left(|d|\phi|^2|^2_g\right)+o_{\epsilon}(1)\\
	&=8\pi\alpha_j^2|\vec{v}(p_j)|^2+o_{\epsilon}(1)
	\end{align*}
	by the same computation as in Theorem \ref{embeddedends}. The rest of the proof follows \cite{indexS3}.
\end{proof}

\begin{rems}
	\begin{itemize}
		\item[(1)] In order to understand fully the indices of Willmore spheres in $S^4$, one would also have to estimate the index of images of (complex) curves in $\P^3$ coming from the Penrose twistor fibration $\P^3\rightarrow S^4$.
		\item[(2)] The same proof applies to $\R^m$ for $m\geq 5$ without any change, but as not all Willmore spheres are conformally minimal (or images of complex curves of $\P^3$ through the Penrose twistor fibration), this result has little interest  (\cite{pengwang}).
	\end{itemize}
\end{rems}

Here, we show as in \cite{indexS3} that there is a well-defined notion of residues at ends of embedded minimal surfaces in arbitrary codimension. 

\begin{prop}
	Let $\Sigma$ be a closed Riemann surface and $\phi:\Sigma\setminus\ens{p_1,\cdots,p_n}\rightarrow \R^m$ be a complete minimal surface with embedded planar ends. Fix a covering $(U_1,\cdots,U_n)$ of $\ens{p_1,\cdots,p_n}\subset\C$. Then the limit
	\begin{align*}
	\lim\limits_{\epsilon \rightarrow 0}\left(-\frac{\epsilon^2}{4\pi}\int_{\partial B_{\epsilon}(p_i)}\ast\, d\left(4|\phi|^2-\frac{1}{2}|d|\phi||^2_g\right)\right)
	\end{align*}
	is a positive real number independent depending only on $(U_1,\cdots,U_n)$ and $\phi$ and we denote it $\mathrm{Res}_{p_j}(\Sigma,U_j)$.
\end{prop}
\begin{proof}
	As the ends are embedded and planar, there exists $\vec{A}_0\in \C^n\setminus\ens{0}$, $\vec{B}_0\in \C^n$, and $\vec{C}_0\in \R^n$ we can assume that
	\begin{align*}
	\phi(z)=2\,\Re\left(\frac{\vec{A}_0}{z}+\vec{B}_0z\right)+\vec{C}_0+O(|z|^2)
	\end{align*}
	and we obtain
	\begin{align*}
	\p{z}\phi=-\frac{\vec{A}_0}{z^2}+\vec{B}_0+O(|z|),
	\end{align*}
	and as $\phi$ is conformal, we have
	\begin{align*}
	0=\s{\p{z}\phi}{\p{z}\phi}=\frac{\s{\vec{A}_0}{\vec{A}_0}}{z^4}-2\frac{\s{\vec{A}_0}{\vec{B}_0}}{z^2}+O\left(\frac{1}{|z|}\right),
	\end{align*}
	which implies that
	\begin{align*}
	\s{\vec{A}_0}{\vec{A}_0}=\s{\vec{A}_0}{\vec{B}_0}=0.
	\end{align*}
	In particular, we obtain
	\begin{align*}
	|\phi(z)|^2=\frac{2|\vec{A}_0|^2}{|z|^2}+4\,\Re\left(\frac{\s{\vec{A}_0}{\vec{C}_0}}{z}\right)+2\,\Re\left(\s{\vec{A}_0}{\bar{\vec{B}_0}}\frac{\z}{z}\right)+O(|z|).
	\end{align*}
	Now, to simplify notations, write
	\begin{align}
	\alpha^2=|\vec{A}_0|^2,\quad \beta=\s{\vec{A}_0}{\bar{\vec{B}_0}},\quad \gamma=\s{\vec{A}_0}{\vec{C}_0},
	\end{align}
	which implies that
	\begin{align*}
	|\phi(z)|^2=\frac{2\alpha^2}{|z|^2}+4\,\Re\left(\frac{\gamma}{z}\right)+2\left({\beta}\frac{z}{\z}\right)+O(|z|)
	\end{align*}
	We obtain
	\begin{align*}
	\partial|\phi(z)|^2&=\left(-\frac{2\alpha^2}{z|z|^2}-\frac{2\gamma}{z^2}+\frac{{\beta}}{\z}-\bar{\beta}\frac{\z}{z^2}+O(1)\right)dz\\
	&=-\frac{2\alpha^2}{|z|^2}\left(\frac{1}{z}-\frac{\gamma}{\alpha^2}\frac{\z}{z}+\frac{\beta}{2\alpha^2}z-\frac{\bar{\beta}}{2\alpha^2}\frac{\z^2}{z}+O(|z|^2)\right)dz.
	\end{align*}
	Therefore, we have
	\begin{align*}
	|\partial|\phi|^2|^2&=\frac{4\alpha^4}{|z|^4}\left(\frac{1}{|z|^2}-2\,\Re\left(\frac{\gamma}{\alpha^2}\frac{{1}}{z}\right)+2\,\Re\left(\frac{\beta}{\alpha^2}\frac{z}{\z}\right)-2\,\Re\left(\frac{\bar{\beta}}{\alpha^2}\frac{\z}{z}\right)+O(|z|)\right)\\
	&=\frac{4\alpha^4}{|z|^6}\left(1-2\,\Re\left(\frac{{\gamma}}{\alpha^2}\z\right)+O(|z|^3)\right)
	\end{align*}
	We also compute
	\begin{align*}
	e^{2\lambda}&=2|\p{z}\phi|^2=\frac{2|\vec{A}_0|^2}{|z|^4}-4\,\Re\left(\frac{\s{\vec{A}_0}{\bar{\vec{B}_0}}}{z^2}\right)+O\left(\frac{1}{|z|}\right)=\frac{2\alpha^2}{|z|^4}\left(1-2\Re\left(\frac{\beta}{\alpha^2} z^2\right)+O(|z|^3)\right)
	\end{align*}
	and we obtain finally by
	\begin{align*}
	&|d|\phi|^2|_g^2=4e^{-2\lambda}|\partial|\phi|^2|^2=\frac{2|z|^4}{\alpha^2}\left(1+2\Re\left(\frac{\beta}{\alpha^2}z^2\right)+O(|z|^3)\right)\times \frac{4\alpha^4}{|z|^6}\left(1-2\,\Re\left(\frac{\gamma}{\alpha^2}\z\right)+O(|z|^3)\right)\\
	&=\frac{8\alpha^2}{|z|^2}\left(1-2\,\Re\left(\frac{\gamma}{\alpha^2}\z\right)+2\Re\left(\frac{\beta}{\alpha^2}z^2\right)+O(|z|^3)\right)\\
	&=\frac{8\alpha^2}{|z|^2}-16\,\Re\left(\frac{\gamma}{z}\right)+16\,\Re\left(\beta\frac{z}{\z}\right)+O\left(|z|\right).
	\end{align*}
	Therefore, we have
	\begin{align*}
	4|\phi|^2-\frac{1}{2}|d|\phi|^2|^2_g&=4\left(\frac{2\alpha^2}{|z|^2}+4\,\Re\left(\frac{\gamma}{z}\right)+2\left({\beta}\frac{z}{\z}\right)+O(|z|)\right)-\frac{1}{2}\left(\frac{8\alpha^2}{|z|^2}-16\,\Re\left(\frac{\gamma}{z}\right)+16\,\Re\left(\beta\frac{z}{\z}\right)+O\left(|z|\right)\right)\\
	&=\frac{4\alpha^2}{|z|^2}+24\,\Re\left(\frac{\gamma}{z}\right)+O(|z|).
	\end{align*}
	Therefore, we obtain
	\begin{align*}
	\partial\left(4|\phi|^2-\frac{1}{2}|d|\phi|^2|^2_g\right)=-\frac{4\alpha^2}{|z|^2}\frac{dz}{z}-24\gamma \frac{dz}{z^2}+O(1),
	\end{align*}
	which implies that
	\begin{align*}
	\Im\int_{S^1(0,\epsilon)}\partial\left(4|\phi|^2-\frac{1}{2}|d|\phi|^2|^2_g\right)=-8\pi \frac{\alpha^2}{\epsilon^2}+O(\epsilon).
	\end{align*}
	Therefore, we obtain
	\begin{align*}
	\lim_{\epsilon\rightarrow 0}\left(-\frac{\epsilon^2}{4\pi}\int_{S^1(0,\epsilon)}\star \,d \left(4|\phi|^2-\frac{1}{2}|d|\phi|^2|^2_g\right)\right)=4\alpha^2>0,
	\end{align*}
	and concludes the proof of the Proposition.
\end{proof}
\begin{rem}
	Although this quantity is independent on the coordinate, it depends on the covering $\ens{U_1,\cdots,U_n}$ of $\ens{p_1,\cdots,p_n}\subset \Sigma$.
\end{rem}

\section{Explicit renormalised energy for ends of multiplicity $2$}\label{multiplicity2}

\subsection{Restriction on the Weierstrass parametrisation}

\begin{lemme}\label{diag2}
	Let $\Sigma$ be a closed Riemann surface, $p_1,\cdots,p_n\in \Sigma$ be distinct points and  $\phi:\Sigma\setminus\ens{p_1,\cdots,p_n}\rightarrow \R^3$ be a complete minimal immersion with finite total curvature and \emph{zero flux}. Suppose that $U\subset \Sigma$ is a fix chart domain containing some point $p_i\in \Sigma$ ($1\leq i\leq n$ fixed), and that end $p_i$ has multiplicity $m= 2$. Let $(\vec{A}_0,\vec{A}_1)\in \C^n\setminus\ens{0}\times \C^n$ be such that in a chart $\varphi:U\rightarrow D^2\subset \C$ such that $\varphi(p_i)=0$ we have the expansion
	\begin{align*}
	\phi(z)=\Re\left(\frac{\vec{A}_0}{z^2}+\frac{\vec{A}_1}{z}\right)+O\left(1\right).
	\end{align*}
	Then $\vec{A}_1\in \mathrm{Span}(\vec{A}_0)$.                                                                                  
\end{lemme}
\begin{rem}
	This Lemma does not hold for $m\geq 3$ in general, as for the Enneper surface 
	\begin{align*}
	\phi(z)=\Re\left(\frac{\vec{A}_0}{z^3}+\frac{\vec{A}_1}{z^2}+\frac{\vec{A}_2}{z}\right)+O(1).
	\end{align*}
	where (up to scaling) $\vec{A}_0=(-1,i,0)$, and $\vec{A}_1=(0,0,3)$ (see \cite{dierkes}). In general, one can also check that for an end of multiplicity $3$ of a complete minimal surface without flux there are no linear relations between $\vec{A}_0,\vec{A}_1,\vec{A}_2\in \C^3$.
\end{rem}
\begin{proof}
	
	If $\phi:\Sigma\setminus\ens{p_1,\cdots,p_n}\rightarrow \R^3$ is a complete minimal surface with \emph{no flux} and $p_i$ (for some $1\leq i\leq n$) has multiplicity $2$, then we have thanks to the Weierstrass parametrisation in a conformal parametrisation around $p_i$ the expression
	\begin{align*}
	\phi(z)=\Re\left(\int_{\ast}^z\left(1-g^2,i\left(1+g^2\right),2g\right)\omega\right)
	\end{align*}
	for some meromorphic function $g$ (this is the stereographic projection of a the Gauss map $\n:\Sigma\rightarrow S^2$) and a meromorphic $1$-form $\omega=f(z)dz$.
	As $\phi$ has no flux, we see that the maximum multiplicity of the pole at zero of the following $1$-forms
	\begin{align*}
	\omega, g\omega, g^2\omega
	\end{align*}
	is exactly equal to $3$ (as $p_i$ has multiplicity $2$), and each of these $1$-forms must be exact (this is equivalent to the exactness of $(1-g^2)\omega,i(1+g^2)\omega$ and $2g\omega$). 
	
	\textbf{Case 1:} The function $g$ has a pole at $0$ of order $k\geq 2$,
	
	Let $j\in \Z\setminus\ens{0}$ and $\lambda_{-k}\in \C\setminus\ens{0}$, $\omega_j\in \C\setminus\ens{0}$ be such that
	\begin{align*}
	&g(z)=\frac{\lambda_{-k}}{z^k}+O\left(\frac{1}{|z|^{k-1}}\right)\\
	&\omega=\left(\omega_jz^j+O(|z|^j)\right)dz
	\end{align*}
	Then we have for some 
	\begin{align*}
	&g\omega=\left(\frac{\lambda_{-k}\omega_j}{z^{k-j}}+O\left(\frac{1}{|z|^{2k-j-1}}\right)\right)dz\\
	&g^2\omega=\left(\frac{\lambda_{-k}^2\omega_j}{z^{2k-j}}+O\left(\frac{1}{|z|^{2k-j-1}}\right)\right)dz.
	\end{align*}
	And as $g^2\omega$ has a pole of order at most $3$, we deduce that $j\geq 2k-3\geq 1$, so $\omega$ is holomorphic at $z=0$. As $(1,g,g^2)\omega$ has a pole of order exactly $3$ at zero, we deduce that $g^2\omega$ has a pole of order $3$ at $0$. Therefore, $j=2k-3$, and
	\begin{align}\label{cont1}
	g\omega=\left(\frac{\lambda_{-k}\omega_{2k-3}}{z^{3-k}}+O(1)\right)dz
	\end{align}
	As $g\omega$ has no residue, $k=2$ is excluded by \eqref{cont1} (as $\lambda_{-k}\omega_{2k-3}\neq 0$). 
	As $k\geq 3$, we deduce by \eqref{cont1} that $g\omega$ is holomorphic at $0$, and as $\omega$ is also holomorphic, we have the expansion
	\begin{align}\label{disp}
	(1-g^2,i(1+g^2),2g)\omega=(-1,i,0)g^2\omega+O(1).
	\end{align}
	Therefore, if $(\mu_0,\mu_1)\in \C^{\ast}\times \C$ are such that (recall that $g^2\omega$ has no residue)
	\begin{align*}
	g^2\omega=\left(-2\frac{\mu_0}{z^3}-\frac{\mu_1}{z^2}+O(1)\right)dz
	\end{align*}
	we obtain by \eqref{disp}
	\begin{align*}
	\int_{\ast}^z(1-g^2,i(1+g^2),2g)\omega=\left(-1,i,0\right)\left(\frac{\mu_0}{z^2}+\frac{\mu_1}{z}\right)+O(1)
	\end{align*}
	and
	\begin{align*}
	\phi(z)=\Re\left(\int_{\ast}^z(1-g^2,i(1+g^2),2g)\omega\right)=\Re\left(\frac{\vec{A}_0}{z^2}+\frac{\vec{A}_1}{z}\right)+O(1)
	\end{align*} 
	where
	\begin{align*}
	&\vec{A}_0=\mu_0\left(-1,i,0\right)\\
	&\vec{A}_1=\mu_1\left(-1,i,0\right)=\frac{\mu_1}{\mu_0}\vec{A}_0.
	\end{align*}
	Therefore, we have proved the Lemma in the first special case.
	
	\textbf{Case 2:} $g$ has a pole of order $1$. Then $g^2\omega$ must have a pole of order $3$, but $\omega$ is exact, so this is impossible.
	
	\textbf{Case 3 :} $g$ is holomorphic. Therefore, $\omega$ must have a pole of order $3$ at $0$ (as the maximum order of the pole of the $1$-forms $\omega,g\omega,g^2\omega$ at $0$ is the one of $\omega$ - $g$ may vanish at $0$), so there exists $\lambda_j,\omega_j\in \C$ such that (recall that $\omega$ is exact)
	\begin{align*}
	&g(z)=\lambda_0+\lambda_1z+\lambda_2z^2+O(|z|^3)\\
	&\omega=\left(\frac{\omega_{-3}}{z^3}+\frac{\omega_{-2}}{z^2}+O(1)\right)dz.
	\end{align*}
	Then we compute
	\begin{align*}
	g\omega=\left(\frac{\lambda_0\omega_{-3}}{z^3}+\frac{\lambda_0\omega_{-2}+\lambda_1\omega_{-3}}{z^2}+\frac{\lambda_1\omega_{-2}+\lambda_2\omega_{-3}}{z}\right)dz.
	\end{align*}
	As $g\omega$ is exact, we have 
	\begin{align}\label{syst11}
	\lambda_1\omega_{-2}+\lambda_{2}\omega_{-3}=0.
	\end{align}
	Now, we have
	\begin{align*}
	&g^2=\lambda_0^2+2\lambda_0\lambda_1z+\left(2\lambda_0\lambda_2+\lambda_1^2\right)z^2+O(|z|^3)\\
	&\omega=\left(\frac{\omega_{-3}}{z^3}+\frac{\omega_{-2}}{z^2}+O(1)\right)dz
	\end{align*}
	so
	\begin{align*}
	g^2\omega=\left(\frac{\lambda_0^2\omega_{-3}}{z^3}+\frac{\lambda_0^2\omega_{-2}+2\lambda_1\omega_{-3}}{z^2}+\frac{2\lambda_0\lambda_1\omega_{-2}+\left(2\lambda_0\lambda_2+\lambda_1^2\right)\omega_{-3}}{z}+O(1)\right)dz.
	\end{align*}
	The exactness of $g^2\omega$ and \eqref{syst11} imply that
	\begin{align*}
	0=2\lambda_0\lambda_1\omega_{-2}+\left(2\lambda_0\lambda_2+\lambda_1^2\right)\omega_{-3}=2\lambda_0\left(\lambda_1\omega_{-2}+\lambda_2\omega_{-3}\right)+\lambda_1^2\omega_{-3}=\lambda_1^2\omega_{-3}
	\end{align*}
	and as $\omega_{-3}\neq 0$ we obtain $\lambda_1=0$. Therefore \eqref{syst11} becomes
	\begin{align*}
	\lambda_2\omega_{-3}=0
	\end{align*}
	so $\lambda_2=0$ (recall that $\omega_{-3}\neq 0$). Finally, we see that
	\begin{align*}
	g(z)=\lambda_0+O(|z|^3)
	\end{align*}
	and if $\lambda=\lambda_0$ we find
	\begin{align*}
	\left(1-g^2,i(1+g^2),2g\right)\omega=\left(\left(1-\lambda^2,i\left(1+\lambda^2\right),2\lambda\right)\left(\frac{\omega_{-3}}{z^3}+\frac{\omega_{-2}}{z^2}\right)+O(1)\right)dz
	\end{align*}
	Therefore, we have
	\begin{align*}
	\int_{\ast}^z\left(1-g^2,i(1+g^2),2g\right)\omega=-\frac{1}{2}\left(1-\lambda^2,i(1+\lambda^2),2\lambda\right)\frac{\omega_{-3}}{z^2}-\left(1-\lambda^2,i(1+\lambda^2),2\lambda\right)\frac{\omega_{-2}}{z}+O(1).
	\end{align*}
	Finally, defining (recall that $\omega_{-3}\neq 0$)
	\begin{align*}
	\vec{A}_0=-\frac{1}{2}\left(1-\lambda^2,i(1+\lambda^2),2\lambda\right)\omega_{-3},\quad \vec{A}_1=-\left(1-\lambda^2,i(1+\lambda^2),2\lambda\right)\omega_{-2}=\frac{2\omega_{-2}}{\omega_{-3}}\vec{A}_0,
	\end{align*}
	we obtain
	\begin{align*}
	\phi(z)=\Re\left(\frac{\vec{A}_0}{z^2}+\frac{\vec{A}_1}{z}\right)+O(1).
	\end{align*}
	Therefore, the two coefficients $\vec{A}_0$ and $\vec{A}_1$ are linearly dependent, which concludes the proof of the Lemma.
\end{proof}

\subsection{Explicit computation}

\normalsize

Let $\phi:D^2\setminus\ens{0}\rightarrow \R^n$ be a conformal parametrisation of an end of $\phi$. Thanks to the Weierstrass parametrisation, there exists $\vec{A}_0\in \C^n\setminus\ens{0}$ and $\vec{A}_1,\vec{A}_2,\vec{A}_3,\vec{A}_4\in \C^n$ such that
\begin{align*}
\phi(z)=2\,\Re\left(\frac{\vec{A}_0}{z^2}+\frac{\vec{A}_1}{z}+\vec{A}_2z+\vec{A}_3z^2\right)+O(|z|^3)
\end{align*}
Notice that by a translation the constant term can be taken equal to $0$.
We compute
\begin{align*}
\p{z}\phi=-2\frac{\vec{A}_0}{z^3}-\frac{\vec{A}_1}{z^2}+\vec{A}_2+2\vec{A}_3z+O(|z|^2)
\end{align*}
and we have by conformity of $\phi$ the identity
\begin{align*}
0=\s{\p{z}\phi}{\p{z}\phi}=4\frac{\s{\vec{A}_0}{\vec{A}_0}}{z^6}+4\frac{\s{\vec{A}_0}{\vec{A}_1}}{z^5}+\frac{\s{\vec{A}_1}{\vec{A}_1}}{z^4}-4\frac{\s{\vec{A}_0}{\vec{A}_2}}{z^3}-2\frac{4\s{\vec{A}_0}{\vec{A}_3}+\s{\vec{A}_1}{\vec{A}_2}}{z^2}+O\left(\frac{1}{|z|}\right)
\end{align*}
Therefore, we have
\begin{align}\label{cancel2}
\s{\vec{A}_0}{\vec{A}_0}=\s{\vec{A}_0}{\vec{A}_1}=\s{\vec{A}_0}{\vec{A}_2}=\s{\vec{A}_1}{\vec{A}_1}=4\s{\vec{A}_0}{\vec{A}_3}+\s{\vec{A}_1}{\vec{A}_2}=0.
\end{align}
Now, we compute
\begin{align*}
\frac{e^{2\lambda}}{2}&=|\p{z}\phi|^2=4\frac{|\vec{A}_0|^2}{|z|^6}+4\,\Re\left(\s{\bar{\vec{A}_0}}{\vec{A}_1}z^{-2}\z^{-3}\right)+\frac{|\vec{A}_1|^2}{|z|^4}\\
&-4\,\Re\left(\s{\bar{\vec{A}_0}}{\vec{A}_2}\z^{-3}\right)-8\,\Re\left(\s{\bar{\vec{A}_0}}{\vec{A}_3}\frac{z}{\z^3}\right)-4\,\Re\left(\s{\bar{\vec{A}_1}}{\vec{A}_2}\z^{-2}\right)+O\left(\frac{1}{|z|}\right)\\
&=\frac{1}{|z|^6}\left(4|\vec{A}_0|^2+|\vec{A}_1|^2|z|^2+4\,\Re\left(\s{\bar{\vec{A}_0}}{\vec{A}_1}z-\s{\bar{\vec{A}_0}}{\vec{A}_2}z^3-2\s{\bar{\vec{A}_0}}{\vec{A}_3}z^4-\s{\bar{\vec{A}_1}}{\vec{A}_2}z^3\z\right)+O(|z|^5)\right)\\
&=\frac{2\beta_0}{|z|^6}\left(1+\beta_1|z|^2+2\,\Re\left(\alpha_0 z-\alpha_1z^3-2\alpha_2z^4-\alpha_3z^3\z\right)+O(|z|^5)\right)
\end{align*}
where
\begin{align*}
&\beta_0=2|\vec{A}_0|^2>0,\quad \beta_0\beta_1=\frac{1}{2}|\vec{A}_1|^2\\
&\alpha_0\beta_0=\s{\bar{\vec{A}_0}}{\vec{A}_1},\quad \alpha_1\beta_0=\s{\bar{\vec{A}_0}}{\vec{A}_2},\quad \alpha_2\beta_0=\s{\bar{\vec{A}_0}}{\vec{A}_3},\quad
\alpha_3\beta_0=\frac{1}{2}\s{\bar{\vec{A}_1}}{\vec{A}_2}.	
\end{align*}
\begin{rem}\label{simplifies2}
	Notice that 
	\begin{align*}
	\beta_1=\frac{1}{4}\frac{|\vec{A}_1|^2}{|\vec{A}_0|^2},\quad |\alpha_0|^2=\frac{1}{4}\frac{|\s{\bar{\vec{A}_0}}{\vec{A}_1}|^2}{|\vec{A}_0|^4}.
	\end{align*}
	Therefore, we have by Cauchy-Schwarz inequality $\beta_1\geq |\alpha_0|^2$. Furthermore, as by Lemma \ref{diag2} we have $\vec{A}_1\in \mathrm{Span}(\vec{A}_0)$, the equality $\beta_1=|\alpha_0|^2$ holds for $n=3$. However, for the sake of verifiability of the proof, we shall only use this relation at the end of the computation.
\end{rem}
Therefore, we have
\begin{align}\label{conf1}
e^{2\lambda}=\frac{4\beta_0}{|z|^6}\left(1+\beta_1|z|^2+2\,\Re\left(\alpha_0z-\alpha_1z^3-2\alpha_2z^4-\alpha_3z^3\z\right)+O(|z|^5)\right).
\end{align}
Now notice that for all $\vec{A},\vec{B}\in \C^n$
\begin{align}\label{re2}
&\left|2\,\Re\left(\vec{A}+{\vec{B}}\right)\right|^2=\s{\vec{A}+\bar{\vec{A}}+\vec{B}+\bar{\vec{B}}}{\vec{A}+\bar{\vec{A}}+\vec{B}+\bar{\vec{B}}}\nonumber\\
&=2|\vec{A}|^2+2|\vec{B}|^2+\s{\vec{A}}{\vec{A}}+\s{\bar{\vec{A}}}{\bar{\vec{A}}}+\s{{\vec{B}}}{\vec{B}}+\s{\bar{\vec{B}}}{\bar{\vec{B}}}+2\s{\vec{A}}{\vec{B}}+2\s{\vec{A}}{\bar{\vec{B}}}+2\s{\bar{\vec{A}}}{\vec{B}}+2\s{\bar{\vec{A}}}{\bar{\vec{B}}}\nonumber\\
&=2\left(|\vec{A}|^2+|\vec{B}|^2\right)+2\,\Re\left(\s{\vec{A}}{\vec{A}}+\s{\vec{B}}{\vec{B}}\right)+4\,\Re\left(\s{\vec{A}}{\vec{B}}+\s{\vec{A}}{\bar{\vec{B}}}\right).
\end{align}
We also compute thanks to \eqref{cancel2} and \eqref{re2}
\begin{align*}
|\phi(z)|^2&=2\frac{|\vec{A}_0|^2}{|z|^4}+2\frac{|\vec{A}_1|^2}{|z|^2}+4\,\Re\left(\s{\bar{\vec{A}_0}}{\vec{A}_1}z^{-1}\z^{-2}\right)+4\,\Re\left(\s{\bar{\vec{A}_0}}{\vec{A}_2}z\z^{-2}\right)+4\,\Re\left(\s{\bar{\vec{A}_1}}{\vec{A}_2}z\z^{-1}\right)\\
&+4\,\Re\left(\s{\bar{\vec{A}_0}}{\vec{A}_3}z^2\z^{-2}\right)+2\,\Re\left(\s{\vec{A}_0}{\vec{A}_3}+\s{\vec{A}_1}{\vec{A}_2}\right)+O(|z|)\\
&=\frac{1}{|z|^4}\left(2|\vec{A}_0|^2+2|\vec{A}_1|^2|z|^2+4\,\Re\left(\s{\bar{\vec{A}_0}}{\vec{A}_1}z+\s{\bar{\vec{A}_0}}{\vec{A}_2}z^3+\s{\bar{\vec{A}_0}}{\vec{A}_3}z^4+\s{\bar{\vec{A}_1}}{\vec{A}_2}z^3\z\right)+O(|z|^5)\right)\\
&=\frac{\beta_0}{|z|^4}\left(1+4\beta_1|z|^2+4\,\Re\left(\alpha_0 z+\alpha_1z^3+\alpha_2z^4+2\alpha_3z^3\z\right)+O(|z|^5)\right)+\beta_2
\end{align*}
where
\begin{align*}
\beta_2=2\,\Re\left(\s{\vec{A}_0}{\vec{A}_3}+\s{\vec{A}_1}{\vec{A}_2}\right)=-6\,\Re\left(\s{\vec{A}_0}{\vec{A}_3}\right)
\end{align*}
thanks to \eqref{cancel2}.
Therefore, we compute
\begin{align}\label{devdelphi}
\partial|\phi|^2&=\frac{\beta_0}{z|z|^4}\left(-2-4\beta_1|z|^2-2\alpha_0z-4\bar{\alpha_0}\,\z+2\alpha_1z^3-4\bar{\alpha_1}\,\z^3+8i\,\Im\left(\alpha_2z^4+\alpha_3z^3\z\right)+O(|z|^5)\right)
\end{align}
As 
\begin{align}\label{trivial}
\int_{\partial B(0,\epsilon)}z^k\z^l\frac{dz}{z}=2\pi i\,\epsilon^{2k}\delta_{k,l}
\end{align}
where $\delta_{k,l}$ is the Kronecker symbol, we directly obtain
\begin{align*}
\Im\int_{\partial B(0,\epsilon)}4\,\partial|\phi|^2=-\frac{16\pi\beta_0}{\epsilon^4}-\frac{32\pi\beta_0\beta_1}{\epsilon^2}+O(\epsilon).
\end{align*}
Now, we will compute the singular residue for an admissible normal variation $\vec{v}=v\n\in \mathscr{E}_{\vec{\Psi}}'(\Sigma,\R^3)$. First assume that $v\in C^4(\Sigma)$.  We see that thanks to the previous section $v$ must admit the following development
\begin{align*}
v=v(p_i)+\gamma|z|^4+2\,\Re\left(\zeta_0 z^2+\zeta_1z^2\z+\zeta_2z^3+\zeta^3z^4+\zeta_4z^3\z\right)+o(|z|^4).
\end{align*}
To simplify the different estimates, we see that it is equivalent to assume that $v\in C^5(\Sigma)$, so that 
\begin{align*}
	v=v(p_i)+\gamma|z|^4+2\,\Re\left(\zeta_0 z^2+\zeta_1z^2\z+\zeta_2z^3+\zeta^3z^4+\zeta_4z^3\z\right)+O(|z|^5).
\end{align*}
Now, we see that all functions $|\phi|^2,e^{2\lambda},v$ have the following general development for some $m\in \Z$
\begin{align*}
|z|^{2m}\left(\mu_0+\mu_1|z|^2+\mu_2|z|^4+\Re\left(\nu_0z+\nu_1z^2+\nu_2z^3+\nu_3z^2\z+\nu_4z^4+\nu_5z^3\z\right)+O(|z|^5)\right)
\end{align*}
and $K_g$ two (but up to $O(|z|^3)$ order).  Therefore, as the singular residue is a quadratic expression of derivatives of these functions and by \eqref{trivial}, we see that may assume that $\nu_2=\nu_3=\nu_4=\nu_5=0$ in all these developments. By abuse of notation all such coefficients appearing in computations will be taken equal to $0$ (a formal notation would be to write an equalities $(\mathrm{mod}\;\, \mathrm{Span}_{\C}(z^3,\z^3,z^4,\z^4,z^3\z,z\z^3))$ instead of the equality symbols, but we find it both heavy and unnecessary). 
With these new conventions, we obtain
\begin{align*}
&|\phi|^2=\frac{\beta_0}{|z|^4}\left(1+4\,\Re\left(\alpha_0z\right)+4\beta_1|z|^2+O(|z|^5)\right)+\beta_2\\
&e^{2\lambda}=\frac{4\beta_0}{|z|^6}\left(1+2\,\Re\left(\alpha_0z
\right)+\beta_1|z|^2+O(|z|^5)\right)\\
&v=v(p_i)+\gamma|z|^4+2\,\Re\left(\zeta_0z^2+\zeta_1z^2\z\right)+O(|z|^5).
\end{align*}
Now we compute
\begin{align*}
|\phi|^2v&=\frac{\beta_0}{|z|^4}\bigg(v(p_i)+4\beta_1v(p_i)|z|^2+2\,\Re\Big(2\alpha_0v(p_i)z+\zeta_0z^2+\zeta_1z^2\z+2\alpha_0\zeta_0z^3+2\bar{\alpha_0}\zeta_0z^2\z+8\beta_1\zeta_0z^3\z+2\alpha_0\zeta_1z^3\z\Big)\\
&+O(|z|^5)\bigg)+\beta_3\\
&=\frac{\beta_0}{|z|^4}\left(v(p_i)+4\beta_1v(p_i)|z|^2+2\,\Re\left(2\alpha_0v(p_i)z+\zeta_0z^2+\left(2\bar{\alpha_0}\zeta_0+\zeta_1\right)z^2\z\right)+O(|z|^5)\right)+\beta_3
\end{align*}
for some constant $\beta_3\in \R$. Therefore, we have
\begin{align*}
\p{z\z}^2\left(|\phi|^2v\right)&=\frac{\beta_0}{|z|^6}\left(4v(p_i)+4\beta_1v(p_i)|z|^2+8\,\Re\left(\alpha_0v(p_i)z\right)+O(|z|^5)\right)\\
&=\frac{4\beta_0}{|z|^6}\left(1+\beta_1|z|^2+2\,\Re\left(\alpha_0z\right)\right)v(p_i)+O(|z|^5)=e^{2\lambda}v(p_i)+O(|z|^5)
\end{align*}
so that
\begin{align*}
\Delta_g\left(|\phi|^2v\right)=4v(p_i)+O(|z|^5).
\end{align*}
This implies that there exists some $\lambda_j\in \C$ such that
\begin{align*}
\partial\left(|\phi|^2v\right)=-\frac{2\beta_0}{z|z|^4}\left(\left(1+2\beta_1|z|^2+\alpha_0z+2\bar{\alpha_0}\z\right)v(p_i)+\lambda_1z^2+\lambda_2\z^2+\lambda_3z^2\z+\lambda_4z\z^2+O(|z|^5)\right)dz.
\end{align*}
This finally implies that
\begin{align}\label{end21}
\Im\int_{\partial B(0,\epsilon)}\Delta_g\left(|\phi|^2v\right)\partial\left(|\phi|^2v\right)=-\frac{16\pi\beta_0}{\epsilon^4}v^2(p_i)-\frac{32\pi\beta_0\beta_1}{\epsilon^2}v^2(p_i)+O(\epsilon).
\end{align}

Here, we see that no constant term occurs. However, we will see that they do occur  for other contributions of the residue. Let $u:D^2\rightarrow \R$ such that 
\begin{align*}
e^{2\lambda}=\frac{4\beta_0}{|z|^6}e^{2u}
\end{align*}
Then
\begin{align*}
-\Delta u=e^{2\lambda}K_g,
\end{align*}
and we have
\begin{align*}
2u=\log\left(1+\beta_1|z|^2+2\,\Re\left(\alpha_0z-\alpha_1z^3-2\alpha_2z^4-\alpha_3z^3\z\right)+O(|z|^3)\right)
\end{align*}
Therefore, first compute
\begin{align*}
2\left(\p{z}u\right)=\frac{\alpha_0+\beta_1\z-3\alpha_1z^2-8\alpha_2z^3-3\alpha_3z^2\z-\bar{\alpha_3}\,\z^3+O(|z|^4)}{1+\beta_1|z|^2+2\,\Re\left(\alpha_0z-\alpha_1z^3-2\alpha_2z^4-\alpha_3z^3\z\right)+O(|z|^5)}
\end{align*}
Then
\begin{align*}
2\left(\p{z\z}^2u\right)&=\frac{\beta_1-3\alpha_3z^2-3\bar{\alpha_3}\,\z^2+O(|z|^3)}{1+\beta_1|z|^2+2\,\Re\left(\alpha_0z-\alpha_1z^3-2\alpha_2z^4-\alpha_3z^3\z\right)+O(|z|^5)}\\
&-\frac{\left|\alpha_0-3\alpha_1z^2-4\alpha_2z^3-3\alpha_3z^2\z+\beta_1\z+O(|z|^4)\right|^2}{1+\beta_1|z|^2+2\,\Re\left(\alpha_0z-\alpha_1z^3-2\alpha_2z^4-\alpha_3z^3\z\right)+O(|z|^5)}\\
&=e^{-2u}\left(\beta_1-6\,\Re\left(\alpha_3z^2\right)+O(|z|^3)\right)-e^{-4u}\left|\alpha_0+\beta_1\z-3\alpha_1z^2+O(|z|^3)\right|^2
\end{align*}
Now we have 
\begin{align*}
&\Re(\alpha_0z)^2=\frac{1}{4}\left(\alpha_0^2z^2+2|\alpha_0|^2|z|^2+\bar{\alpha_0}^2\z^2\right)=\frac{1}{2}\Re\left(\alpha_0^2z^2\right)+\frac{1}{2}|\alpha_0|^2|z|^2\\
&|\alpha_0+\beta_1\z-3\alpha_1z^2+O(|z|^3)|^2=|\alpha_0|^2+2\,\Re\left(\beta_1\alpha_0z-3\bar{\alpha_0}\alpha_1z^2\right)+\beta_1^2|z|^2+O(|z|^3).
\end{align*}
Furthermore, as $-K_g=e^{-2\lambda}\Delta u=4e^{-2\lambda}\p{z\z}^2u$, and $e^{2\lambda}=\dfrac{4\beta_0}{|z|^6}e^{2u}$, we have
\begin{align*}
-K_g&=2\cdot \frac{|z|^6}{4\beta_0}e^{-2u}\left(2\left(\p{z\z}^2u\right)\right)\\
&=\frac{|z|^6}{2\beta_0}\bigg(e^{-4u}\left(\beta_1-6\,\Re\left(\alpha_3z^2\right)+O(|z|^3)\right)-e^{-6u}\left(|\alpha_0|^2+2\,\Re\left(\beta_1\alpha_0z-3\bar{\alpha_0}\alpha_1z^2\right)+\beta_1^2|z|^2+O(|z|^3)\right)\bigg)
\end{align*}
Now, we compute
\begin{align*}
e^{-4u}&=\left(1+2\,\Re(\alpha_0z)+\beta_1|z|^2\right)^{-2}=1-2\left(2\,\Re\left(\alpha_0z\right)+\beta_1|z|^2\right)+3\cdot \left(2\,\Re\left(\alpha_0z\right)\right)^2+O(|z|^3)\\
&=1+2\,\Re\left(-2\alpha_0z+3\alpha_0^2z^2\right)+2\left(3|\alpha_0|^2-\beta_1\right)|z|^2+O(|z|^3)\\
e^{-6u}&=1-3\left(2\,\Re\left(\alpha_0z\right)+\beta_1|z|^2\right)+6\cdot\left(2\,\Re\left(\alpha_0z\right)\right)^2+O(|z|^3)\\
&=1+2\,\Re\left(-3\alpha_0z+6\alpha_0^2z^2\right)+3\left(4|\alpha_0|^2-\beta_1\right)|z|^2+O(|z|^3).
\end{align*}
Therefore, we have
\begin{align*}
&e^{-4u}\left(\beta_1-6\,\Re\left(\alpha_3z^2\right)+O(|z|^3)\right)=\left(1+2\,\Re\left(-2\alpha_0z+3\alpha_0^2z^2\right)+2\left(3|\alpha_0|^2-\beta_1\right)|z|^2+O(|z|^3)\right)\\
&\times \left(\beta_1-6\,\Re\left(\alpha_3z^2\right)+O(|z|^3)\right)\\
&=\beta_1+2\beta_1\,\Re\left(-2\alpha_0z+3\alpha_0^2z^2\right)+2\beta_1\left(3|\alpha_0|^2-\beta_1\right)|z|^2-6\,\Re\left(\alpha_3z^2\right)+O(|z|^3)\\
&=\beta_1+2\,\Re\left(-2\beta_1\alpha_0z+3\left(\beta_1\alpha_0^2-\alpha_3\right)z^2\right)+2\beta_1\left(3|\alpha_0|^2-\beta_1\right)|z|^2+O(|z|^3)
\end{align*}
and
\begin{align*}
&e^{-6u}\left(|\alpha_0|^2+2\,\Re\left(\beta_1\alpha_0z-3\bar{\alpha_0}\alpha_1z^2\right)+\beta_1^2|z|^2+O(|z|^3)\right)=\left(1+2\,\Re\left(-3\alpha_0z+6\alpha_0^2z^2\right)+3\left(4|\alpha_0|^2-\beta_1\right)|z|^2\right)\\
&\times\left(|\alpha_0|^2+2\,\Re\left(\beta_1\alpha_0z-3\bar{\alpha_0}\alpha_1z^2\right)+\beta_1^2|z|^2\right)+O(|z|^3)\\
&=|\alpha_0|^2+2\,\Re\left(\beta_1\alpha_0z-3\bar{\alpha_0}\alpha_1z^2\right)+\beta_1^2|z|^2-12\beta_1\,\Re\left(\alpha_0z\right)^2+2\,\Re\left(-3|\alpha_0|^2\alpha_0z+6|\alpha_0|^2\alpha_0^2z^2\right)\\
&+3|\alpha_0|^2\left(4|\alpha_0|^2-\beta_1\right)|z|^2+O(|z|^3)\\
&=|\alpha_0|^2+2\,\Re\left(\left(\beta_1-3|\alpha_0|^2\right)\alpha_0z+3\left(\left(2|\alpha_0|^2-\beta_1\right)\alpha_0^2-\bar{\alpha_0}\alpha_1\right)z^2\right)+\left(\beta_1^2-6|\alpha_0|^2\beta_1+3|\alpha_0|^2\left(4|\alpha_0|^2-\beta_1\right)\right)|z|^2\\
&+O(|z|^3)\\
&=|\alpha_0|^2+2\,\Re\left(\left(\beta_1-3|\alpha_0|^2\right)\alpha_0z+3\left(\left(2|\alpha_0|^2-\beta_1\right)\alpha_0^2-\bar{\alpha_0}\alpha_1\right)z^2\right)+\left(\beta_1^2+12|\alpha_0|^4-9|\alpha_0|^2\beta_1\right)|z|^2+O(|z|^3). 
\end{align*}
Finally, we get as
\begin{align}\label{gauss2}
&-2\beta_1-\left(\beta_1-3|\alpha_0|^2\right)=3\left(|\alpha_0|^2-\beta_1\right)\nonumber\\
&\left(\beta_1\alpha_0^2-\alpha_3\right)-\left(\left(2|\alpha_0|^2-\beta_1\right)\alpha_0^2-\bar{\alpha_0}\alpha_1\right)=2\left(\beta_1-|\alpha_0|^2\right)\alpha_0^2+\bar{\alpha_0}\alpha_1-\alpha_3\nonumber\\
&2\beta_1\left(3|\alpha_0|^2-\beta_1\right)-\left(\beta_1^2+12|\alpha_0|^4-9|\alpha_0|^2\beta_1\right)=15|\alpha_0|^2\beta_1-12|\alpha_0|^2-3\beta_1^2=3\left(5|\alpha_0|^2\beta_1-4|\alpha_0|^4-\beta_1^2\right)
\end{align}
the expansion
\begin{align*}
-K_g&=\frac{|z|^6}{2\beta_0}\bigg(\beta_1-|\alpha_0|^2+6\,\Re\Big(\left(|\alpha_0|^2-\beta_1\right)\alpha_0z+\left(2\left(\beta_1-|\alpha_0|^2\right)\alpha_0^2+\bar{\alpha_0}\alpha_1-\alpha_3\right)z^2\Big)\\
&+3\left(5|\alpha_0|^2\beta_1-4|\alpha_0|^4-\beta_1^2\right)|z|^2+O(|z|^3)\bigg).
\end{align*}
Now, recall that
\begin{align*}
|\phi|^2v&=\frac{\beta_0}{|z|^4}\left(v(p_i)+4\beta_1v(p_i)|z|^2+2\,\Re\left(2\alpha_0v(p_i)z+\zeta_0z^2+\left(2\bar{\alpha_0}\zeta_0+\zeta_1\right)z^2\z\right)+O(|z|^5)\right)+\beta_3\\
&=\frac{\beta_0}{|z|^4}\left(\left(1+4\beta_1|z|^2+4\,\Re\left(\alpha_0z
\right)\right)v(p_i)+2\,\Re\left(\zeta_0z^2\right)+O(|z|^3)\right).
\end{align*}
Therefore, we have as $16\,\Re\left(\alpha_0z\right)^2=8|\alpha_0|^2|z|^2+8\,\Re\left(\alpha_0^2z^2\right)$
\begin{align*}
|\phi|^4v^2&=\frac{\beta_0}{|z|^8}\left(\left(1+8\beta_1|z|^2+16\,\Re\left(\alpha_0z\right)^2+8\,\Re\left(\alpha_0z
\right)\right)v^2(p_i)+4\,\Re\left(\zeta_0v(p_i)z^2\right)+O(|z|^3)\right)\\
&=\frac{\beta_0}{|z|^8}\left(\left(1+8\left(\beta_1+|\alpha_0|^2\right)|z|^2+8\,\Re\left(\alpha_0z\right)\right)v^2(p_i)+4\,\Re\left(\left(\zeta_0v(p_i)+2\alpha_0^2v^2(p_i)\right)z^2\right)+O(|z|^3)\right)\\
&=\frac{\beta_0}{|z|^8}\left(\left(1+8\left(\beta_1+|\alpha_0|^2\right)|z|^2+8\,\Re\left(\alpha_0z\right)\right)v^2(p_i)+2\,\Re\left(\zeta_3 z^2\right)+O(|z|^3)\right),
\end{align*}
where
\begin{align*}
\zeta_3=2\left(\zeta_0v(p_i)+2\alpha_0^2v^2(p_i)\right).
\end{align*}
This implies that
\begin{align*}
\partial\left(|\phi|^4v^2\right)&=-\frac{4\beta_0}{z|z|^8}\left(\left(1+8\left(\beta_1+|\alpha_0|^2\right)|z|^2+8\,\Re\left(\alpha_0z\right)\right)v^2(p_i)+4\,\Re\left(\zeta_3z^2\right)+O(|z|^3)\right)dz\\
&+\frac{\beta_0}{z|z|^8}\left(\left(8\left(\beta_1+|\alpha_0|^2\right)|z|^2+4\alpha_0z\right)v^2(p_i)+4\zeta_3z^2+O(|z|^3)\right)dz\\
&=-\frac{4\beta_0}{z|z|^8}\left(\left(1+6\left(\beta_1+|\alpha_0|^2\right)|z|^2+3\alpha_0z+4\bar{\alpha_0}\,\z\right)v^2(p_i)+\zeta_3z^2+2\bar{\zeta_3}\z^2+O(|z|^3)\right)dz.
\end{align*}
Therefore, we have for some $\lambda_j\in \C$
\begin{align*}
&K_g\,\partial\left(|\phi|^4v^2\right)=\frac{2\beta_0}{z|z|^2}\left(\left(\beta_1-|\alpha_0|^2\right)+3\left(5|\alpha_0|^2\beta_1-4|\alpha_0|^4-\beta_1^2\right)|z|^2+6\,\Re\left(\left(|\alpha_0|^2-\beta_1\right)\alpha_0z\right)+\Re\left(\zeta_4z^2\right)\right)\\
&\times \left(\left(1+6\left(\beta_1+|\alpha_0|^2\right)|z|^2+3\alpha_0z+4\bar{\alpha_0}\,\z\right)v^2(p_i)+\zeta_3z^2+2\bar{\zeta_3}\z^2+O(|z|^3)\right)dz\\
&=\frac{2\beta_0}{z|z|^2}\bigg(\Big\{\left(\beta_1-|\alpha_0|^2\right)+6\left(\beta_1-|\alpha_0|^2\right)\left(\beta_1+|\alpha_0|^2\right)|z|^2+3\left(|\alpha_0|^2-\beta_1\right)\alpha_0z\cdot 4\bar{\alpha_0}\,\z+3\left(|\alpha_0|^2-\beta_1\right)\bar{\alpha_0}\z\,3\alpha_0z\\
&+3\left(5|\alpha_0|^2\beta_1-4|\alpha_0|^4-\beta_1^2\right)|z|^2\Big\}v^2(p_i)+\lambda_1z+\lambda_2\z+\lambda_3z^2+\lambda_4\z^2+O(|z|^3)\bigg)dz\\
&=\frac{2\beta_0}{z|z|^2}\bigg(\left(\beta_1-|\alpha_0|^2\right)v^2(p_i)+3\left(\beta_1-|\alpha_0|^2\right)^2|z|^2v^2(p_i)+\lambda_1z+\lambda_2\z+\lambda_3z^2+\lambda_4\z^2+O(|z|^3)\bigg)dz
\end{align*}
as
\begin{align*}
6\left(\beta_1^2-|\alpha_0|^4\right)+21|\alpha_0|^2\left(|\alpha_0|^2-\beta_1\right)+3\left(5|\alpha_0|^2\beta_1-4|\alpha_0|^4-\beta_1^2\right)=3\beta_1^2+3|\alpha_0|^4-6|\alpha_0|^2\beta_1=3\left(\beta_1-|\alpha_0|^2\right)^2.
\end{align*}
for some $\lambda_j\in \C$. Therefore, we have
\begin{align}\label{end22}
\Im\int_{\partial B(0,\epsilon)}2K_g\left(|\phi|^2v\right)\,\partial\left(|\phi|^2v\right)=\frac{4\pi\beta_0\left(\beta_1-|\alpha_0|^2\right)}{\epsilon^2}v^2(p_i)+12\pi\left(\beta_1-|\alpha_0|^2\right)^2v^2(p_i)+O(\epsilon).
\end{align}
We find the appearance of the square remarkable. Now, recall that
\begin{align*}
\partial\left(|\phi|^2v\right)=-\frac{2\beta_0}{z|z|^4}\left(\left(1+2\beta_1|z|^2+\alpha_0z+2\bar{\alpha_0}\z\right)v(p_i)+\lambda_1z^2+\lambda_2\z^2+\lambda_3z^2\z+\lambda_4z\z^2+O(|z|^5)\right)dz.
\end{align*}
This implies that
\begin{align*}
\left|\partial\left(|\phi|^2v\right)\right|^2&=\frac{4\beta_0^2}{|z|^{10}}\bigg(\left(1+4\beta_1|z|^2+|\alpha_0|^2|z|^2+4|\alpha_0|^2|z|^2+2\,\Re\left(\alpha_0z\right)+4\,\Re\left(\alpha_0z\right)\right)v^2(p_i)\\
&+\Re\left(\mu_1z^2+\mu_2z^2\z\right)+\beta_4|z|^4+O(|z|^5)\bigg)\\
&=\frac{4\beta_0^2}{|z|^{10}}\bigg(\left(1+\left(4\beta_1+5|\alpha_0|^2\right)|z|^2+6\,\Re\left(\alpha_0z\right)\right)v^2(p_i)+\Re\left(\mu_1z^2+\mu_2z^2\z\right)+\beta_4|z|^4+O(|z|^5)\bigg)
\end{align*}
and as 
\begin{align*}
e^{2\lambda}=\frac{4\beta_0}{|z|^6}\left(1+\beta_1|z|^2+2\,\Re\left(\alpha_0z-\alpha_1z^3-2\alpha_2z^4-\alpha_3z^3\z\right)+O(|z|^5)\right)
\end{align*}
we also have
\begin{align}\label{c2}
e^{-2\lambda}&=\frac{|z|^6}{4\beta_0}\left(1+\beta_1|z|^2+2\,\Re\left(\alpha_0z-\alpha_1z^3-2\alpha_2z^4-\alpha_3z^3\z\right)+O(|z|^5)\right)^{-1}\nonumber\\
&=\frac{|z|^{6}}{4\beta_0}\bigg(1-\beta_1|z|^2-2\,\Re\left(\alpha_0z-\alpha_1z^3-2\alpha_2z^4-\alpha_3z^3\z\right)
+\beta_1^2|z|^2+4\,\Re\left(\alpha_0z\right)^2\nonumber\\
&-8\,\Re\left(\alpha_0z\right)\Re\left(\alpha_1z^3\right)-8\,\Re\left(\alpha_0z
\right)^3-3\cdot\left(2\,\Re\left(\alpha_0z\right)\right)^2\cdot \beta_1|z|^2+16\,\Re\left(\alpha_0z\right)^4+O(|z|^5)\bigg)\nonumber\\
&=\frac{|z|^6}{4\beta_0}\bigg(1+\left(2|\alpha_0|^2-\beta_1\right)|z|^2+\left(6|\alpha_0|^4+\beta_1^2\right)|z|^4+2\,\Re\Big(-\alpha_0z+\alpha_0^2z^2+\left(\alpha_1-2\alpha_0^3\right)z^3\nonumber\\
&-3|\alpha_0|^2\alpha_0z^2\z+2\left(\alpha_2-\alpha_0\alpha_1+\alpha_0^4\right)z^4+\left(-2\bar{\alpha_0}\alpha_1+\left(4|\alpha_0|^4-3\beta_1\right)\alpha_0^2\right)z^3\z\Big)+O(|z|^5)\bigg)\nonumber\\
&=\frac{|z|^6}{4\beta_0}\bigg(1+\left(2|\alpha_0|^2-\beta_1\right)|z|^2+2\,\Re\left(-\alpha_0z+\alpha_0^2z^2-3|\alpha_0|^2\alpha_0z^2\z\right)+\beta_5|z|^4+O(|z|^5)\bigg)
\end{align}
Finally, we have
\begin{align*}
\left|\partial\left(|\phi|^2v\right)\right|_g^2&=\frac{\beta_0}{|z|^4}\left(\left(1+\left(3\beta_1+|\alpha_0|^2\right)|z|^2\right)v^2(p_i)+\Re\left(\mu_1'z^2+\mu_2'z^2\z\right)\right)+\beta_6+O(|z|)
\end{align*}
and 
\begin{align*}
-\partial\left|d\left(|\phi|^2v\right)\right|_g^2=-4\,\partial\left|\partial\left(|\phi|^2v\right)\right|_g^2=\frac{4\beta_0}{z|z|^4}\left(\left(2+\left(3\beta_1+|\alpha_0|^2\right)|z|^2\right)v^2(p_i)+\Re\left(\mu_1'z^2+\mu_2'z^2\z\right)+O(|z|^5)\right).
\end{align*}
Therefore, 
\begin{align}\label{end23}
\Im\int_{\partial B(0,\epsilon)}-\partial\left|d\left(|\phi|^2v\right)\right|_g^2=\frac{16\pi\beta_0}{\epsilon^4}v^2(p_i)+\frac{8\pi\beta_0\left(3\beta_1+|\alpha_0|^2\right)}{\epsilon^2}v^2(p_i)+O(\epsilon).
\end{align}
Finally, we have
\begin{align*}
&\int_{\partial B(0,\epsilon)}\left(\Delta_g\left(|\phi|^2v\right)+2\,K_g|\phi|^2v\right)\star d\left(|\phi|^2v\right)-\frac{1}{2}\star d\left|d\left(|\phi|^2v\right)\right|_g^2\\
&=\Im\int_{\partial B(0,\epsilon)}2\left(\Delta_g\left(|\phi|^2v\right)+2\,K_g|\phi|^2v\right)\partial\left(|\phi|^2v\right)-\partial\left|d\left(|\phi|^2v\right)\right|_g^2\\
&=2\left(-\frac{16\pi\beta_0}{\epsilon^4}v^2(p_i)-\frac{32\pi\beta_0\beta_1}{\epsilon^2}v^2(p_i)\right)+2\left(\frac{4\pi\beta_0\left(\beta_1-|\alpha_0|^2\right)}{\epsilon^2}v^2(p_i)+12\pi\left(\beta_1-|\alpha_0|^2\right)^2v^2(p_i)\right)\\
&+\frac{16\pi\beta_0}{\epsilon^4}v^2(p_i)+\frac{8\pi\beta_0\left(3\beta_1+|\alpha_0|^2\right)}{\epsilon^2}v^2(p_i)+O(\epsilon)\\
&=-\frac{16\pi\beta_0}{\epsilon^4}v^2(p_i)-\frac{32\pi\beta_0\beta_1}{\epsilon^2}v^2(p_i)+24\pi\beta_0\left(\beta_1^2-|\alpha_0|^2\right)^2v^2(p_i)+O(\epsilon).
\end{align*}
Finally, thanks to Lemma \ref{diag2} and Remark \ref{simplifies2}, if $n=3$ (recall that $n$ is the ambient dimension if the immersed minimal surface $\phi$) we have $\beta_1=|\alpha_0|^2$, so 
\begin{align*}
\int_{\partial B(0,\epsilon)}\left(\Delta_g\left(|\phi|^2v\right)+2\,K_g|\phi|^2v\right)\star d\left(|\phi|^2v\right)-\frac{1}{2}\star d\left|d\left(|\phi|^2v\right)\right|_g^2=-\frac{16\pi\beta_0}{\epsilon^4}v^2(p_i)-\frac{32\pi\beta_0\beta_1}{\epsilon^2}v^2(p_i)+O(\epsilon).
\end{align*}
Furthermore, notice that replacing $|\alpha_0|^2=\beta_1$ we find by \eqref{gauss2}
\begin{align*}
-K_g&=\frac{|z|^6}{2\beta_0}\bigg(\beta_1-|\alpha_0|^2+6\,\Re\Big(\left(|\alpha_0|^2-\beta_1\right)\alpha_0z+\left(2\left(\beta_1-|\alpha_0|^2\right)\alpha_0^2+\bar{\alpha_0}\alpha_1-\alpha_3\right)z^2\Big)\\
&+3\left(5|\alpha_0|^2\beta_1-4|\alpha_0|^4-\beta_1^2\right)|z|^2+O(|z|^3)\bigg)\\
&=\frac{|z|^6}{2\beta_0}\bigg(6\,\Re\left(\left(\bar{\alpha_0}\alpha_1-\alpha_3\right)z^2\right)+O(|z|^3)\bigg)
\end{align*}
We have proved the following result.
\begin{prop}
	Let $\phi:\Sigma\setminus\ens{p_1,\cdots,p_n}\rightarrow \R^3$ be a complete minimal surface with finite total curvature and assume that $p_i$ ($1\leq i\leq n$) is an end of multiplicity $2$. If $U_i\subset \Sigma$ is an open subset such that $p_i\in U_i$ and such that there exists a complex chart $\varphi:U\rightarrow \C$ such that $\varphi(U_i)=D^2\subset \C$ and $\varphi(p_i)=0$. Then there exists $\alpha_0^2>0$ and $\alpha_1^2
	\geq 0$ (independent of $\varphi:U_i\rightarrow \C$ such that $\varphi(U_i)=D^2$ and $\varphi(p_i)=0$) such that for all $0<\epsilon\leq 1$ and for all $v\in W^{2,2}(\Sigma)\cap C^4(\Sigma)$
	\begin{align}\label{res2}
	&\Im\int_{\partial B(0,\epsilon)}2\left(\Delta_g\left(|\phi|^2v\right)+2\,K_g|\phi|^2v\right)\partial\left(|\phi|^2v\right)-\partial\left|d\left(|\phi|^2v\right)\right|_g^2=-\frac{16\pi\alpha_0^2}{\epsilon^4}\left(1+\alpha_2^2\epsilon^2\right)v^2(p_i)+O(\epsilon)
	\end{align} 
	We write these coefficients $\alpha_k^2(U_j,p_j)$.
\end{prop}
\begin{proof}
	The independence on the chart is clear as change of charts are rotations $D^2\rightarrow D^2$ under which the expression in \eqref{res2} is unchanged as the $\alpha_0,\alpha_1,\alpha_2$ are norms of coefficients scalar products of $\phi$ in the expressed chart, so they are rotationally invariant.
\end{proof}
\begin{rem}
	Notice that these \enquote{residues} are not independent of $U$.
\end{rem}

We will now state most the following theorems for spheres for simplicity. 

\begin{theorem}
	Let $\Sigma$ be a closed Riemann surface,  $\phi:\Sigma\setminus\ens{p_1,\cdots, p_n}\rightarrow \R^3$ be a complete minimal surface with finite total curvature and zero flux and $\vec{\Psi}:\Sigma\rightarrow \R^3$ be a compact branched Willmore surface such that $\vec{\Psi}=\iota\circ \phi$. Assume that the ends $p_1,\cdots ,p_{m}$ of $\phi$ are flat (where $0\leq m\leq n$ is a fixed integer) and that $p_{m+1},\cdots p_{m}$ $\leq m_2\leq n$ have multiplicity $2$, and fix a covering $U_1,\cdots, U_n\subset \Sigma$ of $\ens{p_1,\cdots,p_n}\subset \Sigma$.
	Then we have for all $v\in W^{2,2}(\Sigma)\cap C^4(\Sigma)$ and for all normal variation $\vec{v}=v\n_{\vec{\Psi}}\in \mathscr{E}_{\phi}(\Sigma,\R^3)$ 
	\begin{align*}
	D^2W(\vec{\Psi})(\vec{v},\vec{v})&=\lim\limits_{\epsilon\rightarrow 0}\left(\frac{1}{2}\int_{\Sigma_{\epsilon}}\left(\Delta_gu-2\,K_gu\right)^2d\vg-8\pi\sum_{i=1}^{m}\frac{\alpha_{i,0}^2}{\epsilon^2}v^2(p_i)-16\pi\sum_{i=m+1}^{n}\frac{\alpha_{i,0}^2}{\epsilon^4}\left(1+\alpha_{i,2}^2\epsilon^2\right)v^2(p_i)\right)
	\end{align*} 
	where $u=|\phi|^2v$ and $\alpha_{j,k}^2=\alpha_{k}^2(U_j,p_j)$.
\end{theorem}

We can improve Theorem \ref{general} by showing that the diagonal coefficient of the universal matrix vanishes for ends of multiplicity $2$.

\begin{theorem}\label{mult2}
	Let $\Sigma$ be a closed Riemann surface,  $\phi:\Sigma\setminus\ens{p_1,\cdots, p_n}\rightarrow \R^3$ be a complete minimal surface with finite total curvature and zero flux and $\vec{\Psi}:\Sigma\rightarrow \R^3$ be a compact branched Willmore surface such that $\vec{\Psi}=\iota\circ \phi$ and assume that the ends of $\phi$ have multiplicity at most $2$. Then there a universal symmetric matrix $\Lambda=\Lambda(\vec{\Psi})=\ens{\lambda_{i,j}}_{1\leq i,j\leq n}$ with \emph{zero diagonal entries} such that for all $v\in W^{2,2}(\Sigma)\cap C^4(\Sigma)$ and normal (admissible) variations $\vec{v}=v\n_{\vec{\Psi}}\in \mathscr{E}_{\phi}(\Sigma,\R^3)$
	\begin{align*}
	Q_{\vec{\Psi}}(v)&=\frac{1}{2}\int_{\Sigma}\left(\lg\left(|\phi|^2v_0\right)\right)^2d\vg+4\pi\sum_{1\leq i,j\leq n}^{}\lambda_{i,j}v(p_i)v(p_j)\\
	&=Q_{\vec{\Psi}}(v_0)+4\pi\sum_{1\leq i,j\leq n}^{}\lambda_{i,j}v(p_i)v(p_j),
	\end{align*}
	for some $v_0\in W^{2,2}(\Sigma)$ such that $v(p_j)=0$ for all $1\leq j\leq n$.
	In particular, we have
	\begin{align*}
	\mathrm{Ind}_W(\vec{\Psi})\leq n-1=\frac{1}{4\pi}W(\vec{\Psi})-\frac{1}{2\pi}\int_{\Sigma}K_{g}d\mathrm{vol}_{g}+\chi(\Sigma)
	\end{align*}
	and if $\vec{\Psi}:\Sigma\rightarrow \R^3$ is assumed to have no branched points, then
	\begin{align*}
	\mathrm{Ind}_{W}(\vec{\Psi})\leq \frac{1}{4\pi}W(\vec{\Psi})-1.
	\end{align*}
\end{theorem}
\begin{proof}
	We have already treated the case of embedded ends. Furthermore, as the expansion is universal (\textit{i.e.} independent of the multiplicity, see the proof of Theorem \ref{explicit}), we only need to compute
	\begin{align}
	Q_{\epsilon}(u_{\epsilon}^i)&=\frac{1}{2}\int_{\Sigma_{\epsilon}}(\Delta_gu_{\epsilon }^i-2K_gu_{\epsilon}^i)^2d\vg=\frac{1}{2}\int_{\partial B(0,\epsilon)}u_{\epsilon}^i\,\partial_{\nu}\left(\mathscr{L}_gu_{\epsilon}^i\right)-\left(\partial_{\nu}u_{\epsilon}^i\right)\mathscr{L}_gu_{\epsilon}^id\mathscr{H}^1\nonumber\\
	&=\frac{1}{2}\int_{\partial B(0,\epsilon)}u\,\partial_{\nu}\left(\mathscr{L}_gu_{\epsilon}^i\right)-\left(\partial_{\nu}u\right)\mathscr{L}_gu_{\epsilon}^id\mathscr{H}^1
	\end{align}
	as $u_{\epsilon}^i=\partial_{\nu}u_{\epsilon}^i=0$ on $\partial B_{\epsilon}(p_j)$ for $j\neq i$ and $\mathscr{L}_g^2u_{\epsilon}^i=0$. Let $v_{\epsilon}^i\in C^{\infty}(\bar{\Sigma_{\epsilon}})$ such that $u_{\epsilon}^i=|\phi|^2v_{\epsilon}^i$. Then we have
	\begin{align}\label{firstterm}
	\Delta_gu_{\epsilon}^i=4v+2e^{-2\lambda}\D|\phi|^2\cdot \D v_{\epsilon}^i+|\phi|^2\Delta_gv_{\epsilon}^i.
	\end{align}
	Furthermore, as we know that $Q_{\epsilon}(u_{\epsilon}^i)$ is only a function of $\epsilon,\alpha_0=\alpha_{0,i},\alpha_1=\alpha_{1,i},\alpha_2=\alpha_{2,i}$ and $v(p_i)$, we can (by an abuse of notation) replace all terms $\partial_{\nu}v$ by $0$ and $v$ by $v(p_i)$. Furthermore, as 
	\begin{align*}
	v_{\epsilon}^i=v,\quad \text{and}\;\, \partial_{\nu} v_{\epsilon}^i=\partial_{\nu}v\quad \text{on}\;\, \partial B(0,\epsilon).
	\end{align*}
	Furthermore, observe that for all smooth function $f:\bar{B}(0,\epsilon)\rightarrow \R$, we 
	\begin{align*}
	\partial_{\nu}f&=\D f(x)\cdot\frac{x}{|x|}=\frac{1}{|z|}\left(\p{x_1}f\cdot x_1+\p{x_2}f\cdot x_2\right)\\
	&=\frac{1}{2|z|}\left(\left(z+\z\right)\left(\partial+\bar{\partial}\right)+\frac{\left(z-\z\right)}{i}i\left(\partial-\bar{\partial}\right)\right)f\\
	&=\frac{1}{|z|}\left(z\cdot \partial+\z\cdot \bar{\partial}\right)f=\frac{2}{|z|}\,\Re\left(z\cdot \p{z}f\right).
	\end{align*}
	Now, recall that the volume form on $\partial B(0,\epsilon)$ is
	\begin{align}
	\frac{x_1dx_2-x_2dx_1}{|x|}=\frac{1}{4i|z|}\left(\left(z+\z\right) \left(dz-d\z\right)-(z-\z)(dz+d\z)\right)=\frac{1}{2i|z|}\left(\z dz-z d\z\right)=\frac{1}{|z|}\Im\left(\z dz\right).
	\end{align}
	Finally, we deduce that for all $f,g\in C^{\infty}(\bar{B}(0,\epsilon))$
	\begin{align*}
	\int_{\partial B(0,\epsilon)}g\,\partial_{\nu}f=\Im\int_{\partial B(0,\epsilon)}g\,\frac{2}{|z|}\,\Re\left(z\cdot \p{z}f\right)\frac{\z}{|z|}dz=2\,\Im\int_{\partial B(0,\epsilon)}g\,\Re\left(z\cdot \p{z}f\right)\frac{dz}{z}.
	\end{align*}
	Therefore, we have
	\begin{align}\label{exp}
	\Delta_gu_{\epsilon}^i=4v(p_i)+2\,e^{-2\lambda}\D|\phi|^2\cdot \D v_{\epsilon}^i+|\phi|^2\Delta_g v_{\epsilon}^i.
	\end{align}
	By the preceding remarks, we have
	\begin{align}\label{renorm2}
	Q_{\epsilon}(u_{\epsilon}^i)&=\frac{1}{2}v(p_i)\int_{\partial B(0,\epsilon)}|\phi|^2\partial_{\nu}\left(\mathscr{L}_gu_{\epsilon}^i\right)-\left(\partial_{\nu}|\phi|^2\right)\mathscr{L}_gu_{\epsilon}^id\mathscr{H}^1\nonumber\\
	&=v(p_i)\,\Im\int_{\partial B(0,\epsilon)}\left\{|\phi|^2\Re\left(z\cdot\p{z}\left(\mathscr{L}_gu_{\epsilon}^i\right)\right)-\Re\left(z\cdot \p{z}|\phi|^2\right)\mathscr{L}_gu_{\epsilon}^i\right\}\frac{dz}{z}
	\end{align}
	Now, define $w_{\epsilon}^i=\mathscr{L}_gu_{\epsilon}^i$. Then one checks directly by the expansion \eqref{exp} that
	\begin{align*}
	\int_{\Sigma\setminus \bar{B}_{\epsilon}(p_i)}|w_{\epsilon}^i|^2d\mathrm{vol}_{g_0}\leq C
	\end{align*}
	for some constant $C>0$ independent of $\epsilon$ as $w_{\epsilon}^i=4v(p_i)+O(|z|)$ in a conformal annulus around $\partial B_{\epsilon}(p_i)$. In particular, as $\epsilon\rightarrow 0$ we have $w_{\epsilon}^i\xrightharpoonup[\epsilon\rightarrow 0]{}w_{0}^i$ for some $w_{0}^i\in L^2(\Sigma,d\mathrm{vol}_{g_{0}})$. Furthermore, $w_{0}^i$ satisfies in the distributional sense 
	\begin{align*}
	\mathscr{L}_gw_{0}^i=0,\quad \text{in}\;\, \mathscr{D}'(\Sigma\setminus\ens{p_1,\cdots,p_n}).
	\end{align*}
	Now, let $\alpha:\Sigma\rightarrow \R$ be a conformal parameter such that 
	\begin{align*}
	g=e^{2\alpha}g_{0}
	\end{align*}
	for some constant Gauss curvature metric $g_0$ of unit volume on $\Sigma$. Then we have
	\begin{align*}
	\mathscr{L}_g=e^{-2\alpha}\left(\Delta_{g_0}-2e^{2\alpha}K_g\right)=e^{-2\alpha}\left(\Delta_{g_0}+V\right)
	\end{align*}
	where $V=-2e^{2\alpha}K_g$ is a real-analytic Schr\"{o}dinger potential (by the Weierstrass parametrisation for example). In particular, we have  in the distributional sense
	\begin{align*}
	\Delta_{g_0}w_{0}^i+Vw_0^i=0,\quad \text{in}\;\, \mathscr{D}'(\Sigma\setminus\ens{p_1,\cdots,p_n}).
	\end{align*}
	As $w_0^i\in L^2(\Sigma,g_0)$ and $V\in L^{\infty}(\Sigma)$, we have $\Delta_{g_0}w_{\epsilon}^i\in L^2(\Sigma,g_0)$. By an immediate bootstrap argument we obtain
	\begin{align*}
	w_{0}^i\in C^{\infty}(\Sigma).
	\end{align*}
	Furthermore, by direct elliptic estimate thanks to Theorem, for almost all $\epsilon_0>0$ small enough and $0<\epsilon<\epsilon_0$, we have
	\begin{align*}
	\int_{\Sigma_{\epsilon_0}}\left(w_{\epsilon}^i\right)^2d\mathrm{vol}_{g}\leq C
	\end{align*}
	and $\mathscr{L}_gw_{\epsilon}^i=0$ implies that for all $K\subset \Sigma_{\epsilon_0}$ there exists $C_k<\infty$ such that 
	\begin{align*}
	\int_{K}|\D^kw_{\epsilon}^i|^2d\mathrm{vol}_{g_0}\leq C_k
	\end{align*}
	so $w_{\epsilon}^i\conv{\epsilon\rightarrow 0}w_0^i$ in $C^k_{\mathrm{loc}}(\Sigma\setminus\ens{p_1,\cdots,p_n})$. In particular, as $w_{0}^i\in C^{\infty}(\Sigma)$, this implies that $w_{\epsilon}^i$ admits a Taylor expansion in the annulus (for some $\epsilon_0$ fixed and small enough) $B_{\epsilon_0}(p_i)\setminus \bar{B}_{\epsilon}(p_i)$ of the form (the first term is given by \eqref{firstterm})
	\begin{align*}
	w_{\epsilon}^i=4v(p_i)+\gamma_0|z|^2+\gamma_1|z|^4+2\,\Re\left(\zeta_0z+\zeta_1z^2+\zeta_2z^2\z+\zeta_3z^3+\zeta_4z^4+\zeta_5z^3\z\right)+O(|z|^5). 
	\end{align*}
	so that no singular power $\Re(\lambda z^m\z^{n})$ for some $n<0$ or $m<0$ occurs (these coefficients depend \emph{a priori} on $\epsilon$ but converge when $\epsilon\rightarrow 0$ so they are not singular in $\epsilon>0$ small enough). Now recall that
	\begin{align*}
	&-K_g=\frac{|z|^6}{2\beta_0}\bigg(6\,\Re\left(\left(\bar{\alpha_0}\alpha_1-\alpha_3\right)z^2\right)+O(|z|^3)\bigg)\\
	&e^{2\lambda}=\frac{4\beta_0}{|z|^6}\left(1+\beta_1|z|^2+2\,\Re\left(\alpha_0z\right)+O(|z|^3)\right).
	\end{align*}
	This implies that
	\begin{align*}
	&-2e^{2\lambda}K_g=24\,\Re\left(\left(\bar{\alpha_0}\alpha_1-\alpha_3\right)z^2\right)+O(|z|^3)\\
	&-2e^{2\lambda}K_g=4\left(24\,\Re\left(\left(\bar{\alpha_0}\alpha_1-\alpha_3\right)v(p_i)z^2\right)\right)+O(|z|^3)
	\end{align*}
	Furthermore, a direct computation shows that
	\begin{align*}
	\Delta w_{\epsilon}^i=4\p{z\z}^2w_{\epsilon}^i=4\left(\gamma_0+4\gamma_1|z|^2+2\,\Re\left(\zeta_2z+3\zeta_5z^2\right)+O(|z|^3)\right)
	\end{align*}
	Therefore
	\begin{align*}
	0=e^{2\lambda}\mathscr{L}_g w_{\epsilon}^i=\Delta w_{\epsilon}^i-2e^{2\lambda}K_gw_{\epsilon}^i=4\left(\gamma_0+4\gamma_1|z|^2+2\,\Re\left(\zeta_2z+3\left(8\left(\bar{\alpha_0}\alpha_1-\alpha_3\right)v(p_i)\right)+\zeta_5\right)z^2\right)+O(|z|^3).
	\end{align*}
	Therefore, we have 
	\begin{align*}
	\gamma_0=\gamma_1=\zeta_2=0,
	\end{align*}
	and $\zeta_5$ is a function of $\alpha_j$ and $v(p_i)$, but this latter fact is of no importance. In particular, we deduce that $w_{\epsilon}^i$ reduces to
	\begin{align*}
	w_{\epsilon}^i=4v(p_i)+2\,\Re\left(\zeta_0z+\zeta_1z^2+\zeta_3z^3+\zeta_4z^4+\zeta_5z^3\z\right)+O(|z|^5).
	\end{align*}
	Now, we have
	\begin{align*}
	|w_{\epsilon}^i|^2&=16v^2(p_i)+16v(p_i)\,\Re\left(\zeta_0z+\zeta_1z^2+\zeta_3z^3+\zeta_4z^4+\zeta_5z^3\z\right)\\
	&+2|\zeta_0|^2|z|^2+2|\zeta_1|^2|z|^4+2\,\Re\left(\zeta_0^2z^2+\zeta_0\zeta_1z^3+\bar{\zeta_0}\zeta_1z^2\z+\left(\zeta_0\zeta_3+\zeta_1^2\right)z^4+\bar{\zeta_0}\zeta_3z^3\z\right)+O(|z|^5)\\
	&=16v^2(p_i)+2|\zeta_0|^2|z|^2+2|\zeta_1|^2|z|^4+2\,\Re\left(8v(p_i)\zeta_0z+\mu_1z^2+\bar{\zeta_0}\zeta_1z^2\z+\mu_3z^3+\mu_4z^4+\mu_5z^3\z\right)+O(|z|^5),
	\end{align*}
	for some unimportant $\mu_j \in \C$.
	By \eqref{conf1}, we have
	\begin{align}
	e^{2\lambda}=\frac{4\beta_0}{|z|^6}\left(1+\beta_1|z|^2+2\,\Re\left(\alpha_0z-\alpha_1z^3-2\alpha_2z^4-\alpha_3z^3\z\right)+O(|z|^5)\right).
	\end{align}
	Therefore, we have for some $\nu_j\in \C$
	\begin{align*}
	|w_{\epsilon}^i|^2e^{2\lambda}&=\frac{4\beta_0}{|z|^6}\bigg(16v^2(p_i)+16\beta_1v^2(p_i)|z|^2+2|\zeta_0|^2|z|^2+2|\zeta_1|^2|z|^4+2|\zeta_0|^2\beta_1|z|^4+16\,\Re(\bar{\alpha_0}\zeta_0)v(p_i)|z|^2\\
	&+2\,\Re\left(\bar{\alpha_0}\,\bar{\zeta_0}\zeta_1\right)|z|^4+\Re\left(\nu_0z+\nu_1z^2+\nu_2z^2\z+\nu_3z^3+\nu_4z^4+\nu_5z^3\z\right)+O(|z|^5)\bigg)\\
	&=\frac{4\beta_0}{|z|^6}\bigg(16v^2(p_i)+\left(16\beta_1v^2(p_i)+2|\zeta_0|^2+16\,\Re\left(\bar{\alpha_0}\zeta_0\right)v(p_i)\right)|z|^2+\left(2|\zeta_1|^2+2|\zeta_0|^2\beta_1+2\,\Re\left(\bar{\alpha_0}\,\bar{\zeta_0}\zeta_1\right)\right)|z|^4\\
	&+\Re\left(\nu_0z+\nu_1z^2+\nu_2z^2\z+\nu_3z^3+\nu_4z^4+\nu_5z^3\z\right)+O(|z|^5)\bigg)
	\end{align*}
	Now, we know that 
	\begin{align}\label{renorm}
	\frac{1}{2}\int_{\Sigma\setminus\bar{B}_{\epsilon}(p_i)}\left(\mathscr{L}_gu_{\epsilon}^i\right)^2d\vg=\frac{1}{2}\int_{\Sigma\setminus\bar{B}_{\epsilon}(p_i)}|w_{\epsilon}^i|^2d\vg=\frac{16\pi\beta_0}{\epsilon^4}v^2(p_i)+\frac{32\pi\beta_0\beta_1}{\epsilon^2}v^2(p_i)+O(1).
	\end{align}
	Otherwise, as $u_{\epsilon}$ and $u_{\epsilon}^j$ (for all $j\neq i$) are independent of $\epsilon>0$, we would obtain in the limit an infinite quantity, although $Q_{\phi}(v)$ is finite, a contradiction.
	Now, we have by polar coordinates
	\begin{align}\label{magic2}
	&\frac{1}{2}\int_{B(0,1)\setminus \bar{B}(0,\epsilon)}|w_{\epsilon}^i|^2d\vg=\frac{1}{2}\int_{B(0,1)\setminus \bar{B}(0,\epsilon)}|w_{\epsilon}^i|^2e^{2\lambda}|dz|^2\nonumber\\
	&=4\pi\beta_0\int_{\epsilon}^{1}\left(\frac{16v^2(p_i)}{r^5}+\frac{16\beta_1v^2(p_i)+2|\zeta_0|^2+16\,\Re\left(\bar{\alpha_0}\zeta_0\right)v(p_i)}{r^3}+\frac{2|\zeta_1|^2+2|\zeta_0|^2\beta_1+2\,\Re\left(\bar{\alpha_0}\,\bar{\zeta_0}\zeta_1\right)}{r}\right)dr+O(1)\nonumber\\
	&=4\pi\beta_0\left(\frac{4v^2(p_i)}{\epsilon^4}+\frac{8\beta_1v^2(p_i)+|\zeta_0|^2+16\,\Re\left(\bar{\alpha_0}\zeta_0\right)v(p_i)}{\epsilon^2}+\left(2|\zeta_1|^2+2|\zeta_0|^2\beta_1+2\,\Re\left(\bar{\alpha_0}\bar{\zeta_0}\zeta_1\right)\right)\log\left(\frac{1}{\epsilon}\right)\right)+O(1)\nonumber\\
	&=\frac{16\pi\beta_0}{\epsilon^4}v^2(p_i)+\frac{32\pi\beta_0\beta_1}{\epsilon^2}v^2(p_i)+\frac{8\pi\beta_0\left(|\zeta_0|^2+8\,\Re(\bar{\alpha_0}\zeta_0)v(p_i)\right)}{\epsilon^2}\nonumber\\
	&+\left(2|\zeta_1|^2+2|\zeta_0|^2\beta_1+2\,\Re\left(\bar{\alpha_0}\bar{\zeta_0}\zeta_1\right)\right)\log\left(\frac{1}{\epsilon}\right)+O(1)\nonumber\\
	&=\frac{16\pi\beta_0}{\epsilon^4}v^2(p_i)+\frac{32\pi\beta_0\beta_1}{\epsilon^2}v^2(p_i)+O(1)
	\end{align}
	where the last equality comes from \eqref{renorm}.
	Therefore, \eqref{magic2} gives the two equalities
	\begin{align}\label{endorder2}
	\left\{\begin{alignedat}{1}
	&|\zeta_0|^2+8\,\Re\left(\bar{\alpha_0}\zeta_0\right)v(p_i)=0\\
	&2|\zeta_1|^2+2|\zeta_0|^2\beta_1+2\,\Re\left(\bar{\alpha_0}\bar{\zeta_0}\zeta_1\right)=0.
	\end{alignedat}\right.
	\end{align}
	Now, by Cauchy's inequality, and as $|\alpha_0|^2=\beta_1$ we have
	\begin{align*}
	\left|2\,\Re\left(\bar{\alpha_0}\bar{\zeta_0}\zeta_1\right)\right|\leq |\zeta_1|^2+|\alpha_0|^2|\zeta_0|^2=|\zeta_1|^2+|\zeta_0|^2\beta_1.
	\end{align*}
	This implies that
	\begin{align*}
	0=2|\zeta_1|^2+2|\zeta_0|^2\beta_1+2\,\Re\left(\bar{\alpha_0}\bar{\zeta_0}\zeta_1\right)\geq 2|\zeta_1|^2+2|\zeta_0|^2\beta_1-\left|2\,\Re\left(\bar{\alpha_0}\bar{\zeta_0}\zeta_1\right)\right|\geq |\zeta_1|^2+\beta_1|\zeta_0|^2
	\end{align*}
	so $\zeta_0=0$ and $\zeta_1=0$ if $\beta_1\neq 0$. However, if $\beta_1=0$, then $\alpha_0$ (as $\beta_1=|\alpha_0|^2$ and the first equation of \eqref{endorder2} becomes
	\begin{align*}
	0=|\zeta_0|^2+8\,\Re\left(\bar{\alpha_0}\zeta_0\right)v(p_i)=|\zeta_0|^2
	\end{align*}
	so $\zeta_0=0$ in all cases.

	Therefore, we have $\zeta_0=\zeta_1=0$ and $w_{\epsilon}^i$ reduces to 
	\begin{align}\label{wnew1}
	w_{\epsilon}^i=4v(p_i)+2\,\Re\left(\zeta_3z^3+\zeta_4z^4+\zeta_5z^3\z\right)+O(|z|^5).
	\end{align}
	Finally, \eqref{magic2} becomes
	\begin{align*}
	\frac{1}{2}\int_{\Sigma\setminus\bar{B}_{\epsilon}(p_i)}|w_{\epsilon}^i|^2d\vg=4\pi\beta_0\left(\frac{4v^2(p_i)}{\epsilon^4}+\frac{8\beta_1v^2(p_i)}{\epsilon^2}\right)+O(1)=\frac{16\pi\beta_0}{\epsilon^4}v^2(p_i)+\frac{32\pi\beta_0\beta_1}{\epsilon^2}v^2(p_i)+O(1)
	\end{align*}
	as expected. Now, as the factors in $\Re(z^3)$, $\Re(z^4)$ and $\Re(z^3\z)$ do not contributes to the renormalised energy in \eqref{renorm2}, we deduce by \eqref{wnew1} that
	\begin{align}\label{end2fin1}
	Q_{\epsilon}(u_{\epsilon}^i)&=v(p_i)\,\Im\int_{\partial B(0,\epsilon)}\left\{|\phi|^2\Re\left(z\cdot\p{z}\left(\mathscr{L}_gu_{\epsilon}^i\right)\right)-\Re\left(z\cdot \p{z}|\phi|^2\right)\mathscr{L}_gu_{\epsilon}^i\right\}\frac{dz}{z}\nonumber\\
	&=4v^2(p_i)\,\Im\int_{\partial B(0,\epsilon)}-\,\Re\left(z\cdot \p{z}|\phi|^2\right)\frac{dz}{z}.
	\end{align}
	Now, recall that by \eqref{devdelphi}
	\begin{align*}
	\partial |\phi|^2=\frac{\beta_0}{z|z|^4}\left(-2-4\beta_1|z|^2-2\alpha_0z-4\bar{\alpha_0}\,\z+2\alpha_1z^3-4\bar{\alpha_1}\,\z^3+8i\,\Im\left(\alpha_2z^4+\alpha_3z^3\z\right)+O(|z|^5)\right).
	\end{align*}
	Therefore, we have (notice that the purely imaginary term cancels)
	\begin{align*}
	-\Re\left(z\cdot \p{z}|\phi|^2\right)=\frac{\beta_0}{|z|^4}\left(2+4\beta_1|z|^2+2\,\Re\left(3\alpha_0z-\alpha_1z^3\right)+O(|z|^5)\right).
	\end{align*}
	Therefore, we directly obtain
	\begin{align}\label{end2fin2}
	\Im\int_{\partial B(0,\epsilon)}-\,\Re\left(z\cdot \p{z}|\phi|^2\right)\frac{dz}{z}=\frac{4\pi\beta_0}{\epsilon^4}+\frac{8\pi\beta_0\beta_1}{\epsilon^2}+O(\epsilon)
	\end{align}
	and finally by \eqref{end2fin1} and \eqref{end2fin2}
	\begin{align*}
	Q_{\epsilon}(u_{\epsilon}^i)=\frac{16\pi\beta_0}{\epsilon^4}v^2(p_i)+\frac{32\pi\beta_0\beta_1}{\epsilon^2}v^2(p_i)+O(\epsilon),
	\end{align*}
	so that no constant term occurs.    
\end{proof}

\begin{rem}
	Notice that the absence of diagonal entries is a consequence (and is equivalent) that the minimal surface $\phi:\Sigma\setminus\ens{p_1,\cdots,p_n}\rightarrow \R^3$ has zero flux.
\end{rem}

\section{Appendix}

\subsection{Estimates for some weighted elliptic operators}

We fix an integer $m\geq 2$. Let $\omega:\R^n\rightarrow\R_+$ a measurable function and for all $k\in\N$ and $1\leq p<\infty$ define the weighted Sobolev space
\begin{align*}
W^{k,p}_{\omega}(\R^m)=L^p(\R^m)\cap \ens{u: \Vert u\Vert_{W_{\omega}^{k,p}}<\infty }
\end{align*}
where
\begin{align*}
\Vert u\Vert_{W_{\omega}^{k,p}}=\left(\int_{\R^m}|u|^pd\leb^m+\sum_{j=0}^{k}\int_{\R^m}|\D^ju|^p\omega^{p(k-j)}d\leb^m\right)^{\frac{1}{p}}.
\end{align*}
By the classical Gagliardo-Nirenberg inequality, we have a continuous injection $W^{k,p}_{\omega}(\R^m)\hookrightarrow W^{k,p}(\R^m)$.

\begin{lemme}\label{lemme1}
	Let $\delta>0$ be a fixed real number. For all $u\in W^{2,2}(\R^m\setminus\bar{B}_{\delta}(0))$ such that either $u=0$ or $\partial_{\nu}u=0$ on $\partial B_{\delta}(0)$, for all $1\leq\alpha<\infty$ we
	\begin{align}\label{borne}
	\np{\frac{\D u}{|x|^\alpha}}{2}{\R^m}\leq 2\alpha\np{\frac{u}{|x|^{\alpha+1}}}{2}{\R^m}+\np{\frac{u}{|x|^{\alpha+1}}}{2}{\R^m}^{\frac{1}{2}}\np{\frac{\Delta u}{|x|^{\alpha-1}}}{2}{\R^m}^{\frac{1}{2}}
	\end{align}
	provided the integrals on the right-hand side of \eqref{borne} be finite. In particular, if $\omega:\R^n\rightarrow\R$ is such that $\omega(x)=|x|^{-1}$, we have a continuous injection 
	\begin{align}\label{inj}
	W^{2,2}(\R^m)\cap L^2_{\omega}(\R^m)\hookrightarrow W^{2,2}_{\omega}(\R^m).
	\end{align}
\end{lemme}
\begin{proof}
	We first assume $u\in W^{2,2}\cap C^{\infty}(\R^m\setminus\bar{B}_{\delta}(0))$ such that either $u=0$ or $\partial_{\nu}u=0$ on $\partial_{\nu}u=0$ on $\partial B_{\delta}(0)$ (so that $u\,\partial_{\nu}u=0$ on $\partial B_{\delta}(0)$). Then we have
	\begin{align}\label{id0}
	\dive(u\D u |x|^{-2\alpha})=|\D u|^2|x|^{-2\alpha}+u\Delta u|x|^{-2\alpha}-2\alpha u (\D u\cdot x)|x|^{-2(\alpha+1)}
	\end{align}
	Therefore, fixing $0\leq \theta_1,\theta_2\leq 1$, $1<p<\infty$, we have by Cauchy-Schwarz inequality and as $u\,\partial_{\nu}u=0$ on $\partial B_{\delta}(0)$
	\begin{align*}
	\int_{\R^m\setminus\bar{B}_{\delta}(0)}^{}\frac{|\D u|^2}{|x|^{2\alpha}}dx&=2\alpha\int_{\R^m\setminus\bar{B}_{\delta}(0)}^{}u\frac{\D u\cdot x}{|x|^{2(\alpha+1)}}dx-\int_{\R^m\setminus\bar{B}_{\delta}(0)}\frac{u\Delta u}{|x|^{2\alpha}}dx\\
	&\leq 2\alpha\left(\int_{\R^m\setminus\bar{B}_{\delta}(0)}^{}\frac{u^2}{|x|^{2(2\alpha+1)\theta_1}}dx\right)^{\frac{1}{2}}\left(\int_{\R^m\setminus\bar{B}_{\delta}(0)}^{}\frac{|\D u|^2}{|x|^{2(2\alpha+1)(1-\theta_1)}}dx\right)^{\frac{1}{2}}\\
	&+\left(\int_{\R^m\setminus\bar{B}_{\delta}(0)}^{}\frac{u^2}{|x|^{4\alpha\theta_2}}dx\right)^{\frac{1}{2}}\left(\int_{\R^m\setminus\bar{B}_{\delta}(0)}^{}\frac{(\Delta u)^2}{|x|^{4\alpha(1-\theta_2)}}dx\right)^{\frac{1}{2}}
	\end{align*}
	As we want to recover the same exponent for $|x|$ in the denominator of $u^2$ (and $|\D u|^2$) on both sides, we choose $\theta_1$ such that
	\begin{align*}
	(2\alpha+1)(1-\theta_1)=\alpha
	\end{align*}
	\textit{i.e.}
	\begin{align*}
	\theta_1=\frac{\alpha+1}{2\alpha+1}
	\end{align*}
	and $\theta_2$ such that
	\begin{align*}
	2\alpha\theta_2=(2\alpha+1)\theta_1
	\end{align*}
	so
	\begin{align*}
	\theta_2=\frac{\alpha+1}{2\alpha}\in [0,1]
	\end{align*}
	for all $\alpha\geq 1$. Finally, we get if
	\begin{align*}
	&X=\left(\int_{\R^m\setminus\bar{B}_{\delta}(0)}^{}\frac{|\D u|^2}{|x|^{2\alpha}}dx\right)^{\frac{1}{2}},\quad
	a=2\alpha\left(\int_{\R^m\setminus\bar{B}_{\delta}(0)}^{}\frac{u^2}{|x|^{2(\alpha+1)}}dx\right)^{\frac{1}{2}},\\
	&b=\left(\int_{\R^m\setminus\bar{B}_{\delta}(0)}^{}\frac{u^2}{|x|^{2(\alpha+1)}}dx\right)^{\frac{1}{2}}\left(\int_{\R^m\setminus\bar{B}_{\delta}(0)}^{}\frac{(\Delta u)^2}{|x|^{2(\alpha-1)}}dx\right)^{\frac{1}{2}}
	\end{align*}
	\begin{align*}
	X^2\leq a X+b
	\end{align*}
	Therefore, 
	\begin{align*}
	X\leq \frac{1}{2}\left(a+\sqrt{a^2+4b}\right)\leq a+\sqrt{b},
	\end{align*}
	or
	\begin{align*}
		\np{\frac{\D u}{|x|^{\alpha}}}{2}{\R^m\setminus\bar{B}_{\delta}(0)}\leq 2\alpha\np{\frac{u}{|x|^{\alpha+1}}}{2}{\R^m\setminus\bar{B}_{\delta}(0)}+\np{\frac{u}{|x|^{\alpha+1}}}{2}{\R^m\setminus\bar{B}_{\delta}(0)}^{\frac{1}{2}}\np{\frac{\Delta u}{|x|^{\alpha-1}}}{2}{\R^m\setminus\bar{B}_{\delta}(0)}^{\frac{1}{2}}.
	\end{align*}
	Notice that this inequality cannot be improved by scaling argument because of the singular weights. The general inequality for $u\in W^{2,2}(\R^2\setminus \bar{B}_{\delta}(0))$ such that either $u=0$ or $\partial_{\nu}u=0$ on $\partial B_{\delta}(0)$ follows by standard regularisation. This concludes the proof of the lemma.
    \end{proof}
    
    \begin{lemme}\label{indicielles}
    	Let $\delta>0$ be a fixed real number, and define for all for all $m\geq 1$ the second order elliptic differential operator
    	\begin{align*}
    	\mathscr{L}_m=\Delta-2(m+1)\frac{x}{|x|^2}\cdot \D +\frac{(m+1)^2}{|x|^2}.
    	\end{align*} 
    	Let $u\in W^{2,2}(\R^2\setminus\bar{B}_{\delta}(0))$ be such that $u=\partial_{\nu}u=0$ on $\partial B_{\delta}(0)$ and assume that $\mathscr{L}_m u\in L^2(\R^2)$. Then we have the identities for all $m\geq 1$ the identity
    	\begin{align}\label{ipp0}
    		&\int_{\R^2\setminus \bar{B}_{\delta}(0)}\left(\mathscr{L}_mu\right)^2dx=\int_{\R^2\setminus\bar{B}_{\delta}(0)}\left(\Delta u-2(m+1)\frac{x}{|x|^2}\cdot \D u+(m+1)^2\frac{u}{|x|^2}\right)^2dx\\
    		&=\int_{\R^2\setminus\bar{B}_{\delta}(0)}\left(\Delta u+(m+1)(m-1)\frac{u}{|x|^2}\right)^2dx+4(m+1)(m-1)\int_{\R^2\setminus\bar{B}_{\delta}(0)}\left(\frac{x}{|x|^2}\cdot \D u-\frac{u}{|x|^2}\right)^2dx.\nonumber
    	\end{align}
        In particular, we have for $m=1$ 
        \begin{align*}
        	\int_{\R^2\setminus\bar{B}_{\delta}(0)}\left(\mathscr{L}_1u\right)^2dx=\int_{\R^2\setminus\bar{B}_{\delta}(0)}\left(\Delta u-4\frac{x}{|x|^2}\cdot \D u+\frac{4}{|x|^2}u\right)^2dx=\int_{\R^2\setminus\bar{B}_{\delta}(0)}\left(\Delta u\right)^2dx.
        \end{align*}
        Furthermore, if $m>1$, then
        \begin{align*}
        	&\np{\frac{u}{|x|^2}}{2}{\R^2\setminus\bar{B}_{\delta}(0)}\leq \frac{1}{(m+1)(m-1)}\np{\Delta u}{2}{\R^2\setminus\bar{B}_{\delta}(0)}+\frac{1}{(m+1)(m-1)}\np{\mathscr{L}_mu}{2}{\R^2\setminus\bar{B}_{\delta}(0)}.	
        \end{align*}
        If $m\geq 3$, 
        \begin{align*}
        	\np{\frac{\D u\cdot x}{|x|^2}}{2}{\R^2\setminus \bar{B}_{\delta}(0)}\leq \frac{1}{2(m+1)}\np{\Delta u}{2}{\R^2\setminus \bar{B}_{\delta}(0)}+\frac{1}{2(m+1)}\np{\mathscr{L}_m u}{2}{\R^2\setminus \bar{B}_{\delta}(0)}
        \end{align*}
        while for $1<m\leq 3$
        \begin{align*}
        	\np{\frac{\D u\cdot x}{|x|^2}}{2}{\R^2\setminus \bar{B}_{\delta}(0)}\leq \frac{1}{(m+1)(m-1)}\np{\Delta u}{2}{\R^2\setminus \bar{B}_{\delta}(0)}+\frac{1}{(m+1)(m-1)}\np{\mathscr{L}_m}{2}{\R^2\setminus\bar{B}_{\delta}(0)}.
        \end{align*}
    \end{lemme}
	\begin{proof}
	\textbf{Step 1: Equalities.}
	Observe that for all $x\in \R^2\setminus\ens{0}$, we have
	\begin{align*}
		\Delta \frac{1}{|x|^{2\alpha}}=\frac{4\alpha^2}{|x|^{2\alpha+2}}
	\end{align*}
	and assuming that $u\in W^{2,2}\cap C^{\infty}(\R^2\setminus\bar{B}_{\delta}(0))$ without loss of generality, we have as $\Delta u^2=2u\,\Delta u+2|\D u|^2$
	\begin{align*}
		\int_{\R^2\setminus\bar{B}_{\delta}(0)}\frac{u^2}{|x|^{2\alpha+2}}dx=\int_{\R^2\setminus\bar{B}_{\delta}(0)}u^2\frac{1}{4\alpha^2}\Delta\frac{1}{|x|^{2\alpha}}dx=\int_{\R^2\setminus\bar{B}_{\delta}(0)}\frac{\Delta u^2}{4\alpha^2|x|^{2\alpha}}dx=\int_{\R^2\setminus\bar{B}_{\delta}(0)}\frac{u\Delta u+|\D u|^2}{2\alpha^2|x|^{2\alpha}}dx
	\end{align*}
	Furthermore, recall that by \eqref{id0}
	\begin{align*}
		\int_{\R^2\setminus\bar{B}_{\delta}(0)}\frac{u\Delta u+|\D u|^2}{|x|^{2\alpha}}dx=\int_{\R^2\setminus\bar{B}_{\delta}(0)}2\alpha \frac{u\,\left(\D u\cdot x\right)}{|x|^{2\alpha+2}}dx.
	\end{align*}
	Therefore, we find
	\begin{align*}
		\int_{\R^2\setminus\bar{B}_{\delta}(0)}\frac{u^2}{|x|^{2\alpha+2}}dx=\frac{1}{\alpha}\int_{\R^2\setminus\bar{B}_{\delta}(0)}\frac{u\left(\D u\cdot x)\right)}{|x|^{2\alpha+2}}dx\leq \frac{1}{\alpha}\left(\int_{\R^2\setminus\bar{B}_{\delta}(0)}\frac{u^2}{|x|^{2\alpha+2}}dx\right)^{\frac{1}{2}}\left(\int_{\R^2\setminus\bar{B}_{\delta}(0)}\frac{(\D u\cdot x)^2}{|x|^{2\alpha+2}}dx\right)^{\frac{1}{2}}
	\end{align*}
	which implies that
	\begin{align}\label{id1}
		\np{\frac{u}{|x|^{\alpha+1}}}{2}{\R^2\setminus\bar{B}_{\delta}(0)}\leq \frac{1}{\alpha}\np{\frac{\D u\cdot x}{|x|^{\alpha+1}}}{2}{\R^2\setminus\bar{B}_{\delta}(0)}.
	\end{align}
	Now, if $\alpha=1$, we find equivalently
	\begin{align*}
	    &\int_{\R^2\setminus\bar{B}_{\delta}(0)}\frac{u^2}{|x|^4}dx=\int_{\R^2\setminus\bar{B}_{\delta}(0)}\frac{u\left(\D u \cdot x\right)}{|x|^4}dx\\
		&\int_{\R^2\setminus\bar{B}_{\delta}(0)}\frac{u}{|x|^2}\left(\frac{u}{|x|^2}-\frac{\D u\cdot x}{|x|^2}\right)dx=0.
	\end{align*}
	Now, compute for all $u\in C^{\infty}_c(\R^2)$
	\begin{align*}
	\Delta \left(\frac{u}{|x|^2}\right)=\frac{1}{|x|^2}\Delta u+2\,\D\left(\frac{1}{|x|^2}\right)\cdot \D u+u\,\Delta\left(\frac{1}{|x|^2}\right)=\frac{1}{|x|^2}\left(\Delta-4\frac{x}{|x|^2}\cdot \D+\frac{4}{|x|^2}\right)u.
	\end{align*}
	We also have 
	\begin{align*}
	&\p{x_1}\left(\frac{x_1}{|x|^2}\p{x_1}u\right)=\left(\frac{1}{|x|^2}-\frac{2x_1^2}{|x|^4}\right)\p{x_1}u+\frac{1}{|x|^2}\p{x_1}^2u\\
	&\p{x_1}^2\left(\frac{x_1}{|x|^2}\p{x_1}u\right)=\left(-\frac{6x_1}{|x|^4}+\frac{8x_1^3}{|x|^6}\right)\p{x_1}u+2\left(\frac{1}{|x|^2}-\frac{2x_1^2}{|x|^4}\right)\p{x_1}^2u+\frac{x_1}{|x|^2}\p{x_1}^3u\\
	&\p{x_2}\left(\frac{x_1}{|x|^2}\p{x_1}u\right)=-\frac{2x_1x_2}{|x|^4}\p{x_1}u+\frac{x_1}{|x|^2}\p{x_1,x_2}^2u\\
	&\p{x_2}^2\left(\frac{x_1}{|x|^2}\p{x_1}u\right)=\left(-\frac{2x_1}{|x|^4}+\frac{8x_1x_2^2}{|x|^6}\right)\p{x_1}u-\frac{4x_1x_2}{|x|^4}\p{x_1,x_2}^2u+\frac{x_1}{|x|^2}\p{x_1}\p{x_2}^2u.
	\end{align*}
	Therefore, we find
	\begin{align*}
	\Delta\left(\frac{x_1}{|x|^2}\p{x_1}u\right)&=\left(-\frac{8x_1}{|x|^4}+\frac{8x_1\left(x_1^2+x_2^2\right)}{|x|^6}\right)\p{x_1}u+\frac{x_1}{|x|^2}\p{x_1}\left(\p{x_1}^2u+\p{x_2}^2u\right)+\frac{2}{|x|^2}\p{x_1}^2u\\
	&-4\left(\frac{x_1^2}{|x|^4}\p{x_1}^2u+\frac{x_1x_2}{|x|^4}\p{x_1,x_2}^2u\right)\\
	&=\frac{x_1}{|x|^2}\p{x_1}\Delta u+\frac{2}{|x|^2}\p{x_1}^2u-4\left(\frac{x_1^2}{|x|^4}\p{x_1}^2u+\frac{x_1x_2}{|x|^4}\p{x_1,x_2}^2u\right),
	\end{align*}
	and by symmetry this implies that 
	\begin{align*}
	\Delta\left(\frac{x}{|x|^2}\cdot \D u\right)&=\frac{x}{|x|^2}\cdot \D\Delta u+\frac{2}{|x|^2}\Delta u-4\left(\frac{x_1^2}{|x|^4}\p{x_1}^2u+\frac{2x_1x_2}{|x|^4}\p{x_1,x_2}^2u+\frac{x_2^2}{|x|^4}\p{x_2}^2u\right)\\
	&=\frac{x}{|x|^2}\cdot \D\Delta u+\frac{2}{|x|^2}\Delta u-4\left(\frac{x}{|x|^2}\right)^t\cdot\D^2 u\cdot\left(\frac{x}{|x|^2}\right).
	\end{align*}
	Therefore, we deduce that 
	\begin{align*}
	\Delta\left(\mathscr{L}_m\right)&=\Delta^2-2(m+1)\frac{x}{|x|^2}\cdot \D \Delta u-\frac{4(m+1)}{|x|^2}\Delta u+8(m+1)\left(\frac{x}{|x|^2}\right)^t\cdot\D^2u\cdot\left(\frac{x}{|x|^2}\right)\\
	&+\frac{(m+1)^2}{|x|^2}\left(\Delta u-4\frac{x}{|x|^2}\cdot \D u+\frac{4}{|x|^2}u\right)\\
	&=\Delta^2-2(m+1)\frac{x}{|x|^2}\cdot \D\Delta u+\frac{(m+1)^2-4(m+1)}{|x|^2}\Delta u+8(m+1)\left(\frac{x}{|x|^2}\right)^t\cdot\D^2u\cdot \left(\frac{x}{|x|^2}\right)\\
	&-4(m+1)^2\frac{x}{|x|^4}\cdot \D u+\frac{4(m+1)^2}{|x|^4}u.
	\end{align*}
	Now, we have
	\begin{align*}
	\frac{x}{|x|^2}\cdot \D\left(\frac{1}{|x|^2}u\right)=-\frac{2}{|x|^4}u+\frac{x}{|x|^4}\cdot \D u,
	\end{align*}
	and
	\begin{align*}
	x_1\p{x_1}\left(\frac{x}{|x|^2}\cdot \D u\right)&=x_1\left\{\left(\frac{1}{|x|^2}-\frac{2x_1^2}{|x|^4}\right)\p{x_1}u-\frac{2x_1x_2}{|x|^4}\p{x_2}u+\frac{x}{|x|^4}\cdot \D\p{x_1}u\right\}\\
	&=\frac{x_1\left(-x_1^2+x_2^2\right)}{|x|^4}\p{x_1}u-\frac{2x_1^2x_2}{|x|^4}\p{x_2}u+\frac{x_1^2}{|x|^2}\p{x_1}^2u+\frac{x_1x_2}{|x|^2}\p{x_1,x_2}^2u\\
	x_2\p{x_2}\left(\frac{x}{|x|^2}\cdot \D u\right)&=\frac{x_2\left(x_1^2-x_2^2\right)}{|x|^4}\p{x_2}u-\frac{2x_1x_2^2}{|x|^4}\p{x_1}u+\frac{x_1^2}{|x|^2}\p{x_1}^2u+\frac{x_1x_2}{|x|^2}\p{x_1,x_2}^2u\\
	\frac{x}{|x|^2}\cdot \D\left(\frac{x}{|x|^2}\cdot \D u\right)&=-\frac{x}{|x|^4}\cdot \D u+\left(\frac{x}{|x|^2}\right)^t\cdot \D^2 u\cdot \left(\frac{x}{|x|^2}\right),
	\end{align*}
	which implies that 
	\begin{align*}
	\frac{x}{|x|^2}\cdot \D\left(\mathscr{L}_mu\right)&=\frac{x}{|x|^2}\cdot \D\Delta u-2(m+1)\left(\frac{x}{|x|^2}\right)^t\cdot \D^2 u\cdot\left(\frac{x}{|x|^2}\right)+\left\{(m+1)^2+2(m+1)\right\}\frac{x}{|x|^4}\cdot \D u\\
	&-\frac{2(m+1)^2}{|x|^4}u.
	\end{align*}
	Therefore, 
	\begin{align*}
	\mathscr{L}_m^{\ast}\mathscr{L}_m&u=\left(\Delta+2(m+1)\frac{x}{|x|^2}\cdot \D+\frac{(m+1)^2}{|x|^2}\right)\left(\mathscr{L}_mu\right)\\
	&=\Delta^2u-\colorcancel{2(m+1)\frac{x}{|x|^2}\cdot \D\Delta u}{blue}+\frac{(m+1)^2-4(m+1)}{|x|^2}\Delta u+8(m+1)\left(\frac{x}{|x|^2}\right)^t\cdot \D^2u\cdot \left(\frac{x}{|x|^2}\right)\\
	&-\ccancel{4(m+1)^2\frac{x}{|x|^2}\cdot \D u}+\frac{4(m+1)^2}{|x|^4}u\\
	&+\colorcancel{2(m+1)\frac{x}{|x|^2}\cdot \D\Delta u}{blue}-4(m+1)^2\left(\frac{x}{|x|^2}\right)^t\cdot \D^2u\cdot\left(\frac{x}{|x|^2}\right)+\Big\{\ccancel{2(m+1)^3}+\ccancel{4(m+1)^2}\Big\}\frac{x}{|x|^2}\cdot \D u\\
	&-\frac{4(m+1)^3}{|x|^4}u\\
	&+\frac{(m+1)^2}{|x|^2}\Delta u-\ccancel{2(m+1)^3\frac{x}{|x|^2}\cdot \D u}+\frac{(m+1)^4}{|x|^4}u\\
	&=\Delta^2u+\frac{2(m+1)(m-1)}{|x|^2}\Delta u-4(m+1)(m-1)\left(\frac{x}{|x|^2}\right)^t\cdot \D^2u\cdot\left(\frac{x}{|x|^2}\right)+\frac{(m+1)^2(m-1)^2}{|x|^4}u,
	\end{align*}
	and we indeed recover $\mathscr{L}^{\ast}\mathscr{L}_1=\Delta^2$.
	We deduce that for all $u\in W^{2,2}\cap C^{\infty}(\R^2\setminus\bar{B}_{\delta}(0))$ such that $u=\partial_{\nu}u=0$ on $\partial B_{\delta}(0)$, 
	\begin{align}\label{newipp}
	&\int_{\R^2\setminus \bar{B}_{\delta}(0)}\left(\Delta u-2(m+1)\frac{x}{|x|^2}\cdot \D u+\frac{(m+1)^2}{|x|^2}u\right)^2dx=\int_{\R^2\setminus\bar{B}_{\delta}(0)}\left(\mathscr{L}_mu\right)^2dx\nonumber\\
	&=\int_{\R^2\setminus\bar{B}_{\delta}(0)}\bigg(u\,\Delta^2u+2(m+1)(m-1)\frac{u}{|x|^2}\Delta u-4(m+1)(m-1)u\,\left(\frac{x}{|x|^2}\right)^t\cdot\D^2u\cdot\left(\frac{x}{|x|^2}\right)\nonumber\\
	&+\frac{(m+1)^2(m-1)^2}{|x|^4}u^2\bigg)dx
	+\int_{\partial B_{\delta}(0)}\left(u\,\partial_{\nu}\left(\mathscr{L}_mu\right)-\partial_{\nu}\left(\mathscr{L}_mu\right)\right)d\mathscr{H}^1\nonumber\\
	&=\int_{\R^2\setminus \bar{B}_{\delta}(0)}\left(\left(\Delta u\right)^2+2(m+1)(m-1)\Delta u\,\frac{u}{|x|^2}+(m+1)^2(m-1)^2\frac{u^2}{|x|^4}\right)dx\nonumber\\
	&-4(m+1)(m-1)\int_{\R^2\setminus \bar{B}_{\delta}(0)}u\,\left(\frac{x}{|x|^2}\right)^t\cdot\D^2u\cdot\left(\frac{x}{|x|^2}\right)dx\nonumber\\
	&=\int_{\R^2\setminus\bar{B}_{\delta}(0)}\left(\Delta u+(m+1)(m-1)^2\frac{u}{|x|^2}\right)^2dx-4(m+1)(m-1)\int_{\R^2\setminus \bar{B}_{\delta}(0)}u\,\left(\frac{x}{|x|^2}\right)^t\cdot\D^2u\cdot\left(\frac{x}{|x|^2}\right)dx.
	\end{align}
	Now, observe that by \eqref{harmonic}
	\begin{align}\label{newipp2}
	&\int_{\R^2\setminus \bar{B}_{\delta}(0)}u\,\left(\frac{x}{|x|^2}\right)^t\cdot\D^2u\cdot\left(\frac{x}{|x|^2}\right)dx\nonumber\\
	&=\int_{\R^2\setminus\bar{B}_{\delta}(0)}\frac{x_1}{|x|^2}u\left(\frac{x_1}{|x|^2}\p{x_1}\left(\p{x_1}u\right)+\frac{x}{|x|^2}\p{x_2}\left(\p{x_1}u\right)\right)+\frac{x_2}{|x|^2}u\left(\frac{x_1}{|x|^2}\p{x_1}\left(\p{x_2}u\right)+\frac{x_2}{|x|^2}\p{x_2}\left(\p{x_1}u\right)\right)dx\nonumber\\
	&=\int_{\R^2\setminus\bar{B}_{\delta}(0)}\left(\frac{x_1}{|x|^2}u\frac{x}{|x|^2}\cdot \D(\p{x_1}u)+\frac{x_2}{|x|^2}u\frac{x}{|x|^2}\cdot \D\left(\p{x_2}u\right)\right)dx\nonumber\\
	&=-\int_{\R^2\setminus \bar{B}_{\delta}(0)}\left(\frac{x_1}{|x|^2}\left(\frac{x}{|x|^2}\cdot \D u\right)\p{x_1}u+\frac{x_2}{|x|^2}\left(\frac{x}{|x|^2}\cdot \D u\right)\p{x_2}u\right)dx+\int_{\R^2\setminus\bar{B}_{\delta}(0)}u\,\frac{x}{|x|^4}\cdot \D u\,dx\nonumber\\
	&-\int_{\partial B_{\delta}(0)}\left(\frac{x_1}{|x|^3}u\,\partial_{\nu}(\p{x_1}u)+\frac{x_2}{|x|^3}u\,\partial_{\nu}(\p{x_2}u)\right)d\mathscr{H}^1\nonumber\\
	&=-\int_{\R^2\setminus\bar{B}_{\delta}(0)}\left(\frac{\D u\cdot x}{|x|^2}\right)^2dx+\int_{\R^2\setminus\bar{B}_{\delta}(0)}u\,\frac{x}{|x|^4}\cdot \D u\,dx\leq 0
	\end{align}
	where we used
	\begin{align*}
	\D\left(\frac{x_1}{|x|^2}\right)\cdot \frac{x}{|x|^2}=\frac{1}{|x|^2}\left(\frac{x_1}{|x|^2}-\frac{2x_1^3}{|x|^4}-\frac{2x_1x_2^2}{|x|^4}\right)=-\frac{x_1}{|x|^4},
	\end{align*}
	The last inequality come from the following observations (see the computations before \eqref{id1} for an alternative derivation)
	\begin{align*}
	\int_{\R^2\setminus\bar{B}_{\delta}(0)}u\frac{x}{|x|^4}\cdot \D udx&=\int_{\R^2\setminus \bar{B}_{\delta}(0)}\frac{u}{|x|^2}\dive\left(\frac{x}{|x|^2}u\right)dx=-\int_{\R^2\setminus \bar{B}_{\delta}(0)}\left(\left(\frac{x}{|x|^4}\cdot \D u\right)\,u+\D\left(\frac{1}{|x|^2}\right)\cdot \frac{x}{|x|^2}u^2\right)dx\\
	&=-\int_{\R^2\setminus \bar{B}_{\delta}(0)}u\frac{x}{|x|^4}\cdot \D u\,dx+2\int_{\R^2\setminus \bar{B}_{\delta}(0)}\frac{u^2}{|x|^4}dx,
	\end{align*}
	so that 
	\begin{align}\label{identity}
	\int_{\R^2\setminus\bar{B}_{\delta}(0)}\frac{u^2}{|x|^4}dx=\int_{\R^2\setminus\bar{B}_{\delta}(0)}u\frac{x}{|x|^4}\cdot \D u\,dx,
	\end{align}
	Therefore, thanks to \eqref{identity}, we rewrite \eqref{newipp2} as  
	\begin{align}\label{newipp4}
	\int_{\R^2\setminus \bar{B}_{\delta}(0)}\left(\frac{\D u\cdot x}{|x|^2}\right)^2dx-\int_{\R^2\setminus\bar{B}_{\delta}(0)}u\frac{x}{|x|^4}\cdot \D udx&=\int_{\R^2\setminus \bar{B}_{\delta}(0)}\left(\left(\frac{\D u \cdot x}{|x|^2}\right)^2-2\left(\frac{\D u\cdot x }{|x|^2}\right)\frac{u}{|x|^2}+\frac{u^2}{|x|^4}\right)dx\nonumber\\
	&=\int_{\R^2\setminus \bar{B}_{\delta}(0)}\left(\frac{x}{|x|^2}\cdot \D u-\frac{u}{|x|^2}\right)^2dx.
	\end{align}
	Finally, we deduce by \eqref{newipp}, \eqref{newipp2} and \eqref{newipp4} that 
	\begin{align}\label{parfait}
	&\int_{\R^2\setminus \bar{B}_{\delta}(0)}\left(\mathscr{L}_m u\right)^2dx=\int_{\R^2\setminus \bar{B}_{\delta}(0)}\left(\Delta u-2(m+1)\frac{x}{|x|^2}\cdot \D u+\frac{(m+1)^2}{|x|^2}u\right)^2dx\\
	&=\int_{\R^2\setminus \bar{B}_{\delta}(0)}\left(\Delta u+(m+1)(m-1)\frac{u}{|x|^2}\right)^2dx
	+4(m+1)(m-1)\int_{\R^2\setminus \bar{B}_{\delta}(0)}\left(\frac{x}{|x|^2}\cdot \D u-\frac{u}{|x|^2}\right)^2dx.\nonumber
	\end{align}
	\textbf{Step 2: Inequalities.}
	Now we have thanks to \eqref{identity}
	\begin{align}\label{id3}
		&\int_{\R^2\setminus\bar{B}_{\delta}(0)}\left(-2(m+1)\frac{x}{|x|^2}\cdot \D u+\frac{(m+1)^2}{|x|^2}u\right)^2dx=\int_{\R^2\setminus\bar{B}_{\delta}(0)}4(m+1)^2\left(\frac{\D u\cdot x}{|x|^2}\right)^2dx\nonumber\\
		&+\int_{\R^2\setminus\bar{B}_{\delta}(0)}(m+1)^4\frac{u^2}{|x|^4}dx
		-4(m+1)^3\int_{\R^2\setminus\bar{B}_{\delta}(0)}\frac{u\left(\D u\cdot x\right)}{|x|^4}dx\nonumber\\
		&=4(m+1)^2\int_{\R^2\setminus\bar{B}_{\delta}(0)}\frac{(\D u\cdot x)^2}{|x|^4}dx+((m+1)^4-4(m+1)^3)\int_{\R^2\setminus\bar{B}_{\delta}(0)}\frac{u^2}{|x|^4}dx\\
		&\geq \left((m+1)^4+4(m+1)^2-4(m+1)^3\right)\int_{\R^2\setminus\bar{B}_{\delta}(0)}\frac{u^2}{|x|^4}dx=(m+1)^2(m-1)^2\int_{\R^2\setminus\bar{B}_{\delta}(0)}\frac{u^2}{|x|^4}dx\nonumber
	\end{align}
	so for $m>1$, we find
	\begin{align*}
		\np{\frac{u}{|x|^2}}{2}{\R^2\setminus \bar{B}_{\delta}(0)}\leq \frac{1}{(m+1)(m-1)}\np{\Delta u}{2}{\R^2\setminus \bar{B}_{\delta}(0)}+\frac{1}{(m+1)(m-1)}\np{\mathscr{L}_mu}{2}{\R^2\setminus \bar{B}_{\delta}(0)}.
	\end{align*}
	Therefore, if $m\geq 3$, we have $(m+1)^4-4(m+1)^3=(m+1)^3(m-3)\geq 0$, so \eqref{id3} implies that
	\begin{align*}
		\int_{\R^2\setminus \bar{B}_{\delta}(0)}\frac{(\D u\cdot x)^2}{|x|^4}dx&\leq \frac{1}{4(m+1)^2}\int_{\R^2\setminus \bar{B}_{\delta}(0)}\left(-2(m+1)\frac{x}{|x|^2}\cdot \D u+\frac{(m+1)^2}{|x|^2}u\right)^2dx\\
		&=\frac{1}{4(m+1)^2}\int_{\R^2}\left(\left(\mathscr{L}_m-\Delta\right)u\right)^2dx
	\end{align*}
	which implies by the triangle inequality that
	\begin{align*}
		\np{\frac{\D u\cdot x}{|x|^2}}{2}{\R^2\setminus \bar{B}_{\delta}(0)}\leq \frac{1}{2(m+1)}\np{\Delta u}{2}{\R^2\setminus \bar{B}_{\delta}(0)}+\frac{1}{2(m+1)}\np{\mathscr{L}_mu}{2}{\R^2\setminus \bar{B}_{\delta}(0)}.
	\end{align*}
	If $1<m\leq 3$, then $(m+1)^4-4(m+1)^3\leq 0$, so we have by \eqref{id1} and \eqref{id1}
	\begin{align*}
		\int_{\R^2\setminus \bar{B}_{\delta}(0)}\left(\left(\mathscr{L}_2-\Delta\right)u\right)^2dx&=4(m+1)^2\int_{\R^2\setminus \bar{B}_{\delta}(0)}\frac{(\D u\cdot x)^2}{|x|^4}dx+((m+1)^4-4(m+1)^3)\int_{\R^2\setminus \bar{B}_{\delta}(0)}\frac{u^2}{|x|^4}dx\\
		&\geq ((m+1)^2+(m+1)^4-4(m+1)^3)\int_{\R^2\setminus \bar{B}_{\delta}(0)}\frac{(\D u\cdot x)^2}{|x|^4}dx\\
		&=(m+1)^2(m-1)^2\int_{\R^2\setminus\bar{B}_{\delta}(0)}\frac{(\D u\cdot x)^2}{|x|^4}dx,
	\end{align*}
	so that
	\begin{align*}
		\np{\frac{\D u\cdot x}{|x|^2}}{2}{\R^2}\leq \frac{1}{(m+1)(m-1)}\np{\Delta u}{2}{\R^2}+\frac{1}{(m+1)(m-1)}\np{\mathscr{L}_mu}{2}{\R^2}.
	\end{align*}
    This completes the proof of the theorem. 
\end{proof} 

\subsection{Admissible variations for branched Willmore surfaces}\label{counter}

Recall (\cite{index3}) that for a branched Willmore surface $\vec{\Psi}:\Sigma\rightarrow \R^3$ we defined the index as the maximal dimension of the space of normal variations $\vec{v}=v\n$ where $v\in W^{2,2}\cap W^{1,\infty}(\Sigma)$ satisfying the additional conditions
\begin{align}\label{cond}
|d\vec{v}|_g\in L^{\infty}(\Sigma),\quad \Delta_g^{\perp}\vec{v}\in L^2(\Sigma,d\vg).
\end{align}
such that $D^2W(\phi)(\vec{v},\vec{v})<0$. First, the condition $|d\vec{v}|_g\in L^{\infty}(\Sigma)$ is really necessary to define the weakest notion of index (for continuous paths, see Lemma $3.11$ \cite{index3}), and we want to show here that the second condition $\Delta_g\vec{v}\in L^2(\Sigma,d\vg)$ cannot be relaxed in general. If $p\in \Sigma$ is a branch point of $\vec{\Psi}$ of multiplicity $\theta_0\geq 2$, taking a complex chart $z:U\subset \Sigma\rightarrow \C$ such that $z(p)=0$, there exists $\alpha>0$ such that
\begin{align*}
e^{2\lambda}=2|\p{z}\phi|^2=\frac{\alpha}{|z|^{2\theta_0-2}}\left(1+O(|z|)\right),
\end{align*}
and $\Delta_g\vec{v}\in L^2(\Sigma,d\vg)$ is equivalent to
\begin{align}\label{l2}
\frac{\Delta v}{|z|^{\theta_0-1}}\in L^2(D^2).
\end{align}
Notice that this implies if $v$ is smooth that $v$ must have the following Taylor expansion (for some $\zeta\in \C$)
\begin{align}\label{optimal}
v=v(p)+\,\Re\left(\zeta z^{\theta_0}\right)+O(|z|^{\theta_0+1})
\end{align}
If $L^2$ is replaced by $L^{\infty}$ in \eqref{l2} we obtain the same expansion (see \cite{classification}), with a $O(|z|^{\theta_0+1}\log^2|z|)$ error term instead. We will check that restriction to the case of smooth variations for the simplest example of plane with multiplicity $m\geq 1$, the expansion \eqref{optimal} is the largest space for which the second derivative makes sense. Indeed, thanks to the pointwise conformal invariance of the Willmore energy, and recalling that for all branched immersions $\vec{\Psi}:\Sigma\rightarrow\R^3$ which is the inversion of a complete minimal surface $\phi:\Sigma\setminus\ens{p_1,\cdots,p_n}\rightarrow\R^3$ (we state the result in codimension $1$,  but they would hold in general thanks to the expression of the second derivative of the Gauss curvature obtained in \cite{index3}) and 
\begin{align*}
\mathscr{W}(\vec{\Psi})=\int_{\Sigma}\left(|\H_{g_{\vec{\Psi}}}|^2-K_{g_{\vec{\Psi}}}\right)d\mathrm{vol}_{g_{\vec{\Psi}}},
\end{align*}
then for all admissible normal variation $\vec{v}=v\n_{\vec{\Psi}}$ such that $D^2\mathscr{W}(\vec{\Psi})(\vec{v},\vec{v})$ is well-defined, then
\begin{align*}
D^2\mathscr{W}(\vec{\Psi})(\vec{v},\vec{v})=\int_{\Sigma}\bigg\{\frac{1}{2}\left(\Delta_gu-2K_gu\right)^2d\vg-d\,\Im\left(2\left(\Delta_gu+2K_gu\right)\partial u-\partial|du|_g^2\right)\bigg\},
\end{align*}
where $g=g_{\phi}=\phi^{\ast}g_{\R^3}$ and $u=|\phi|^2v$. In particular, if $\Sigma_{\epsilon}$ is defined as previously (with respect to some arbitrary covering $(U_1,\cdots,U_n)$ as the ends $p_1,\cdots,p_n$), then by Stokes theorem
\begin{align}\label{limitvar}
D^2\mathscr{W}(\vec{\Psi})(\vec{v},\vec{v})=\lim\limits_{\epsilon\rightarrow 0}\bigg\{\frac{1}{2}\int_{\Sigma_{\epsilon}}\left(\Delta_gu-2K_gu\right)^2d\vg-\Im\int_{\partial \Sigma_{\epsilon}}2\left(\Delta_gu+2K_gu\right)\partial u-\partial|du|_g^2\bigg\}
\end{align}
so the limit on the right-hand side of \eqref{limitvar} exists and is finite for all admissible variation $\vec{v}$ as previously.

\begin{lemme}
	Let $\phi:\C\setminus\ens{0}\rightarrow \R^3$ be a plane with multiplicity $2$ such that for some $\vec{A}_0\in \C^3\setminus\ens{0}$ and $\vec{B}_0\in \R^3\setminus\ens{0}$ such that $\s{\vec{A}_0}{\vec{B}_0}=0$ we have
	\begin{align*}
	\phi(z)=2\,\Re\left(\frac{\vec{A}_0}{z^2}\right)+\vec{B}_0
	\end{align*}
	and let $\vec{\Psi}:S^2\rightarrow \R^3$ be a round sphere with multiplicity $2$ which is the inversion at $0$ of $\phi$. Then the normal variations $\vec{v}=v\n_{\vec{\Psi}}$ where $v\in W^{2,2}(S^2)\cap C^{\infty}(S^2)$ is defined for some $\beta\in \R\setminus\ens{0}$ and some compactly supported \emph{radial} smooth cut-off $\rho:\C\rightarrow \R$ such that $\rho=1$ in a neighbourhood of $0$ by
	\begin{align*}
	v=v(0)+\beta|z|^2\rho
	\end{align*}
	is not admissible for $\vec{\Psi}$.
\end{lemme}
\begin{proof}
	We first have $\s{\vec{A}_0}{\vec{A}_0}=0$ as $\phi$ is conformal ($\s{\p{z}\phi}{\p{z}\phi}=0$), so we have as $\s{\vec{A}_0}{\vec{B}_0}=0$ the identity
	\begin{align*}
	|\phi|^2=\frac{2|\vec{A}_0|^2}{|z|^4}+|\vec{B}_0|^2.
	\end{align*}
	so we normalise for convenience $2|\vec{A}_0|^2=1$ and we let $\gamma=|\vec{B}_0|^2>0$. Furthermore, notice that $\s{\phi(z)}{\vec{B}_0}=|\vec{B}_0|^2$ so that
	\begin{align*}
	\left|\vec{\Psi}(z)-\frac{\vec{B}_0}{2|\vec{B}_0|^2}\right|^2=\left|\frac{\phi(z)}{|\phi(z)|^2}-\frac{\vec{B}_0}{2|\vec{B}_0|^2}\right|=\frac{1}{|\phi(z)|^2}-\frac{1}{|\vec{B}_0|^2}\frac{\s{\phi(z)}{\vec{B}_0}}{|\phi(z)|^2}+\frac{1}{4|\vec{B}_0|^2}=\frac{1}{4|\vec{B}_0|^
		2}
	\end{align*}
	so $\vec{\Psi}$ is a sphere of multiplicity $2$ of centre $\dfrac{\vec{B}_0}{2|\vec{B}_0|^2}$ and radius $\dfrac{1}{2|\vec{B}_0|}$.
	
	As $\phi$ is harmonic, we find
	\begin{align*}
	e^{2\lambda}=\p{z\z}^2|\phi|^2=\frac{4}{|z|^6}.
	\end{align*}
	This also implies by the Liouville equation
	\begin{align*}
	K_g=-\Delta_g\lambda=0.
	\end{align*}
	Now, let $\beta\in \R\setminus\ens{0}$ fixed and $v(z)=v(0)+\beta|z|^2$. Then $v\in C^{\infty}(\C)\cap W^{2,2}(\C,d\mathrm{vol}_{g_0})$, where 
	\begin{align*}
	g_0=\frac{4}{\left(1+|z|^2\right)^2}|dz|^2
	\end{align*}
	is the metric of the sphere $S^2$ after stereographic projection. Notice that $v$ satisfies the first condition of \eqref{cond} but not the second one as (assuming without loss of generality that $\rho=1$ on $D^2$)
	\begin{align*}
	\frac{\Delta v}{|z|^{\theta_0-1}}=\frac{\Delta v}{|z|}=\frac{4\beta}{|z|}\notin L^2(D^2).
	\end{align*}
	Now assume by contradiction that $v$ is an admissible variation of the inversion $\vec{\Psi}$ of $\phi$. Then we have (by \cite{indexS3})
	\begin{align*}
	D^2W(\vec{\Psi})(v\vec{n}_{\vec{\Psi}},v\n_{\vec{\Psi}})&=\lim\limits_{\epsilon\rightarrow 0}\bigg\{\frac{1}{2}\int_{\C\setminus\bar{B}(0,\epsilon)}\left(\Delta_g\left(|\phi|^2v\right)-2K_g(|\phi|^2v)\right)^2d\vg\\
	&+\Im\int_{\partial B(0,\epsilon)}2\left(\Delta_g\left(|\phi|^2v\right)+2K_g(|\phi|^2v)\right)\partial\left(|\phi|^2v\right)-\partial\left|d\left(|\phi|^2v\right)\right|_g^2\bigg\}
	\end{align*}
	We first have as $v=v(0)+\beta|z|^2\rho$
	\begin{align*}
	|\phi|^2v=\frac{v}{|z|^4}=\frac{v(0)}{|z|^4}+\frac{\beta}{|z|^2}\rho(z)+\gamma v(0)+\beta\gamma|z|^2\rho(z)
	\end{align*}
	so
	\begin{align*}
	\p{z\z}^2\left(|\phi|^2v\right)=\frac{4v(0)}{|z|^6}+\frac{\beta}{|z|^4}\rho(z)+\beta\gamma \rho(z)-\frac{2\beta}{|z|^2}\Re\left(\frac{\p{\z}\rho(z)}{z}\right)+\frac{\beta}{|z|^2}\p{z\z}^2\rho(z)+2\beta\gamma\,\Re\left(z\p{\z}\rho(z)\right)+\beta\gamma|z|^2\p{z\z}^2\rho(z).
	\end{align*}
	Therefore, we have
	\begin{align*}
	\Delta_g(|\phi|^2v)&=\frac{|z|^6}{4}\cdot 4 \p{z\z}^2\left(|\phi|^2v\right)=4v(0)+\beta|z|^2\rho(z)+\beta\gamma|z|^6\rho(z)-2\beta|z|^4\,\Re\left(\z\cdot\p{\z}\rho(z)\right)+\frac{\beta}{4}|z|^4\Delta\rho(z)\\
	&+2\beta\gamma|z|^6\,\Re\left(z\cdot\p{\z}\rho(z)\right)+\frac{\beta\gamma}{4}|z|^8\Delta\rho(z).
	\end{align*}
	Furthermore, as $\rho=1$ on $D^2$ on $\rho$ has compact support on $\C$, all square of terms involving derivatives of $\rho$ once integrating on $\C$ with respect to the metric $g=e^{2\lambda}|dz|^2$ are finite. 
	In other words, we have
	\begin{align*}
	&\frac{1}{2}\int_{\C\setminus\bar{B}(0,\epsilon)}\left(\Delta_g\left(|\phi|^2v\right)-2K_g|\phi|^2v\right)^2d\vg=\frac{1}{2}\int_{\C\setminus\bar{B}(0,\epsilon)}\left(4v(0)+\beta|z|^2\rho+\beta\gamma|z|^6\rho(z)\right)^2\frac{4|dz|^2}{|z|^6}+O
	(1)\\
	&=\frac{1}{2}\int_{\C\setminus\bar{B}(0,\epsilon)}\left(4v(0)+\beta|z|^2\rho(z)\right)^2\frac{4|dz|^2}{|z|^6}+O(1)\\
	&=2\int_{\C\setminus\bar{B}(0,\epsilon)}\left(\frac{16v^2(0)}{|z|^6}+\frac{8\beta v(0)}{|z|^4}+\frac{\beta^2}{|z|^2}\rho^2(z)\right)|dz|^2+O(1)\\
	&=4\pi\int_{\epsilon}^{\infty}\left(\frac{16v^2(0)}{r^5}+\frac{4\beta v(0)}{r^3}+\frac{\beta^2}{r}\rho^2(r)\right)dr+O(1)\\
	&=4\pi\left(\frac{4v^2(0)}{\epsilon^4}+\frac{4\beta v(0)}{\epsilon^2}\rho(z)+\beta^2\log\left(\frac{1}{\epsilon}\right)\right)+O(1)\\
	&=\frac{16\pi}{\epsilon^4}v^2(0)+\frac{16\pi\beta }{\epsilon^2}v(0)+4\pi \beta^2\log\left(\frac{1}{\epsilon}\right)+O(1),
	\end{align*}
	where $O(1)$ is a bounded quantity as $\epsilon\rightarrow 0$.
	Now we have as $\rho=1$ the following identities on $D^2$
	\begin{align}\label{dplan}
	&|\phi|^2v=\frac{v(0)}{|z|^4}+\frac{\beta}{|z|^2}+\gamma v(0)+\beta\gamma|z|^2\nonumber\\
	&\Delta_g\left(|\phi|^2v\right)=4v(0)+\beta|z|^2+\beta\gamma|z|^6\nonumber\\
	&\partial\left(|\phi|^2v\right)=-\frac{1}{z|z|^4}\left(2v(0)+\beta|z|^2-\beta\gamma|z|^6\right)dz
	\end{align}
	so
	\begin{align*}
	\left(\Delta_g\left(|\phi|^2v\right)+2K_g|\phi|^2v\right)\partial\left(|\phi|^2v\right)&=\Delta_g\left(|\phi|^2v\right)\partial\left(|\phi|^2v\right)\\
	&=-\frac{1}{z|z|^4}\left(8v^2(0)+6\beta v(0)|z|^2+\beta^2|z|^4+O(|z|^6)\right)dz
	\end{align*}
	and
	\begin{align}\label{plan1}
	\Im\int_{\partial B(0,\epsilon)}2\left(\Delta_g\left(|\phi|^2v\right)+2K_g|\phi|^2v\right)\partial\left(|\phi|^2v\right)=-\frac{32\pi }{\epsilon^4}v^2(0)-\frac{24\pi\beta }{\epsilon^2}v(0)-4\pi\beta^2+O(\epsilon).
	\end{align}
	Thanks to \eqref{dplan}, we find
	\begin{align*}
	\left|\partial\left(|\phi|^2v\right)\right|^2=\frac{1}{|z|^{10}}\left(4v^2(0)+4\beta v(0)|z|^2+\beta^2|z|^4+O(|z|^6)\right)|dz|^2
	\end{align*}
	so
	\begin{align*}
	\left|d\left(|\phi|^2v\right)\right|^2=4e^{-2\lambda}|\partial\left(|\phi|^2v\right)|^2=\frac{1}{|z|^4}\left(4v^2(0)+4\beta v(0)|z|^2+O(|z|^6)\right)+\beta^2.
	\end{align*}
	Therefore, we have
	\begin{align*}
	-\partial\left|d\left(|\phi|^2v\right)\right|_g^2=\frac{1}{z|z|^4}\left(8v^2(0)+4\beta v(0)|z|^2+O(|z|^6)\right)dz
	\end{align*}
	and
	\begin{align}\label{plan2}
	\int_{\partial B(0,\epsilon)}-\partial\left|d\left(|\phi|^2v\right)\right|_g^2=\frac{16\pi}{\epsilon^2}v^2(0)+\frac{8\pi \beta}{\epsilon^2}v(0)+O(\epsilon).
	\end{align}
	Finally, we have by 
	\begin{align*}
	\Im\int_{\partial B(0,\epsilon)}2\left(\Delta_g\left(|\phi|^2v\right)+2K_g(|\phi|^2v)\right)\partial\left(|\phi|^2v\right)-\partial\left|d\left(|\phi|^2v\right)\right|_g^2=-\frac{16\pi}{\epsilon^4}v^2(0)-\frac{16\pi\beta}{\epsilon^2}v(0)-4\pi\beta^2+O(\epsilon).
	\end{align*}
	Therefore, we obtain
	\begin{align*}
	&\frac{1}{2}\int_{\C\setminus\bar{B}(0,\epsilon)}\left(\Delta_g\left(|\phi|^2v\right)-2K_g(|\phi|^2v)\right)^2d\vg\\
	&+\Im\int_{\partial B(0,\epsilon)}2\left(\Delta_g\left(|\phi|^2v\right)+2K_g(|\phi|^2v)\right)\partial\left(|\phi|^2v\right)-\partial\left|d\left(|\phi|^2v\right)\right|_g^2=4\pi\beta^2\log\left(\frac{1}{\epsilon}\right)+O(1)\conv{\epsilon\rightarrow 0}\infty,
	\end{align*}
	which shows that $v$ is not an admissible variation of $\vec{\Psi}$. 
\end{proof}

\subsubsection{Admissible smooth variation of the plane of multiplicity $2m$}

In general, if $m\geq 1$ is any fixed integer, $\vec{A}_0\in \C^3\setminus\ens{0}$ is such that $2|\vec{A}_0|^2=1$ and $\phi:\C\setminus\ens{0}\rightarrow\R^3$ is defined by 
\begin{align*}
\phi(z)=2\,\Re\left(\frac{\vec{A}_0}{z^{2m}}\right),
\end{align*}
then one checks easily that for all $\beta\in \R\setminus\ens{0}$, the normal variation $\vec{v}=v\n$, where $v=v(0)+\beta|z|^{2m}\rho$ is not admissible (for some smooth radial cut-off function $\rho$ identically equal to $1$ in an open neighbourhood of $0\in \C$), as 
\begin{align*}
&\frac{1}{2}\int_{\C\setminus\bar{B}(0,\epsilon)}\left(\Delta_g\left(|\phi|^2v\right)-2K_g(|\phi|^2v)\right)^2d\vg\\
&+\Im\int_{\partial B(0,\epsilon)}2\left(\Delta_g\left(|\phi|^2v\right)+2K_g(|\phi|^2v)\right)\partial\left(|\phi|^2v\right)-\partial\left|d\left(|\phi|^2v\right)\right|_g^2=4\pi m^2\beta^2\log\left(\frac{1}{\epsilon}\right)+O(1)\conv{\epsilon\rightarrow 0}\infty.
\end{align*}

\subsubsection{Admissible smooth variations of the plane with multiplicity $m$}

\begin{prop}
	Let $\vec{\Psi}:S^2=\C\cup\ens{\infty}\rightarrow \R^3$ be a round sphere with multiplicity $m\geq 1$, let $\phi:\C\setminus\ens{0}\rightarrow \R^3$ be a plane with multiplicity $m\geq 1$ such that $\vec{\Psi}$ be the inversion at $0$ of $\phi$. Furthermore, fix $\beta\in \R\setminus\ens{0}$, $a,b\in \ens{1,\cdots,m-1}$ such that $a+b=m$, and let $\rho:\C\cup\ens{\infty}\rightarrow \R$ be a smooth radial cut-off function such that $\rho=1$ in a neighbourhood of $0$ and let $v\in W^{2,2}\cap C^{\infty}(S^2)$ such that in the usual meromorphic coordinate $z$ on $S^2=\C\cup\ens{\infty}$ we have
	\begin{align*}
	v=v(0)+2\beta\,\Re\left(z^{a}\z^{b}\right)\rho(z).
	\end{align*}
	Then $\vec{v}=v\n_{\vec{\Psi}}$ is not an admissible variation of $\vec{\Psi}$.
\end{prop}
\begin{proof}
	Thanks to the previous subsection, we can assume that $m\geq 3$, $a\neq b$, and as $\rho$ is radial, we have for all integer $k\geq 0$ and $l\in \Z$ and  for all $\epsilon>0$ the identity
	\begin{align}\label{neglect}
	\int_{\C\setminus\bar{B}(0,\epsilon)}\Re\left(z^{a}\z^{b}\right)\rho^k(z)|z|^{l}|dz|^2=0.
	\end{align}
	As previously, assume that $\vec{A}_0\in \C^3\setminus\ens{0}$ and $\vec{B}_0\in \R^3\setminus\ens{0}$ are such that $2|\vec{A}_0|^2=1$,  $\s{\vec{A}_0}{\vec{B}_0}=0$, and 
	\begin{align*}
	|\phi(z)|^2=\frac{2|\vec{A}_0|^2}{|z|^{2m}}+|\vec{B}_0|^2=\frac{1}{|z|^{2m}}+\gamma
	\end{align*}
	as $\s{\vec{A}_0}{\vec{A}_0}=0$, where we noted $\gamma=|\vec{B}_0|^2>0$. Then we have 
	\begin{align*}
	e^{2\lambda}=\frac{m^2}{|z|^{2m+2}}.
	\end{align*}
	Now we compute directly
	\begin{align*}
	|\phi|^2v=\frac{v(0)}{|z|^{2m}}+2\beta\,\Re\left(z^{a-m}\z^{b-m}\right)\rho(z)+\gamma v(0)+2\beta\gamma\, \Re\left(z^{a}\z^{b}\right)\rho(z).
	\end{align*}
	As $\rho$ is radial and smooth, there exists a smooth function $f:\R_+\rightarrow\R_+$ such that $\rho(z)=f(|z|^2)$. Therefore, we have
	\begin{align*}
	&\p{z}\rho(z)=\z\,f'(|z|^2),\quad \Re(z\cdot \p{z}\rho(z))=z\cdot \p{z}\rho(z)=|z|^2f'(|z|^2)\\
	&\Delta\rho(z)=4\,f'(|z|^2)+4|z|^2\,f''(|z|^2)
	\end{align*} 
	so $z\cdot \p{z}\rho$ and $\Delta\rho$ are also radial.
	Therefore, we have
	\begin{align*}
	&\p{z\z}^2\left(|\phi|^2v\right)=\frac{m^2v(0)}{|z|^{2m+2}}+2\beta(m-a)(m-b)\,\Re\left(z^{a-m-1}\z^{b-m-1}\right)\rho(z)+2\beta\gamma\,\Re\left(z^{a-1}\z^{b-1}\right)\rho(z)\\
	&+\beta\left((a-m)z^{a-m-1}\z^{b-m}+(b-m)z^{b-m-1}\z^{a-m}\right)\p{\z}\rho(z)\\
	&+\beta\left((b-m)z^{a-m}\z^{b-m-1}+(a-m)z^{b-m}\z^{a-m-1}\right)\p{z}\rho(z)+2\beta\,\Re\left(z^{a-m}\z^{b-m}\right)\p{z\z}^2\rho(z)\\
	&+\beta\gamma\left(az^{a-1}\z^{b}+bz^{b-1}\z^a\right)\p{\z}\rho(z)+\beta\gamma\left(bz^{a}\z^{b-1}+az^{b}\z^{a-1}\right)\p{z}\rho(z)+2\beta\gamma\Re\left(z^{a}\z^{b}\right)\p{z\z}^2\rho(z)\\
	&=\frac{m^2v(0)}{|z|^{2m+2}}+2\beta(m-a)(m-b)\,\Re\left(z^{b-m-1}\z^{b-m-1}\right)\rho(z)+2\beta\gamma\,\Re\left(z^{a-1}\z^{b-1}\right)\rho(z)\\
	&+2\beta(a+b-2m)\,\Re\left(z^{a-m}\z^{b-m}\right)f'(|z|^2)+\frac{\beta}{2}\,\Re\left(z^{a-m}\z^{b-m}\right)\Delta\rho(z)\\
	&+2\beta\gamma(a+b)\,\Re\left(z^{a}\z^{b}\right)f'(|z|^2)+\frac{\beta\gamma}{2}\,\Re\left(z^{a}\z^{b}\right)\Delta\rho(z).
	\end{align*}
	As $a+b=m$ we obtain
	\begin{align*}
	&\Delta_g\left(|\phi|^2v\right)=\frac{4}{m^2}|z|^{2m+2}\p{z\z}^2\left(|\phi|^2v\right)=4v(0)+8\beta\bigg(1-\frac{a}{m}\bigg)\bigg(1-\frac{b}{m}\bigg)\,\Re\left(z^{a}\z^{b}\right)\rho(z)\\
	&+\frac{8}{m^2}\beta\gamma\,\Re\left(z^{a+m}\z^{b+m}\right)\rho(z)
	-\frac{8}{m}\,\Re\left(z^{a+1}\z^{b+1}\right)f'(|z|^2)+\frac{2\beta}{m^2}\,\Re\left(z^{a+1}\z^{b+1}\right)\Delta\rho(z)\\
	&+\frac{8\beta\gamma}{m}\,\Re\left(z^{a+m+1}\z^{b+m+1}\right)f'(|z|^2)+\frac{2\beta\gamma}{m^2}\,\Re\left(z^{a+m+1}\z^{b+m+1}\right)\Delta\rho(z).
	\end{align*}
	As mentioned previously in the special case of ends of multiplicity $2$ we can neglect as a constant term bounded independently of $\epsilon\rightarrow 0$ all components of $\left(\Delta_g(|\phi|^2v)\right)^2$ involving derivatives of $\rho$. Furthermore, we can also neglect the term involving $\,\Re\left(z^{a+m}\z^{b+m}\right)\rho(z)$ as $a+b\geq m$ so this term is integrable at $0$ in $L^1$ and $L^2$ with respect to the singular metric $g=e^{2\lambda}|dz|^2$ and of compact support.
	Finally, we find by \eqref{neglect} and as $2\,\Re\left(z^{a}\z^{b}\right)^2=\Re\left(z^{2a}\z^{2b}\right)+|z|^{2a+2b}=\Re(z^{2a}\z^{2b})+|z|^{2m}$
	\begin{align*}
	&\frac{1}{2}\int_{\C\setminus\bar{B}(0,\epsilon)}\left(\Delta_g\left(|\phi|^2v\right)-2K_g\left(|\phi|^2v\right)\right)^2d\vg\\
	&=\frac{1}{2}\int_{\C\setminus\bar{B}(0,\epsilon)}\left(4v(0)+8\beta\bigg(1-\frac{a}{m}\bigg)\bigg(1-\frac{b}{m}\bigg)\,\Re\left(z^{a}\z^{b}\right)\rho(z)\right)^2\frac{m^2}{|z|^{2m+2}}|dz|^2+O(1)\\
	&=\frac{1}{2}\int_{\C\setminus\bar{B}(0,\epsilon)}\left(16v^2(0)+32\beta^2\bigg(1-\frac{a}{m}\bigg)^2\bigg(1-\frac{b}{m}\bigg)^2|z|^{2m}\rho^2(z)\right)\frac{m^2}{|z|^{2m+2}}|dz|^2+O(1)\\
	&=\pi m^2\int_{\epsilon}^{\infty}\left(\frac{16v^2(0)}{r^{2m+1}}+32\beta^2\bigg(1-\frac{a}{m}\bigg)^2\bigg(1-\frac{b}{m}\bigg)^2\frac{\rho^2(r)}{r}\right)dr+O(1)\\
	&=\frac{8\pi m}{\epsilon^{2m}}v^2(0)+32\pi m^2\bigg(1-\frac{a}{m}\bigg)^2\bigg(1-\frac{b}{m}\bigg)^2\beta^2\log\left(\frac{1}{\epsilon}\right)+O(1).
	\end{align*}
	Now, as without loss of generality we can assume $\rho=1$ in $D^2$ we have the following identities on $D^2$
	\begin{align*}
	&|\phi|^2v=\frac{v(0)}{|z|^{2m}}+2\beta\,\Re\left(z^{a-m}\z^{b-m}\right)+\gamma v(0)+O(|z|)\\
	&\partial\left(|\phi|^2v\right)=-\frac{mv(0)}{z|z|^{2m}}dz+\beta\left((a-m)z^{a-m-1}\z^{b-m}+(b-m)z^{m-b-1}\z^{a-m}\right)dz\\
	&=-\frac{1}{z|z|^{2m}}\left(mv(0)+\beta\left((m-a)z^{a}\z^{b}+(m-b)z^{b}\z^{a}\right)\right)dz\\
	&\Delta_g\left(|\phi|^2v\right)=4v(0)+8\beta\bigg(1-\frac{a}{m}\bigg)\bigg(1-\frac{b}{m}\bigg)\,\Re\left(z^a\z^{b}\right)+O(|z|^{2m+1})
	\end{align*}
	Therefore, there exists $\zeta_0,\zeta_1\in \C$ such that
	\begin{align*}
	\bigg(\Delta_g\left(|\phi|^2v\right)+2K_g\left(|\phi|^2v\right)\bigg)\partial\left(|\phi|^2v\right)&=\Delta_g\left(|\phi|^2v\right)\partial\left(|\phi|^2v\right)=\\
	&=-\frac{1}{z|z|^{2m}}\left(4mv^2(0)+\zeta_0z^{a}\z^{b}+\zeta_1z^{b}\z^{a}+O\left(|z|^{2m+1}\right)\right)
	\end{align*}
	so that (as $a\neq b$)
	\begin{align*}
	\Im\int_{\partial B(0,\epsilon)}\bigg(\Delta_g\left(|\phi|^2v\right)+2K_g\left(|\phi|^2v\right)\bigg)\partial\left(|\phi|^2v\right)=-\frac{8\pi m}{\epsilon^{2m}}v^2(0)+O(\epsilon).
	\end{align*}
	Likewise, we have for some $\delta\in \R$
	\begin{align*}
	\left|\left(\partial|\phi|^2v\right)\right|^2=\frac{1}{|z|^{4m+2}}\left(m^2v^2(0)+\delta v(0) \,\Re\left(z^{a}\z^{b}\right)+O(|z|^{2m+1})\right)
	\end{align*}
	so
	\begin{align*}
	&\left|d\left(|\phi|^2v\right)\right|^2_g=4|\left(\partial|\phi|^2v\right)|^2_g=\frac{1}{|z|^{2m}}\left(4v^2(0)+\frac{4\delta}{m^2}v(0)\,\Re\left(z^{a}\z^{b}\right)\right)\\
	\end{align*}
	and
	\begin{align*}
	\Im\int_{\partial B(0,\epsilon)}-\partial\left|d\left(|\phi|^2v\right)\right|_g^2=\frac{8\pi m}{\epsilon^{2m}}v^2(0)+O(\epsilon).
	\end{align*}
	Finally, we have 
	\begin{align*}
	\Im\int_{\partial B(0,\epsilon)}2\left(\Delta_g\left(|\phi|^2v\right)+2K_g(|\phi|^2v)\right)\partial\left(|\phi|^2v\right)-\partial\left|d\left(|\phi|^2v\right)\right|_g^2=-\frac{8\pi m}{\epsilon^{2m}}+O(\epsilon)
	\end{align*}
	and
	\begin{align*}
	&\frac{1}{2}\int_{\C\setminus\bar{B}(0,\epsilon)}\left(\Delta_g\left(|\phi|^2v\right)-2K_g(|\phi|^2v)\right)^2d\vg\\
	&+\Im\int_{\partial B(0,\epsilon)}2\left(\Delta_g\left(|\phi|^2v\right)+2K_g(|\phi|^2v)\right)\partial\left(|\phi|^2v\right)-\partial\left|d\left(|\phi|^2v\right)\right|_g^2\\
	&=32\pi m^2\bigg(1-\frac{a}{m}\bigg)^2\bigg(1-\frac{b}{m}\bigg)^2\beta^2\log\left(\frac{1}{\epsilon}\right)+O(1)\conv{\epsilon\rightarrow 0}\infty.
	\end{align*}
	This concludes the proof of the theorem.
\end{proof}

\begin{rem}
	Notice that these results imply that the variations considered in \cite{index3} cannot be taken more general. For smooth variations, they admit the following expansion at a branch point $p$ of multiplicity $\theta_0\geq 1$
	\begin{align*}
	\vec{w}=\vec{w}(p)+\Re\left(\vec{\gamma}z^{\theta_0}\right)+O(|z|^{\theta_0+1}).
	\end{align*}
\end{rem}

\nocite{}
\bibliographystyle{plain}
\bibliography{biblioindex}

\end{document}